\documentclass[11pt,reqno]{amsart}

\usepackage[latin1]{inputenc}
\usepackage{amsmath,amsthm,amssymb,amscd,mathrsfs,amsbsy,ulem}
\usepackage{mathtools,color,esint,bm,float}
\usepackage{booktabs}
\usepackage{nicematrix}
\usepackage{tikz}
\usepackage[perpage]{footmisc}
\usepackage{subfig,caption}
\usepackage[shortlabels]{enumitem}
\usepackage{cancel}

\allowdisplaybreaks[4]

\newcolumntype{?}{!{\vrule width 0.4pt}}

\DeclareRobustCommand{\SkipTocEntry}[5]{}

\definecolor{darkgreen}{rgb}{0.0, 0.6, 0.13}

\usepackage{hyperref}

\newtheorem{thm}{Theorem}[section]
 \newtheorem{cor}[thm]{Corollary}
 \newtheorem{lem}[thm]{Lemma}
 \newtheorem{prop}[thm]{Proposition}

 \theoremstyle{definition}
 \newtheorem{df}[thm]{Definition}
 \theoremstyle{remark}
 \newtheorem{rem}[thm]{Remark}
 
 \numberwithin{equation}{section}

\newcommand{\Ab}{\mathbb A}
\newcommand{\Bb}{\mathbb B}
\newcommand{\Cb}{\mathbb C}

\newcommand{\Eb}{\mathbb E}

\newcommand{\Gb}{\mathbb G}
\newcommand{\Hb}{\mathbb H}

\newcommand{\Kb}{\mathbb K}
\newcommand{\Lb}{\mathbb L}
\newcommand{\Mb}{\mathbb M}

\newcommand{\Pb}{\mathbb P}

\newcommand{\Rb}{\mathbb R}

\newcommand{\Tb}{\mathbb T}
\newcommand{\Ub}{\mathbb U}
\newcommand{\Vb}{\mathbb V}
\newcommand{\Wb}{\mathbb W}
\newcommand{\Xb}{\mathbb X}
\newcommand{\Yb}{\mathbb Y}
\newcommand{\Zb}{\mathbb Z}

\newcommand{\Bc}{\mathcal B}
\newcommand{\Cc}{\mathcal C}
\newcommand{\Dc}{\mathcal D}

\newcommand{\Fc}{\mathcal F}
\newcommand{\Gc}{\mathcal{G}}
\newcommand{\Hc}{\mathcal H}
\newcommand{\Ic}{\mathcal I}
\newcommand {\Jc}{\mathcal{J}}
\newcommand{\Kc}{\mathcal K}
\newcommand{\Lc}{\mathcal L}
\renewcommand{\Mc}{\mathcal M}
\newcommand{\Nc}{\mathcal N}
\newcommand{\Oc}{\mathcal O}
\newcommand{\Pc}{\mathcal P}
\newcommand{\Qc}{\mathcal Q}

\newcommand{\Sc}{\mathcal S}
\newcommand{\Tc}{\mathcal T}
\newcommand{\Uc}{\mathcal U}
\newcommand{\Vc}{\mathcal V}
\newcommand{\Wc}{\mathcal W}
\newcommand{\Xc}{\mathcal X}

\newcommand{\Zc}{\mathcal Z}

\newcommand{\Bs}{\mathscr B}
\newcommand{\Cs}{\mathscr C}
\newcommand{\Ds}{\mathscr D}
\newcommand{\Es}{\mathscr E}
\newcommand{\Fs}{\mathscr F}
\newcommand{\Gs}{\mathscr G}
\newcommand{\Hs}{\mathscr H}
\newcommand{\Is}{\mathscr I}

\newcommand{\Ks}{\mathscr K}
\newcommand{\Ls}{\mathscr L}

\newcommand{\Ns}{\mathscr N}
\newcommand{\Os}{\mathscr O}
\newcommand{\Ps}{\mathscr P}
\newcommand{\Qs}{\mathscr Q}
\newcommand{\Rs}{\mathscr R}

\newcommand{\Ts}{\mathscr T}
\newcommand{\Us}{\mathscr U}
\newcommand{\Vs}{\mathscr V}
\newcommand{\Ws}{\mathscr W}
\newcommand{\Xs}{\mathscr X}
\newcommand{\Ys}{\mathscr Y}
\newcommand{\Zs}{\mathscr Z}
\newcommand{\Af}{\mathfrak A}
\newcommand{\Bf}{\mathfrak B}
\newcommand{\Cf}{\mathfrak C}
\newcommand{\Df}{\mathfrak D}
\newcommand{\Ef}{\mathfrak E}
\newcommand{\Ff}{\mathfrak F}
\newcommand{\Gf}{\mathfrak G}

\newcommand{\If}{\mathfrak I}

\newcommand{\Lf}{\mathfrak L}
\newcommand{\Mf}{\mathfrak M}
\newcommand{\Nf}{\mathfrak N}

\newcommand{\Pf}{\mathfrak P}

\newcommand{\Rf}{\mathfrak R}
\newcommand{\Sf}{\mathfrak S}

\newcommand{\Xf}{\mathfrak X}

\newcommand{\ff}{\mathfrak f}
\newcommand{\gf}{\mathfrak g}

\newcommand{\lf}{\mathfrak l}
\newcommand{\mf}{\mathfrak m}
\newcommand{\nf}{\mathfrak n}

\newcommand{\rf}{\mathfrak r}

\newcommand{\uf}{\mathfrak u}

\newcommand{\wf}{\mathfrak w}

\newcommand{\dirac}{\boldsymbol{\delta}}

\begin{document}

\title{Long time justification of wave turbulence theory}
\author{Yu Deng}
\address{\textsc{Department of Mathematics, University of Southern California, Los Angeles, CA, USA}}
\email{\texttt{yudeng@usc.edu}}
\author{Zaher Hani}
\address{\textsc{Department of Mathematics, University of Michigan, Ann Arbor, MI, USA}}
\email{\texttt{zhani@umich.edu}}
\date{}
\maketitle

\begin{abstract} In a series of previous works \cite{DH21,DH21-2,DH23}, we gave a rigorous derivation of the homogeneous wave kinetic equation (WKE) up to small multiples of the \emph{kinetic timescale}, which corresponds to short time solutions to the wave kinetic equation. In this work, we extend this justification to arbitrarily long times that cover the full lifespan of the WKE. This is the first long-time derivation ever obtained in any large data nonlinear (particle or wave)  collisional kinetic limit.
\end{abstract}
\tableofcontents
\section{Introduction} The problem of rigorously justifying the laws of kinetic theory and more generally statistical physics, starting from Hamiltonian first principles, was first raised to prominence by Hilbert in his famous list of problems announced at the ICM in 1900.

More than a century has passed since Hilbert announced this problem---now called \emph{Hilbert's Sixth Problem}---and during this time, a tremendous amount of progress has been achieved towards its end. We shall review this progress more thoroughly in Section \ref{history}, but it can be summarized briefly as follows:
In the classical context of particle kinetic theory, the turning point was Lanford's result \cite{Lan75} in 1975 which justified Boltzmann's kinetic theory up to small kinetic times; this result was later completed and developed upon in many works, most notably \cite{GST14}. In the context of wave kinetic theory (also known as wave turbulence theory), where colliding particles are replaced by nonlinearly interacting waves, the parallel results were only recently achieved by the authors \cite{DH21,DH21-2,DH23} following several important developments after the turn of the century in linear or equilibrium settings. However, for all this time, one of the biggest challenges in kinetic theory remained open, which is to resolve Hilbert's sixth problem for a nonlinear system (such as classical Boltzmann or wave turbulence) over arbitrarily long time intervals in the kinetic timescale. 

In this paper, we give the \emph{first ever large data long-time justification result} of a kinetic limit for a nonlinear Hamiltonian (particle or wave) collisional system: we derive the homogeneous wave kinetic equation (WKE), from the cubic nonlinear Schr\"{o}dinger (NLS) equation, up to \emph{large} multiples of the kinetic time $T_{\mathrm{kin}}$, that cover the \emph{full lifespan} of the WKE. We also obtain associated results such as propagation of chaos, density evolution, and the hierarchical evolution of higher moments over the same long time range. This extends the results of \cite{DH21,DH21-2,DH22,DH23}, which concern only sufficiently small multiples of $T_{\mathrm{kin}}$.
\subsection{Setup and the main result}\label{intro-nls} We start by describing the setup and main result of this paper, before moving to the background and discussions in Section \ref{history}. Following \cite{DH23}, in dimension $d\geq 3$, consider the cubic nonlinear Schr\"{o}dinger equation
\begin{equation}\label{nls}\tag{NLS}
\left\{
\begin{split}&(i\partial_t-\Delta)u+\alpha|u|^2u=0,\quad x\in \Tb_L^d=[0,L]^d,\\
&u(0,x)=u_{\mathrm{in}}(x)
\end{split}
\right.
\end{equation} 
on the square torus\footnote{All results and proofs extend without change to arbitrary rectangular tori as long as $0<\gamma<1$ (see below). For $\gamma=1$ we need a genericity condition on the torus.} $\Tb_L^d=[0,L]^d$ of size $L$. Here $\alpha$ is a parameter indicating the strength of the nonlinearity, and \[\Delta:=\frac{1}{2\pi}(\partial_{x_1}^2+\cdots +\partial_{x_d}^2)\] is the normalized Laplacian. Let $\Zb_L^d:=(L^{-1}\Zb)^d$ be the dual of $\Tb_L^d$, and fix the space Fourier transform as
\begin{equation}\label{fourier}
\widehat u(t, k) =\frac{1}{L^{d/2}}\int_{\Tb^d_L} u(t, x) e^{-2\pi i k\cdot x} \, dx, \qquad u(t,x) =\frac{1}{L^{d/2}}\sum_{k\in\Zb_L^d}\widehat{u}(t,k)e^{2\pi ik\cdot x}.
\end{equation}

Assume the initial data of (\ref{nls}) is given by
\begin{equation}
\label{data}\tag{DAT}u_{\mathrm{in}}(x)=\frac{1}{L^{d/2}}\sum_{k\in\Zb_L^d}\widehat{u_{\mathrm{in}}}(k)e^{2\pi ik\cdot x},\quad \widehat{u_{\mathrm{in}}}(k)=\sqrt{\varphi_{\mathrm{in}}(k)}\cdot \gf_k,
\end{equation}
where $\varphi_{\mathrm{in}}:\Rb^d\to[0,\infty)$ is a given Schwartz function, and $\{\gf_k\}$ is a collection of i.i.d. random variables. For concreteness, we will assume each $\gf_k$ is a standard normalized Gaussian. The results remain true with suitable modifications in non-Gaussian cases, see Remark \ref{suppresult} and \cite{DH21-2}.

Define the \emph{kinetic (or Van Hove) time}
\[T_{\mathrm{kin}}:=\frac{1}{2\alpha^2}.\] In this paper we will study the dynamics of (\ref{nls}) under the limit $L\to\infty,\alpha\to 0$. We assume that they are related by $\alpha=L^{-\gamma}$, so in particular $T_{\mathrm{kin}}=\frac{1}{2} L^{2\gamma}$, where $\gamma\in(0,1]$ is a fixed value called the \emph{scaling law} between $L$ and $\alpha$. This is the full range of admissible scaling laws for the problem as we  explain in Remark \ref{scalingrem} below.
\subsubsection{The wave kinetic equation} The wave kinetic equation is given by:
\begin{equation}\label{wke}\tag{WKE}
\left\{
\begin{split}&\partial_\tau\varphi(\tau,k)=\Kc(\varphi(\tau),\varphi(\tau),\varphi(\tau))(k),\\
&\varphi(0,k)=\varphi_{\mathrm{in}}(k),
\end{split}
\right.
\end{equation} where $\varphi_{\mathrm{in}}$ is as above and $k\in\Rb^d$, and the nonlinearity $\Kc$ (i.e. cubic collision operator) is given by
\begin{align}
\Kc(\varphi_1,\varphi_2,\varphi_3)(k)&=\int_{(\Rb^d)^3}\big\{\varphi_1(k_1)\varphi_2(k_2)\varphi_3(k_3)-\varphi_1(k)\varphi_2(k_2)\varphi_3(k_3)\nonumber\\&\qquad\qquad\qquad\quad+\varphi_1(k_1)\varphi_2(k)\varphi_3(k_3)-\varphi_1(k_1)\varphi_2(k_2)\varphi_3(k)\big\}\nonumber\\\label{wke2}\tag{COL}&\qquad\qquad\times\dirac(k_1-k_2+k_3-k)\dirac(|k_1|^2-|k_2|^2+|k_3|^2-|k|^2)\,\mathrm{d}k_1\mathrm{d}k_2\mathrm{d}k_3.
\end{align}
Here and below $\dirac$ denotes the Dirac delta, and we define
\[|k|^2:=\langle k,k\rangle,\quad \langle k,\ell\rangle:=k^1 \ell^1+\cdots +k^d\ell^d,\] where $k=(k^1,\cdots,k^d)$ and $\ell=(\ell^1,\cdots,\ell^d)$ are $\Zb_L^d$ or $\Rb^d$ vectors. 

The local well-posedness of (\ref{wke}) is proved in \cite{GIT20}. In general, global well-posedness is not true; there are known examples of finite-time blowup \cite{EV15,EV15-2,CL19}. In fact such blowup, in the form of $\delta$ singularity, is expected under mild conditions on initial data, and is connected to the formation of condensates. See Section \ref{diffusion} for more related discussions.
\subsubsection{The main result} Our main result is as follows.
\begin{thm}\label{main} Fix $d\geq 3$ and $\gamma\in(0,1)$ (or assume $\gamma=1$ but replace $\Tb_L^d$ by a rectangular torus and assume a genericity condition on its aspect ratios), and fix Schwartz initial data $\varphi_{\mathrm{in}}\geq 0$. We also fix $\tau_*\in(0,\tau_{\mathrm{max}})$, where $\tau_{\max}$ is the maximal time of existence for (\ref{wke}) (which may be finite or $\infty$), see Proposition \ref{wkelwp} for the precise definition. Consider the equation (\ref{nls}) with random initial data (\ref{data}), and assume $\alpha=L^{-\gamma}$ so that $T_{\mathrm{kin}}=\frac{1}{2} L^{2\gamma}$.

Then, for sufficiently large $L$, (\ref{nls}) has a smooth solution $u(t,x)$ up to time \[T=\tau_*\cdot\frac{L^{2\gamma}}{2}=\tau_*\cdot T_{\mathrm{kin}},\] with probability $\geq 1-e^{-(\log L)^{40d}}$. Moreover we have
\begin{equation}\label{limit}\sup_{t\in[0,T]}\sup_{k\in\Zb_L^d}\left|\Eb\,|\widehat{u}(t,k)|^2-\varphi\bigg(\frac{t}{T_{\mathrm{kin}}},k\bigg)\right|\leq L^{-\theta},
\end{equation} where $\theta>0$ depends only on $(d,\gamma)$ (see the line following (\ref{parameters}) in Section \ref{setupparam}), $\widehat{u}$ is as in \eqref{fourier}, and $\varphi(\tau,k)$ is the solution to (\ref{wke}).  In (\ref{limit}) and below we understand that the expectation $\Eb$ is taken assuming (\ref{nls}) has a smooth solution on $[0,T]$, which is an event with overwhelming probability.
\end{thm}
\begin{rem} The choice of the equation (\ref{nls}) is not essential, and ideas of the proof work for any semilinear dispersive equation modulo technicalities\footnote{For equations with derivative nonlinearity or quasilinear equations (such as water wave), it is physically natural to add vanishing dissipations to the system, or to put ultraviolet cutoffs at very high frequencies that converge to the identity in the kinetic limit. With such considerations, our method of proof also work for these equations.}. The reason we choose to work with (\ref{nls}) is that it can be viewed as the universal Hamiltonian equation, in the sense that any other Hamiltonian equation reduces to (\ref{nls}) in a certain limit regime of solutions (see \cite{SS99}). Moreover, the kinetic theory for (\ref{nls}) is closely related to both classical and quantum Boltzmann equations. In fact, the integration domain in the collision kernel in \eqref{wke2} is the same as that in the Boltzmann equation (see Section \ref{history} below for further discussions). All this makes (NLS) the perfect choice as the underlying system in our studies.
\end{rem}
\begin{rem}
A major significance of Theorem \ref{main} is that it allows all the rich long-time dynamics of (\ref{wke}), see Section \ref{diffusion}, to manifest via (\ref{nls}). In the case when (\ref{wke}) has finite time blowup, Theorem \ref{main} is clearly optimal (unless we want to extend beyond the blowup time). When (\ref{wke}) has global solution, we can allow the $\tau_*$ in Theorem \ref{main} to grow with $L$, and an interesting question is to determine the optimal dependence of $\tau_*$ on $L$. In general this depends on the long-time behavior of the solution to (\ref{wke}); for uniformly bounded solutions, our proof yields $\tau_*\sim\log\log L$. This may be improved to $\tau_*\sim \log L/\log\log L$ by refining the arguments in the proof (we made no such effort in this paper).
\end{rem}
\begin{rem}\label{scalingrem}The role of the scaling law $\gamma$ has been vague in early physics literature, but has recently been clarified by the authors in \cite{DH21,DH21-2,DH22,DH23}. For the equation (\ref{nls}) on an arbitrary torus (i.e. without imposing genericity condition on its aspect ratios), the admissible range of scaling laws is $\gamma\in(0,1)$, with the endpoint $\gamma=0$ also being admissible in the \emph{discrete} setting (see \cite{LS11}) and $\gamma=1$ also being admissible if we impose genericity condition on the torus \cite{DH21}. This full range is covered in \cite{DH23} and the current paper. Two specific scaling laws are of particular importance: $\gamma=1$ ($T_{\mathrm{kin}}\sim L^2$) which is linked to the notable Gibbs measure invariance problem in dimension $d=3$, and $\gamma=1/2$ ($T_{\mathrm{kin}}\sim L$), which corresponds to the Boltzmann-Grad scaling in the particle case, and is a natural scaling law in the inhomogeneous setting. See \cite{DH22,DH23} for more discussions.
\end{rem}
\begin{rem}\label{suppresult} As in the earlier papers \cite{DH21,DH21-2,DH23}, we can also obtain the accompanying results to Theorem \ref{main} pertaining to higher order statistics, using the same arguments in \cite{DH21-2}. These include propagation of chaos (which in fact directly follows from the proof of Theorem \ref{main}), probability density evolution in the case of non-Gaussian initial data, and derivation of the wave kinetic hierarchy; see \cite{DH21-2} for more details.
\end{rem}
\begin{rem} Unlike in \cite{DH21,DH21-2,DH23}, here the Schwartz assumption in Theorem \ref{main} is necessary (at least the level of regularity and decay required for $\varphi_{\mathrm{in}}$ should depend on $\tau_*$), because in our inductive argument (Section \ref{sectioninduct}), the norms involved in the bounds get successively weaker with each iteration.
\end{rem}
\subsection{Background and history}\label{history} The work of Boltzmann \cite{Bol72} in the 1870s marked the formal beginning of \emph{kinetic theory}. In general, it studies the evolution of macroscopic\footnote{In some of the literature, a distinction is made between the \emph{mesoscopic} observables studied in kinetic theory (such as the density function), and the macroscopic ones obtained by taking averages of the mescoscopic quantitites (such as fluid velocity).} observables in the kinetic limit, where certain statistical averages of a given microscopic system (e.g. $N$ colliding particles) are studied in the limit of infinite system size ($N\to \infty$). In the classical case of colliding particles, the microscopic dynamics is Newtonian and the corresponding macroscopic limit is given by the \emph{Boltzmann equation}.
\smallskip

The parallel theory for waves, under the name of \emph{wave kinetic theory} or \emph{wave turbulence theory}, was initiated in the work of Peierls \cite{Pei29} in the 1920s. Here particles are replaced by wave modes (or wave packets), collisions are replaced by nonlinear wave interactions such as (\ref{nls}), and the Boltzmann equation is replaced by the wave kinetic equation. With advances in the last century, the wave kinetic theory has now become a vast subject, which provides systematic treatments of wave interactions and has major scientific applications across various disciplines \cite{Has62,Has63,Zak65,BS66,Ved67,ZS67,BN69,Dav72,GS79,ZLF92,Sec98,Jan08}, see Nazarenko \cite{Naz11} for a relatively modern survey. In particular, we mention the influential papers of Hasselmann \cite{Has62,Has63} and Zakharov \cite{Zak65}: the former initiates the study of kinetic theory for water waves, leading to highly successful applications to oceanography. The latter one discovers the \emph{Zakharov spectrum} in analogy to the Kolmogorov spectrum in hydrodynamic turbulence, which gives the theory its common name of \emph{wave turbulence}.

Another aspect of wave kinetic theory appears when we replace classical particles in the Boltzmann picture with quantum particles (Bosons or Fermions). This leads to the quantum Boltzmann or \emph{Boltzmann-Nordheim equation}, see the papers of Nordheim \cite{Nor28} and Uehling-Uhlenbeck \cite{UU33} in the 1920--30s. The quantum Boltzmann equation is closely related to both the classical Boltzmann equation and the wave kinetic equation (\ref{wke}) associated with (\ref{nls}); in fact its collision operator is an exact linear combination of the other two, and all three collision operators are given by integrations in the same domain. On another hand, one can also consider classical or quantum particle systems where the nonlinear interactions is rescaled by the size of the system, in what is called the \emph{mean field limit}. This leads to Vlasov-type equations in the classical setting, and nonlinear Schr\"odinger equations in the quantum setting. Such mean-field limits are relatively better understood, and we shall not elaborate on them in our discussion here (cf. \cite{Spo80, ESY07, KM08, Sch17, J14, HJ15, Go16} for a sample of results and some review articles).
\smallskip

Mathematically speaking, the main challenge is to rigorously derive the kinetic equations, as well as subsequent diffusion and other asymptotic limits, from the underlying microscopic systems. For classical particles, this was explicitly stated in \emph{Hilbert's Sixth Problem} back in 1900 \cite{Hil00}. The state of art in this particle setting has been the rigorous derivation of the Boltzmann equation for \emph{short time} (cf. Lanford \cite{Lan75}, King \cite{Kin75}, which have been completed and improved more recently in Gallagher-Saint-Raymond-Texier \cite{GST14}) or \emph{small data} (Illner-Pulvirenti \cite{IP86,IP89}, Pulvirenti \cite{Pul87}); see also \cite{PS16, BGSS20,BGSS20-1,BGSS22} for some recent developments. The next step, which would be the justification of Boltzmann equation for arbitrarily long time, has been a major open problem in this subject since the 1975 paper of Lanford \cite{Lan75} \footnote{More can be said when the microscopic particle system is replaced by a random collision model as in \cite{OVY93}.}. 

\smallskip
On the wave turbulence side, the mathematical progress started much slower, but is quickly catching up in recent years. Earlier results such as Spohn \cite{Spo77} and Erd\"{o}s-Yau \cite{EY00} focused on linear problems where the Feynman diagram expansion has a simpler structure. These were then extended beyond the kinetic time scale in the breakthrough of Erd\"{o}s-Salmhofer-Yau \cite{ESY08}.

The next major advancement was Lukkarinen-Spohn \cite{LS11}. Though only involving equilibrium settings, it was the first one to treat the full nonlinear problem, and inspired many of the subsequent studies. In nonlinear,  out-of-equilibrium settings, partial results were obtained by several different groups including the authors \cite{BGHS19,DH19,CG19,CG20,DeS22}, which covered successively better time scales but still fell short of the kinetic time $T_{\mathrm{kin}}$.

The state of the art prior to this work has been the authors' recent papers \cite{DH21,DH21-2,DH23}. These provide the first rigorous derivation of (\ref{wke}) from (\ref{nls}) that reaches (a small multiple of) $T_{\mathrm{kin}}$, thus matching the state of the art in the particle case \cite{GST14}. We also mention a few results concerning the inhomogeneous setting and stochastic variants etc., see \cite{Fao20,DK21,ST21,Her22,HRST22,DK23,ACG21,Ma22}, and concerning the solution theory to (\ref{wke}), see \cite{EV15,EV15-2,CL19,GIT20,ST20,CDG22}. Finally, the derivation of the Boltzmann-Nordheim equation (see \cite{ESY04} for precise formulation) is still completely open, even for short time (for recent progress, cf. \cite{CHH23,CCH23, CHHol23} and references therein).

\smallskip
We conclude that, in all the \emph{large data, nonlinear} settings discussed above, the results have been limited to short time, even for the equilibrium result in \cite{LS11}. In this context, the current paper provides the first rigorous derivation of any collisional nonlinear kinetic equation (particle or wave) that goes beyond the small data or short time perturbative setting.
\subsection{The difficulty}\label{diff} As discussed in Section \ref{history}, in both particle and wave settings, the justification of the kinetic approximation has been limited to short time, and extending this to arbitrarily long time has been a major unsolved problem. The only exception is Erd\"{o}s-Salmhofer-Yau \cite{ESY08} (extending Erd\"{o}s-Yau \cite{EY00}), which concerns a \emph{linear} equation.

The fundamental difficulty involved in these long-time problems is the {\bf divergence of Taylor series}, which we explain below. In general, there are two approaches to the short-time justification of the kinetic approximation: the \emph{hierarchical approach} \cite{Lan75,Kin75,GST14,ESY08,LS11}, which relies on the infinite hierarchy system satisfied by the $n$-th order statistics (i.e. $n$-th moments or $n$-particle distributions),
and the \emph{dynamical approach} \cite{BGHS19,DH19,CG19,CG20,DH21,DH21-2,DH23}, which relies on approximating the dynamics of the microscopic system.

In the dynamical approach, the solution to the microscopic system (such as \ref{nls})) is approximated by high order Duhamel expansions; this is essentially a Taylor series in time, which will have to converge if we want this approach to be valid. In the hierarchical approach, a key challenge is to obtain exponential a priori bounds for high order target quantities (i.e. $n$-th order statistics for very large $n$), which also requires the absolute convergence of a certain Taylor-type series expansion. Thus in either case, {\bf as long as the approach is based on the structure of time $0$ data, the result cannot go beyond the radius of convergence of the Taylor series at time $0$.}

For \emph{linear} equations such as \cite{ESY08} this is not a problem, because under suitable assumptions the above radius of convergence is \emph{always} $\infty$. However, even for the simplest nonlinear equation, it is very likely that a regular solution exists on a longer time interval, but the radius of convergence of the Taylor expansion at time $0$ is much smaller. In such cases, it is not possible to go beyond this radius of convergence, by simply setting up an argument ``based on time $0$". On the other hand, if we want to exploit any ``base time" $t_*$ greater than 0, then we are faced with the difficulty that the data at time $t_*$ does not have independent states or modes, and the behavior of the target quantities at time $t_*$ relies on the full nonlinear dynamics on $[0,t_*]$. This difficulty is intimately linked to the emergence of \emph{time irreversibility}, which we will discuss in Sections \ref{strategyintro} and \ref{subsec:irreversibility} below.
\subsection{The strategy}\label{strategyintro} Our main theorem, Theorem \ref{main}, resolves the difficulty in Section \ref{diff} and completes the long-time justification of the kinetic limit in the wave turbulence setting. This relies on the following strategy, which is summarized here, and expanded upon in Section \ref{intro2} below.
\medskip

\uwave{Idea 1: Shifting base times.} We follow the dynamical approach. As discussed in Section \ref{diff}, to go beyond the radius of convergence, we need to shift the base time, and perform expansions at later times $t>0$. In practice, for fixed $\tau_*$ as in Theorem \ref{main}, we may choose $\delta$ small enough depending on $\tau_*$, such that $\tau_*=\Df\delta$ with $\Df\in\Zb$, and the solution $\varphi(\tau,k)$ has a convergent Taylor expansion on $[p\delta,(p+1)\delta]$ with base point $p\delta$, for each $0\leq p<\Df$. This formally allows us to expand $u(\tau\cdot T_{\mathrm{kin}})$ as a Taylor series of $u(p\delta\cdot T_{\mathrm{kin}})$ for $\tau\in[p\delta,(p+1)\delta]$, so \emph{if $u(p\delta\cdot T_{\mathrm{kin}})$ were to have independent Fourier coefficients as $u(0)=u_{\mathrm{in}}$ does in (\ref{data})}, then repeating the same arguments as in \cite{DH21,DH21-2,DH23} would immediately yield (\ref{limit}) for $t\in[p\delta,(p+1)\delta]$. Of course, such independence is not true due to the nonlinear evolution (\ref{nls}) on $[0,p\delta\cdot T_{\mathrm{kin}}]$, so the first central component of our proof is to {\bf examine the sense in which the Fourier coefficients of $u(p\delta\cdot T_{\mathrm{kin}})$ can be viewed as nearly independent.} This near independence condition is to be proved \emph{inductively in $p$}, which means it will get successively weaker and involve successively lower order quantities as $p$ increases; at each fixed $p$ it has to be strong enough to imply the validity of kinetic approximation, but also has to be verified by the actual time evolution of (\ref{nls}). More importantly, this condition has to recognize the irreversibility of the kinetic equation, which means that, at each $p$ it has to indicate a {\bf unique preferred direction} of time in which the derivation can proceed. This emergence of time irreversibility is a subtle point on which we shall elaborate below and in Section \ref{subsec:irreversibility}.

\medskip

\uwave{Idea 2: Quantifying ``near independence".} Now the question becomes how to measure the dependence of Fourier coefficients of $u(p\delta\cdot T_{\mathrm{kin}})$. One possibility, suggested by the proof of propagation of chaos in \cite{DH21-2}, is to express this in terms of mixed moments, or more precisely \emph{cumulants}, which quantify how far the behavior of the mixed higher order moments is from independence and Gaussianity. 

Recall the definition of cumulants $\Kb(X_1,\cdots,X_r)$ (see \cite{LS11}, or Definition \ref{defcm} below). Now consider
\begin{equation}\label{introcm}\Kb\big(\widehat{u}(p\delta\cdot T_{\mathrm{kin}},k_1)^{\zeta_1}\,,\widehat{u}(p\delta\cdot T_{\mathrm{kin}},k_2)^{\zeta_2}\,, \cdots, \widehat{u}(p\delta\cdot T_{\mathrm{kin}},k_r)^{\zeta_r}\big),\end{equation} where $k_j\in\Zb_L^d$ and $\zeta_j\in\{\pm\}$ (indicating possible conjugates, see Section \ref{setupnotat} below). By gauge and space translation invariance, it is easy to see that (\ref{introcm}) vanishes unless $\sum_j\zeta_j=\sum_{j}\zeta_jk_j=0$ (so $r$ must be even); moreover (\ref{introcm}) equals $\Eb|\widehat{u}(p\delta\cdot T_{\mathrm{kin}},k)|^2\approx \varphi(p\delta,k)$ if $r=2$ and $k_1=k_2=k$. As such, we know that the independence (and Gaussianity) of the Fourier coefficients of $u(p\delta\cdot T_{\mathrm{kin}})$ exactly corresponds to the vanishing of the cumulants (\ref{introcm}) for $r\geq 4$.

Now, for near independence, it is natural to think that such vanishing should be replaced by suitable smallness or decay conditions for (\ref{introcm}). Indeed, it follows from the proof in \cite{DH21-2} that
\begin{equation}\label{introcm2}|(\ref{introcm})|\leq L^{-cr},\quad \forall r\geq 4
\end{equation} for some small value $c>0$. Unfortunately, this is too weak as we shall see in an example below. On the other hand, in the Gibbs measure case \cite{LS11}, (\ref{introcm2}) is replaced by the much stronger \emph{physical space} $L^1$ bound for cumulants; however, this is too strong and a calculation shows that such bounds cannot hold away from equilibrium (see Appendix \ref{physicalL1}).

In fact, there is a deep reason why the bound (\ref{introcm2}) or the one in \cite{LS11}, or {\bf any norm upper bound} for that matter, will {\bf not be enough to run a successful inductive argument} following Idea 1 above. This is because such a norm upper bound does not indicate an intrinsic direction of time that would allow the derivation to proceed \emph{only} in that direction but not in the other.

More precisely, if such a norm bound for cumulants at time $p\delta\cdot T_{\mathrm{kin}}$ was sufficient to derive (\ref{wke}) on the time interval $[p\delta\cdot T_{\mathrm{kin}}, (p+1)\delta\cdot T_{\mathrm{kin}}]$, then since (\ref{nls}) is time-reversible, the same argument would also give a derivation of the \emph{time-reversed} version of (\ref{wke}) on the time interval $[(p-1)\delta\cdot T_{\mathrm{kin}}, p \delta\cdot T_{\mathrm{kin}}]$. But this would contradict the fact that we actually have the \emph{original} (\ref{wke}) on the time interval $[(p-1)\delta\cdot T_{\mathrm{kin}}, p \delta\cdot T_{\mathrm{kin}}]$. Such a ``paradox" is well-known in the kinetic literature, and presents a major challenge to any long-time derivation attempt for kinetic equations following Idea 1 above.

In the context of (\ref{nls}) and (\ref{wke}), this difficulty can be demonstrated by a concrete example, whose details we leave to Section \ref{intro2-1} as it requires some preliminaries. Basically, for $p=1$, we can find a term $\Kb_0$ in the Taylor series expansion of (\ref{introcm}) with $r=4$, such that when we expand $\widehat{u}(2\delta\cdot T_{\mathrm{kin}})$ into a Taylor series consisting of multilinear expressions of $\widehat{u}(\delta\cdot T_{\mathrm{kin}})$, the contribution of $\Kb_0$ to $\Eb|\widehat{u}(2\delta\cdot T_{\mathrm{kin}},k)|^2$ is a term that \emph{does not appear in} the leading term approximation $\varphi(2\delta,k)$. However, this contribution has the same order of magnitude as this leading term! This anomaly, presented by $\Kb_0$, still vanishes in the limit, but only because of the cancellations coming from the \emph{exact structure} of $\Kb_0$, namely its own Taylor series expansion on the interval $[0, \delta\cdot T_{\mathrm{kin}}]$.

In other words, such cancellation relies crucially on the fact that $\Kb_0$ comes from a forward-in-time expansion of the (\ref{nls}) flow. If we choose to run the \emph{backward} evolution of (\ref{nls}) from time $\delta\cdot T_{\mathrm{kin}}$ to time $0$ (instead of from time $\delta\cdot T_{\mathrm{kin}}$ to time $2\delta\cdot T_{\mathrm{kin}}$), this cancellation will fail and we will get a different effective equation due to the contribution of $\Kb_0$. This is explained in more detail in Sections \ref{intro2-1} and \ref{subsec:irreversibility}.

\medskip

\uwave{Idea 3: Setting up the ansatz.} We now need to set up the precise ansatz for the cumulants (\ref{introcm}) at time $p\delta\cdot T_{\mathrm{kin}}$, which involves not only upper bounds but also exact structures. This ansatz will contain \emph{a separate expansion} for each term (\ref{introcm}) up to some high order, plus a remainder that is sufficiently small. As explained in Idea 2 above, such expansion at time $p\delta\cdot T_{\mathrm{kin}}$ also needs to distinguish between going forward to time $(p+1)\delta\cdot T_{\mathrm{kin}}$ and going backward to time $(p-1)\delta\cdot T_{\mathrm{kin}}$. This is achieved by having this expansion \emph{memorize the whole history of interactions by (\ref{nls}) on the time interval $[0,p\delta\cdot T_{\mathrm{kin}}]$}. More precisely, the expansion will be a suitable iteration of short-time Taylor expansions on time intervals $[q,\delta\cdot T_{\mathrm{kin}},(q+1)\delta\cdot T_{\mathrm{kin}}]$ for $0\leq q<p$.

Schematically, we may choose $R_p\gg R_{p+1}$, and consider the ansatz
\begin{equation}\label{cmansatz}\Kb\big(\widehat{u}(p\delta\cdot T_{\mathrm{kin}},k_1)^{\zeta_1}\,, \cdots, \widehat{u}(p\delta\cdot T_{\mathrm{kin}},k_r)^{\zeta_r}\big)=\sum(\mathrm{terms\ at\ step\ }p)+(\mathrm{small\ remainder})
\end{equation} for all $r\leq R_p$; the inductive argument amounts to proving that if (\ref{cmansatz}) is true for $p$, then it is true for $p+1$. Of course, this is based on the fact that $\widehat{u}((p+1)\delta\cdot T_{\mathrm{kin}})$ can be expanded into Taylor series consisting of multilinear expressions of $\widehat{u}(p\delta\cdot T_{\mathrm{kin}})$. As in all our previous works, this Taylor series is organized into ternary trees; after using the technical Lemma \ref{propertycm}, we then obtain a ``pre-Feynman diagram" expansion of the cumulants \eqref{introcm} for $p+1$ in terms of those for $p$, which is organized into multiple ternary trees.

The exact ansatz (\ref{cmansatz}) is then determined by inductively iterating the above pre-expansion. By looking at the first few $p$'s, we may guess that (\ref{cmansatz}) can be organized into \emph{gardens} (i.e. multiple trees with their leaves paired, see \cite{DH21-2} or Definition \ref{defgarden} below). These gardens $\Gc$ at level $p+1$ (which we denote by $\Gc_{p+1}$) can be constructed by induction as follows. Consider the multiple ternary trees occurring in the pre-expansion above; upon using Lemma \ref{propertycm}, their leaves are divided into two-element pairs (corresponding to the two-point cumulant $\Eb|\widehat{u}(p\delta\cdot T_{\mathrm{kin}},k_\lf)|^2$) and four-or-more-element subsets (corresponding to higher order cumulants (\ref{introcm}) for $p$ with $r\geq 4$). We then construct the garden $\Gc_{p+1}$ by keeping all the leaf pairs, and replacing each four-or-more-element subset by a garden of form $\Gc_p$.

This strategy can be described as a ``partial time series expansion", i.e. we stop the time series expansion at time $p\delta$ if we encounter a two-point cumulant, and continue with the expansion if we encounter higher order cumulants. We replace the two-point cumulant $\Eb|\widehat{u}(p\delta\cdot T_{\mathrm{kin}},k_\lf)|^2$ by its approximation $\varphi(p\delta,k_\lf)$, and replace the higher order cumulants by the induction hypothesis in \eqref{cmansatz}. This partial time series expansion has critical importance in our proof, as it allows us to make use of the uniform bound of $\varphi$ \emph{outside the radius of convergence at $t=0$}, and avoid falling into the trap of divergent Taylor series.

Eventually, these iteration steps lead to {\bf completely new combinatorial structures, which we call layered gardens (Definition \ref{deflayer}) and canonical layered gardens (Definition \ref{defcanon})}. We need to work out all the properties of these objects from scratch, including the combinatorial ones (Propositions \ref{canonequiv}--\ref{layerreg2}, \ref{lftwistprop}, \ref{layerlad}--\ref{layervine}) and the analytic ones (Propositions \ref{proplayer1}, \ref{proplayer2}--\ref{proplayer4}, \ref{vineest}, \ref{ladderl1new}--\ref{ladderl1old}). These are the most central parts of this paper.

Finally, once the ansatz for the cumulants is proved inductively for all $p$, we can conclude that the main contribution to $\Eb |\widehat u(p\delta \cdot T_{\mathrm{kin}}, k)|^2$ comes from the order $r=2$ cumulants in (\ref{introcm}) from the previous time step $(p-1)\delta\cdot T_{\mathrm{kin}}$, which are themselves approximated by the solution $\varphi((p-1)\delta, k)$ of the \eqref{wke}. Following the arguments in \cite{DH21, DH23}, the resulting contribution of these two-point correlations can be shown to sum up to the power series expansion of $\varphi(p\delta, k)$ in terms of $\varphi((p-1)\delta,k)$, thus establishing the approximation of $\Eb |\widehat u(p\delta \cdot T_{\mathrm{kin}}, k)|^2$ with $\varphi(p\delta, k)$ by induction.

For a detailed discussion of the example mentioned in Idea 2 above and emergence of the arrow of time, the motivation of the ansatz, the main components of the proof, and the relation to the previous works \cite{DH21,DH21-2,DH23}, see Section \ref{intro2} below.
\subsection{Future horizons} Theorem \ref{main} and its proof opens a door to many problems that were previously out of reach. In this section we only mention the three most important avenues that should be now open for investigation.
\subsubsection{Extension of Lanford's theorem} With the analogy drawn between particle and wave kinetic theory, it is natural to expect that the particle counterpart of Theorem \ref{main}, i.e. the long-time extension of Lanford's theorem \cite{Lan75,Kin75,GST14}, would also be true. The authors are investigating this in an upcoming paper \cite{DHM} jointly with Xiao Ma. 
\subsubsection{Fluctuations and large deviations} The cumulant estimates in this paper should be sufficient to obtain the fluctuation and large deviation dynamics around the wave kinetic limit given by \eqref{wke}, over the same long time interval as in Theorem \ref{main}.  Such fluctuation and large deviation dynamics were described from a physics point of view in \cite{GBE22}. The parallel results in the particle setting were proved recently in \cite{BGSS20} for short times.

\subsubsection{Diffusion and other asymptotic limits}\label{diffusion} Once the kinetic approximation is justified for arbitrarily long time, the next step in the Hilbert's Sixth Problem, in both particle and wave settings, is then to reach the time scale for diffusion and other asymptotic limits.

More precisely, in the case of Boltzmann equation, it is known that certain rescaled limits of solutions converge to solutions to fluid equations, such as the Euler and Navier-Stokes equations, see \cite{GS04,SR09,GT19} and references therein. Deriving those fluid equations directly from Newtonian dynamics would then require extending Lanford's theorem to the adequate time scale of this limit.

In the case of \emph{inhomogeneous} wave kinetic equation or Boltzmann-Nordheim equation\footnote{We remark that the derivation of these inhomogeneous equations has not been complete, at least in the setting without noise (see \cite{ACG21,HRST22,HSZ23} for some related progress); however we expect this to follow from similar arguments as in \cite{DH21,DH21-2,DH23} and the current paper, modulo technical differences.}, i.e. when $\varphi=\varphi(\tau,x,k)$ in (\ref{wke}) and $\partial_\tau\varphi$ is replaced by $(\partial_\tau+k\cdot\nabla_x)\varphi$, it is not clear whether a diffusion limit similar to Boltzmann exists. What is the possible asymptotic behavior of the kinetic equation, and what this may imply about the behavior of (\ref{nls}) \emph{beyond the kinetic time scale}, is a major open avenue of investigation in this subject.

Finally, for the \emph{homogeneous} wave kinetic equation or the Boltzmann-Nordheim equation, there are known finite time blowups where $\tau_{\mathrm{max}}<\infty$ with formation of $\delta$ singularity (see \cite{EV15,EV15-2,CL19}). In fact this is expected to be fairly general, and corresponds to the formation of condensates. It may be possible to continue (\ref{wke}) weakly after the blowup time, but the sense in which the approximation \eqref{limit} holds probably needs to be modified. How to make this rigorous is another outstanding open problem.

There also exist (formal) global power-law solutions to \eqref{wke} known as the Kolmogorov-Zakharov cascade spectra, which were recently investigated in \cite{CDG22}. These solutions are constant-flux cascade solutions for (\ref{wke}), and their rigorous understanding is yet another fascinating question. Once this cascade behavior is understood at the level of \eqref{wke}, Theorem \ref{main} will immediately allow to prove energy cascade at the level of (\ref{nls}), i.e. rigorously establish the generic cascade phenomenon for nonlinear Schr\"{o}dinger equations. This links to the extensively studied, yet still open, question of growth of Sobolev norms for (\ref{nls}) and similar nonlinear dispersive equations.

\addtocontents{toc}{\SkipTocEntry}
\subsection*{Acknowledgements} The authors are partly supported by a Simons Collaboration Grant on Wave Turbulence. The first author is supported in part by NSF grant DMS-2246908 and Sloan Fellowship. The second author is supported in part by NSF grant DMS-1936640. The authors would like to thank the anonymous referees for suggestions that substantially improved the exposition.

\section{Description of main ideas}\label{intro2} 

In this section we detail out some of the key ideas leading to the main ansatz (Proposition \ref{propansatz}), which was briefly stated in Section \ref{strategyintro}. Throughout this section we will use the (now standard) concepts of trees, couples, gardens and decorations etc., see Figure \ref{fig:gardenintro}; the reader may consult the authors' earlier papers \cite{DH21,DH21-2,DH23} for the definitions of these terms, or see Definitions \ref{deftree}--\ref{defdec} below.
\begin{figure}[h!]
\includegraphics[scale=0.4]{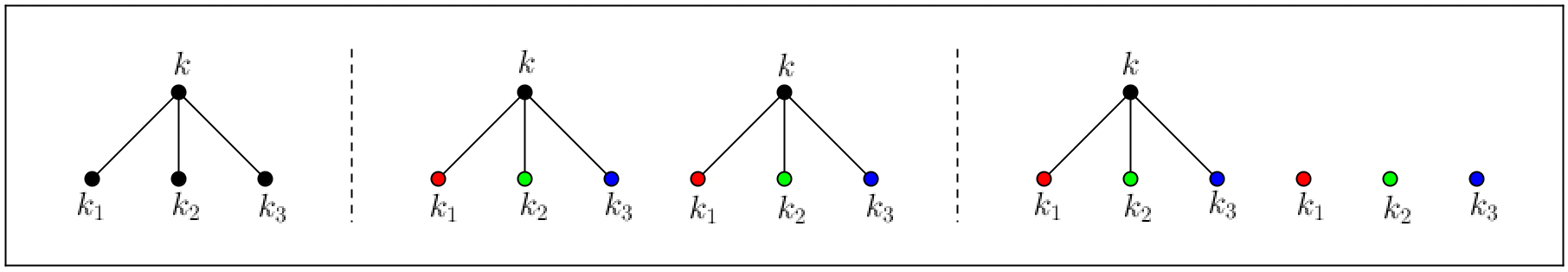}
\caption{From left to right: a tree, a couple and a garden of $4$ trees, with corresponding decorations (leaves of same color indicate pairings). See Definitions \ref{deftree}--\ref{defdec} below or \cite{DH21,DH21-2,DH23} for precise definitions.}
\label{fig:gardenintro}
\end{figure}
\subsection{An example}\label{intro2-1} We start by discussing the example mentioned in Section \ref{strategyintro}, Idea 2. Recall that the Fourier coefficients $\widehat{u}(0,k)$ are independent for different $k$, but $\widehat{u}(\delta\cdot T_{\mathrm{kin}},k)$ are no longer independent due to the nonlinear evolution. The first deviation from independence is characterized by the order $4$ cumulant (all odd ordered cumulants vanish as shown in Section \ref{strategyintro}), say
\begin{equation}\label{introex1}\Kb\big(\overline{\widehat{u}(\delta\cdot T_{\mathrm{kin}},k_1)},\widehat{u}(\delta\cdot T_{\mathrm{kin}},k_2),\overline{\widehat{u}(\delta\cdot T_{\mathrm{kin}},k_3)},\widehat{u}(\delta\cdot T_{\mathrm{kin}},k_4)\big).\end{equation} Recall that $\widehat{u}(\delta\cdot T_{\mathrm{kin}})$ can be expanded into a Taylor series consisting of multilinear expressions of $\widehat{u}(0)$, where these terms are organized using ternary trees (see Definition \ref{deftree}, and (\ref{defjt}) and (\ref{defjn}) below). If all four factors in (\ref{introex1}) are replaced by the \emph{linear} evolution, then the independence at time $0$ will carry over and the resulting cumulant will be $0$; thus the first nonzero contribution to (\ref{introex1}) is given by the case when three of the four factors are replaced by the linear evolution, and the remaining one is replaced by the cubic term (i.e. the first iterate, which corresponds to the ternary tree with only one branching node).

We know the tree expansion is calculated by iterating the nonlinearity in (\ref{nls}), for example
\begin{equation}\label{introex2}\widehat{u}(t,k_1)\sim e^{2\pi i|k_1|^2t}\bigg[\widehat{u}(0,k_1)+iL^{-d}\alpha\sum_{k_2-k_3+k_4=k_1}\int_0^{t}e^{2\pi i\Omega s}\,\mathrm{d}s\cdot\widehat{u}(0,k_2)\overline{\widehat{u}(0,k_3)}\widehat{u}(0,k_4)\bigg]\end{equation} (we only show the first two terms in the expansion), where $\Omega:=|k_2|^2-|k_3|^2+|k_4|^2-|k_1|^2$. Now, if we replace the first factor in (\ref{introex1}) by the cubic term, and apply Lemma \ref{propertycm}, we would obtain (part of) the first nonzero contribution to (\ref{introex1}), which has the form
\begin{equation}\label{introex3}\Kb_0(k_1,k_2,k_3,k_4):=e^{2\pi i\Omega \cdot\delta T_{\mathrm{kin}}}\cdot L^{-d-\gamma}\cdot\mathbf{1}_{k_1-k_2+k_3-k_4=0}\cdot\prod_{j=2}^4\varphi_{\mathrm{in}}(k_j)\cdot\int_0^{\delta\cdot T_{\mathrm{kin}}}e^{-2\pi i\Omega s}\,\mathrm{d}s.\end{equation} Note that this term (as well as the other terms in the full cumulant expansion (\ref{introex1})) is bounded pointwise by $L^{-(d-\gamma)}$ (using that $T_{\mathrm{kin}}=L^{2\gamma}$); however, knowing \emph{only} this upper bound will not be sufficient, which is what we will demonstrate through this example.

To see this, consider the expansion similar to (\ref{introex2}) but with starting time $\delta\cdot T_{\mathrm{kin}}$ rather than 0. Now we calculate the quantity $\Eb|\widehat{u}(2\delta\cdot T_{\mathrm{kin}},k)|^2$, in which we replace one of the two factors by the linear evolution of $\widehat{u}(\delta\cdot T_{\mathrm{kin}})$, and the other factor by the term that is cubic in $\widehat{u}(\delta\cdot T_{\mathrm{kin}})$. This then leads to the contribution
\begin{equation}\label{introex4}
\Yb:=L^{-d-\gamma}\sum_{k_1-k_2+k_3=k}e^{-2\pi i\Omega\cdot \delta T_{\mathrm{kin}}}\cdot\int_{\delta\cdot T_{\mathrm{kin}}}^{2\delta\cdot T_{\mathrm{kin}}}e^{2\pi i\Omega t}\,\mathrm{d}t\cdot\Kb_0(k,k_1,k_2,k_3),
\end{equation} where $\Omega$ is defined for $(k,k_1,k_2,k_3)$. In this sum we may restrict to the region $|\Omega|\leq L^{-2\gamma}$ (in other words $|t\Omega|\leq 1$), so the $\mathrm{d}t$ integral has saturated size $\sim T_{\mathrm{kin}}$; a dimension counting argument then yields that the number of choices of $(k_1,k_2,k_3)$ satisfying $k_1-k_2+k_3=k$ and $|\Omega|\leq L^{-2\gamma}$ for fixed $k$ is at most $L^{2(d-\gamma)}$ (see Lemma \ref{countlem}). Now, if we only have the upper bound $L^{-(d-\gamma)}$ for $\Kb_0$, then restricting the sum in (\ref{introex4}) to the region $|\Omega|\leq L^{-2\gamma}$ yields a sum whose upper bound is $L^{-d-\gamma}\cdot T_{\mathrm{kin}}\cdot L^{2(d-\gamma)}\cdot L^{-(d-\gamma)}\sim 1$, which has \emph{the same size} as the leading term; however it is clearly \emph{not} a part of the leading term approximation to $\Eb|\widehat{u}(2\delta\cdot T_{\mathrm{kin}},k)|^2$.

The conclusion is that, it is impossible to exclude the ``fake leading term" $\Yb$ and obtain the correct asymptotics, if we only use the upper bound for $\Kb_0$ defined in (\ref{introex3}). To eliminate $\Yb$ in the limit, we need to use the exact structure (\ref{introex3}), which (after a time rescaling) leads to the expression
\begin{equation}\label{Mark}
\Yb= L^{-2(d-\gamma)}\sum_{k_1-k_2+k_3=k}\prod_{j=1}^3\varphi_{\mathrm{in}}(k_j)\cdot\int_{0<s<\delta<t<2\delta} e^{\pi i\Omega\cdot L^{2\gamma}(t-s)}\,\mathrm{d}t\mathrm{d}s \to 0\,\,\,(L\to\infty).
\end{equation}
 Now $\Yb$ disappears in the limit, because for fixed $(t,s)$, the asymptotics of the summation in $(k_1,k_2,k_3)$ contains a factor $\dirac(t-s)$, whose integral in $(t,s)$ then vanishes due to the restriction $s<\delta<t$. In other words, the contribution of $\Kb_0$ vanishes in the limit only if we take into account its own structure (\ref{introex3}), which memorizes the interaction history on the time interval $[0,\delta\cdot T_{\mathrm{kin}}]$.

The above cancellation leading to the vanishing of $\Yb$ is due to the fact that the intervals $[0,\delta]$ and $[\delta,2\delta]$ have no common interior; this draws interesting analogy with the properties of Markov processes (for example, the Brownian motion whose increments on disjoint intervals are independent and have vanishing correlation). In fact, what the above analysis demonstrates is exactly a kind of \emph{Markovian property}; namely, to the first order, the two-point correlations at time $2\delta\cdot T_{\mathrm{kin}}$ depend only on the two-point correlations at time $\delta\cdot T_{\mathrm{kin}}$, but not on the nonlinear evolution history on $[0,\delta\cdot T_{\mathrm{kin}}]$ which is recorded by the higher order cumulants $\Kb_0$. Nevertheless, such Markovian property holds only if we go forward in time, and is clearly false if we go backward in time and replace $t\in[\delta,2\delta]$ by $t\in[0,\delta]$ in (\ref{Mark}), see Section \ref{subsec:irreversibility}.

\subsection{Building up gardens: A naive (but incorrect) attempt}\label{intro2-2} Recall the notations introduced in Section \ref{strategyintro} and our choice of $\delta$ so that the interval $[0,\tau_*]$ is split into $\Df$ subintervals of length $\delta$. As explained in Section \ref{strategyintro}, a central step in estimating $\Eb|\widehat{u}(p\delta\cdot T_{\mathrm{kin}},k)|^2$, and hence in the proof of Theorem \ref{main}, is to analyze and control the exact structure of the cumulants of form (\ref{introcm}) (these cumulants are related to the corresponding moments as in Definition \ref{defcm} below).

As explained in Section \ref{strategyintro}, Idea 2 and Idea 3, such structure should also memorize the full interaction history of (\ref{nls}) on the time interval $[0,p\delta\cdot T_{\mathrm{kin}}]$. The example in Section \ref{intro2-1} already illustrates a step in the construction of such history, namely calculating moments and cumulants at time $(p+1)\delta\cdot T_{\mathrm{kin}}$, using their Taylor expansions at time $p\delta\cdot T_{\mathrm{kin}}$ or earlier. The same example also shows that the information carried by such cumulant expansion is crucial to showing that their contribution vanishes in the limit. Generalizing this time iterative step then naturally leads to the construction of \emph{layered gardens and couples} (i.e. each node is in a unique time layer indicated by nonnegative integers), as briefed in Section \ref{strategyintro}, Idea 3. To motivate, in this section we will first describe a naive but incorrect attempt, and leave the real construction, which is only one step away, to Section \ref{intro2-3}.

Recall the $R_p$ described in Section \ref{strategyintro}; let $N_p$ be such that $N_p\gg R_p\gg N_{p+1}$. For each $p<\Df$, we can always perform an order $N_{p+1}$ Duhamel expansion of the solution at time $(p+1)\delta \cdot T_{\mathrm{kin}}$, which expresses it as a sum of iterates of the time $p\delta\cdot T_{\mathrm{kin}}$ data  plus a remainder term. The iterates are multilinear expressions in the solution at time $p\delta \cdot T_{\mathrm{kin}}$ and can be arranged as sums of ternary trees whose order (i.e. number of branching nodes) equal the order of the iteration (cf. equations \eqref{defbk}, \eqref{defjt}, and Definitions \ref{deftree} and \ref{defdec}). Note that for $p=0$, all the cumulants
\[\Kb\big(\widehat{u}(0,k_1)^{\zeta_1}\,,\widehat{u}(0,k_2)^{\zeta_2}\,, \cdots, \widehat{u}(0,k_r)^{\zeta_r}\big)\]at time $0$ and of order $r\geq 4$ should vanish due to exact independence and Gaussianity in (\ref{data}).

For $p=1$, by performing the above expansion, and applying Lemma \ref{propertycm}, we can write the main (non-remainder) contribution to the cumulant \begin{equation}
\label{time1cm}\Kb\big(\widehat{u}(\delta\cdot T_{\mathrm{kin}},k_1)^{\zeta_1}\,,\widehat{u}(\delta\cdot T_{\mathrm{kin}},k_2)^{\zeta_2}\,, \cdots, \widehat{u}(\delta\cdot T_{\mathrm{kin}},k_r)^{\zeta_r}\big)\end{equation}as a sum over \emph{irreducible gardens} $\Gc$ of width $r$. Here an irreducible garden is a garden of width $r$ (i.e. a collection of  $r$ trees whose leaves are completely paired with each other, see Definition \ref{defgarden} below), such that no proper subset of these $r$ trees have their leaves completely paired. Note that this irreducibility property precisely corresponds to the definition of the cumulant $\Kb$ in (\ref{time1cm}); instead, the corresponding moment in (\ref{time1cm}) would consist of the sum over all (reducible and irreducible) gardens. This gives the explicit structure of the cumulants at time $\delta \cdot T_{\mathrm{kin}}$ as a sum over irreducible gardens $\Gc$, which we denote by\footnote{We use the notion $\Ns$ because the construction here is \emph{not} the correct ansatz, but a naive attempt; see Section \ref{intro2-3} for more discussions.} $\Gc\in\Ns_1$. We shall put all the nodes in this garden $\Gc\in \Ns_1$ into layer $0$; then each term occurring in this summation is of form $\Kc_\Gc$ which consists of a summation over all decorations $(k_\nf)$ and an integration over time variables $(t_\nf)$ defined by the layering (which are all $0$ for $\Gc\in\Ns_1$). See (\ref{defkg0}) in Definition \ref{defkg} below for the exact expression $\Kc_\Gc$; note that all input functions $F_{\Lf_\lf}(k_\lf)$ should be replaced by $\varphi_{\mathrm{in}}(k_\lf)=\varphi(0,k_\lf)$.

Now let us move to $p=2$. Again we express the solution $\widehat u(2\delta \cdot T_{\mathrm{kin}}, k)$ as a Duhamel expansion, but now \emph{with initial time} $\delta \cdot T_{\mathrm{kin}}$. Up to a remainder term that we ignore for now, this writes $\widehat u(2\delta \cdot T_{\mathrm{kin}}, k)$ as a sum of terms associated with ternary trees whose leaves are associated with $\widehat u(\delta \cdot T_{\mathrm{kin}}, k_\lf)$. As such, this allows to express the cumulant 
\begin{equation}\label{time2cm}
\Kb\big(\widehat u(2\delta \cdot T_{\mathrm{kin}},k_1)^{\zeta_1},\cdots, \widehat u(2\delta \cdot T_{\mathrm{kin}},k_r)^{\zeta_r}\big)=\sum \Kb\big((\Jc_{\Tc_1})_{k_1}, \cdots, (\Jc_{\Tc_r})_{k_r}\big)
\end{equation}
where the sum is taken over all trees $\Tc_j$ of sign $\zeta_j$ and order at most $N_2$, and each $\Jc_{\Tc_j}$ is the Duhamel iterate associated with the interaction history described by $\Tc_j$, whose exact expression is as in \eqref{defjt} below. Then, the cumulants in the sum in (\ref{time2cm}) can be written as linear combinations of 
\begin{equation}\label{time2cm2}
\Kb\bigg(\prod_{\lf \in \Lc_1} \widehat u(\delta \cdot T_{\mathrm{kin}}, k_\lf)^{\sigma_\lf}, \cdots, \prod_{\lf \in \Lc_r}\widehat u(\delta \cdot T_{\mathrm{kin}}, k_\lf)^{\sigma_\lf}\bigg),
\end{equation}
where $\Lc_j$ denotes the set of leaf nodes of the tree $\Tc_j$ and $\sigma_\lf\in\{\pm\}$ is the sign of each leaf $\lf$. By applying Lemma \ref{propertycm} again, we an express (\ref{time2cm2}) as a sum of form 
\begin{equation}\label{time2cm3}
\sum_\Pc\prod_{B\in \Pc} \Kb\big((\widehat u(\delta \cdot T_{\mathrm{kin}}, k_\lf)^{\sigma_\lf})_{\lf \in B}\big).
\end{equation}
Here the summation in (\ref{time2cm3}) is taken over all irreducible partitions $\Pc$ of the set $\Lc_1\cup \cdots \Lc_r$ (i.e. no nontrivial union of sets in $\Pc$ equals any nontrivial union of sets $\Lc_j$), and the cumulant for each $B$ involves the random variables $\widehat u(\delta \cdot T_{\mathrm{kin}}, k_\lf)^{\sigma_\lf}$ for $\lf \in B$.

Now, we can use the formula deduced above for the cumulants (\ref{time1cm}) at time $\delta\cdot T_{\mathrm{kin}}$, namely the expansion into terms associated with $\Gc\in\Ns_1$, to replace each individual term in (\ref{time2cm3}), and subsequently obtain an expression for (\ref{time2cm}). Combinatorially, this expression is associated with a new, bigger garden $\Gc_2$, which is obtained as follows: start from the collection of trees $(\Tc_j:1\leq j\leq r)$, then for each set $B\in\Pc$ of leaves, we attach a garden $\Gc=\Gc(B)\in\Ns_1$ at these leaves, so that these leaves become the roots of the trees in this garden. Moreover, we put all \emph{branching nodes} of the trees $\Tc_j$ into layer $1$, and keep all nodes of each $\Gc(B)$ in layer $0$, to make $\Gc_2$ a layered garden, and denote all such layered gardens by $\Ns_2$. Analytically, by the definitions of $\Jc_\Tc$ and $\Kc_\Gc$ (see (\ref{defjt}) and (\ref{defkg0}) below), it is easy to see that the expression associated with $\Gc_2$ is precisely $\Kc_{\Gc_2}$ in (\ref{defkg0}) below, where again the time integration is defined by the layering, and all inputs $F_{\Lf_\lf}(k_\lf)$ should be replaced by $\varphi(0,k_\lf)$.

A simple instance of the above construction is already contained in the example described in Section \ref{intro2-1}, where we have two trees $(\Tc_1,\Tc_2)$ with one being trivial (only one leaf) and the other having exactly one branch and three leaves, and attach the garden in the right of Figure \ref{fig:gardenintro}, to obtain the couple in the middle of Figure \ref{fig:gardenintro} with suitable layering; note that the garden we attach has \emph{width $4$}. Another instance of this construction is given in Figure \ref{fig:reglay2}, where we have two trees $(\Tc_1,\Tc_2)$ each with two branches, and each of the five gardens in $\Ns_1$ we attach has \emph{width $2$} (i.e. is a couple) with three being trivial and two being nontrivial.
\begin{figure}[h!]
\includegraphics[scale=0.5]{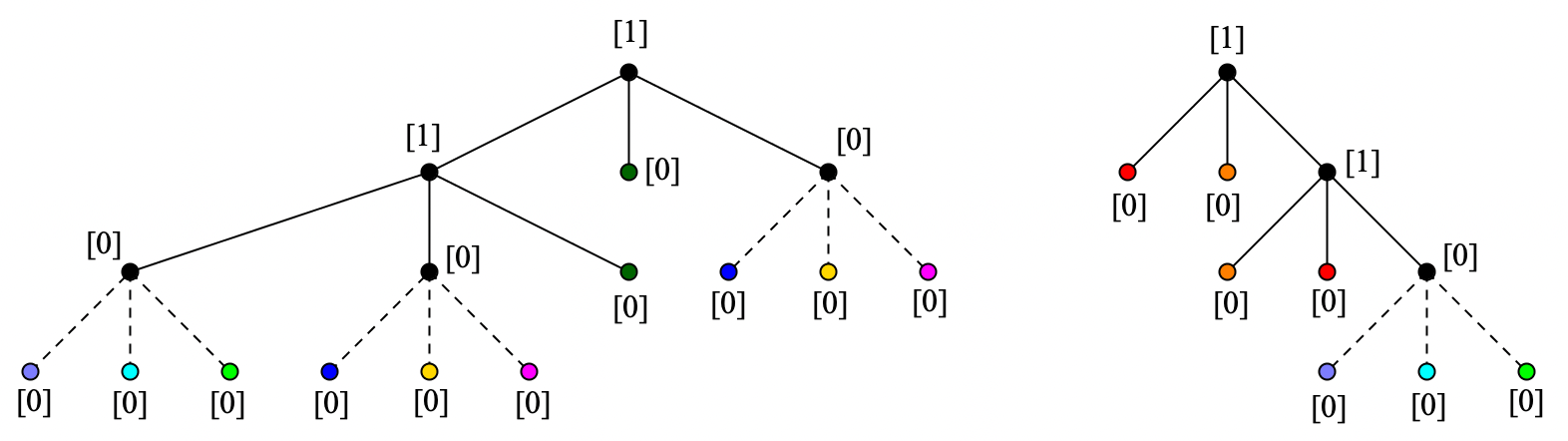}
\caption{An example of a layered couple as described in Section \ref{intro2-2} which appears in the expansion of $\widehat u(2\delta \cdot T_{\mathrm{kin}})$ (the numbers [0] or [1] denote the layer of each node). This is a regular couple (see Definition \ref{defreg} below) but is \emph{not} canonical (see Section \ref{intro2-3}). In the modified construction in Section \ref{intro2-3}, this will be replaced by the canonical regular couple in Figure \ref{fig:reglay3}.}
\label{fig:reglay2}
\end{figure}

At this point, it is now clear how to continue this process inductively and construct $\Ns_{p+1}$ from $\Ns_p$ etc., which ultimately gives (up to acceptable remainder terms) an exact \emph{structural} description of the cumulants at any time $p\delta\cdot T_{\mathrm{kin}}$, as a sum of terms of form $\Kc_\Gc$ over all irreducible layered gardens $\Gc\in \Ns_p$. This clearly accounts for the example in Section \ref{intro2-1}; however this naive attempt still does not give the correct ansatz, due to one fatal weakness. We discuss this weakness, as well as the one simple step to fix it, in Section \ref{intro2-3} below.
\subsection{Canonical layering and the true ansatz}
\label{intro2-3} An indication why the construction in Section \ref{intro2-2} does not work is that, by this construction, all leaf nodes \emph{must} be in layer $0$ (in other words, all the inputs in $\Kc_\Gc$ in (\ref{defkg0}) must be replaced by the initial data $\varphi(0,k_\lf)$), which is exactly what we want to avoid in Section \ref{diff} (i.e. base all the analysis on time $0$ data).

To see an instance of this major issue, we take a closer look at the layered couple $\Qc$ in Figure \ref{fig:reglay2}. Note that it is a \emph{regular couple} (see Definition \ref{defreg} below), and does not incorporate the cancellation shown in the example in Section \ref{intro2-1} (i.e. it is \emph{coherent} in the terminology of Definition \ref{defcoh} below); thus there is no gain of powers of $L$ in the associated $\Kc_\Qc$ expression. Moreover, if the layer $1$ nodes in $\Qc$ are replaced by layer $p$, then the roots of the two sub-couples with dashed lines can have \emph{arbitrary} layer between $0$ and $p$ while still being consistent with the construction in Section \ref{intro2-2}. Therefore, the \emph{number of layerings} of a given regular couple with order $2n$ can be as large as $\Df^n$. Since the $\Kc_\Qc$ quantity for each individual regular couple $\Qc$ gains only a factor of $(C_1\delta)^n$ (where $C_1$ is a large constant depending on initial data $\varphi_{\mathrm{in}}$, in the same way as in \cite{DH21,DH21-2,DH23}), this will lead to serious divergence issues as $(C_1\Df\delta)^n=(C_1\tau_*)^n$ grows in $n$ unless $\tau_*$ is sufficiently small.

Of course, the above weakness comes from the fact that we are trying to perform the Duhamel expansion all the way to time $0$ \emph{in every case}. Indeed, whenever we encounter the \emph{two-point correlation} $\Eb|\widehat{u}(p\delta\cdot T_{\mathrm{kin}},k_\lf)|^2$ in the expression, we should simply replace it by its (expected) approximation $\varphi(p\delta,k_\lf)$, instead of further Duhamel-expanding it by earlier time data and going all the way to time $0$. In this way we can exploit the fact that $\varphi(p\delta,k_\lf)$ is uniformly bounded in $p$, despite that it cannot be expanded as a power series involving only $\varphi_{\mathrm{in}}=\varphi(0,\cdot)$.

Now, recall the reduction and construction process described in Section \ref{intro2-2}. By Lemma \ref{propertycm}, it is clear that, we would encounter the two-point correlation $\Eb|\widehat{u}(p\delta\cdot T_{\mathrm{kin}},k_\lf)|^2$ \emph{only when} one of the sets $B\in\Pc$ as in Section \ref{intro2-2} has exactly cardinality $2$. This leads to the following minor but crucial modification to the construction process in Section \ref{intro2-2}, which is consistent with the brief description in Section \ref{strategyintro}, Idea 3: namely, in reducing (\ref{time2cm3}) (and same with $\delta$ replaced by arbitrary $p\delta$), we further expand the cumulant $\Kb\big(((\widehat{u}(\delta\cdot T_{\mathrm{kin}},k_\lf))^{\sigma_\lf})_{\lf\in B}\big)$ using earlier time ansatz \emph{only when} $|B|\geq 4$. When $|B|=2$, we simply replace the cumulant, which is now just $\Eb|\widehat{u}(\delta\cdot T_{\mathrm{kin}},k_\lf)|^2$, by $\varphi(\delta,k_\lf)$ plus a decaying sub-leading term. Note that this minor modification addresses both issued mentioned above: now we are stopping expansion at time $p\delta\cdot T_{\mathrm{kin}}$ whenever we encounter a two-point correlation at this time, and a combinatorial argument allows us to control the number of coherent regular couples (see Section \ref{intro2-3-1} below).

Combinatorially, the new construction leads to a new class of layered gardens (and couples) $\Gc$, which we refer to as \emph{canonical} layered gardens or $\Gc\in\Gs_p$, as follows: for $p=1$, $\Gs_1$ still denotes the set of all irreducible gardens with all nodes in layer $0$ (which is same as $\Ns_1$ in Section \ref{intro2-2}); for $p\geq 1$, $\Gs_{p+1}$ is constructed from $\Gs_p$ in the same way as $\Ns_{p+1}$ is constructed from $\Ns_p$ in Section \ref{intro2-2}, except that we attach the garden $\Gc(B)\in\Gs_p$ \emph{only when} $|B|\geq 4$, while for $|B|=2$ we simply \emph{turn these two leaves into a leaf pair and put both of them in layer $p$}.

For example, the example described in Section 2.1 is canonical, because the only garden we attach has width $4$; on the other hand, the layered garden in Figure \ref{fig:reglay2} is not canonical, as the attached gardens all have width $2$. Indeed, starting from two trees each with 2 branches as in Section \ref{intro2-2}, instead of attaching the five width $2$ gardens (i.e. couples), the new construction procedure requires to turn each two-leaf set $B$ into a leaf pair in layer $1$, which results in the garden as illustrated in Figure \ref{fig:reglay3}, which is a canonical layered garden in $\Gs_2$. See also Figures \ref{fig:layergarden} and \ref{fig:canongarden} for a canonical garden in $\Gs_3$ and its construction process.
\begin{figure}[h!]
\includegraphics[scale=0.5]{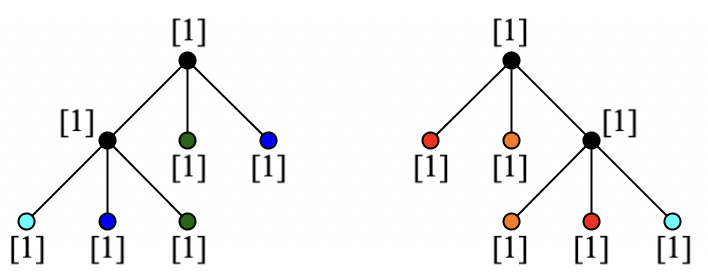}
\caption{The canonical (regular) couple we get, if we start from the same setup in Figure \ref{fig:reglay2} in Section \ref{intro2-2}, but follow the new construction procedure in Section \ref{intro2-3}. Here we do not attach any width $2$ garden (i.e. couple) but turn any two-leaf subset $B$ into a leaf pair in layer $1$.}
\label{fig:reglay3}
\end{figure}

It can be shown, see Proposition \ref{canonequiv}, that a layered garden is canonical if and only if for any two nodes $(\nf,\nf')$ such that all their descendant leaves are completely paired, we always have $\max(\Lf_\nf,\Lf_{\nf'})\geq\min(\Lf_{\nf^{\mathrm{pr}}},\Lf_{(\nf')^{\mathrm{pr}}})$, where $\Lf_\nf$ is the layer of $\nf$ and $\nf^{\mathrm{pr}}$ is the parent node of $\nf$ etc. This provides another easy way to test if a given layered garden is canonical (for example the garden in Figure \ref{fig:reglay2} clearly violates this condition), and also allows to bound the number of layerings satisfying the coherent condition in Definition \ref{defcoh}, see Proposition \ref{layerreg2}.

Analytically, with the new reduction process, and by similar arguments as in Section \ref{intro2-2}, we see that the cumulants at time $p\delta\cdot T_{\mathrm{kin}}$ can be expressed as a sum of $\Kc_\Gc$ over all canonical layered gardens $\Gc\in\Gs_p$, where $\Kc_\Gc$ now is defined exactly as in (\ref{defkg0}), with time integration defined by the layering, and the inputs given by\footnote{Compared to Section \ref{intro2-2}, here the inputs are not replaced by the initial data $\varphi(0,\cdot)$ because we may not be expanding all the way to time $0$.} $F_{\Lf_\lf}(k_\lf)$. Here $\Lf_\lf$ is the layer of $\lf$ and in practice $F_p(k)$ equals $\varphi(p\delta,k)$ plus a decaying sub-leading term. For an example (associated with the garden in Figures \ref{fig:layergarden} and \ref{fig:canongarden}) of deducing the expression $\Kc_\Gc$ for $\Gc\in\Gs_3$ from expressions $\Kc_{\Gc'}$ for $\Gc'\in\Gs_2$, see Section \ref{secduhamel}.
\subsection{Emergence of the arrow of time}\label{subsec:irreversibility} With the detailed description of the ansatz in Sections \ref{intro2-1}--\ref{intro2-3}, the following statement should now become clear: {\bf the structure of the cumulants in (\ref{introcm}) at time $p\delta\cdot T_{\mathrm{kin}}$ depends on (and memorizes) the whole interaction history of (\ref{nls}) on the time interval $[0,p\delta\cdot T_{\mathrm{kin}}]$.} Indeed, the ansatz is constructed in terms of canonical layered gardens, which are formed by carefully designed iteration procedures involving the short-time Taylor expansions of (\ref{nls}) on $[q\delta\cdot T_{\mathrm{kin}},(q+1)\delta\cdot T_{\mathrm{kin}}]$ for $0\leq q<p$, and thus sees the whole interaction history on $[0,p\delta\cdot T_{\mathrm{kin}}]$.

Now we can demonstrate, with such an ansatz that memorizes the interaction history, how an intrinsic arrow of time is created and resolves the ``paradox" described in Section \ref{strategyintro}, Idea 2. In fact, suppose we start at time $p\delta\cdot T_{\mathrm{kin}}$. The full statistics of $\widehat{u}(p\delta\cdot T_{\mathrm{kin}},k)$ is then given by two-point correlations, which match the solution $\varphi(p\delta ,k)$ of (\ref{wke}), plus the cumulant terms $\Kb_0$, which are lower order but memorize the full history on $[0,p\delta\cdot T_{\mathrm{kin}}]$.

Note that, regardless of whether we go forward to time $(p+1)\delta\cdot T_{\mathrm{kin}}$ or go backward to time $(p-1)\delta\cdot T_{\mathrm{kin}}$, the contribution of the two-point correlation terms $\varphi(p\delta ,k)$ to the new two-point correlations will \emph{always be} $\varphi((p+1)\delta,k)$, which matches the solution to the original (\ref{wke}) if we go forward in time, and matches the solution to the \emph{time-reversed} (\ref{wke}) if we go backward in time. However, whether such approximation is valid still depends on the contribution of the cumulant terms $\Kb_0$ to the new two-point correlations, which we denote by $\Yb$.

Now, if we go forward, then the terms $\Yb$ will be combinations of the \emph{existing history} on $[0,p\delta\cdot T_{\mathrm{kin}}]$ already seen by $\Kb_0$, and the \emph{new history} coming from interactions of (\ref{nls}) on $[p\delta\cdot T_{\mathrm{kin}},(p+1)\delta\cdot T_{\mathrm{kin}}]$. As the existing history and the new history belong to \emph{disjoint time intervals} $[0,p\delta\cdot T_{\mathrm{kin}}]$ and $[p\delta\cdot T_{\mathrm{kin}},(p+1)\delta\cdot T_{\mathrm{kin}}]$, we know $\Yb$ will vanish in the limit, as illustrated in Section \ref{intro2-1}, where $\Kb_0$ is as in (\ref{introex3}) and $\Yb$ as in (\ref{Mark}). Therefore, the two-point correlations in the full statistics of $\widehat{u}((p+1)\delta\cdot T_{\mathrm{kin}},k)$ still matches the solution $\varphi((p+1)\delta,k)$ to (\ref{wke}). It can be shown that the new cumulant terms are still lower order, but they will memorize the the full history of $[0,(p+1)\delta\cdot T_{\mathrm{kin}}]$, which is built up by both $[0,p\delta\cdot T_{\mathrm{kin}}]$ and $[p\delta\cdot T_{\mathrm{kin}},(p+1)\delta\cdot T_{\mathrm{kin}}]$, and the induction goes on.

If instead we go backward to time $(p-1)\delta\cdot T_{\mathrm{kin}}$, then the terms $\Yb$ {\bf will not vanish in the limit}, because we are repeating existing history on $[(p-1)\delta\cdot T_{\mathrm{kin}},p\delta\cdot T_{\mathrm{kin}}]$ instead of exploring new history on $[p\delta\cdot T_{\mathrm{kin}},(p+1)\delta\cdot T_{\mathrm{kin}}]$. The Markovian property in Section \ref{intro2-1} then fails because $\Kb_0$ ``recognizes" the history it has experienced; see for example (\ref{Mark}), where the leading term of $\Yb$ coming from $\dirac(t-s)$ will {\bf not be zero} if the ``new history'' $t\in[\delta,2\delta]$ is changed to $t\in[0,\delta]$ and ``repeats the existing history" $s\in [0,\delta]$. In other words, instead of (\ref{Mark}), we now have
\begin{equation}\label{Mark2}
\Yb=-L^{-2(d-\gamma)}\sum_{k_1-k_2+k_3=k}\prod_{j=1}^3\varphi_{\mathrm{in}}(k_j)\cdot\int_{0<s<\delta;\,0<t<\delta} e^{\pi i\Omega\cdot L^{2\gamma}(t-s)}\,\mathrm{d}t\mathrm{d}s \not\to 0\,\,\,(L\to\infty).
\end{equation} In such cases, due to these contributions $\Yb$, the two-point correlations in the full statistics of $\widehat{u}((p-1)\delta\cdot T_{\mathrm{kin}},k)$ will {\bf not} match $\varphi((p+1)\delta,k)$ given by the time-reversed (\ref{wke}).

In summary, we have demonstrated how our canonical layered garden ansatz has played a crucial role in creating the arrow of time and resolving the time irreversibility ``paradox". It is also clear that the iteration steps in the current proof are {\bf fundamentally different from} the short-time proof in our earlier works \cite{DH21,DH21-2,DH23}, as the latter just corresponds to the special case $p=0$, which has {\bf empty history}.
\subsection{Main components of the proof}\label{intro2-3-1} With the key definition of canonical layered gardens introduced and the ansatz for cumulants (\ref{introcm}) fixed in Section \ref{intro2-3}, we now describe the main components of the rest of the proof. The proof requires to (i) estimate and analyze the expressions $\Kc_\Gc$ for individual canonical layered gardens $\Gc$ (or suitable combinations thereof), and (ii) control the number of canonical layered gardens under suitable assumptions.
\begin{itemize}
\item \uwave{Non-regular couples:} It is known since \cite{DH21,DH21-2,DH23} that the \emph{regular couples} $\Gc=\Qc$ (see Definition \ref{defreg}) provide the main contributions to $\Kc_\Gc$. For non-regular couples $\Gc$, and non-regular parts of arbitrary gardens and couples (apart from the special structures of \emph{vines} discussed below), the corresponding $\Kc_\Gc$ can be estimated by applying a particular counting algorithm. This algorithm is a major component in the proofs of \cite{DH21,DH23}, and exactly the same algorithm and proof applies also in this paper. This is summarized in the \emph{rigidity theorem} (Proposition \ref{lgmolect}), which we will use as a black box.
\item \uwave{Non-coherent regular couples:} Now we restrict to regular couples. A key new notion in this paper is that of \emph{incoherency}, see Definition \ref{defcoh} below. This notion of incoherency relies on the layering structure, and basically corresponds to the time disjointness property exhibited in the example in Section \ref{intro2-1}; heuristically, such incoherency will occur in regular couples each time we attach a garden of width at least $4$ to a set $B$ of leave as described in Section \ref{intro2-2} and above. Due to the Markovian property observed in Section \ref{intro2-1}, we can prove that each occurrence of incoherency leads to a gain of power $L^{-c}$ for some absolute constant $c$ (see Proposition \ref{proplayer1}), which suffices for our proof.
\item \uwave{Coherent regular couples:} Turning to regular couples with no incoherency (i.e. coherent), we can prove a structure theorem for them (see Propositions \ref{canonequiv}--\ref{layerreg1}), which essentially reduces them to regular couples in a single layer (as the one in Figure \ref{fig:reglay3}) which are already treated in \cite{DH21,DH23}. This allows us to obtain the full asymptotics and cancellation for $\Kc_\Gc$ associated with such couples $\Gc=\Qc$, see Propositions \ref{proplayer2}--\ref{proplayer4}. In particular, the leading contribution in $\Eb|\widehat{u}(p\delta\cdot T_{\mathrm{kin}},k)|^2$ can be precisely calculated using these asymptotics, which is then shown to be $\varphi(p\delta,k)$ by induction in time.

In addition, we can show that, for any coherent regular part of arbitrary gardens and couples that has size $n$, the number of choices for this part, including possible layerings, is bounded by $C_0^n$ where $C_0$ is a constant \emph{independent} of $\delta$ (see Proposition \ref{layerreg2}). This is a key property that relies on the modified construction process described above (and is certainly not true in the setting of Section \ref{intro2-2}), and allows to avoid the divergence in Section \ref{intro2-2} caused by the factor $\Df^n$.
\item \uwave{Vines and twists:} The structure of \emph{vines} and \emph{vine chains}, which was discovered in \cite{DH23} (with a special case contained in \cite{DH21}), is the main obstacle in the analysis of non-regular couples. Such structures require a delicate cancellation argument, where we combine the $\Kc_\Gc$ terms for different gardens $\Gc$, that are related by a special combinatorial operation called \emph{twist}.

In this paper, all these arguments need to be adapted to the new layering structure (consistency between twisting and layering structures is proved in Proposition \ref{lftwistprop}). We also need to introduce similar incoherency notions for vines and vine chains (Definition \ref{cohmol}), and establish similar gains for each occurrence of incoherency, as with regular couples above, see Proposition \ref{vineest}. The number of coherent vines and vine chains can again be controlled, see Proposition \ref{layervine}.
\item \uwave{Ladders:} The \emph{ladders} are neutral objects observed in \cite{DH21,DH23} that do not gain or lose any powers of $L$. As such, in this paper we only need to avoid $\Df^n$ type loss in Section \ref{intro2-2}; just like for regular couples and vines, we can define the incoherency notion for ladders (Definition \ref{cohmol}), and prove similar decay estimates and counting bounds for the given number of occurrences of incoherency, see Propositions \ref{layerlad} and \ref{ladderl1new}. Finally we need an $L^1$ estimate for time integrals, similar to \cite{DH21,DH23} but adapted to the layered case, see Proposition \ref{ladderl1old}.
\end{itemize}
\subsubsection{Structure of the paper} The rest of this paper is organized as follows. In Section \ref{section3} we introduce the relevant notations (couples, gardens etc.), including the fundamental objects of layered gardens and canonical layered gardens, and write down the main ansatz. Then Sections \ref{layerobject1}--\ref{seclayer3} are devoted to the study of regular couples: Section \ref{layerobject1} contains the combinatorial results, Section \ref{seclayer2} contains the decay estimate for incoherent regular couples, and Section \ref{seclayer3} contains the asymptotics for the coherent ones. In Sections \ref{section7}--\ref{section8} we define the notions of molecules, vines and twists, and prove relevant combinatorial results for layerings of these objects. In Section \ref{cancelvine} we prove the main estimate for vines and vine chains. Putting these together, in Sections \ref{stage1red}--\ref{redcount} we perform the necessary reduction steps to the molecule associated with a given garden $\Gc$, and apply the blackbox algorithm in \cite{DH21,DH23} to complete the desired estimates for $\Kc_\Gc$. Finally, in Section \ref{endgame} we provide the estimates for the linearization operator and the remainder term, and complete the proof of Theorem \ref{main}. In the whole proof process we will need various auxiliary lemmas (combinatorial, analytical, number theoretic and counting), which are listed in Appendix \ref{appendmisc}. Then Appendix \ref{physicalL1} illustrates the failure of some physical space $L^1$ bound in \cite{LS11} (discussed in Section \ref{strategyintro}, Idea 2), Appendix \ref{schwartzwke} proves the persistence of regularity for solutions to \eqref{wke} which gives existence of solutions in Schwartz space, and Appendix \ref{appendtable} contains three tables collecting some of the important notations used in this paper.

\subsubsection{Relation with previous works \cite{DH21,DH21-2,DH23}} As pointed out at the end of Section \ref{subsec:irreversibility}, the long-time proof in the current paper is not an iteration of the short-time proof in our previous works \cite{DH21,DH21-2,DH23}; in fact the inductive step in this paper is fundamentally different from the short time analysis in \cite{DH21,DH21-2,DH23}. On the other hand, the current work does rely on many technical ingredients developed in \cite{DH21,DH21-2,DH23}, which we briefly describe below.

 First, our proof adopts the same Feynman diagram expansion paradigm from \cite{DH21,DH21-2,DH23}, where iterates of (\ref{nls}) are expanded in trees, two-point correlations of these iterates are expanded in \emph{couples} of trees \cite{DH21, DH23}, and multi-point correlations (and cumulants) are expanded in \emph{gardens} of trees \cite{DH21-2}. We still give self-contained definitions of all those objects and the quantities associated to them in Section \ref{section3}. 
 
Next, at a number of instances, especially concerning the common combinatorial objects that appear in both \cite{DH21,DH21-2,DH23} and the current paper, we are directly citing the lemmas and propositions that are proved there. These mainly include (i) Propositions \ref{propstructure}--\ref{deflink} and \ref{regref1}--\ref{regref3}, concerning structural properties of regular couples and exact asymptotics of certain expressions associated with them, (ii) Parts of Section \ref{section7} concerning structural properties of vines and their relation to couples and gardens, and (iii) Proposition \ref{lgmolect} concerning counting estimates for certain molecules, which is proved in Proposition 9.6 of \cite{DH23} and used as a black box here.

At each place where we directly cite from \cite{DH21,DH21-2,DH23}, we make sure that the relevant proof in the current paper is either exactly the same as in the cited papers, or requires at most minimal and essentially trivial adaptations. For any result whose proof requires even minor modifications from \cite{DH21,DH21-2,DH23}, we have included a full self-contained proof in the current paper.
\section{Preparations for the proof}\label{section3}
\subsection{First reductions} Let $\delta$ be a small parameter to be fixed below, such that $\tau_*=\Df\delta,\,\Df\in\Zb$ (see Section \ref{setupparam}). We perform the first reduction for the equation (\ref{nls}) as follows. Let $u$ be a solution to (\ref{nls}), and recall $\alpha=L^{-\gamma}$. Let $M=\fint|u|^2$ be the conserved mass of $u$ (where $\fint$ takes the average on $\Tb_L^d$), and define $v:=e^{-2iL^{-\gamma} Mt}\cdot u$, then $v$ satisfies the Wick ordered equation
\begin{equation}(i\partial_t-\Delta)v+L^{-\gamma}\bigg(|v|^2v-2\fint|v|^2\cdot v\bigg)=0.
\end{equation} By switching to Fourier space, rescaling in time and reverting the linear Schr\"{o}dinger flow, we define
\begin{equation}\label{reducedcoef}a_k(t)=e^{-\pi i\cdot\delta L^{2\gamma}|k|^2t}\cdot\widehat{v}( \delta T_{\mathrm{kin}}\cdot t,k)
\end{equation} with $\widehat{v}$ as in (\ref{fourier}). Then $\textit{\textbf{a}}:=a_k(t)$, where $t\in[0,\Df]$ and $k\in\Zb_L^d$, will satisfy the equation
 \begin{equation}\label{akeqn}
 \left\{
\begin{aligned}
\partial_ta_k &= \Cc_+(\textit{\textbf{a}},\overline{\textit{\textbf{a}}},\textit{\textbf{a}})_k(t),\\
a_k(0) &=(a_k)_{\mathrm{in}}=\sqrt{\varphi_{\mathrm{in}}(k)}\cdot \gf_k,
\end{aligned}
\right.
\end{equation} with the nonlinearity
\begin{equation}\label{akeqn2} \Cc_\zeta(\textit{\textbf{f}},\textit{\textbf{g}},\textit{\textbf{h}})_k(t):=\frac{\delta}{2L^{d-\gamma}}\cdot(i\zeta)\sum_{k_1-k_2+k_3=k}\epsilon_{k_1k_2k_3}
e^{\zeta\pi i\cdot\delta L^{2\gamma}\Omega(k_1,k_2,k_3,k)t}f_{k_1}(t)g_{k_2}(t)h_{k_3}(t)
\end{equation}for $\zeta\in\{\pm\}$. Here in (\ref{akeqn2}) and below, the summation is taken over $(k_1,k_2,k_3)\in(\Zb_L^d)^3$, and 
\begin{equation}\label{defcoef0}\epsilon_{k_1k_2k_3}=
\left\{
\begin{aligned}+&1,&&\mathrm{if\ }k_2\not\in\{k_1,k_3\};\\
-&1,&&\mathrm{if\ }k_1=k_2=k_3;\\
&0,&&\mathrm{otherwise},
\end{aligned}
\right.\end{equation} and the resonance factor (assuming $k_1-k_2+k_3=k$) is
\begin{equation}\label{res}
\Omega=\Omega(k_1,k_2,k_3,k):=|k_1|^2-|k_2|^2+|k_3|^2-|k|^2=2\langle k_1-k,k-k_3\rangle.\end{equation} Note that $\epsilon_{k_1k_2k_3}$ is always supported in the set \begin{equation}\label{defset}\Gf:=\big\{(k_1,k_2,k_3):\mathrm{\ either\ }k_2\not\in\{k_1,k_3\},\mathrm{\ or\ }k_1=k_2=k_3\big\}.\end{equation}

The rest of this paper is focused on the system (\ref{akeqn})--(\ref{akeqn2}) for $\textit{\textbf{a}}$, with the relevant terms defined in (\ref{defcoef0})--(\ref{res}), in the time interval $t\in[0,\Df]$.
\subsection{Moments and cumulants} Let $a_k(t)$ be the (random) solution to (\ref{akeqn})--(\ref{akeqn2}). By gauge ($u\mapsto e^{i\theta}u$) and space translation ($u(x)\mapsto u(x-x_0)$) invariance it is easy to see that, for any $k_j\in\Zb_L^d$, $t_j\in\Rb$ and $\zeta_j\in\{\pm\}$ we have
\begin{equation}\label{gaugeinv}\Eb\bigg(\prod_{j=1}^ra_{k_j}(t_j)^{\zeta_j}\bigg)=0
\end{equation} unless $\zeta_1+\cdots +\zeta_r=0$ (in particular $r$ is even) and $\zeta_1k_1+\cdots +\zeta_rk_r=0$ (in particular $k_1=k_2=k$ when $r=2$). Note that any event $\Ff$ can be interpreted as a subset of initial data $u_{\mathrm{in}}$ as in (\ref{data}). We say an event $\Ff$ has gauge or space translation symmetry, if the corresponding subset is gauge or space translation invariant. Now we recall the standard notion of cumulants (see \cite{LS11}):
\begin{df}[Cumulants \cite{LS11}]\label{defcm} For any complex random variables $X_j\,(1\leq j\leq r)$, define the \emph{cumulant}
\begin{equation}\label{eqncm}\Kb(X_1,\cdots, X_r):=\sum_{\Pc}(-1)^{|\Pc|-1}(|\Pc|-1)!\prod_{A\in\Pc}\Eb\bigg(\prod_{j\in A}X_j\bigg),
\end{equation} where the sum is taken over all partitions $\Pc$ of $\{1,\cdots,r\}$. We then have the inversion formula (see for example \cite{LS11})
\begin{equation}\label{eqncm2}\Eb(X_1\cdots X_r):=\sum_{\Pc}\prod_{A\in\Pc}\Kb\big((X_j)_{j\in A}\big).
\end{equation}

Note that, due to (\ref{gaugeinv}), if $X_j=a_{k_j}(t_j)^{\zeta_j}$, then (\ref{eqncm}) also vanishes unless $\zeta_1+\cdots +\zeta_r=0$ and $\zeta_1k_1+\cdots +\zeta_rk_r=0$, and in both (\ref{eqncm}) and (\ref{eqncm2}) we can always assume that $(\zeta_j:j\in A)$ contains half $+$ and half $-$ for each $A\in\Pc$. In addition, if all the moments and cumulants are taken assuming a particular event $\Ff$ (i.e. taken with an additional factor $\mathbf{1}_{\Ff}$) with the same gauge and space translation symmetries, then all the above conclusions remain true.
\end{df}
\subsection{Norms, parameters, and notations} We next define the necessary norms, fix the parameters and introduce the relevant notations.
\subsubsection{Norms}\label{setupnorms} In the rest of this paper, we use $\widehat{\cdot}$ \emph{only} for the time Fourier transform, defined as\[\widehat{u}(\lambda)=\int_\Rb u(t) e^{-2\pi i\lambda t}\,\mathrm{d}t,\quad u(t)=\int_\Rb \widehat{u}(\lambda)e^{2\pi i\lambda t}\,\mathrm{d}\lambda,\] and similarly for higher dimensional versions.

For $\theta\in\Rb$, integers $\Lambda_1,\Lambda_2\geq 0$ and a function $F=F(t_1,\cdots,t_r,k)$ on $\Rb^r\times\Rb^d$, where $t_j\in\Rb$ are the time variables and $k\in\Rb^d$ is the vector variable, define the $\Xf^{\theta,\Lambda_1,\Lambda_2}$ norm 
\begin{equation}\label{normX}\|F\|_{\Xf^{\theta,\Lambda_1,\Lambda_2}}=\max_{|\rho|\leq\Lambda_1}\int(\max(\langle\lambda_1\rangle,\cdots,\langle\lambda_r\rangle))^{\theta}\cdot\|\langle k\rangle^{\Lambda_2}\partial_k^\rho\widehat{F}(\lambda_1,\cdots,\lambda_r,k)\|_{L_k^\infty}\,\mathrm{d}\lambda_1\cdots\mathrm{d}\lambda_r.\end{equation} In special cases where $F=F(t_1,\cdots,t_r)$ or $F=F(k)$, define accordingly the norms
\begin{equation}\label{normX1}\|F\|_{\Xf^\theta}=\int(\max(\langle\lambda_1\rangle,\cdots,\langle\lambda_r\rangle))^{\theta}\cdot|\widehat{F}(\lambda_1,\cdots,\lambda_r)|\,\mathrm{d}\lambda_1\cdots\mathrm{d}\lambda_r
\end{equation}
\begin{equation}\label{normX2}\|F\|_{\Sf^{\Lambda_1,\Lambda_2}}=\max_{|\rho|\leq\Lambda_1}\|\langle k\rangle^{\Lambda_2}\partial_k^\rho F(k)\|_{L_k^\infty}.
\end{equation}

Suppose $F=F(t_1,\cdots,t_r,k)$ is a function on $\Bc\times\Zb_L^d$ (or any subset of $\Rb^d$ instead of $\Zb_L^d$), where $\Bc\subset\Rb^r$ is a set of time variables $(t_j)$, then we may define the $\Xf^{\theta,\Lambda_1,\Lambda_2}(\Bc)$ norm of $F$ by
\begin{equation}\label{locnorm}\|F\|_{\Xf^{\theta,\Lambda_1,\Lambda_2}(\Bc)}=\inf\big\{\|G\|_{\Xf^{\theta,\Lambda_1,\Lambda_2}}:G|_{\Bc\times\Zb_L^d}=F\big\}.
\end{equation}
The same applies to norms $\Xf^\theta$ (so we can define the $\Xf^\theta(\Bc)$ norm) and $\Sf^{\Lambda_1,\Lambda_2}$ (so it can be defined for functions on $\Zb_L^d$). In addition, if $\Lambda_1=\Lambda_2=0$, we can define the $\Xf^{\theta,0,0}$ and $\Xf^{\theta,0,0}(\Bc)$ norms for functions $F$ depending on multiple variables $k_j\,(1\leq j\leq q)$ with the same formula.

Finally, for any interval $\Bc$, integer $\Lambda\geq 0$ and any function $\textit{\textbf{a}}=a_k(t)$ where $t\in \Bc$ and $k\in\Zb_L^d$, define the norm
\begin{equation}\label{defznorm}\|\textit{\textbf{a}}\|_{Z^\Lambda(\Bc)}=\sup_{t\in\Bc}\sup_{k\in\Zb_L^d}\langle k\rangle^{\Lambda}|a_k(t)|.
\end{equation}
\subsubsection{Parameters}\label{setupparam}
First recall the local well-posedness result proved in \cite{GIT20}. For our purpose, we state it below in the form of an (equivalent) blowup criterion.
\begin{prop}\label{wkelwp} Given any Schwartz initial data $\varphi_{\mathrm{in}}$, there exists a maximal time $\tau_{\max}>0$ (may be $\infty$) such that (\ref{wke}) has a unique solution $\varphi=\varphi(\tau,k)$ in Schwartz class, for $\tau\in[0,\tau_{\max})$. Moreover, if $\tau_{\mathrm{max}}<\infty$ then we have the blowup criterion
\begin{equation}\label{blowuptest}\lim_{\tau\to\tau_{\max}}\|\langle k\rangle^{s}\varphi(\tau,k)\|_{L_k^{2}}=\infty,\quad \forall s>d/2-1.
\end{equation}
\end{prop}
\begin{proof} See Appendix \ref{schwartzwke}.
\end{proof}

\smallskip
Throughout this paper, we will fix $\tau_{\mathrm{max}}$ as in Proposition \ref{wkelwp}, and fix $\tau_*\in(0,\tau_{\mathrm{max}})$ as in Theorem \ref{main}. We also fix the solution $\varphi(\tau,k)$ to (\ref{wke}) as in Theorem \ref{main} and Proposition \ref{wkelwp}. We will use the notation $C_0$ to denote arbitrary large constant that depends only on $(d,\gamma)$, and fix $\eta$ that is small enough depending only on $C_0$. Define
\begin{equation}\label{defC}\Cf:=\sup_{\tau\in[0,\tau_*]}\|\varphi(\tau)\|_{\Sf^{40d,40d}}<\infty,
\end{equation} with $\Sf^{40d,40d}$ defined as above, which depends only on $\tau_*$ and the initial data $\varphi_{\mathrm{in}}$. We use $C_1$ to denote arbitrary large constant that depends only on $(\tau_*,\eta)$ and $\Cf$. Fix also $\nu$ that is small enough depending on $C_1$, and $\delta$ that is small enough depending on $\nu$, such that
\begin{equation}\label{defD}\Df:=\frac{\tau_*}{\delta}\in\Zb.
\end{equation} We will use $C_2$ to denote arbitrary large constant that depends only on $\Df$ and $\delta$ (and also $\varphi_{\mathrm{in}}$), and assume $L$ is large enough depending on $C_2$. For integer $0\leq p\leq\Df$, define
\begin{equation}\label{parameters} \theta_p:=\eta^{20}(1+2^{-p}),\quad\Lambda_p:=(40d)^{\Df-p+1},
\end{equation} note that $\theta_{p+1}<\frac{4\theta_{p+1}+\theta_p}{5}<\frac{3\theta_{p+1}+2\theta_p}{5}<\frac{2\theta_{p+1}+3\theta_p}{5}<\frac{\theta_{p+1}+4\theta_p}{5}<\theta_p$. We choose the $\theta$ in Theorem \ref{main} as $\theta=\eta^{20}$, which is smaller than all the $\theta_p$ in (\ref{parameters}). Define also $N_p$ and $R_p$ such that\footnote{This choice is just for convenience; the same proof allows for any choice of $R_\Df$ between $(\log L)^{40d}$ and $L^{\nu^2(40d)^{-2\Df}}$.}\begin{equation}\label{parameters2}R_\Df:=\lfloor(\log L)^{40d}\rfloor;\quad N_p=R_p^{40d}\,(0\leq p\leq\Df),\quad R_{p-1}=N_p^{40d}\,(1\leq p\leq \Df).\end{equation} Finally, define
\begin{equation}\label{defgamma1}\gamma_0:=\min(\gamma,1-\gamma),\quad \gamma_1:=\min(2\gamma,1,2(d-1)(1-\gamma)).
\end{equation}
\subsubsection{Notations}\label{setupnotat} We will use $X=\Oc_j(Y)$, $X\lesssim_jY$ and $X\ll_jY$ etc., where $j\in\{0,1,2\}$, to indicate that the implicit constants depend only on $C_j$ (where $C_j$ are defined in Section \ref{setupparam}). Note that by definition we have $\nu\ll_11$ and $\delta^\nu\ll_11$. For multi-index $\rho=(\rho_1,\cdots,\rho_m)$, denote $|\rho|=\rho_1+\cdots+\rho_m$ and $\rho!=(\rho_1)!\cdots(\rho_m)!$, etc. For an index set $A$, we use the vector notation $\alpha[A]=(\alpha_j)_{j\in A}$ and $\mathrm{d}\alpha[A]=\prod_{j\in A}\mathrm{d}\alpha_j$, etc. The sign $\zeta\in\{\pm\}$ is understood as $\pm 1$ in algebraic manipulations. Denote $x^j$ as the coordinates of vector $(x^1,\cdots,x^d)\in\Rb^d$, and $z^+=z$ for a complex number $z$, and $z^-=\overline{z}$.

For $x\in\Rb$ fix the cutoff function
\[\chi_0=\chi_0(x)=\left\{
\begin{aligned}&0,&&\mathrm{for\ }|x|\geq 1;\\
&1,&&\mathrm{for\ }|x|\leq 1/2;\\
&\mathrm{smooth\ in\ }[0,1],&&\mathrm{for\ }1/2\leq|x|\leq 1.
\end{aligned}
\right.\] For $x\in\Rb^d$ define $\chi_0(x)=\chi_0(x^1)\cdots \chi_0(x^d)$, and $\chi_\infty=1-\chi_0$. In general, we may use $\chi$ to denote other cutoff functions with slightly different supports. These functions, as well as the other cutoff functions, will be in Gevrey class $2$ (i.e. the $m$-th order derivatives are bounded by $(2m)!$).
\subsection{Combinatorial structures} We define the combinatorial structures that will be used throughout this paper, including trees, gardens, couples and their decorations. For illustrations of structures defined in \cite{DH21,DH21-2,DH23} the reader may refer to those papers; here we only provide illustration for the new structures introduced in this paper.
\subsubsection{Trees, gardens and couples} We start with the notion of trees in \cite{DH21}, which will be restricted to ternary trees throughout this paper.
\begin{df}[Trees \cite{DH21}]\label{deftree} A \emph{ternary tree} $\Tc$ (we will simply say a \emph{tree} below) is a rooted tree where each non-leaf (or \emph{branching}) node has exactly three children nodes, which we shall distinguish as the \emph{left}, \emph{mid} and \emph{right} ones. A node $\mf$ is a \emph{descendant} of a node $\nf$, or $\nf$ is an \emph{ancestor} of $\mf$, if $\mf$ belongs to the subtree rooted at $\nf$ (we allow $\mf=\nf$). We say $\Tc$ is \emph{trivial} (and write $\Tc=\bullet$) if it has only the root, in which case this root is also viewed as a leaf.

We denote generic nodes by $\nf$, generic leaves by $\lf$, the root by $\rf$, the set of leaves by $\Lc$ and the set of branching nodes by $\Nc$. The \emph{order} of a tree $\Tc$ is defined by $n(\Tc)=|\Nc|$, so if $n(\Tc)=n$ then $|\Lc|=2n+1$ and $|\Tc|=3n+1$. Denote the parent node of a non-root node $\nf$ by $\nf^{\mathrm{pr}}$.

A tree $\Tc$ may have sign $+$ or $-$. If its sign is fixed then we decide the signs of its nodes as follows: the root $\rf$ has the same sign as $\Tc$, and for any branching node $\nf\in\Nc$, the signs of the three children nodes of $\nf$ from left to right are $(\zeta,-\zeta,\zeta)$ if $\nf$ has sign $\zeta\in\{\pm\}$. Once the sign of $\Tc$ is fixed, we will denote the sign of $\nf\in\Tc$ by $\zeta_\nf$. Define $\zeta(\Tc)=\prod_{\nf\in\Nc}(i\zeta_\nf)$. We also define the conjugate $\overline{\Tc}$ of a tree $\Tc$ to be the same tree but with opposite sign.
\end{df}
Next we define gardens (and the special case of couples), which are collections of trees with their leaves paired. The definition of couples first appeared in \cite{DH21}, and general gardens in \cite{DH21-2}.
\begin{df}[Gardens and couples \cite{DH21,DH21-2}]\label{defgarden} Given a sequence $(\zeta_1,\cdots,\zeta_{2R})$, where $\zeta_j\in\{\pm\}$ and exactly half of them are $+$, we define a \emph{garden} $\Gc$ of \emph{signature} $(\zeta_1,\cdots,\zeta_{2R})$, to be an ordered collection of trees $(\Tc_1,\cdots,\Tc_{2R})$, such that $\Tc_j$ has sign $\zeta_j$ for $1\leq j\leq 2R$, together with a partition $\Ps$ of the set of leaves in all $\Tc_j$ into two-element subsets (called \emph{pairings}) such that the two paired leaves have opposite signs. The \emph{width} of the garden is defined to be $2R$, which is always an even number. The \emph{order} $n(\Gc)$ of a garden $\Gc$ is the sum of orders of all $\Tc_j\,(1\leq j\leq 2R)$. We denote $\Lc=\Lc_1\cup\cdots\cup\Lc_{2R}$ to be the set of leaves and $\Nc=\Nc_1\cup\cdots\cup\Nc_{2R}$ to be the set of branching nodes, where $\Lc_j$ and $\Nc_j$ are the sets of leaves and branching nodes of $\Tc_j$, and define $\zeta(\Gc)=\prod_{\nf\in\Nc}(i\zeta_\nf)$.

A garden of width $2$ (i.e. $R=1$) is called a \emph{couple}, and denoted by $\Qc$. A non-couple garden (i.e. $R\geq 2$) is called \emph{proper}. For a couple we will always assume $(\zeta_1,\zeta_2)=(+,-)$ and denote $(\Tc_1,\Tc_2)$ by $(\Tc^+,\Tc^-)$. We also define a \emph{paired tree} to be a tree $\Tc$ whose leaves are all paired to each other according to the same pairing rule for gardens and couples, except for \emph{one unpaired leaf} which is called the \emph{lone leaf}. Note that if $\Tc$ is a paired tree, then it automatically forms a couple with the trivial tree $\bullet$. Define the conjugate of a couple $\Qc = (\Tc^+,\Tc^-)$ as $\overline{\Qc} = (\overline{\Tc^-},\overline{\Tc^+})$; for a paired tree $\Tc$ we also define its conjugate as $\overline{\Tc}$ with the same pairings, where $\overline{\Tc}$ is as in Definition \ref{deftree}.

We say a garden $\Gc$ is \emph{irreducible} if there exists no proper subset $A\subset\{1,\cdots,2R\}$ (i.e. $A\neq\varnothing$ and $A\neq\{1,\cdots,2R\}$) such that the leaves of $\Tc_j\,(j\in A)$ are all paired with each other. Clearly any garden $\Gc$ can be uniquely decomposed into irreducible gardens corresponding to a partition $(A_i)$ of $\{1,\cdots,2R\}$, where $|A_i|$ is even and $(\zeta_j:j\in A_i)$ contains half $+$ and half $-$ for each $i$. These are called the \emph{components} of $\Gc$. We say $\Gc$ is a \emph{multi-couple} if all its components are couples (i.e. have width $2$), and \emph{trivial} if each tree $\Tc_j$ is trivial. Note that a trivial garden must be a multi-couple formed by $R$ trivial couples.
\end{df}
Next we define decorations of couples and gardens (this is defined in \cite{DH21} for couples, and \cite{DH21-2} for general gardens).
\begin{df}[Decorations \cite{DH21,DH21-2}]\label{defdec} For any tree $\Tc$, a \emph{decoration} $\Ds$ of $\Tc$ is a mapping $k[\Tc]:\Tc\to\Zb_L^d$ consisting of $k_\nf$ for each node $\nf$ of $\Tc$, such that
\[k_{\nf_1}-k_{\nf_2}+k_{\nf_3}=k_\nf,\quad\mathrm{equivalently}\quad \zeta_{\nf_1}k_{\nf_1}+\zeta_{\nf_2}k_{\nf_2}+\zeta_{\nf_3}k_{\nf_3}=\zeta_{\nf}k_{\nf}\] for any branching node $\nf$ and its three children nodes $(\nf_1,\nf_2,\nf_3)$ from left to right (recall $\zeta_\nf$ defined in Definition \ref{deftree}). Clearly a decoration $\Ds$ is uniquely determined by the values of $(k_\lf)_{\lf\in\Lc}$. For $k\in\Zb_L^d$, we say $\Ds$ is a $k$-decoration if $k_\rf=k$ for the root $\rf$.

For any garden $\Gc=(\Tc_1,\cdots,\Tc_{2R})$, a \emph{decoration} $\Is$ of $\Gc$ is a mapping $k[\Gc]:\Gc\to\Zb_L^d$ consisting of $k_\nf$ for each node $\nf$ of $\Gc$, such that the restriction to each $\Tc_j$ is a decoration of $\Tc_j$, and $k_\lf=k_{\lf'}$ for each leaf pair $(\lf,\lf')$ of $\Gc$. For any $k_j\in\Zb_L^d\,(1\leq j\leq 2R)$, we say $\Is$ is a $(k_1,\cdots,k_{2R})$-decoration if $k_{\rf_j}=k_j$ where $\rf_j$ is the root of the tree $\Tc_j$. Note that in this case we necessarily have $\zeta_1k_1+\cdots +\zeta_{2R}k_{2R}=0$, where $\zeta_j$ are as in Definition \ref{defgarden}; in particular for couples we must have $k_1=k_2:=k$, thus for any $k\in\Zb_L^d$ we may also talk about $k$-decorations for couples. We can define decorations of paired trees in the same way (again requiring that $k_\lf=k_{\lf'}$ for each leaf pair $(\lf,\lf')$); in this case we must have $k_\rf=k_{\lf_{\mathrm{lo}}}:=k$ for the root $\rf$ and lone leaf $\lf_{\mathrm{lo}}$, so we may also talk about $k$-decorations of paired trees. Moreover, \emph{for couples and paired trees}, we shall define $k$-decorations for \emph{general vectors} $k\in\Rb^d$, by requiring $k_\nf-k\in\Zb_L^d$ instead of $k_\nf\in\Zb_L^d$, and keeping all other requirements.

For any decoration $\Ds$ or $\Is$ defined above, we define the coefficient
\begin{equation}\label{defepscoef}\epsilon_\Ds \,(\mathrm{or\ }\epsilon_\Is)=\prod_{\nf\in\Nc}\epsilon_{k_{\nf_1}k_{\nf_2}k_{\nf_3}}\end{equation} where $(\nf_1,\nf_2,\nf_3)$ is as above and $\epsilon_{k_1k_2k_3}$ is as in (\ref{defcoef0}); in the support of $\epsilon_\Ds$ (or $\epsilon_\Is$) we have $(k_{\nf_1},k_{\nf_2},k_{\nf_3})\in\Gf$ as in (\ref{defset}) for any $\nf\in\Nc$. Define the resonance factor $\Omega_\nf$ for each $\nf\in\Nc$ by
\begin{equation}\label{defomega}\Omega_\nf=\Omega_\nf(k_{\nf_1},k_{\nf_2},k_{\nf_3},k_\nf)=|k_{\nf_1}|^2-|k_{\nf_2}|^2+|k_{\nf_3}|^2-|k_\nf|^2.\end{equation} Note that, for a $k$-decoration of couples and paired trees, if we translate all the $k_\nf$, including $k$, by the same constant vector $\overline{k}\in\Rb^d$ (which may not be in $\Zb_L^d$), then the decoration conditions, as well as the values of $\Omega_\nf$, do not change.
\end{df}
\subsubsection{Layerings} Now we introduce the key new structure of \emph{layerings} imposed on gardens (and couples and paired trees), which corresponds to iterating the solution at different time slices.
\begin{df}[Layered gardens]\label{deflayer} A \emph{layering} of a garden $\Gc$ is a mapping $\Lf[\Gc]:\Gc\to\Zb$ consisting of $\Lf_\nf$ for each node $\nf$ of $\Gc$, such that $0\leq\Lf_{\nf'}\leq\Lf_{\nf}$ whenever $\nf'$ is a child of $\nf$ and $\Lf_{\lf}=\Lf_{\lf'}$ whenever $(\lf,\lf')$ are two paired leaves, see Figure \ref{fig:layergarden}. We call $\Lf_\nf$ the \emph{layer} of $\nf$. A garden $\Gc$ with a layering fixed is called a \emph{layered} garden. We also define layerings of paired trees in the same way. Moreover, define a \emph{pre-layering} (of a garden or a paired tree) to be the same as layering except we define $\Lf_\nf$ only for \emph{branching nodes} $\nf$, and according to the same rules as above.
\begin{figure}[h!]
\includegraphics[scale=0.5]{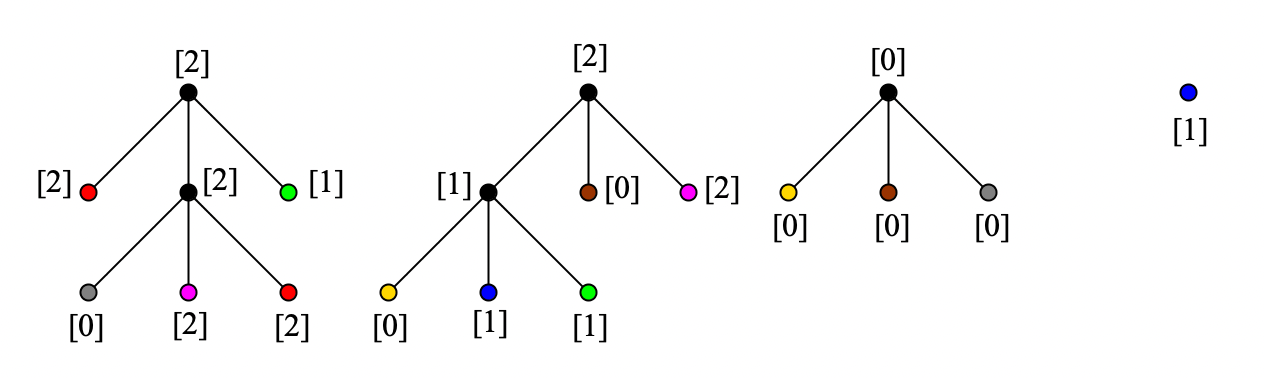}
\caption{A layered garden, as in Definition \ref{deflayer}. The signs of trees are $(+,-,+,-)$ from left to right.}
\label{fig:layergarden}
\end{figure}
\end{df}
We will consider an important class of layered gardens, namely the \emph{canonical layered gardens}, which includes the ones that occur in the actual inductive scheme (see Proposition \ref{propansatz}).
\begin{df}[Canonical layered gardens]\label{defcanon} Let $p\geq 0$. We inductively define a collection $\Gs_p$ of irreducible proper layered gardens (see Definition \ref{defgarden}); gardens in $\Gs_p$ will be called \emph{canonical layered gardens of depth $p$}. Let $\Gs_0=\varnothing$. Suppose $\Gs_p$ is defined for some $p\geq 0$, to define $\Gs_{p+1}$ we consider the following process.

First choose any $2R\,(R\geq 2)$ trees $\Tc_j^p\,(1\leq j\leq 2R)$ with signs $\zeta_j$ as in Definition \ref{defgarden} (half $+$ and half $-$), and put all their branching nodes in  layer $p$. Next, divide the set of leaves of all these trees into subsets $B_i$ such that the signs of leaves in each subset are half $+$ and half $-$; we also require that there is no proper subset $A\subset\{1,\cdots,2R\}$ such that the set of leaves of $\Tc_j^p$ for all $j\in A$ equals the union of \emph{some of the sets} $B_i$. Then, we turn any \emph{two-element} subset $B_i$ into a pairing, putting the involved leaves again in layer $p$. For any subset $B_i$ of \emph{at least four elements}, we replace its leaves by an irreducible proper layered garden in $\Gs_p$, such that the roots of the trees in this garden (including the signs) are just the elements of $B_i$. This results in a big proper layered garden of width $2R$, which is easily shown to be irreducible, and we define $\Gs_{p+1}$ to consist of all the irreducible proper layered gardens that can be obtained in this way, see Figure \ref{fig:canongarden}.

We also define, for $p\geq 0$, a collection $\Cs_{p+1}$ of layered \emph{couples}, by choosing $R=1$ and following the same procedure in the above construction (i.e. choosing two trees of signs $+$ and $-$, dividing the set of leaves into pairings and subsets of at least four elements, and replacing the latters by layered gardens in $\Gs_p$). Couples in $\Cs_{p+1}$ are called \emph{canonical layered couples}.
\begin{figure}[h!]
\includegraphics[scale=0.36]{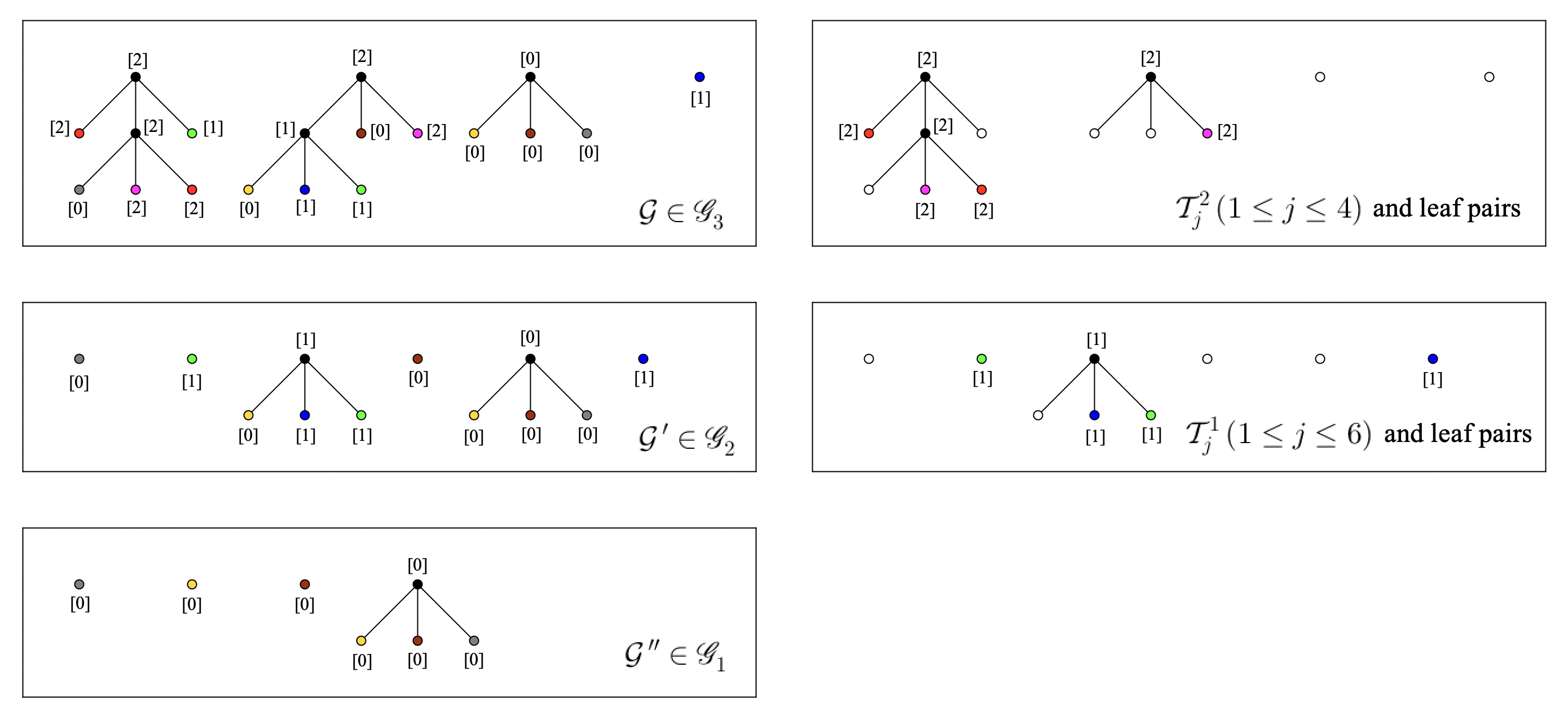}
\caption{An illustration showing that the layered garden in Figure \ref{fig:layergarden} is canonical, following the inductive construction in Definition \ref{defcanon}.}
\label{fig:canongarden}
\end{figure}
\end{df}
Finally, we define an important quantity $\Kc_\Gc$ for gardens $\Gc$ (as well as $\Kc_\Qc^*$ and $\Kc_\Tc^*$ for couples $\Qc$ and paired trees $\Tc$), which is similar to the ones in \cite{DH21,DH21-2}, but also takes into account the layers. This quantity is the main subject of study in the proof of the current paper.
\begin{df}[Definition of $\Kc_\Gc$, $\Kc_\Qc^*$ and $\Kc_\Tc^*$]\label{defkg} Fix an integer $p\geq 0$, let $\Gc$ be a layered garden consisting of trees $\Tc_j\,(1\leq j\leq 2R)$ with sign $\zeta_j$, such that all nodes of $\Gc$ are in layer at most $p$. Fix also the input functions $F_q=F_q(k)$ for $0\leq q\leq p$, then for any $k_1,\cdots,k_{2R}\in\Zb_L^d$ (satisfying $\zeta_1k_1+\cdots +\zeta_{2R}k_{2R}=0$, where $(\zeta_1,\cdots,\zeta_{2R})$ is the signature of $\Gc$) and $t>p$, we may define 
\begin{equation}\label{defkg0}\Kc_\Gc=\Kc_\Gc(t,k_1,\cdots,k_{2R})=\bigg(\frac{\delta}{2L^{d-\gamma}}\bigg)^n\zeta(\Gc)\sum_\Is\epsilon_\Is\int_{\Ic}\prod_{\nf\in\Nc}e^{\pi i\zeta_\nf\cdot\delta L^{2\gamma}\Omega_\nf t_\nf}\,\mathrm{d}t_\nf\cdot\prod_{\lf\in\Lc}^{(+)}F_{\Lf_\lf}(k_\lf).
\end{equation} Here in (\ref{defkg0}) $n$ is the order of $\Gc$, $\zeta(\Gc)$ is as in Definition \ref{defgarden} and $\epsilon_\Is$ as in Definition \ref{defdec}. The sum is taken over all $(k_1,\cdots,k_{2R})$-decorations $\Is$ of $\Gc$, $\Omega_\nf$ is as in (\ref{defomega}), and the last product is taken over all $\lf\in\Lc$ with sign $+$. The time domain $\Ic$ is defined by
\begin{equation}\label{timegarden}\Ic=\left\{t[\Nc]:\Lf_{\nf}<t_{\nf}<\Lf_{\nf}+1\quad\mathrm{and}\quad 0<t_{\nf'}<t_\nf<t,\mathrm{\ whenever\ \nf'\ is\ a\ child\ of\ \nf}\right\},
\end{equation} where $\Lf_\nf$ is the layer of $\nf$ as in Definition \ref{deflayer}. Note that the definition of $\Ic$ depends only on the \emph{pre-layering} of $\Gc$.

Now we can define $\Kc_\Qc$ for couples $\Qc$ via (\ref{defkg0}), in which case we can allow $k_1=k_2:=k$ to be \emph{arbitrary vector in $\Rb^d$, provided that the input functions $F_r=F_r(k)$ for $0\leq r\leq p$ are defined on $\Rb^d$, and that the sum in (\ref{defkg0}) is still taken over all $k$-decorations $\Is$ of $\Qc$ as in Definition \ref{defdec}.} For later use, we will extend this definition to $\Kc_\Qc^*$, and also define similar expression $\Kc_\Tc^*$ for paired trees $\Tc$, as follows.

First, let $\Qc=(\Tc^+,\Tc^-)$ be a layered couple such that the two roots are in layer at most $q$ and $q'$ respectively where $q,q'\leq p$, and fix the input functions $F_r=F_r(k)$ for $0\leq r\leq p$ (which are defined on $\Rb^d$ as stated above). Then for all $k\in\Rb^d$ and $t>q$, $s>q'$, we may define $\Kc_\Qc^*=\Kc_\Qc^*(t,s,k)$, in the same way as (\ref{defkg0}), except that the time domain is now
\begin{multline}\label{timecouple}\Ic^*=\left\{t[\Nc]:\Lf_{\nf}<t_{\nf}<\Lf_{\nf}+1\quad\mathrm{and}\quad 0<t_{\nf'}<t_\nf,\mathrm{\ whenever\ \nf'\ is\ a\ child\ of\ \nf},\right.\\\left.\mathrm{and}\quad t_{\nf}<t\ \mathrm{(or\ }s),\mathrm{\ whenever\ \nf\ is\ a\ node\ of\ \Tc^+\ (or\ \Tc^-)}\right\}.
\end{multline} This gives an extension of the expression $\Kc_\Qc$ in (\ref{defkg0}) when $\Gc=\Qc$ is a couple and $k_1=k_2=k\in\Rb^d$, such that the latter corresponds to the special case $t=s>p$.

Similarly, let $\Tc$ be a layered paired tree such that the root is in layer at most $q$ and the lone leaf is in layer $q'$ where $q'\leq q\leq p$, and fix the input functions $F_r=F_r(k)$ for $0\leq r\leq p$ which are defined on $\Rb^d$. Then \emph{for all $k\in\Rb^d$ and $(t,s)$ satisfying $t>q$, $s<q'+1$ and $t>s$}, we may define $\Kc_\Tc^*=\Kc_\Tc^*(t,s,k)$ in the same way as (\ref{defkg0}), except that (i) the sum is taken over all $k$-decorations $\Ds$ of $\Tc$ where again $k$ may not be in $\Zb_L^d$, (ii) the last product is taken over all leaves $\lf\in\Lc$ with sign $+$ \emph{other than the lone leaf $\lf_{\mathrm{lo}}$}, and (iii) the time domain is now \begin{multline}\label{timepairtree}\Dc^*=\big\{t[\Nc]:\Lf_{\nf}<t_{\nf}<\Lf_{\nf}+1\quad\mathrm{and}\quad 0<t_{\nf'}<t_\nf<t,\mathrm{\ whenever\ \nf'\ is\ a\ child\ of\ \nf},\\\mathrm{and}\quad t_{\nf_{\mathrm{lo}}^{\mathrm{pr}}}>s,\mathrm{\ where\ \nf_{\mathrm{lo}}^{\mathrm{pr}}\ is\ the\ parent\ node\ of\ the\ lone\ leaf\ \lf_{\mathrm{lo}}}\big\}.
\end{multline}
\end{df}
\subsection{The inductive scheme}\label{sectioninduct} With all the above preparations, we can now describe the inductive scheme which is the central component in the proof of Theorem \ref{main}. 

Recall that we are considering the system (\ref{akeqn})--(\ref{akeqn2}) for $t\in[0,\Df]$. Moreover, as the distinction between different $C_j$ in Section \ref{setupparam} will be important, we also recall it here:
\[C_0\ll\eta^{-1}\ll C_1\ll\delta^{-1}\ll C_2.\] Here $C_0$ is any ``absolute" constant depending only on $(d,\gamma)$, which can be absorbed by $\eta$; $C_1$ is much larger than $\eta^{-1}$ and also depends on the uniform bound $\Cf$ of the solution $\varphi$ to (\ref{wke}) defined in (\ref{defC}), but can still be absorbed by the small parameter $\delta$ indicating the length of each short interval; $C_2$ is the largest and may depend on the number $\Df$ of intervals which may be much larger than $\delta^{-1}$, but can still be absorbed by any quantity that increases with $L$.
\begin{prop}\label{propansatz} Recall the parameters defined in (\ref{parameters2}). For $0\leq p\leq\Df$ and $2\leq R\leq R_p$, we can define an event $\Ff_p$, functions $\Rs_p=\Rs_p(k)$ and $\Es_p=\Es_p(k)$, and $\widetilde{\Es_p}=\widetilde{\Es_p}(\zeta_1,\cdots,\zeta_{2R},k_1,\cdots,k_{2R})$, such that $\Rs_p(k)$ is real valued, and the followings hold for each $0\leq p\leq\Df$:
\begin{enumerate}[{(1)}]
\item Assuming $\Ff_p$, the system (\ref{akeqn})--(\ref{akeqn2}) has a unique solution $\textit{\textbf{a}}=a_k(t)$ on $[0,p]$. We also have $\Ff_p\supset\Ff_{p+1}$ and $\Pb(\Ff_p)\geq 1-e^{-N_p}$. Moreover each $\Ff_p$ has gauge and space translation symmetry; the expectations $\Eb$ and cumulants $\Kb$ below are taken assuming $\Ff_p$, and they come with subscript $p$.
\item Let the cumulant $\Kb$ be defined as in (\ref{eqncm}), and $\Kb_p$ as in (1). For any $k,\ell\in\Zb_L^d$ we have
\begin{equation}\label{ansatz1}\Kb_p(a_{k}(p),\overline{a_{\ell}(p)})=\Eb_p(a_{k}(p)\overline{a_{\ell}(p)})=\mathbf{1}_{k=\ell}\cdot\big(\varphi(\delta p,k)+\Rs_p(k)+\Es_p(k)\big),
\end{equation} where $\varphi$ is the solution to (\ref{wke}), the sub-leading term $\Rs_p(k)$ and remainder $\Es_p(k)$ satisfy
\begin{equation}\label{ansatz2}\|\Rs_p(k)\|_{\Sf^{\Lambda_p,\Lambda_p}}\leq L^{-\theta_p},\quad |\Es_p(k)|\leq \langle k\rangle^{-\Lambda_p}\cdot e^{-N_p}.
\end{equation}
\item For any $2\leq R\leq R_p$, signs $\zeta_j\,(1\leq j\leq 2R)$ as in Definition \ref{defgarden}, and $k_j\in\Zb_L^d\,(1\leq j\leq 2R)$, we have 
\begin{equation}\label{ansatz3}\Kb_p\big(a_{k_1}(p)^{\zeta_1},\cdots,a_{k_{2R}}(p)^{\zeta_{2R}}\big)=\sum_\Gc\Kc_\Gc(p,k_1,\cdots,k_{2R})+\widetilde{\Es_p}.\end{equation} Here the sum is taken over all \emph{canonical layered gardens} $\Gc\in\Gs_p$ of signature $(\zeta_1,\cdots,\zeta_{2R})$ and \emph{order at most $N_p$} (in particular all nodes of $\Gc$ are in layer at most $p-1$), and $\Kc_\Gc$ is as in (\ref{defkg0}) in Definition \ref{defkg} with input functions $F_q(k)=\varphi(\delta q,k)+\Rs_q(k)$ for $0\leq q\leq p-1$. The remainder $\widetilde{\Es_p}$ is supported in the set $\zeta_1+\cdots +\zeta_{2R}=0$ and $\zeta_1k_1+\cdots +\zeta_{2R}k_{2R}=0$, and satisfies that
\begin{equation}\label{ansatz4}
|\widetilde{\Es_p}|\leq\big(\langle k_1\rangle\cdots\langle k_{2R}\rangle\big)^{-\Lambda_p}\cdot e^{-N_p}.
\end{equation}
\item In addition, for any $n\leq N_p$, consider the expression on the right hand side of (\ref{ansatz3}), but with the summation taken over all canonical layered gardens $\Gc$ \emph{of order $n$} and without the remainder $\widetilde{\Es_p}$. Then this expression, denoted by $\Ks_{p,n}$, satisfies that
\begin{equation}\label{ansatz5}
|\Ks_{p,n}|\leq \big(\langle k_1\rangle\cdots\langle k_{2R}\rangle\big)^{-\Lambda_p}\cdot L^{-(R-1)/6}\cdot \delta^{\nu n}.
\end{equation}
\end{enumerate}
\end{prop}
\subsection{The Duhamel expansion}\label{secduhamel} The proof of Proposition \ref{propansatz} will occupy most of the remaining parts of this paper. Since the statements (1)--(4) are trivially true for $p=0$ (with $\Rs_0=\Es_0=\widetilde{\Es_0}=0$ and $\Ff_0$ the trivial true event) due to the independence and Gaussian assumption (\ref{data}), we only need to justify the inductive step. Thus, for the rest of this paper (except Section \ref{proofmain}), we will fix $0\leq p\leq\Df-1$ and assume (1)--(4) of Proposition \ref{propansatz} hold for any $0\leq q\leq p$, and we just need to prove the same for $p+1$.

To achieve this, we need to expand the solution $(a_k(p+1))$ in terms of the previous data $(a_k(p))$. Assuming the event $\Ff_p$, for any tree $\Tc$ and any $t>p$, define the expression
\begin{equation}\label{defjt}
(\Jc_\Tc)_k(t)=\bigg(\frac{\delta}{2L^{d-\gamma}}\bigg)^n\zeta(\Tc)\sum_{\Ds}\epsilon_{\Ds}\int_{\Dc}\prod_{\nf\in\Nc}e^{\pi i\zeta_\nf\cdot\delta L^{2\gamma}\Omega_\nf t_\nf}\,\mathrm{d}t_\nf\cdot\prod_{\lf\in\Lc}a_{k_\lf}(p).
\end{equation} Here in (\ref{defjt}), $n$ is the order of $\Tc$, $\zeta(\Tc)$ is as in Definition \ref{deftree} and $\epsilon_\Ds$ as in Definition \ref{defdec}. The sum is taken over all $k$-decorations $\Ds$ of $\Tc$, and $\Omega_\nf$ is as in (\ref{defomega}). The integral domain $\Dc$ is defined by
\begin{equation}\label{timetree}\Dc=\left\{t[\Nc]:p<t_{\nf'}<t_\nf<t,\mathrm{\ whenever\ \nf'\ is\ a\ child\ of\ \nf}\right\}.
\end{equation} Now, define
\begin{equation}\label{defjn}(\Jc_n)_k(t)=\sum_{n(\Tc^+)=n}(\Jc_{\Tc^+})_k(t)
\end{equation} where the sum is taken over all trees $\Tc^+$ of order $n$ and sign $+$. Define the function $\textit{\textbf{b}}=b_k(t)$, where $t\in[p,p+1]$ and $k\in\Zb_L^d$, by
\begin{equation}\label{defbk}a_k(t)=a_k^{\mathrm{tr}}(t)+b_k(t),\quad a_k^{\mathrm{tr}}(t):=\sum_{n=0}^{N_{p+1}}(\Jc_n)_k(t).
\end{equation} It is then easy to see that the system (\ref{akeqn})--(\ref{akeqn2}) for $\textit{\textbf{a}}=a_k(t)$ on $[p,p+1]$ is equivalent to the following equation for $\textit{\textbf{b}}$:
\begin{equation}\label{eqnbk}
\textit{\textbf{b}}=\Ls_0+\Ls_1\textit{\textbf{b}}+\Ls_2(\textit{\textbf{b}},\textit{\textbf{b}})+\Ls_3(\textit{\textbf{b}},\textit{\textbf{b}},\textit{\textbf{b}}),
\end{equation} or equivalently 
\begin{equation}\label{eqnbk2}\textit{\textbf{b}}=(1-\Ls_1)^{-1}(\Ls_0+\Ls_2(\textit{\textbf{b}},\textit{\textbf{b}})+\Ls_3(\textit{\textbf{b}},\textit{\textbf{b}},\textit{\textbf{b}})),
\end{equation}where 
\begin{equation}\label{eqnbk1.5}(\Ls_j(\textit{\textbf{b}},\cdots,\textit{\textbf{b}}))_k(t)=\sum_{(j)}\int_p^t\Cc_+(\textit{\textbf{u}},\textit{\textbf{v}},\textit{\textbf{w}})_k(s)\,\mathrm{d}s.\end{equation} The sums in (\ref{eqnbk1.5}) are taken over $(\textit{\textbf{u}},\textit{\textbf{v}},\textit{\textbf{w}})$, where each of them equals either $\textit{\textbf{b}}$ or $\Jc_n$ for some $0\leq n\leq N_{p+1}$. Assume in the sum $\sum_{(j)}$ that exactly $j$ inputs in $(\textit{\textbf{u}},\textit{\textbf{v}},\textit{\textbf{w}})$ equal $\textit{\textbf{b}}$, and in the sum $\sum_{(0)}$ we further require that $(\textit{\textbf{u}},\textit{\textbf{v}},\textit{\textbf{w}})=(\Jc_{n_1},\Jc_{n_2},\Jc_{n_3})$ with $n_1+n_2+n_3\geq N_{p+1}$. Note that each $\Ls_j$ is a $j$-multilinear form over $\Rb$.

To calculate the cumulants
\begin{equation}\label{cmatp+1}\Kb_{p+1}\big(a_{k_1}(p+1)^{\zeta_1},\cdots,a_{k_{2R}}(p+1)^{\zeta_{2R}}\big)\end{equation} in Proposition \ref{propansatz}, we first calculate the contribution of truncated tree terms, namely
\begin{equation}\label{cmattr}\Kb_{p}\big(a_{k_1}^{\mathrm{tr}}(p+1)^{\zeta_1},\cdots,a_{k_{2R}}^{\mathrm{tr}}(p+1)^{\zeta_{2R}}\big).\end{equation} Now, by  Lemma \ref{propertycm}, Definition \ref{defcanon}, gauge and space translation symmetry of $\Ff_p$, and the induction hypothesis, we can write the cumulant (\ref{cmattr}) as the sum of one explicit term, plus remainder terms that contain one of $\Es_p$ or $\widetilde{\Es_p}$. For $R\geq 2$ this explicit term equals
\[\sum_\Gc\Kc_\Gc(p+1,k_1,\cdots,k_{2R}),\] where $\Kc_\Gc$ is as in (\ref{defkg0}) in Definition \ref{defkg} with input functions $F_q(k)=\varphi(\delta q,k)+\Rs_q(k)$ for $0\leq q\leq p$, and $\Gc$ runs over a specific collection $\Gs_{p+1}^{\mathrm{tr}}$ of canonical layered gardens in $\Gs_{p+1}$; for $R=1$ we should replace $\Gc$ by a couple $\Qc$ which runs over a specific collection $\Cs_{p+1}^{\mathrm{tr}}$ of canonical layered couples in $\Cs_{p+1}$. Here $\Gs_{p+1}^{\mathrm{tr}}$ and $\Cs_{p+1}^{\mathrm{tr}}$ are defined as follows:
\begin{df}\label{defgtr} For any canonical layered garden $\Gc\in\Gs_{p+1}$ (or couple $\Qc\in\Cs_{p+1}$), we define $\Gc\in\Gs_{p+1}^{\mathrm{tr}}$ (or $\Qc\in\Cs_{p+1}^{\mathrm{tr}}$) if and only if, following the construction process in Definition \ref{defcanon}, we have that (i) each initial tree $\Tc_j^p$ in Definition \ref{defcanon} has order at most $N_{p+1}$, and (ii) each layered garden in $\Gs_p$ that replaces the leaves in some subset $B_i$ has order at most $N_p$.
\end{df}
We illustrate the above expansion by a concrete example, which corresponds to the canonical layered garden in Figure \ref{fig:canongarden}. Consider (\ref{cmattr}) with $p=2$ and $R=2$, and $\zeta_j=(-1)^{j-1}$. Replace the four $a^{\mathrm{tr}}$ factors in (\ref{cmattr}), from left to right, by $\Jc_{\Tc_j^2}$ for the four trees $\Tc_j^2$ in Figure \ref{fig:canongarden} (where $\Tc_2^2$ and $\Tc_4^2$ have sign $-$, corresponding to $\zeta_2=\zeta_4=-$). Note that the root of $\Tc_j^2$ is decorated by $k_j$, and we denote the decoration of the nodes in $\Tc_j^2$ by $k_{ji}\,(i\geq 1)$ with numbering from top to bottom and left to right; similarly we denote the time variables associated with $\Tc_j^2$ by $t_j'<t_j$ etc. This then leads to the expression
\begin{align}&\sum_{\substack{k_1=k_{11}-k_{12}+k_{13}\\ k_{12}=k_{14}-k_{15}+k_{16}}}\sum_{k_2=k_{21}-k_{22}+k_{23}}\bigg(\frac{\delta}{2L^{d-\gamma}}\bigg)^3(-i)\cdot\epsilon_{k_{11}k_{12}k_{13}}\epsilon_{k_{14}k_{15}k_{16}}\epsilon_{k_{21}k_{22}k_{23}}\nonumber\\
&\qquad\,\,\times\int_{2<t_1'<t_1<3}e^{\pi i\delta L^{2\gamma}(|k_{12}|^2-|k_{14}|^2+|k_{15}|^2-|k_{16}|^2)t_1'}\cdot e^{\pi i\delta L^{2\gamma}(|k_{11}|^2-|k_{12}|^2+|k_{13}|^2-|k_1|^2)t_1}\,\mathrm{d}t_1'\mathrm{d}t_1\nonumber\\
&\qquad\,\,\times \int_{2<t_2<3}e^{\pi i\delta L^{2\gamma}(|k_{2}|^2-|k_{21}|^2+|k_{22}|^2-|k_{23}|^2)t_2}\,\mathrm{d}t_2\nonumber\\
\label{cumulantexample}&\qquad\,\,\times\Kb_2\big(a_{k_{11}}(2)\overline{a_{k_{14}}(2)}a_{k_{15}}(2)\overline{a_{k_{16}}(2)}a_{k_{13}}(2),
\overline{a_{k_{21}}(2)}a_{k_{22}}(2)\overline{a_{k_{23}}(2)},a_{k_3}(2),\overline{a_{k_4}(2)}\big).
\end{align} Now we apply Lemma \ref{propertycm} to the cumulant expression in (\ref{cumulantexample}); corresponding to the pairings between $\Tc_j^2$ in Figure \ref{fig:canongarden}, we select the term corresponding to
\[\Kb_2(a_{k_{11}}(2),\overline{a_{k_{16}}(2)})\cdot\Kb_2(a_{k_{15}}(2),\overline{a_{k_{23}}(2)})\cdot\Kb_2\big(\overline{a_{k_{14}}(2)},a_{k_{13}}(2),\overline{a_{k_{21}}(2)},a_{k_{22}}(2),a_{k_3}(2),\overline{a_{k_4}(2)}\big),\] which also determines the sets $B_i$ in Definition \ref{defcanon}. By induction hypothesis, the first two cumulants above can be replaced by $\mathbf{1}_{k_{11}=k_{16}}\cdot F_2(k_{11})$ etc. (ignoring $\Es_p$ and $\widetilde{\Es_p}$, same below); the last cumulant is the sum of \[\Kc_{\Gc'}(2,k_{14},k_{13},k_{21},k_{22},k_3,k_4)\] over all $\Gc'\in\Gs_2$ (with signature $(-,+,-,+,+,-)$ and an upper bound for $n(\Gc')$). If we select the garden $\Gc'\in\Gs_2$ in Figure \ref{fig:canongarden}, we get the factor
\begin{align}&\sum_{k_{21}=k_{24}-k_{25}+k_{26}}\sum_{k_3=k_{31}-k_{32}+k_{33}}\bigg(\frac{\delta}{2L^{d-\gamma}}\bigg)^2(+1)\cdot \epsilon_{k_{24}k_{25}k_{26}}\epsilon_{k_{31}k_{32}k_{33}}\nonumber\\
&\qquad\,\,\times\int_{1<t_2'<2}\int_{0<t_3<1}e^{\pi i\delta L^{2\gamma}(|k_{21}|^2-|k_{24}|^2+|k_{25}|^2-|k_{26}|^2)t_2'}\cdot e^{\pi i\delta L^{2\gamma}(|k_{31}|^2-|k_{32}|^2+|k_{33}|^2-|k_{3}|^2)t_3}\,\mathrm{d}t_2'\mathrm{d}t_3\nonumber\\
&\qquad\,\,\times\mathbf{1}_{k_{14}=k_{33}}F_0(k_{14})\cdot \mathbf{1}_{k_{13}=k_{26}}F_1(k_{13})\cdot \mathbf{1}_{k_{24}=k_{31}}F_0(k_{24})\cdot \mathbf{1}_{k_{25}=k_4}F_1(k_{25})\cdot \mathbf{1}_{k_{22}=k_{32}}F_0(k_{22}).\nonumber
\end{align} Plugging this into (\ref{cumulantexample}), we get the total expression
\begin{align*}
&\sum_{\substack{k_1=k_{11}-k_{12}+k_{13}\\ k_{12}=k_{14}-k_{15}+k_{16}}}\sum_{\substack{k_2=k_{21}-k_{22}+k_{23}\\ k_{21}=k_{24}-k_{25}+k_{26}}}\sum_{k_3=k_{31}-k_{32}+k_{33}}\bigg(\frac{\delta}{2L^{d-\gamma}}\bigg)^5(-i)\cdot \epsilon_{k_{11}k_{12}k_{13}}\epsilon_{k_{14}k_{15}k_{16}}\epsilon_{k_{21}k_{22}k_{23}}\\
&\qquad\,\,\times \int_{2<t_1'<t_1<3}e^{\pi i\delta L^{2\gamma}(|k_{12}|^2-|k_{14}|^2+|k_{15}|^2-|k_{16}|^2)t_1'}\cdot e^{\pi i\delta L^{2\gamma}(|k_{11}|^2-|k_{12}|^2+|k_{13}|^2-|k_1|^2)t_1}\,\mathrm{d}t_1'\mathrm{d}t_1\\
&\qquad\,\,\times \int_{1<t_2'<2}\int_{2<t_2<3}e^{\pi i\delta L^{2\gamma}(|k_{2}|^2-|k_{21}|^2+|k_{22}|^2-|k_{23}|^2)t_2}\cdot e^{\pi i\delta L^{2\gamma}(|k_{21}|^2-|k_{24}|^2+|k_{25}|^2-|k_{26}|^2)t_2'}\,\mathrm{d}t_2'\mathrm{d}t_2\\
&\qquad\,\,\times \epsilon_{k_{24}k_{25}k_{26}}\epsilon_{k_{31}k_{32}k_{33}}\int_{0<t_3<1}e^{\pi i\delta L^{2\gamma}(|k_{31}|^2-|k_{32}|^2+|k_{33}|^2-|k_{3}|^2)t_3}\,\mathrm{d}t_3\cdot\mathbf{1}_{k_{11}=k_{16}}F_2(k_{11})\cdot\mathbf{1}_{k_{15}=k_{23}}F_2(K_{15})\\&\qquad\,\,\times \mathbf{1}_{k_{14}=k_{33}}F_0(k_{14})\cdot \mathbf{1}_{k_{13}=k_{26}}F_1(k_{13})\cdot \mathbf{1}_{k_{24}=k_{31}}F_0(k_{24})\cdot \mathbf{1}_{k_{25}=k_4}F_1(k_{25})\cdot \mathbf{1}_{k_{22}=k_{32}}F_0(k_{22}),
\end{align*} which is exactly $\Kc_\Gc(3,k_1,k_2,k_3,k_4)$ for the canonical layered garden $\Gc$ illustrated in Figure \ref{fig:canongarden}.

Now we recall that, due to the smallness of $\delta$ relative to $\Cf$ defined in (\ref{defC}),  the solution $\varphi((p+1)\delta,k)$ to the wave kinetic equation (\ref{wke}) admits a power series expansion in terms of the data $\varphi(p\delta,k)$, namely, for $k\in\Rb^d$:
\begin{equation}\label{wketaylor}\varphi((p+1)\delta,k)=\sum_{n=0}^\infty\Uc_n(p+1,k),
\end{equation} where $\Uc_n(t,k)$ is defined for $t\in[p,p+1]$ such that
\begin{equation}\label{wketaylor2}\Uc_0(t,k)=\varphi(p\delta,k),\quad \Uc_{n}(t,k)=\delta\sum_{n_1+n_2+n_3=n-1}\int_p^t\Kc(\Uc_{n_1}(t'),\Uc_{n_2}(t'),\Uc_{n_3}(t'))(k)\,\mathrm{d}t',
\end{equation} where $\Kc$ is defined as in (\ref{wke2}). It is easy to see that
\begin{equation}\label{taylorbound}\sup_{t\in[p,p+1]}\|\Uc_n(t,\cdot)\|_{\Sf^{4\Lambda_{p+1},4\Lambda_{p+1}}}\leq C_2\cdot (C_1\delta)^n
\end{equation} for any $n\geq 0$.

We can now state the main estimates needed in the inductive step. They are listed in the following propositions:
\begin{prop}\label{layergarden} Given any $1\leq R\leq 10R_{p+1}$ and $n\leq N_{p+1}^4$, consider the expression
\begin{equation}\label{expgarden}\Ks_{p+1,n}^{\mathrm{tr}}(\zeta_1,\cdots,\zeta_{2R},k_1,\cdots,k_{2R}):=\sum_{
\Gc}\Kc_\Gc(p+1,k_1,\cdots,k_{2R}).
\end{equation} Here, for $R\geq 2$, $\Gc$ runs over all canonical layered gardens in $\Gs_{p+1}^{\mathrm{tr}}$ of signature $(\zeta_1,\cdots,\zeta_{2R})$ and order $n$; for $R=1$ we replace $\Gc$ by a couple $\Qc$ which runs over all canonical layered couples in $\Cs_{p+1}^{\mathrm{tr}}$ of order $n$, where $\Gs_{p+1}^{\mathrm{tr}}$ and $\Cs_{p+1}^{\mathrm{tr}}$ are defined in Definition \ref{defgtr}. Note that when $n\leq N_{p+1}^4$, the requirement (ii) in Definition \ref{defgtr}, namely each layered garden in $\Gs_p$ that replaces the leaves in some subset $B_i$ has order at most $N_p$, is redundant. Then we have that:
\begin{enumerate}[{(1)}]
\item If $R\geq 2$, then for any choices of $(\zeta_1,\cdots,\zeta_{2R})$, we have
\begin{equation}\label{estgarden}|\Ks_{p+1,n}^{\mathrm{tr}}|\leq \big(\langle k_1\rangle\cdots\langle k_{2R}\rangle\big)^{-4\Lambda_{p+1}}\cdot L^{-(R-1)/4}\cdot\delta^{2\nu n}.
\end{equation}
\item If $R=1$, the expression $\Ks_{p+1,n}^{\mathrm{tr}}(+,-,k_1,k_2)$ is nonzero only when $k_1=k_2:=k$. In this case, we have
\begin{equation}\label{estcouple}
\begin{aligned}\Ks_{p+1,2n}^{\mathrm{tr}}(+,-,k,k)&=\Uc_n(p+1,k)+\Rs_{p+1,2n}(k),\\
\Ks_{p+1,2n+1}^{\mathrm{tr}}(+,-,k,k)&=\Rs_{p+1,2n+1}(k),
\end{aligned}
\end{equation} where $\Uc_n$ is defined in (\ref{wketaylor2}); moreover for any $n\leq N_{p+1}^4$ we have
\begin{equation}
\label{estcouple2}\|\Rs_{p+1,n}(k)\|_{\Sf^{4\Lambda_{p+1},4\Lambda_{p+1}}}\leq L^{-(4\theta_{p+1}+\theta_p)/5}\cdot \delta^{2\nu n}.
\end{equation}
\end{enumerate}
\end{prop}
\begin{prop}\label{fixedpoint} There exists an event $\Ff_{p+1}\subset\Ff_p$, which has gauge and space translation symmetry, such that $\Pb(\Ff_p\backslash\Ff_{p+1})\leq e^{-4N_{p+1}}$. Moreover, assuming $\Ff_{p+1}$, we have the following estimates:
\begin{enumerate}[{(1)}]
\item For the term $a_k^{\mathrm{tr}}(t)$, and each $k\in\Zb_L^d$ and $t\in[p,p+1]$, we have \begin{equation}\label{fixedest1}|(a_k^{\mathrm{tr}})(t)|\leq \langle k\rangle^{-(3/2)\Lambda_{p+1}}\cdot e^{3N_{p+1}}.\end{equation}
\item For the term $\Ls_0$ defined in (\ref{eqnbk1.5}), and each $k\in\Zb_L^d$ and $t\in[p,p+1]$, we have
\begin{equation}\label{fixedest2}|(\Ls_0)_k(t)|\leq \langle k\rangle^{-(3/2)\Lambda_{p+1}}\cdot e^{-20N_{p+1}}.
\end{equation}
\item The $\Rb$-linear operator $\Ls_1$ defined in (\ref{eqnbk1.5}) satisfies that $1-\Ls_1$ is invertible in the space $Z:=Z^{(5/4)\Lambda_{p+1}}([p,p+1])$ (see (\ref{defznorm}) for definition), and that
\begin{equation}\label{fixedest3}\|(1-\Ls_1)^{-1}\|_{Z\to Z}\leq e^{3N_{p+1}}.
\end{equation}
\end{enumerate}
Finally, assuming (\ref{fixedest1})--(\ref{fixedest3}), the mapping
\begin{equation}\label{fixedest4}
\textit{\textbf{b}}\mapsto(1-\Ls_1)^{-1}(\Ls_0+\Ls_2(\textit{\textbf{b}},\textit{\textbf{b}})+\Ls_3(\textit{\textbf{b}},\textit{\textbf{b}},\textit{\textbf{b}}))
\end{equation} is a contraction mapping from the metric space $\{\textit{\textbf{b}}\in Z:\|\textit{\textbf{b}}\|_{Z}\leq e^{-10N_{p+1}}\}$ to itself.
\end{prop}
\subsection{Proof of Proposition \ref{propansatz} and Theorem \ref{main}}\label{proofmain} Assuming Propositions \ref{layergarden}--\ref{fixedpoint}, we can now prove Proposition \ref{propansatz} and Theorem \ref{main}.
\begin{proof}[Proof of Proposition \ref{propansatz}] As pointed out in Section \ref{secduhamel}, we only need to prove the statements (1)--(4) in Proposition \ref{propansatz} for $p+1$. Let the event $\Ff_{p+1}$ be defined as in Proposition \ref{fixedpoint}. Then, assuming $\Ff_{p+1}$, by the contraction mapping property of  (\ref{fixedest4}), the equation (\ref{eqnbk}) will have a unique solution $\textit{\textbf{b}}=b_k(t)$, which means that $\textit{\textbf{a}}=a_k(t)$ as in (\ref{defbk}) solves (\ref{akeqn})--(\ref{akeqn2}) on $[p,p+1]$. As $\Ff_{p+1}\subset\Ff_p$, we get a unique solution $\textit{\textbf{a}}=a_k(t)$ to (\ref{akeqn})--(\ref{akeqn2}) on $[0,p+1]$, and \[\Pb(\Ff_{p+1})=\Pb(\Ff_p)-\Pb(\Ff_p\backslash\Ff_{p+1})\geq 1-e^{-N_p}-e^{-4N_{p+1}}\geq 1-e^{-N_{p+1}},\] which proves statement (1).

Next, for any $1\leq R\leq R_{p+1}$ and any $(\zeta_1,\cdots,\zeta_{2R},k_1,\cdots,k_{2R})$, consider the cumulant (\ref{cmatp+1}). By gauge and space translation symmetry, we know it satisfies the support conditions required in (2) and (3) of Proposition \ref{propansatz}. As $a_k(t)=a_k^{\mathrm{tr}}(t)+b_k(t)$, we shall first consider the contribution (\ref{cmattr}). Following the discussions in Section \ref{secduhamel}, we expand each $a_{k}^{\mathrm{tr}}$ into tree terms of order at most $N_{p+1}$ (see (\ref{defbk})), apply Lemma \ref{propertycm} and replace each resulting cumulant term using the induction hypothesis; note that these cumulants have degree at most $2R'\leq 10R_{p+1}N_{p+1}\leq R_p/10$ (recall $R_p=N_{p+1}^{40d}$ and $N_{p+1}=R_{p+1}^{40d}$). Using Definition \ref{defcanon} and the induction hypothesis, we get
\begin{equation}\label{proofmain-1}
(\ref{cmattr})=\sum_\Gc\Kc_\Gc(p+1,k_1,\cdots,k_{2R})+\Es_{p+1}^{(1)},
\end{equation} where $\Kc_\Gc$ is as in (\ref{defkg0}) in Definition \ref{defkg} with input functions $F_q(k)=\varphi(\delta q,k)+\Rs_q(k)$ for $0\leq q\leq p$, and $\Gc$ runs over $\Gs_{p+1}^{\mathrm{tr}}$ as defined in Definition \ref{defgtr}; for $R=1$ we should replace $\Gc$ by a couple $\Qc$ which runs over $\Cs_{p+1}^{\mathrm{tr}}$.

The error term $\Es_{p+1}^{(1)}$ in (\ref{proofmain-1}) contains those contributions where at least one of the resulting cumulant terms is replaced by $\Es_p$ or $\widetilde{\Es_p}$. By (\ref{ansatz2}) and (\ref{ansatz4}) for $p$, these terms decay like $e^{-N_p}$, where recall that $N_p=R_p^{40d}$; the other cumulant terms can be estimated using (\ref{ansatz5}) (by summing over all $n\leq N_p$), and they are simply bounded by $L^{-(R-1)/6}\leq 1$. Since the whole expression involves the summation of at most $10R_{p+1}N_{p+1}$ vector variables $k_\lf\in\Zb_L^d$, and the negative powers $\langle k_j\rangle^{-\Lambda_p}$ in (\ref{ansatz2}), (\ref{ansatz4}) and (\ref{ansatz5}) allows us to restrict each $k_\lf$ to a fixed unit ball using summability, we conclude that
\[|\Es_{p+1}^{(1)}|\leq\big(\langle k_1\rangle\cdots\langle k_{2R}\rangle\big)^{-\Lambda_{p}/4}\cdot L^{20R_{p+1}N_{p+1}}\cdot e^{-N_p}\] pointwise, where the negative power of $\max(\langle k_1\rangle,\cdots,\langle k_{2R}\rangle)$ also comes from the $\langle k_\lf\rangle^{-\Lambda_p}$ powers above. This is clearly sufficient for the bounds (\ref{ansatz2}) and (\ref{ansatz4}) for $p+1$.

Next consider the sum of the $\Kc_\Gc$ terms in (\ref{proofmain-1}), and assume $n(\Gc)=n$. If $n>N_{p+1}^4$, then in the process of replacing each cumulant term using the induction hypothesis as described above, one of these cumulants must be replaced by a $\Kc_{\Gc'}$ term as in (\ref{ansatz3}) for some $\Gc'$ with $n(\Gc')\geq N_{p+1}^2$ (because the number of $\Gc'$ involved at most $10R_{p+1}N_{p+1}$, while the sum of all these $n(\Gc')$ is $n\geq N_{p+1}^4$). Since we are also summing in $\Gc'$ with the value of $n(\Gc')$ fixed, using (\ref{ansatz5}) for $p$, we can estimate the sum of $\Kc_{\Gc'}$ by $e^{-N_{p+1}^2}$ and the product of other cumulants simply by $1$ as above, with the same decay factors $\langle k_\lf\rangle^{-\Lambda_p}$. This contribution, denoted by $\Es_{p+1}^{(2)}$, thensatisfies that
\[|\Es_{p+1}^{(2)}|\leq\big(\langle k_1\rangle\cdots\langle k_{2R}\rangle\big)^{-\Lambda_{p}/4}\cdot L^{20R_{p+1}N_{p+1}}\cdot e^{-N_{p+1}^2}\] pointwise, which is also sufficient for (\ref{ansatz2}) and (\ref{ansatz4}) for $p+1$. 

Moreover, if $N_{p+1}<n\leq N_{p+1}^4$, denote this contribution by $\Es_{p+1}^{(3)}$, then we can use (\ref{estgarden}) and (\ref{estcouple2}) from Proposition \ref{layergarden}, to bound it by
\[|\Es_{p+1}^{(3)}|\leq\big(\langle k_1\rangle\cdots\langle k_{2R}\rangle\big)^{-(3/2)\Lambda_{p+1}}\cdot e^{-10N_{p+1}},\] which is again sufficient for (\ref{ansatz2}) and (\ref{ansatz4}) for $p+1$. This finally allows us to restrict to $n\leq N_{p+1}$, in which case both requirements (i) and (ii) in the definition of $\Gs_{p+1}^{\mathrm{tr}}$ and $\Cs_{p+1}^{\mathrm{tr}}$ in Definition \ref{defgtr} become redundant. Therefore, in (\ref{proofmain-1}) we can just sum over all $\Gc\in\Gs_{p+1}$ (or $\Qc\in\Cs_{p+1}$) that has order at most $N_{p+1}$, which matches the description of Proposition \ref{propansatz} for $p+1$. Apart from the difference between (\ref{cmatp+1}) and (\ref{cmattr}), this already proves statement (3) of Proposition \ref{propansatz}; moreover statement (2) follows from (\ref{estcouple})--(\ref{estcouple2}) of Proposition \ref{layergarden} together with (\ref{wketaylor}) and (\ref{taylorbound}) (to get real valued $\Rs_{p+1}$ we simply take real part), and statement (4) follows from (\ref{estgarden}) of Proposition \ref{layergarden}.

Now it remains to consider the difference between (\ref{cmatp+1}) and (\ref{cmattr}), and prove that it belongs to an error term. Recall the definition of $\Kb_{p+1}$ and $\Kb_p$ in Proposition \ref{propansatz} and that $a_k(t)=a_k^{\mathrm{tr}}(t)+b_k(t)$, we know that this difference consists of two terms, denoted by $\Es_{p+1}^{(4)}$ and $\Es_{p+1}^{(5)}$. First $\Es_{p+1}^{(4)}$ contains terms of
\[\Kb_{p+1}\big(c_{k_1}(p+1)^{\zeta_1},\cdots, c_{k_{2R}}(p+1)^{\zeta_{2R}}\big)\] where each $c_k(t)$ is either $a_k^{\mathrm{tr}}(t)$ or $b_k(t)$ but with at least one factor being $b_k(t)$; note that, assuming the event $\Ff_{p+1}$, we always have $\|b_k(t)\|_Z\leq e^{-10N_{p+1}}$ by Proposition \ref{fixedpoint}, where $Z:=Z^{(5/4)\Lambda_{p+1}([p,p+1])}$. Now there are at most $2^{2R}$ choice for $c_j$ which can be easily absorbed, and using (\ref{eqncm}) we can reduce the cumulants to the corresponding moments with a loss of at most $(C_0R_{p+1})!$. Then, by using Cauchy-Schwartz for each moment, we can control $\Es_{p+1}^{(4)}$ by terms of form
\[|\Es_{p+1}^{(4)}|\leq (C_0R_{p+1})!\cdot e^{-10N_{p+1}}\prod_{j\not\in A}\langle k_j\rangle^{-(5/4)\Lambda_{p+1}}\cdot\bigg[\Eb\bigg(\mathbf{1}_{\Ff_{p}}\prod_{j\in A}|a_{k_j}^{\mathrm{tr}}(p+1)|^2\bigg)\bigg]^{1/2}\] for some $A\subset\{1,\cdots,2R\}$. The last moment expression can again be rewritten, using (\ref{eqncm2}), in terms of cumulants of $\{a_{k_j}^{\mathrm{tr}}(p+1)\}$ and $\{\overline{a_{k_j}^{\mathrm{tr}}(p+1)}\}$ at another loss of at most $(C_0R_{p+1})!$, while each of these cumulants has the form (\ref{cmatp+1}) (with $R$ replaced by $2R$, which does not matter). By repeating all the above arguments, we can bound
\begin{equation}\label{proofmain-2}\Eb\bigg(\mathbf{1}_{\Ff_{p}}\prod_{j\in A}|a_{k_j}^{\mathrm{tr}}(p+1)|^2\bigg)\leq (C_0R_{p+1})! \cdot\prod_{j\in A}\langle k_j\rangle^{-3\Lambda_{p+1}},\end{equation} and putting together we get the bound for $\Es_{p+1}^{(4)}$ that is sufficient for (\ref{ansatz2}) and (\ref{ansatz4}) for $p+1$, noticing that $N_{p+1}=R_{p+1}^{40d}$. Finally, the term $\Es_{p+1}^{(5)}$ equals the difference \[(\Kb_{p+1}-\Kb_{p})\big(a_{k_1}^{\mathrm{tr}}(p+1)^{\zeta_1},\cdots,a_{k_{2R}}^{\mathrm{tr}}(p+1)^{\zeta_{2R}}\big).\] By (\ref{eqncm}), we can write this as a sum of products of moments, with at least one moment taken assuming the event $\Ff_p\backslash\Ff_{p+1}$. By using another Cauchy-Schwartz, we can bound each such moment by (\ref{proofmain-2}), with at least one extra factor $\Pb(\Ff_p\backslash\Ff_{p+1})^{1/2}$. This is again sufficient for (\ref{ansatz2}) and (\ref{ansatz4}) for $p+1$, since $\Pb(\Ff_p\backslash\Ff_{p+1})\leq e^{-4N_{p+1}}$ by (\ref{fixedpoint}). This proves Proposition \ref{propansatz}.
\end{proof}
\begin{proof}[Proof of Theorem \ref{main}] Let $t=T_{\mathrm{kin}}\cdot\tau$ where $0<\tau\leq\tau_*$. Since $\tau_*=\Df\cdot\delta$, we may assume $\tau=t'\delta$ with $t'\in[p,p+1]$ for some $0\leq p<\Df$. First assume $t'=p$, then Proposition \ref{propansatz} already implies
\[\Eb\big(\mathbf{1}_{\Ff_p}|\widehat{u}(t,k)|^2\big)=\varphi(\tau,k)+\Oc_0(L^{-\theta_p})\] uniformly in $p$ and $k$, so we only need to consider the expectation with $\mathbf{1}_{\Ff\backslash\Ff_p}$, where $\Ff\supset\Ff_p$ is the event that (\ref{nls}) has a smooth solution on $[0,T]$ (so $\Pb(\Ff)\geq\Pb(\Ff_p)\geq\Pb(\Ff_\Df)\geq 1-e^{-(\log L)^{20d}}$). By mass conservation for (\ref{nls}), we have
\[\Eb\big(\mathbf{1}_{\Ff\backslash\Ff_p}|\widehat{u}(t,k)|^2\big)\leq\sum_{k}\Eb\big(\mathbf{1}_{\Ff\backslash\Ff_p}|\widehat{u_{\mathrm{in}}}(k)|^2\big)\leq2\sum_{k}\varphi_{\mathrm{in}}(k)\cdot\Pb(\Ff\backslash\Ff_p)^{1/2}\leq C_0L^d\cdot\Pb(\Ff_p^c)\to 0\] uniformly in $p$ and $k$, which then proves (\ref{limit}).

In general, suppose $t'\in[p,p+1]$. We use Proposition \ref{propansatz} for $p$, then, assuming $\Ff_{p+1}$, we repeat the arguments applied in the inductive step (which is used to prove Proposition \ref{propansatz} for $p+1$), but apply them to rescaled time $t'$ instead of $p+1$. It is clear from the proof that none of the ingredients is affected by this change, and the result follows in the same way. This proves Theorem \ref{main}.
\end{proof}
The rest of this paper is devoted to the proof of Propositions \ref{layergarden} and \ref{fixedpoint}. Recall that we have fixed $p\in[0,\Df-1]$ as in Section \ref{secduhamel}; this value will be kept the same throughout the proof.
\section{Layered regular objects I: Combinatorics}\label{layerobject1}
\subsection{Regular objects and structure theorem} First recall the definition and basic properties of \emph{regular couples} and \emph{regular trees} in \cite{DH21}, which are of fundamental importance in both \cite{DH21} and the current paper.
\begin{df}[Regular couples and regular trees \cite{DH21}]\label{defreg} Define a \emph{$(1,1)$-mini-couple} to be a couple of order $2$ formed by two trees of order $1$ with no siblings paired. It has two possibilities indicated by codes $00$ and $01$, see Figure \ref{fig:minicpl}. Define a \emph{mini tree} to be a paired tree of order $2$, again with no siblings paired. It has six possibilities indicated by codes $10,11,20,21,30,31$, see Figure \ref{fig:minitree}.

For any couple $\Qc$ we can define two operations: operation A where a leaf pair is replaced by a $(1, 1)$-mini-couple, and operation B where a node is replaced by a mini tree, see Figure \ref{fig:operab}. Then, we define a couple $\Qc$ to be \emph{regular} if it can be formed, starting from the trivial couple, by operations A and B. We also define a paired tree $\Tc$ to be a \emph{regular tree}, if $\Tc$ forms a regular couple with the trivial tree $\bullet$. Clearly the order of any regular couple and regular tree must be even.
\begin{figure}[h!]
\includegraphics[scale=0.47]{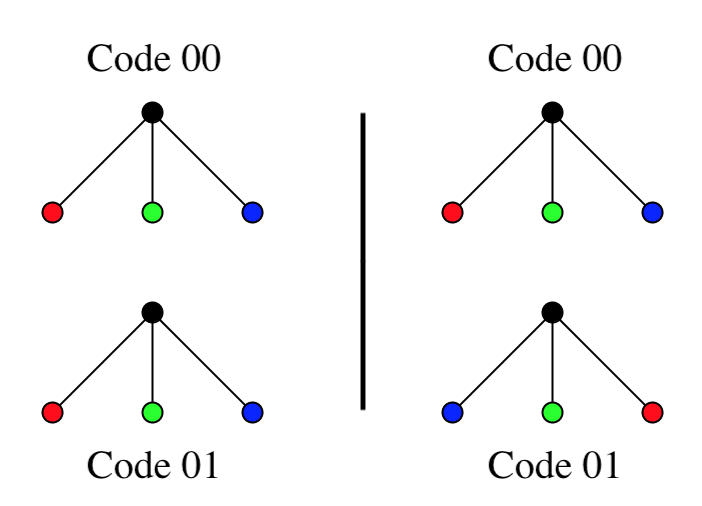}
\caption{A $(1,1)$-mini-couple as in Definition \ref{defreg}; the codes $00$ and $01$ indicate the two different possibilities.}
\label{fig:minicpl}
\end{figure}
\begin{figure}[h!]
\includegraphics[scale=0.45]{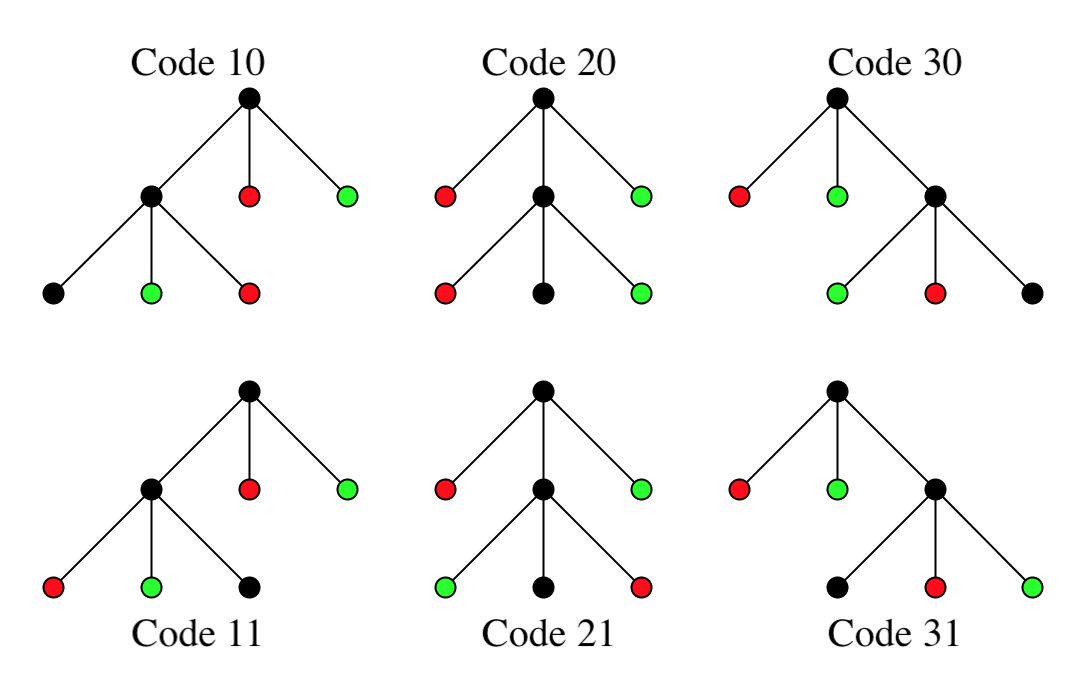}
\caption{A mini-tree as in Definition \ref{defreg}; the codes $10,11,20,21,30,31$ indicate the six different possibilities.}
\label{fig:minitree}
\end{figure}
\begin{figure}[h!]
\includegraphics[scale=0.45]{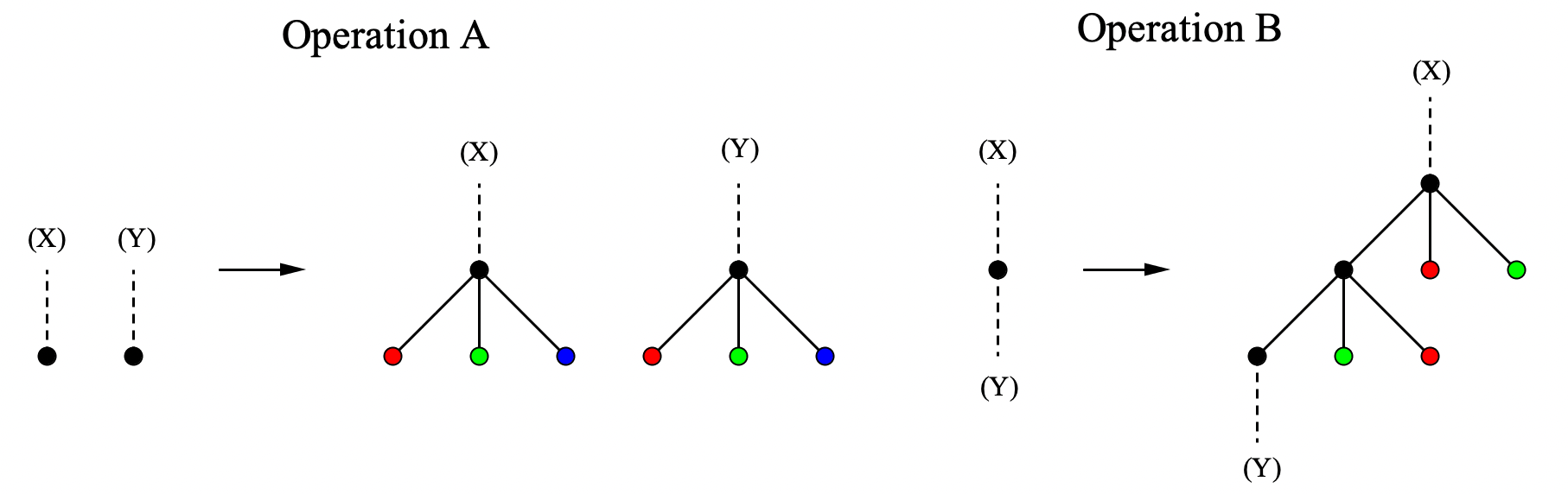}
\caption{Operations $A$ and $B$ as in Definition \ref{defreg}; there are different possibilities depending on the structure of the $(1,1)$-mini-couple and mini-tree, see Figures \ref{fig:minicpl} and \ref{fig:minitree}. The ends $X$ and $Y$ here represent the rest of the couple, which is unaffected by the operation.}
\label{fig:operab}
\end{figure}
\end{df}
We also need the notion of \emph{sub-couples} and \emph{sub-paired-trees}, which are intuitively clear, and are made precise in the following definition.
\begin{df}[Embedded objects]\label{subobject} Let $\Gc$ be a garden, for nodes $\nf$ and $\nf'$, if the leaves in the two trees rooted at $\nf$ and $\nf'$ are all paired to each other, then these two trees form a couple $\Qc$, called a \emph{sub-couple} of $\Gc$; we also say $\Qc$ is \emph{embedded in} $\Gc$. Similarly, if $\nf'$ is a descendant of $\nf$, and all the leaves that are descendants of $\nf$ but not descendants of $\nf'$ are paired to each other, then the tree rooted at $\nf$ becomes a paired tree $\Tc$ after replacing the subtree rooted at $\nf'$ by a single (lone) leaf. In this case we also call $\Tc$ a \emph{sub-paired-tree} and say $\Tc$ is \emph{embedded in} $\Gc$. Also, starting from a garden $\Gc$, when we say we ``replace" a leaf pair by a couple, or a node by a paired tree (such as in Definition \ref{defreg}), then this couple or paired tree will be embedded in the resulting new garden $\Gc_1$. Conversely, we may ``collapse" an embedded couple or paired tree into a leaf pair or a node.
\end{df}
Before introducing the notion of layers, we first recall the definition of regular chains in \cite{DH21}, as well as two structure theorems: first one for regular couples and regular trees, and second one for general gardens. These results are proved in \cite{DH21} and \cite{DH21-2}.
\begin{df}[Regular chains \cite{DH21}]\label{defregchain} Given $m\geq 0$, consider a partition $\Pc$ of $\{1,\cdots,2m\}$ into $m$ pairwise disjoint two-element subsets (or pairs). We say this partition $\Pc$ is \emph{legal} if there do not exist $a<b<c<d$ such that $\{a,c\}\in\Pc$ and $\{b,d\}\in\Pc$; below we will always assume $a<b$ and order these pairs $\{a,b\}$ in the increasing order of $a$. Now, for any legal partition $\Pc$, we define a \emph{regular chain} to be a paired tree of $2m$ branching nodes $\nf_j\,(1\leq j\leq 2m)$, such that $\nf_{j+1}$ is a child of $\nf_j$ for $1\leq j\leq 2m-1$, and some specified child of $\nf_{2m}$ is the lone leaf. Moreover, for any $\{a,b\}\in\Pc$, the two remaining children nodes of $\nf_a$ are paired with the two remaining children nodes of $\nf_b$ with no siblings paired. Define also a \emph{dominant chain} to be a regular chain where $\Pc=\{\{1,2\},\cdots,\{2m-1,2m\}\}$.
\end{df}
\begin{prop}[Structure theorem for regular objects]\label{propstructure} Let $\Qc$ be a nontrivial regular couple. Then exactly one of the two cases happen:
\begin{enumerate}[{(a)}]
\item There exist unique regular couples $\Qc_j\,(1\leq j\leq 3)$ such that $\Qc$ is obtained from the $(1,1)$-mini-couple $\Qc_0$ (see Figure \ref{fig:minicpl}), by replacing the red, green and blue leaf pair with $\Qc_1$, $\Qc_2$ and $\Qc_3$ respectively. In this case we say $\Qc$ is a \emph{type 1} regular couple.
\item There exists a unique nontrivial regular couple $\Qc_0$ which is formed by pairing the two lone leaves of two regular chains $\Xc^\epsilon$ of sign $\epsilon\in\{\pm\}$ and order $2m_\epsilon$, and unique regular couples $\Qc_{j,\epsilon,\iota}$ where $1\leq j\leq m_\epsilon$ and $\iota\in\{1,2\}$, and a unique type 1 or trivial regular couple $\Qc_{\mathrm{lp}}$, such that $\Qc$ is obtained from $\Qc_0$ by the following operations. First we replace the lone leaf pair with $\Qc_{\mathrm{lp}}$. Next, let the branching nodes of $\Xc^{\epsilon}$ be $\nf_a^\epsilon\,(1\leq a\leq 2m_\epsilon)$ with legal partition $\Pc^\epsilon$ as in Definition \ref{defregchain}; for $1\leq j\leq m_\epsilon$, let the $j$-th pair in $\Pc^\epsilon$ (in the ordering of Definition \ref{defregchain}) be $\{a,b\}$, then the nodes $\nf_a^\epsilon$, $\nf_b^\epsilon$ and all their children can be rearranged into a mini-tree as in Figure \ref{fig:minitree} (by setting $\nf_b^\epsilon$ as the child of $\nf_a^\epsilon$ in this new mini-tree). Then, we replace the red leaf pair in this mini-tree with $\Qc_{j,\epsilon,1}$, and replace the green leaf pair in this mini-tree with $\Qc_{j,\epsilon,2}$. In this case we say $\Qc$ is a \emph{type 2} regular couple. See Figure \ref{fig:regchain}.
\end{enumerate}

Similarly, let $\Tc$ be a nontrivial regular tree, then there exists a unique nontrivial regular chain $\Xc_0$ of order $2m_0$, and unique regular couples $\Qc_{j,0,\iota}$ where $1\leq j\leq m_0$ and $\iota\in\{1,2\}$, such that $\Tc$ is obtained from $\Xc_0$ by the following operations. Let the branching nodes of $\Xc_0$ be $\nf_a^0\,(1\leq a\leq 2m_0)$ with legal partition $\Pc^0$; for $1\leq j\leq m_0$, let the $j$-th pair in $\Pc^0$ be $\{a,b\}$, then the nodes $\nf_a^0$, $\nf_b^0$ and all their children can be rearranged into a mini-tree as in Figure \ref{fig:minitree}. Then, we replace the red leaf pair in this mini-tree with $\Qc_{j,0,1}$, and replace the green leaf pair in this mini-tree with $\Qc_{j,0,2}$.

We may also inductively define the notion of \emph{dominant couples}, which is a subclass of regular couples, as follows: a regular couple $\Qc$ is dominant, if and only if it is trivial, or it has type 1 and the corresponding regular couples $\Qc_j\,(1\leq j\leq 3)$ are all dominant, or it has type 2, and the regular chain $\Xc^\epsilon$ is dominant for $\epsilon\in\{\pm\}$, and the regular couples $\Qc_{j,\epsilon,\iota}$ and $\Qc_{\mathrm{lp}}$ are all dominant. Similarly, we define a regular tree $\Tc$ is dominant, if and only if the regular chain $\Xc_0$ is dominant, and the corresponding regular couples $\Qc_{j,0,\iota}$ are all dominant.

Finally, given any integer $m\geq 0$, the number of regular couples of order $2m$ is at most $C_0^m$, and the same holds for regular trees (see notations for $C_j$ in Section \ref{sectioninduct}).
\begin{figure}[h!]
\includegraphics[scale=0.46]{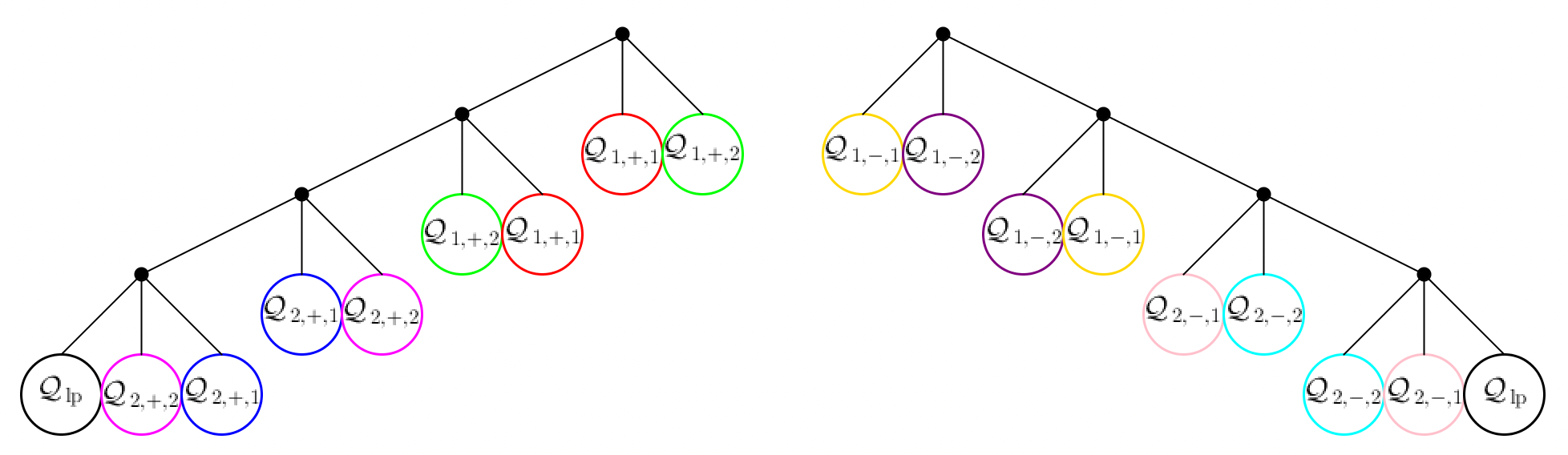}
\caption{A type 2 regular (in fact dominant) couple formed by two regular chains as in Proposition \ref{propstructure}, with the corresponding notations $\Qc_{j,\epsilon,\iota}$ and $\Qc_{\mathrm{lp}}$. The regular tree case is similar.}
\label{fig:regchain}
\end{figure}
\end{prop}
\begin{proof} See Propositions 4.5, 4.7, 4.8 and Corollary 4.9 of \cite{DH21}.
\end{proof}
\begin{prop}\label{propstructure2} Define a garden $\Gc$ to be \emph{prime} if it does not contain any embedded $(1,1)$-mini-couple or mini-tree. Then, for any garden $\Gc$ there is a unique prime garden $\Gc_{\mathrm{sk}}$, called the \emph{skeleton} of $\Gc$, such that $\Gc$ is obtained from $\Gc_{\mathrm{sk}}$ by replacing each leaf pair with a regular couple and each branching node with a regular tree. The choices of these regular couples and regular trees are also unique. Moreover, given any $\Gc_{\mathrm{sk}}$, the number of gardens $\Gc$ that has order m, width 2R and skeleton $\Gc_{\mathrm{sk}}$ is $\leq C_0^{m+R}$.
\end{prop}
\begin{proof} See Propositions 4.4, 4.5 and 4.6 of \cite{DH21-2}.
\end{proof}
Finally, recall from \cite{DH21} that branching nodes in any regular couple and regular tree are also naturally divided into pairs, which we refer to as \emph{links} for distinction:
\begin{prop}[Links of branching nodes]\label{deflink} For any regular couple $\Qc$ and regular tree $\Tc$, there is a unique way to divide all the \emph{branching nodes} into two-element subsets called \emph{links}. This is defined inductively as follows:

For any regular couple $\Qc$ (or any regular tree $\Tc$), let the relevant notions be defined as in Proposition \ref{propstructure}. Then, we link the branching nodes of the couples $\Qc_j\,(1\leq j\leq 3)$, or $\Qc_{j,\epsilon,\iota}$ and $\Qc_{\mathrm{lp}}$, or $\Qc_{j,0,\iota}$ depending on different situations, by the induction hypothesis. Moreover, if $\Qc$ has type 1, then we link the two roots of the two trees of $\Qc$; if $\Qc$ has type 2, then we link $\nf_a^\epsilon$ with $\nf_b^\epsilon$ for any pair $\{a,b\}\in\Pc^\epsilon$, and for regular tree $\Tc$ we link $\nf_a^0$ with $\nf_b^0$ for any pair $\{a,b\}\in\Pc^0$. For later use, we may also fix a set $\Nc^{\mathrm{ch}}\subset\Nc$ by inductively selecting one branching node from each link: in each inductive step, for type 1 regular couples we select the root with sign $+$ into $\Nc^{\mathrm{ch}}$, and for type 2 regular couples or regular trees we select $\nf_a^\epsilon$ or $\nf_a^0$ into $\Nc^{\mathrm{ch}}$ for each pair $\{a,b\}$ with $a<b$.

Then, for any decoration $\Is$ of $\Qc$ (or $\Ds$ of $\Tc$) as in Definition \ref{defdec}, we must have $\zeta_{\nf'}\Omega_{\nf'}=-\zeta_\nf\Omega_\nf$ for any linked branching nodes $(\nf,\nf')$, with $\zeta_\nf$ as in Definition \ref{deftree} and $\Omega_\nf$ as in (\ref{defomega}). Moreover, let $\nf_j$ and $\nf_j'$ be the children nodes of $\nf$ and $\nf'$ as in Definition \ref{defdec}, and define $(x_\nf,y_\nf)=(k_{\nf_1}-k_{\nf},k_\nf-k_{\nf_3})$ and similarly for $\nf'$, then we have $\{x_{\nf'},y_{\nf'}\}=\{\zeta_1 x_\nf,\zeta_2 y_\nf\}$ for some choices of $\zeta_j\in\{\pm\}$.
\end{prop}
\begin{proof} See Proposition 4.3 of \cite{DH21}. By Propositions 4.7 and 4.8 of \cite{DH21}, it is easy to see that the definition of links stated in the proof of Proposition 4.3 of \cite{DH21} coincides with the definition here.
\end{proof}
\subsection{Layered regular objects} Now we study the structures of regular couples and regular trees with layering. Note that the definition of sub-couples and sub-paired-trees in Definition \ref{subobject} can obviously be extended to the corresponding layered objects, so we can talk about a layered couple or layered paired tree being embedded in some layered garden.

In fact, throughout the proof we will mainly be looking at canonical layered gardens as in Definition \ref{defcanon}; therefore the regular couples and regular trees studied will also appear as sub-couples and sub-paired-trees of canonical gardens. We now introduce the following definition, which is of key importance in the study of layered regular objects:
\begin{df}[Coherent regular objects]\label{defcoh} We say a layered regular couple or regular tree is \emph{coherent}, if for any link of branching nodes $(\nf,\nf')$ as in Definition \ref{deflink} we have $\Lf_{\nf}=\Lf_{\nf'}$. In general we define the \emph{incoherency index} to be the number of links of branching nodes $(\nf,\nf')$ such that $\Lf_{\nf}\neq \Lf_{\nf'}$.
\end{df}
Our goal is to classify all coherent layered regular couples and regular trees, as well as those with fixed incoherency index, \emph{that are embedded in canonical layered gardens}. For this, we first need to state a convenient equivalent condition for a layered garden to be canonical.
\begin{prop}\label{canonequiv} Let $\Gc$ be an irreducible proper layered garden. For any $p\geq 0$, we have $\Gc\in\Gs_{p+1}$ if and only if $\Gc$ satisfies the following two conditions: (i) all nodes $\nf\in\Gc$ satisfy $\Lf_\nf\leq p$; (ii) if two nodes $\nf$ and $\nf'$ are such that all leaves of the trees rooted at $\nf$ and $\nf'$ are completely paired, then we must have $\max(\Lf_{\nf},\Lf_{\nf'})\geq \min(\Lf_{\nf^{\mathrm{pr}}},\Lf_{(\nf')^{\mathrm{pr}}})$, where $\nf^{\mathrm{pr}}$ and $(\nf')^{\mathrm{pr}}$ are the parent nodes of $\nf$ and $\nf'$, as in Definition \ref{deftree}; if $\nf$ or $\nf'$ is a root then the corresponding $\Lf_{\nf^{\mathrm{pr}}}$ or $\Lf_{(\nf')^{\mathrm{pr}}}$ is replaced by $p$. Similarly, we have that a couple $\Qc\in\Cs_{p+1}$ if and only if it satisfies both (i) and (ii) above.
\end{prop}
\begin{proof} Consider first the garden case. Note that (i) is obvious by definition (as in Definition \ref{defcanon} we put the branching nodes and paired leaves of the trees $\Tc_j^p$ in layer $p$, and all other nodes are in layer $<p$), we first prove (ii) by induction. If $p=0$ then $\Lf_\nf=0$ for all nodes $\nf$, so (ii) is trivial.  Suppose (ii) holds for $p-1$, now let us consider $\Gc\in\Gs_{p+1}$ and two nodes $\nf$ and $\nf'$ as in (ii). Consider the trees $\Tc_j^p$ in the recursive construction process in Definition \ref{defcanon} (see for example Figure \ref{fig:canongarden}); all their branching nodes and leaf pairs are in layer $p$, while all their unpaired leaves (which may be leaves or branching nodes in the garden $\Gc$) are in layer $\leq p-1$. Moreover, these nodes are divided into groups of at least four, such that the layered trees rooted at all nodes in each individual group form a canonical layered garden in $\Gs_p$. We may assume $\Lf_\nf<p$ and $\Lf_{\nf'}<p$ or otherwise (ii) already holds, so in particular both $\nf$ and $\nf'$ belong to one of these $\Gs_p$ gardens; by the complete pairing assumption for $\nf$ and $\nf'$, we know they must belong to the same $\Gs_p$ garden, say $\Gc_1$.

Next, note that $\Lf_{\nf^{\mathrm{pr}}}=p$ if and only if $\nf$ is the root of a tree in $\Gc_1$ (same for $\nf'$). Consider three cases: if $\max(\Lf_{\nf^{\mathrm{pr}}},\Lf_{(\nf')^{\mathrm{pr}}})<p$, then both $\nf^{\mathrm{pr}}$ and $(\nf')^{\mathrm{pr}}$ are nodes in $\Gc_1$, so the result in (ii) follows from induction hypothesis. If $\Lf_{\nf^{\mathrm{pr}}}=p$ and $\Lf_{(\nf')^{\mathrm{pr}}}<p$, then $\nf$ is the root of a tree in $\Gc_1$ and $(\nf')^{\mathrm{pr}}$ is a node in $\Gc_1$, so the induction hypothesis gives that \[\max(\Lf_\nf,\Lf_{\nf'})\geq \min(p-1,\Lf_{(\nf')^{\mathrm{pr}}})=\Lf_{(\nf')^{\mathrm{pr}}}=\min(p,\Lf_{(\nf')^{\mathrm{pr}}})=\min(\Lf_{\nf^{\mathrm{pr}}},\Lf_{(\nf')^{\mathrm{pr}}}),
\] thus (ii) is still true. If $\Lf_{\nf^{\mathrm{pr}}}=\Lf_{(\nf')^{\mathrm{pr}}}=p$, then $\nf$ and $\nf'$ are the roots of two trees in $\Gc_1$, which contradicts the assumption that $\Gc_1$ is an irreducible proper garden. Note also that if $\nf$ (or $\nf'$) is the root of a tree in $\Gc$, then the proof goes in the same way as in the case $\Lf_{\nf^{\mathrm{pr}}}=p$ (or $\Lf_{(\nf')^{\mathrm{pr}}}=p$). This completes the proof of (ii) in the garden case; the couple case is completely analogous.

Finally, we prove that (i) and (ii) imply $\Gc\in\Gs_{p+1}$ (again the couple case follows from the same proof). If $p=0$ the result is true, because any irreducible proper garden with all nodes in layer $0$ belongs to $\Gs_1$. Suppose the result holds for $p-1$, then consider any irreducible proper layered garden $\Gc$ satisfying (i) and (ii) for $p$. Let $\Uc$ be the set of nodes $\nf\in\Gc$ such that $\Lf_{\nf}<p$ and $\Lf_{\nf^{\mathrm{pr}}}=p$ (including when $\nf$ is the root of a tree in $\Gc$), then none of these $\nf\in\Uc$ can be ancestor or descendant of each other, and the trees roots at these nodes $\nf$ form a garden (leaves in these garden are precisely those leaves in layer $\leq p-1$, which are paired to each other). Moreover all nodes in this garden are in layer $\leq p-1$, and by (ii) and induction hypothesis, we see that this garden can be decomposed into irreducible proper layered gardens (which cannot be couples) that belong to $\Gs_p$. Then, we consider the trees formed by replacing the tree rooted at each node $\nf\in\Uc$ with  a single leaf node $\nf$, and use them as the trees $\Tc_j^p\,(1\leq j\leq 2R)$ in Definition \ref{defcanon}. The branching nodes and leaf pairs in these $\Tc_j^p$ trees then are all in layer $p$; moreover, to form $\Gc$, the unpaired leaves in these trees are divided into groups of at least four, and each group is replaced by a canonical garden in $\Gs_p$. In addition, the requirement in Definition \ref{defcanon} that no union of any of such leaf pairs and leaf groups equals the set of leaves of $\Tc_j^p\,(j\in A)$ for any proper subset $A\subset\{1,\cdots,2R\}$, follows from the irreducibility of $\Gc$. Therefore, we have verified all conditions in Definition \ref{defcanon} and hence $\Gc\in\Gs_{p+1}$. This completes the proof.
\end{proof}
We can now state the two main results of this section, namely the classification of coherent and near-coherent layered regular couples in canonical layered gardens.
\begin{prop}[Structure theorem for coherent regular objects]\label{layerreg1} Suppose a regular couple $\Qc$ is embedded in a \emph{canonical} layered garden $\Gc\in\Gs_{p+1}$, with roots located at nodes $\nf^+$ and $\nf^-$ of $\Gc$. Assume that (i) this layering makes $\Qc$ a \emph{coherent} layered regular couple, and (ii) some $q,q'\in\Zb$ are fixed such that $\Lf_{(\nf^+)^{\mathrm{pr}}}=q$ and $\Lf_{(\nf^-)^{\mathrm{pr}}}=q'$; if $\nf^+$ or $\nf^-$ is the root of a tree in $\Gc$ then the corresponding $q$ or $q'$ is replaced by $p$. Assume $q\geq q'$ (the case $q<q'$ follows by symmetry) and define the relative notions as in Proposition \ref{propstructure}, then the followings hold. If $\Qc$ has type 1 or trivial, then all nodes in $\Qc$ must be in layer $q'$. If $\Qc$ has type 2, then all nodes in the couples $\Qc_{j,-,\iota}\,(1\leq j\leq m_-)$ and the couple $\Qc_{\mathrm{lp}}$, as well as the nodes $\nf_a^-\,(1\leq a\leq 2m_-)$, must be in layer $q'$. We also have $q\geq \Lf_{\nf_1^+}\geq\cdots\geq \Lf_{\nf_{2m_+}^+}\geq q'$; for the $j$-th pair $\{a,b\}$ in $\Pc^+$ (where $1\leq j\leq m_+$), we have $\Lf_{\nf_a^+}=\Lf_{\nf_b^+}$ and all nodes in the couples $\Qc_{j,+,\iota}$ must be in this layer. See Figure \ref{fig:layerregchain}.

Similarly, suppose a regular tree $\Tc$ is embedded in a \emph{canonical} layered garden $\Gc\in\Gs_{p+1}$, with the roots located at node $\nf$ and lone leaf located at node $\nf'$ of $\Gc$. Assume that (i) this layering makes $\Tc$ a \emph{coherent} layered regular tree, and (ii) some $q,q'\in\Zb$ are fixed such that $\Lf_{\nf^{\mathrm{pr}}}=q\geq q'=\Lf_{\nf'}$; if $\nf$ is the root of a tree in $\Gc$ then $q$ is replaced by $p$. Define the relative notions as in Proposition \ref{propstructure}, then the followings hold. We have $q\geq \Lf_{\nf_1^0}\geq\cdots\geq \Lf_{\nf_{2m_0}^0}\geq q'$; for the $j$-th pair $\{a,b\}$ in $\Pc^0$ (where $1\leq j\leq m_0$), we have $\Lf_{\nf_a^0}=\Lf_{\nf_b^0}$ and all the nodes in the couple $\Qc_{j,0,\iota}$ must be in this layer.

In addition, if we fix $(q,q')$ and the structure and layering of rest of the garden, but replace the regular couple $\Qc$ or regular tree $\Tc$ by any \emph{other} regular couple or regular tree with layering satisfying all the above assumptions, then the resulting garden also belong to $\Gs_{p+1}$.\footnote{As a result, when we talk about coherent regular objects in the context of a larger garden, they actually only rely on the values $(q,q')$ and not on the rest structure of the garden.} Finally, all the above results hold if the garden $\Gc$ is replaced by a couple, and $\Gs_{p+1}$ is replaced by $\Cs_{p+1}$.
\begin{figure}[h!]
\includegraphics[scale=0.46]{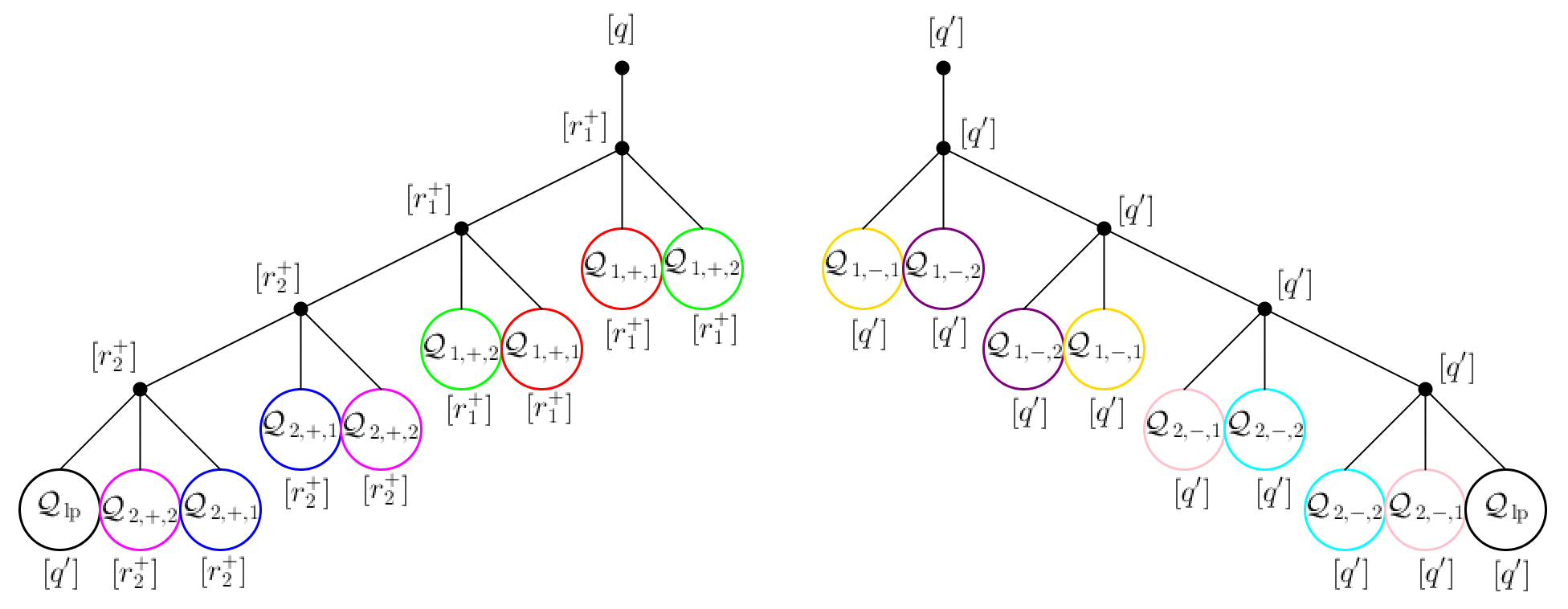}
\caption{A layering of the regular couple in Figure \ref{fig:regchain} that makes it coherent, as in Proposition \ref{layerreg1}. Here $r_j^+=\Lf_{\nf_j^+}$ and $q\geq r_1^+\geq r_2^+\geq q'$. The regular tree case is similar.}
\label{fig:layerregchain}
\end{figure}
\end{prop}
\begin{proof} (1) We prove by induction. If $\Qc$ is the trivial couple, which is a leaf pair where the two leaves are in the same layer by Definition \ref{deflayer}, then these two leaves must be in layer $q'$ by Definition \ref{deflayer} and Proposition \ref{canonequiv}. Suppose the result holds for any smaller couple $\widetilde{\Qc}$, then in particular, when $q=q'$, all nodes in $\widetilde{\Qc}$ must be in layer $q$. Now consider any $\Qc$, and let the the relative notions be defined as in Proposition \ref{propstructure}.

If $\Qc$ has type 1, then by Proposition \ref{deflink} and Definition \ref{defcoh}, the two roots of the two trees of $\Qc$ must be in the same layer, and this layer again must be $q'$ due to Proposition \ref{canonequiv}. Then, we apply the induction hypothesis to the regular couples $\Qc_j\,(1\leq j\leq 3)$, to see that all nodes of each $\Qc_j\,(1\leq j\leq 3)$ must be in layer $q'$. Therefore we obtain that all nodes in $\Qc$ are in layer $q'$.

Now Suppose $\Qc$ has type 2. By Proposition \ref{deflink} and Definition \ref{defcoh}, we know that for $\Lf_{\nf_a^\epsilon}=\Lf_{\nf_b^\epsilon}$ for $\epsilon\in\{\pm\}$ and each pair $\{a,b\}\in\Pc^\epsilon$. Applying the induction hypothesis to $\Qc_{j,\epsilon,\iota}$ where $1\leq j\leq m_\epsilon$ is such that $\{a,b\}$ is the $j$-th pair in $\Pc^\epsilon$, we know that all the nodes in $\Qc_{j,\epsilon,\iota}$ are in layer $\Lf_{\nf_a^\epsilon}$.

Next, consider the lone leaves of $\Xc^\epsilon$ for $\epsilon\in\{\pm\}$, which we denote by $\nf_*^\epsilon$. They are also the roots of trees of the type 1 or trivial couple $\Qc_{\mathrm{lp}}$, so they must be in the same layer (say $q_*$) due to Proposition \ref{deflink} and Definition \ref{defcoh}; again by induction hypothesis, all nodes in $\Qc_{\mathrm{lp}}$ are also in layer $q_*$. Clearly $q_*\leq q'$; if $q_*<q'$ (and by assumption $q'\leq q$), then for $\epsilon\in\{\pm\}$, we may choose the smallest $a=a(\epsilon)\in\{1,\cdots,2m_\epsilon+1\}$ such that $\Lf_{\nf_{a(\epsilon)}^\epsilon}=q_*$ (we understand that $\nf_{2m_\epsilon+1}^\epsilon:=\nf_*^\epsilon$). Using that $\Lf_{\nf_a^\epsilon}=\Lf_{\nf_b^\epsilon}$ for each pair $\{a,b\}\in\Pc^\epsilon$, it is easy to see that the branching nodes in the set $\{\nf_{a'}^\epsilon:a(\epsilon)\leq a\leq 2m_\epsilon\}$ are linked to each other, which means that the leaves of the two trees rooted at $\nf_{a(\epsilon)}^\epsilon\,(\epsilon\in\{\pm\})$ are completely paired, but this contradicts Proposition \ref{canonequiv}.

Now we know that $q_*=q'$, so in particular all nodes $\nf_a^-\,(1\leq a\leq 2m_-)$ and all nodes in the couples $\Qc_{j,-,\iota}\,(1\leq j\leq m_-,\iota\in\{1,2\})$ and $\Qc_{\mathrm{lp}}$ are in layer $q_*$. The rest of the claims regarding $\Lf_{\nf_a^+}$ and layers of nodes in $\Qc_{j,+,\iota}$ is already obtained above, and clearly we have $q\geq \Lf_{\nf_1^+}\geq\cdots\geq \Lf_{\nf_{2m_+}^+}\geq q'$ because $\Lf_{(\nf^+)^{\mathrm{pr}}}=q$ and $\Lf_{\nf_{*}^+}=q_*=q'$.

The case of regular tree $\Tc$ is proved similarly. With the regular chain $\Xc_0$ and relevant notions fixed as in Proposition \ref{propstructure}, we have that $\Lf_{\nf_a^0}=\Lf_{\nf_b^0}$ and this equals the layer of all nodes in the couple $\Qc_{j,0,\iota}$, in the same way as above. Clearly also $q\geq \Lf_{\nf_1^0}\geq\cdots\geq \Lf_{\nf_{2m_0}^0}\geq q'$  because $\Lf_{\nf^{\mathrm{pr}}}=q$ and $\Lf_{\nf'}=q'$. The whole proof works for couples in $\Cs_{p+1}$ instead of gardens in $\Gs_{p+1}$ without any change, and if $\nf^+$ or $\nf$ is the root of a tree in $\Gc$ then we simply replace $\Lf_{(\nf^+)^{\mathrm{pr}}}$ or $\Lf_{\nf^{\mathrm{pr}}}$ by $p$.

(2) To prove the last statement about replacing $\Qc$ or $\Tc$ by any other regular couple or regular tree of the same structure, we first show that attaching a layered regular couple to a leaf pair or attaching a layered regular tree to a node \emph{does not change} the canonicity of a layered garden, provided that all nodes in the regular couple or regular tree (other than the lone leaf) are in \emph{the same} layer.

More precisely, let $\Gc$ be a layered garden (the case of layered couple being exactly the same), and pick a leaf pair $(\nf,\nf')$ such that $\Lf_{\nf}=\Lf_{\nf'}=q$, then we may form a new layered garden $\Gc_1$ by replacing this leaf pair with a layered couple $\Qc$ (with roots of the two trees at $\nf$ and $\nf'$) such that all nodes in $\Qc$ are in layer $q$. Similarly, picking any node $\nf$ of $\Gc$ (which can be leaf or branching node) such that $\Lf_{\nf}=q'$ and $\Lf_{\nf^{\mathrm{pr}}}=q\geq q'$ (with $q$ replaced by $p$ if $\nf$ is the root of a tree in $\Gc$), we may form a new layered garden $\Gc_2$ by replacing $\nf$ with a paired tree $\Tc$, such that the lone leaf of $\Tc$ is still in layer $q'$, while all other nodes in $\Tc$ are in layer $r$ for some $q\geq r\geq q'$. Then, we have that \[\Gc\in\Gs_{p+1}\Leftrightarrow \Gc_1\in\Gs_{p+1}\Leftrightarrow\Gc_2\in\Gs_{p+1}.\] This is intuitively clear, and can be verified using the equivalent condition for canonicity introduced in Proposition \ref{canonequiv}. Consider for example $\Gc$ and $\Gc_2$ (the case of $\Gc_1$ is similar and easier); by Proposition \ref{canonequiv}, to check if they are canonical, we only need to check for instances where
\begin{equation}\label{violation}\min(\Lf_{\mf^{\mathrm{pr}}},\Lf_{(\mf')^{\mathrm{pr}}})>\max(\Lf_\mf,\Lf_{\mf'}),\end{equation} while the leaves of the two trees rooted at $\mf$ and $\mf'$ are completely paired. By our choice, in any such instance, neither $\mf$ nor $\mf'$ can be any node in $\Tc$ other than the root (denoted by $\nf_{\mathrm{rt}}$ here) and the lone leaf (denoted by $\nf_{\mathrm{lo}}$ here); if neither $\mf$ nor $\mf'$ equals $\nf_{\mathrm{rt}}$ or $\nf_{\mathrm{lo}}$ then the form of (\ref{violation}) does not change from $\Gc$ to $\Gc_2$, so equivalence holds in this case. Now consider the cases where (say) $\mf=\nf_{\mathrm{rt}}$ or $\mf=\nf_{\mathrm{lo}}$ in $\Gc_2$, and $\mf=\nf$ in $\Gc$, then we can verify the elementary fact that
\[\min(q,\Lf_{(\mf')^{\mathrm{pr}}})>\max(q',\Lf_{\mf'})\Leftrightarrow [\min(q,\Lf_{(\mf')^{\mathrm{pr}}})>\max(r,\Lf_{\mf'})]\vee [\min(r,\Lf_{(\mf')^{\mathrm{pr}}})>\max(q',\Lf_{\mf'})],\] therefore an instance (\ref{violation}) happens in $\Gc$ if and only if an instance (\ref{violation}) happens in $\Gc_2$.

Now let $\widetilde{\Gc}$ be the garden obtained from $\Gc$ by replacing the regular couple $\Qc$ with one  leaf pair in layer $q'$, or by replacing the regular tree $\Tc$ with one single node in layer $q'$. By the structure of the layered regular couple $\Qc$ or layered regular tree $\Tc$ described above, it is easy to see that $\Gc$ can be formed from $\widetilde{\Gc}$ by repeatedly applying the $\Gc\mapsto\Gc_1$ and $\Gc\mapsto\Gc_2$ transformations described above, so we have that
\[\Gc\in\Gs_{p+1}\Leftrightarrow\widetilde{\Gc}\in\Gs_{p+1}.\] Now if $\Qc$ or $\Tc$ is replaced by any other regular couple or regular tree with layering satisfying all the above assumptions, to form a new garden $\Gc'$, then we still have
\[\Gc\in\Gs_{p+1}\Leftrightarrow\widetilde{\Gc}\in\Gs_{p+1}\Leftrightarrow \Gc'\in\Gs_{p+1}.\] This completes the proof.
\end{proof}
\begin{prop}\label{layerreg2} Consider the same setting as in Proposition \ref{layerreg1}, and with $(q,q')$ fixed, but instead of assuming $\Qc$ (or $\Tc$) is coherent, we assume it has \emph{incoherency index} $g$ with $0\leq g\leq m$, where $2m$ is the order of $\Qc$ (or $\Tc$). Then, for all possible choices of layerings given by $\Gc\in\Gs_{p+1}$ (or $\Cs_{p+1}$), the number of possible layerings it induces on $\Qc$ (or $\Tc$) does not exceed $C_2^g C_0^{m+|q-q'|}$.
\end{prop}
\begin{proof} We will consider the case of regular couples $\Qc$, while the case of regular trees $\Tc$ follows the same way. By Proposition \ref{propstructure}, we may fix the structure (without layering) of $\Qc$ at a cost of $C_0^m$; then we may also fix the exact positions of the $g$ links of branching nodes $(\nf,\nf')$ such that $\Lf_{\nf}\neq \Lf_{\nf'}$, at a cost of $\binom{m}{g}\leq 2^m$. Therefore, we will assume below that $\Qc$ is fixed as a regular couple, and the places where $\Lf_{\nf}\neq \Lf_{\nf'}$ happens are also fixed.

We proceed by induction. If $\Qc$ is trivial then $m=g=0$, and both leaves have to be in layer $q'$ by Proposition \ref{canonequiv}, so the number of choices of layerings is $1$. Suppose this number of choices is bounded by $C_2^g C_0^{m+|q-q'|}$ for any smaller couple $\widetilde{\Qc}$, then consider a couple $\Qc$. If $\Qc$ has type 1, then it is formed by two roots $\nf_+,\nf_-$ and three couples $\Qc_j\,(1\leq j\leq 3)$. Let the $m$ and $g$ values for $\Qc_j$ be $m_j$ and $g_j$, then we have $m=m_1+m_2+m_3+1$ and $g=g_1+g_2+g_3+\xi$, where $\xi$ equals $0$ or $1$ depending on whether $\Lf_{\nf_+}=\Lf_{\nf_-}$ or not. Now, if $\Lf_{\nf_+}=\Lf_{\nf_-}$, then this value must be $q'$ due to Proposition \ref{canonequiv}, and $\xi=0$; by induction hypothesis, the number of choices of layerings for $\Qc_j$ is at most $C_2^{g_j}C_0^{m_j}$, noticing also $\Lf_{\nf_+}=\Lf_{\nf_-}$. Putting together we get that the number of choices of layerings for $\Qc$ is at most
\[C_2^{g_1+g_2+g_3}\cdot C_0^{m_1+m_2+m_3}\leq C_2^g C_0^m.\] If instead $\Lf_{\nf_+}\neq\Lf_{\nf_-}$, then the number of choices for them is at most $(\Df+1)^2$ since we always have $0\leq\Lf_{\nf}\leq p\leq \Df$, and moreover $\xi=1$. By induction hypothesis for each $\Qc_j$, we get that the number of choices of layerings for $\Qc$ is at most
\[(\Df+1)^2\cdot C_2^{g_1+g_2+g_3}C_0^{m_1+m_2+m_3}\cdot C_0^{3\Df}\leq C_2^g C_0^m\] using that $g=g_1+g_2+g_3+1$, provided that $C_2\geq (\Df+1)^2C_0^{3\Df}$.

Now assume $\Qc$ has type 2, let the relevant notions be defined as in Proposition \ref{propstructure}. Let the $m$ and $g$ values for $\Qc_{j,\epsilon,\iota}$ be $m_{j,\epsilon,\iota}$ and $g_{j,\epsilon,\iota}$ etc., and similarly for $\Qc_{\mathrm{lp}}$. Moreover, $\Qc_{\mathrm{lp}}$ is either trivial or has type 1; if it has type 1, then let $\Qc_{\mathrm{lp},j}\,(1\leq j\leq 3)$ be defined from $\Qc_{\mathrm{lp}}$ as in Proposition \ref{propstructure} (a), let $m_{\mathrm{lp},j}$ and $g_{\mathrm{lp},j}$ be corresponding $m$ and $g$ values. We then have
\[m=\sum_{j,\epsilon,\iota}m_{j,\epsilon,\iota}+\sum_{j=1}^3m_{\mathrm{lp},j}+m_++m_-+\mathbf{1}_{m_{\mathrm{lp}}>0};\quad g=\sum_{j,\epsilon,\iota}g_{j,\epsilon,\iota}+\sum_{j=1}^3g_{\mathrm{lp},j}+\xi_++\xi_-+\eta,\] where for $\epsilon\in\{\pm\}$, $\xi_\epsilon$ is defined to be the number of pairs $\{a,b\}\in\Pc^\epsilon$ such that $\Lf_{\nf_a^\epsilon}\neq\Lf_{\nf_b^\epsilon}$, and $\eta$ equals $0$ or $1$ depending on whether $\Lf_{\nf_*^+}=\Lf_{\nf_*^-}$ or not, where $\nf_*^\epsilon$ are the lone leaves (i.e. roots of $\Qc_{\mathrm{lp}}$). When $\Qc_{\mathrm{lp}}$ is trivial we understand that the $m_{\mathrm{lp},j}$ and $g_{\mathrm{lp},j}$ terms are absent (same in the proof below), in which case we also have $\eta=0$. Next we classify the value of $g_0:=\xi_++\xi_-+\eta$.

If $g_0=0$, then we must have $\Lf_{\nf_*^+}=\Lf_{\nf_*^-}$ and $\Lf_{\nf_a^\epsilon}=\Lf_{\nf_b^\epsilon}$ for each pair $\{a,b\}\in\Pc^\epsilon$. This puts us in the same situation as in part (1) of the proof of Proposition \ref{layerreg1}, so the same proof there leads to the conclusion that $\Lf_{\nf_*^+}=\Lf_{\nf_*^-}=q'$ (assuming say $q\geq q'$). Therefore $\Lf_{\nf_a^-}=q'$ for all $1\leq a\leq 2m_-$ just as in Proposition \ref{layerreg1}, and $(\Lf_{\nf_a^+}:1\leq a\leq 2m_+)$ forms a decreasing sequence in $[q',q]$, while an elementary combinatorial argument shows that the number of such sequences is at most $2^{q-q'+2m_+}$. Now as long as $C_0\geq 4$, by using the induction hypothesis for each of the couples $\Qc_{j,\epsilon,\iota}$ and $\Qc_{\mathrm{lp},j}$, we get that the number of choices of layerings for $\Qc$ is at most
\[\prod_{j,\epsilon,\iota}C_2^{g_{j,\epsilon,\iota}}C_0^{m_{j,\epsilon,\iota}}\cdot \prod_{j=1}^3C_2^{g_{\mathrm{lp},j}}C_0^{m_{\mathrm{lp},j}}\cdot 2^{q-q'+2m_+}\leq C_2^g C_0^{m+|q-q'|}.\]

Now assume $g_0\geq 1$. In this case $\Lf_{\nf_*^\epsilon}$ may not be $q'$, but $(\Lf_{\nf_a^\epsilon}:1\leq a\leq 2m_\epsilon)$ for $\epsilon\in\{\pm\}$ are still two decreasing sequences in $[0,\Df]$, so they have at most $2^{2\Df+2m_++2m_-}$ choices. The $(\Lf_{\nf_*^+},\Lf_{\nf_*^-})$ has at most $(\Df+1)^2$ choices, and using again the the induction hypothesis for each of the couples $\Qc_{j,\epsilon,\iota}$ and $\Qc_{\mathrm{lp},j}$, we get that the number of choices of layerings for $\Qc$ is at most
\[\prod_{j,\epsilon,\iota}C_2^{g_{j,\epsilon,\iota}}C_0^{m_{j,\epsilon,\iota}}\cdot \prod_{j=1}^3C_2^{g_{\mathrm{lp},j}}C_0^{m_{\mathrm{lp},j}}\cdot 2^{2\Df+2m_++2m_-}\cdot (\Df+1)^2C_0^{2g_0\Df+3\Df}\leq C_2^g C_0^{m},\] as long as $C_0\geq 4$, and $C_2\geq (4C_0)^{5\Df}$. This completes the proof. 
\end{proof}
\section{Layered regular objects II: Decay for non-coherent objects}\label{seclayer2} Recall the definition of $\Kc_\Qc^*$ and $\Kc_\Tc^*$ for layered regular couples $\Qc$ and regular trees $\Tc$ in Definition \ref{defkg}. The goal of this section and Section \ref{seclayer3} is to analyze these expressions; in this section we consider the non-coherent case, and in Section \ref{seclayer3} we consider the coherent case.

Below and in Section \ref{seclayer3}, we fix a layered regular couple $\Qc$ and a layered regular tree $\Tc$ of order $2n$ and incoherency index $0\leq g\leq n$ (except in Proposition \ref{proplayer4} where the sum of the two orders is $2n$), where $n\leq N_{p+1}^4$; if $g=0$ we shall restrict to those coherent objects described in Proposition \ref{layerreg1}. Consider the expression $\Kc_\Qc^*$ and $\Kc_\Tc^*$ as in Definition \ref{defkg}, with input functions $F_r(k)=\varphi(\delta r,k)+\Rs_r(k)$, as in Proposition \ref{propansatz} and Section \ref{secduhamel}. Assume $\Rs_r(k)$ is \emph{defined for all $k\in\Rb^d$} and \emph{real valued}, and satisfies the estimate (\ref{ansatz2}) for $0\leq r\leq p$, but with $\Lambda_p$ replaced by $40d$. Moreover, let $q$ and $q'$ be as in Definition \ref{defkg} (and in Proposition \ref{layerreg1} if $g=0$), and let $\Bc=[q,q+1]\times[q',q'+1]$ for $\Qc$, and let $\Bc=([q,q+1]\times[q',q'+1])\cap\{t>s\}$ for $\Tc$.
\subsection{The non-coherent estimates}\label{seclayer2-1} We start by proving the main estimates for $\Kc_\Qc^*$ and $\Kc_\Tc^*$ for non-coherent objects $\Qc$ and $\Tc$.
\begin{prop}\label{proplayer1} If $g\geq 1$, then we have
\begin{equation}\label{layerregest1}\|\Kc_\Qc^*\|_{\Xf^{\eta,30d,30d}(\Bc)}+\|\Kc_\Tc^*\|_{\Xf^{\eta,30d,0}(\Bc)}\lesssim_1 (C_1\delta)^n\cdot L^{-(\gamma_1-\sqrt{\eta})\cdot g}.
\end{equation}
\end{prop}
The proof of Proposition \ref{proplayer1} follows the strategy of Section 7 of \cite{DH23}, which is much simpler than the arguments in \cite{DH21}. We need the following two lemmas from \cite{DH23}.
\begin{lem}\label{lemlayer1} We use the notation $e(z)=e^{2\pi i z}$ (same for Lemma \ref{lemlayer2} below), and recall the cutoff function $\chi_0$ defined in Section \ref{setupnotat}. Then for any $|s|\leq L$, and uniformly in $(a,b,\xi,\xi')\in(\Rb^d)^4$, we have
\begin{equation}\label{layerregproof2}\sum_{(g,h)\in\Zb^{2d}}\bigg|\int_{\Rb^{2d}}\chi_0(x-a)\chi_0(y-b)\cdot e[\langle Lg+\xi,x\rangle+\langle Lh+\xi',y\rangle+s\langle x,y\rangle]\,\mathrm{d}x\mathrm{d}y\bigg|\lesssim_0 \langle s\rangle^{-d}.
\end{equation}
\end{lem}
\begin{proof} The proof is an easier version of Lemma 6.5 of \cite{DH23} (with $\Phi$ replaced by $1$ in Lemma 6.5 of \cite{DH23}). For fixed $(a,b,\xi,\xi')$, consider the two cases where (i) $|Lg+\xi+sb|\geq 10dL$ or $|Lh+\xi'+sa|\geq 10dL$, and (ii) $|Lg+\xi+sb|\leq 10dL$ and $|Lh+\xi'+sa|\leq 10dL$. In case (i), define
\[u:=x-a,\quad v:=y-b;\quad a':=a+\frac{Lh+\xi'}{s},\quad b':=b+\frac{Lg+\xi}{s},\] where we may assume $s\neq 0$, then up to a unimodular factor, the integral in $(x,y)$ reduces to
\begin{equation}\label{layerregnew1}
\int_{\Rb^{2d}}\chi_0(u)\chi_0(v)\cdot e(s\langle u+a',v+b'\rangle)\,\mathrm{d}u\mathrm{d}v.
\end{equation} Note that $|s|\leq L$, and $|u|,|v|\leq d$ due to the cutoff, and $\max(|sa'|,|sb'|)\geq 10dL$ due to assumption (i). Thus we can integrate by parts in $u$ (if $|sb'|\geq 10dL$) or $v$ (if $|sa'|\geq 10dL$) a total of $10d$ times to bound the $(x,y)$ integral by \[L^{-3d}(|sa+Lh+\xi'|^{-3d}+|sb+Lg+\xi|^{-3d}),\] which satisfies (\ref{layerregproof2}) upon summation in $(g,h)$, using also $|s|\leq L$.

We are left with case (ii). With fixed $(a,b,\xi,\xi')$ and $s$, this leaves at most $(20d+1)^{2d}\leq C_0$ choices for $(g,h)$, while for fixed $(g,h)$, the standard stationary phase argument implies that the integral in $(x,y)$ is bounded by $C_0\langle s\rangle^{-d}$, so (\ref{layerregproof2}) is also true in this case.
\end{proof}
\begin{lem}\label{lemlayer2} Uniformly in $(a,b,\xi,\xi'
)\in(\Rb^d)^4$, we have
\begin{multline}\label{layerregproof4}
\int_{L\leq |s|\leq \Df\delta L^{2\gamma}}\bigg|\sum_{(x,y)\in\Zb_L^{2d}}\chi_0(x-a)\chi_0(y-b)\cdot e(\langle x,\xi\rangle+\langle y,\xi'\rangle+s\langle x,y\rangle)\bigg|\,\mathrm{d}s\\\lesssim_2 L^{2d-2(d-1)(1-\gamma)+\eta}.
\end{multline}
\end{lem}
\begin{proof} See Lemma 6.6 of \cite{DH23} (with $D=\lambda=1$ and $\Phi\equiv 1$ in Lemma 6.6 of \cite{DH23}). The $\lesssim_2$ inequality is due to the upper bound $|s|\leq \Df\delta L^{2\gamma}$ where $\Df$ is controlled by $C_2$.
\end{proof}
We can now prove Proposition \ref{proplayer1}.
\begin{proof}[Proof of Proposition \ref{proplayer1}] We start with $\Kc_\Qc^*$; the discussion of $\Kc_\Tc^*$ is similar (with some additional twists) and will be left to the last step. We use the notation $e(z)=e^{2\pi i z}$ throughout this proof.

(1) \uwave{Preparations.} Recall the definition of $\Kc_\Qc^*$ and $\Kc_\Tc^*$ as in Definition \ref{defkg}, with the expression similar to (\ref{defkg0}) and time domain as in (\ref{timecouple}). For regular couple $\Qc$ we can link its branching nodes and select the subset $\Nc^{\mathrm{ch}}$ as in Proposition \ref{deflink}, then $|\Nc^{\mathrm{ch}}|=n$ is half the order of $\Qc$. For each $\nf\in\Nc^{\mathrm{ch}}$ define $x_\nf=k_{\nf_1}-k_\nf$ and $y_{\nf}=k_\nf-k_{\nf_3}$ as in Definition \ref{deflink}, where $\nf_j$ and $\nf_j'$ are the children nodes of $\nf$ and $\nf'$ as in Definition \ref{defdec}; note that $\Omega_\nf=2\langle x_\nf,y_\nf\rangle$ and $\zeta_{\nf'}\Omega_{\nf'}=-\zeta_\nf\Omega_\nf$ by Proposition \ref{deflink}.

Now consider the $2n+1$ vectors $k_\lf-k$ for all leaves $\lf\in\Lc$ of sign $+$; by examining the $k$-decoration of the tree with sign $+$, we see that they satisfy one linear equation of form
\begin{equation}\label{incohproof1}
\sum_{\lf\in\Lc}^{(+)}\sigma(\lf)(k_\lf-k)=0,
\end{equation} where $\sigma(\lf)\in\{-1,0,1\}$ with the sum of all the $\sigma(\lf)$ being also $\pm1$; moreover each $k_\lf-k$ is a fixed integer linear combination of $(x_\nf,y_\nf)$ and vice versa, which forms an integer coefficient linear bijection between $(\Zb_L^d)^{2n}$ and the subspace of $(\Zb_L^d)^{2n+1}$ defined by (\ref{incohproof1}). Therefore, the summation over all $k$-decorations $\Is$ of $\Qc$, appearing in (\ref{defkg0}), can be replaced by the summation over all $(x[\Nc^{\mathrm{ch}}],y[\Nc^{\mathrm{ch}}])\in(\Zb_L^d)^{2n}$ (using also the last statement in Proposition \ref{deflink}), so we can reduce (\ref{defkg0}) to
\begin{multline}\label{incohproof2}
\Kc_\Qc^*(t,s,k)=\bigg(\frac{\delta}{2L^{d-\gamma}}\bigg)^{2n}\zeta(\Qc)\int_{\Ic^*}\sum_{(x[\Nc^{\mathrm{ch}}],y[\Nc^{\mathrm{ch}}])}
\epsilon\prod_{\nf\in\Nc^{\mathrm{ch}}}e^{2\pi i\zeta_\nf\cdot\delta L^{2\gamma}(t_\nf-t_{\nf'})\langle x_\nf,y_\nf\rangle}\\\times W(k,x[\Nc^{\mathrm{ch}}],y[\Nc^{\mathrm{ch}}])\,\mathrm{d}t[\Nc].
\end{multline} Here $\Ic^*$ is the same as (\ref{timecouple}) and $(x[\Nc^{\mathrm{ch}}],y[\Nc^{\mathrm{ch}}])\in(\Zb_L^d)^{2n}$, and $\epsilon=\epsilon_\Is$ is as in Definition \ref{defdec} (in particular it only depends on $(x_\nf,y_\nf)$). Moreover $\nf'$ is the branching node linked with $\nf$, and $W$ is the function defined by
\begin{equation}\label{incohproof3}
W(k,x[\Nc^{\mathrm{ch}}],y[\Nc^{\mathrm{ch}}])=\prod_{\lf\in\Lc}^{(+)}F_{\Lf_\lf}(k+z_\lf),
\end{equation} where $z_\lf=k_\lf-k$ is a fixed integer coefficient linear combination of $(x[\Nc^{\mathrm{ch}}],y[\Nc^{\mathrm{ch}}])$.

Note that the expression (\ref{incohproof2}) depends on $k$ only via $W$ and $F_{\Lf_\lf}$ in (\ref{incohproof3}), so we if take any derivative in $k$, the derivative will fall on one of the $F_{\Lf_\lf}$ factors. Since we are taking $30d$ derivatives in (\ref{layerregest1}) and each $F_{\Lf_\lf}(k_\lf)$ is bounded in $\Sf^{40d,40d}$ by assumption, we may absorb these derivatives at the price of weakening the $\Sf^{40d,40d}$ norms by $\Sf^{10d,10d}$. Similarly, we can also absorb the weight $\langle k\rangle^{30d}$ occurring in (\ref{layerregest1}), because $k$ is a linear combination of all the $k_\lf$ by (\ref{incohproof1}), which implies \[\langle k\rangle^{30d}\lesssim_0 n^{30d}\cdot\max_{\lf}\langle k_\lf\rangle^{30d},\] where the $n^{30d}$ factor can be absorbed by $C_1^n$. Therefore, below we will only bound the $\Xf^{\eta,0,0}(\Bc)$ norm of $\Kc_\Qc^*$, assuming only that each ($\Rs_{\Lf_\lf}(k_\lf)$ and) $F_{\Lf_\lf}(k_\lf)$ is bounded in $\Sf^{10d,10d}$. Here note that, strictly speaking, the norm $\Xf^{\eta,30d,30d}(\Bc)$ involves an extension of $\Kc_\Qc^*$ to all values of $(t,s)\in\Rb^2$; however such extension will not depend on the derivatives taken, so this point does not affect the proof. In fact, this extension will be canonically fixed using some fixed smooth cutoff functions (see for example (\ref{incohproof19})), after which we only need to deal with the $\Xf^{\eta,0,0}$ norm. 

(2) \uwave{Localization of $x_\nf$ and $y_\nf$.} Now start with (\ref{incohproof2}). Note that the $\Xf^{\eta,0,0}$ norm, defined in (\ref{normX}), takes $L_k^\infty$ first before $L_{\lambda_j}^1$, which is less convenient here; however, using the extra decay in $k$ and Sobolev we have
\[\|G(k)\|_{L_k^\infty}\lesssim_0\max_{|\rho|\leq 2d}\|\partial_k^\rho G(k)\|_{L_k^1}\lesssim_0\max_{|\rho|\leq 2d}\|\langle k\rangle^{2d}\partial_k^\rho G(k)\|_{L_k^\infty}\] for any function $G(k)$, which then implies that
\begin{equation}\label{incohproof4}\|F\|_{\Xf^{\eta,0,0}}\lesssim_0\max_{|\rho|\leq 2d}\sup_k\langle k\rangle^{2d}\cdot\|\partial_k^\rho F(\cdot,k)\|_{\Xf^\eta},
\end{equation} where in the last norm we fix $k$ and measure the $\Xf^\eta$ norm in the time variables. The $2d$ derivatives and weight $\langle k\rangle^{2d}$ can again be absorbed at the price of weakening the $\Sf^{10d,10d}$ norm of $F_{\Lf_\lf}(k_\lf)$ to $\Sf^{8d,8d}$, which then allows us to fix $k$ and only deal with the $\Xf^\eta$ norm (where, again, the extension does not depend on the extra derivatives nor on $k$).

With fixed $k$, we may fix some value $k_\lf^0\in\Zb^d$ for each $\lf\in\Lc$ of sign $+$, and restrict to the region $|k_\lf-k_\lf^0|\leq 1$ for each such $\lf$, where $k_\lf=k+z_\lf$ as above; such term carries a quickly decaying and summable coefficient in $(k_\lf^0)$, due to the decay in $k_\lf$ provided by the $\Sf^{8d,8d}$ norm bounds of each $\Rs_{\Lf_\lf}(k_\lf)$ and $F_{\Lf_\lf}(k_\lf)$. Then, with fixed $(k_\lf^0)$, we may apply Lemma \ref{treelemma}, which allows to decompose each single term with fixed $(k_\lf^0)$ into at most $C_0^n$ sub-terms, such that for each sub-term there exist constant vectors $(a_\nf,b_\nf)$ for $\nf\in\Nc^{\mathrm{ch}}$ such that \[|x_\nf-a_\nf|\leq 1/2\mathrm{\ and\ }\quad |y_\nf-b_\nf|\leq 1/2 \mathrm{\ for\ all\ }\nf\in\Nc^{\mathrm{ch}}.\]

Now we will focus on one of the sub-terms (the factor $C_0^n$ can be absorbed) and fix $(a_\nf,b_\nf)$; in particular, using a partition of unity we can insert the cutoff functions $\chi_0(x_\nf-a_\nf)\chi_0(y_\nf-b_\nf)$. Next, notice that the two sets of vectors $(k_\lf)$ and $(k,x_\nf,y_\nf)$, where $\lf\in\Lc$ has sign $+$ and $\nf\in\Nc^{\mathrm{ch}}$, are related by a (volume preserving) integer coefficient linear transform, all of whose coefficients are at most $1$ in absolute value. Now the function
\begin{equation}\label{incohproof5}
W=\prod_{\lf\in\Lc}^{(+)}F_{\Lf_\lf}(k_\lf)=\prod_{\lf\in\Lc}^{(+)}(\varphi(\delta\Lf_\lf,k+z_\lf)+\Rs_{\Lf_\lf}(k+z_\lf)),
\end{equation} viewed in the variables $(k_\lf)$, has Fourier $L^1$ norm $\lesssim_1C_1^n$ due to the factorization structure and $\Sf^{8d,8d}$ bound assumptions; with the change of variables, we can bound the Fourier $L^1$ norm in the $(k,x_\nf,y_\nf)$ variables, namely
\begin{equation}\label{incohproof6}
W(k,x[\Nc^{\mathrm{ch}}],y[\Nc^{\mathrm{ch}}])=\int_{(\Rb^d)^{2n+1}}\widehat{W}(\lambda,\xi[\Nc^{\mathrm{ch}}],\xi'[\Nc^{\mathrm{ch}}])\cdot e(\langle \lambda,k\rangle)\prod_{\nf\in\Nc^{\mathrm{ch}}}e(\langle \xi_\nf,x_\nf\rangle+\langle \xi_\nf',y_\nf\rangle)\,\mathrm{d}\xi_\nf\mathrm{d}\xi_\nf'\mathrm{d}\lambda,
\end{equation} where
\begin{equation}\label{incohproof7}
\int_{(\Rb^d)^{2n+1}}|\widehat{W}(\lambda,\xi[\Nc^{\mathrm{ch}}],\xi'[\Nc^{\mathrm{ch}}])|\,\mathrm{d}\lambda\mathrm{d}\xi[\Nc^{\mathrm{ch}}]\mathrm{d}\xi'[\Nc^{\mathrm{ch}}]\lesssim_1 C_1^n.
\end{equation} As $C_1^n$ can be absorbed by the right hand side of (\ref{layerregest1}), we can replace the function $W$ in (\ref{incohproof2}) by products of $\chi_0$ and phase-shift factors, and we only need to estimate
\begin{multline}\label{incohproof8}
\widetilde{\Kc}(t,s):=\bigg(\frac{\delta}{2L^{d-\gamma}}\bigg)^{2n}\zeta(\Qc)\int_{\Ic^*}\sum_{(x[\Nc^{\mathrm{ch}}],y[\Nc^{\mathrm{ch}}])}
\epsilon\prod_{\nf\in\Nc^{\mathrm{ch}}}\bigg(\chi_0(x_\nf-a_\nf)\chi_0(y_\nf-b_\nf)\\\times e\big[\zeta_\nf\cdot\delta L^{2\gamma}(t_\nf-t_{\nf'})\langle x_\nf,y_\nf\rangle+\langle \xi_\nf,x_\nf\rangle+\langle \xi_\nf',y_\nf\rangle\big]\bigg)\,\mathrm{d}t[\Nc]
\end{multline} uniformly in all choices of $(\xi_\nf,\xi_\nf')$.

(3) \uwave{Estimating the summation.} Now consider $\widetilde{\Kc}(t,s)$, with fixed $k$, defined in (\ref{incohproof8}). With fixed $(a_\nf,b_\nf,\xi_\nf,\xi_\nf')$, we can write (where $\nf'$ is the branching node linked with $\nf$)
\begin{equation}\label{incohproof9}
\widetilde{\Kc}(t,s)=\bigg(\frac{\delta}{2L^{d-\gamma}}\bigg)^{2n}\zeta(\Qc)\int_{\Ic^*}\prod_{\nf\in\Nc^{\mathrm{ch}}}\Sc_{\nf}(t_{\nf},t_{\nf'})\,\mathrm{d}t[\Nc],
\end{equation} where
\begin{equation}\label{incohproof10}
\Sc_{\nf}(t_{\nf},t_{\nf'})=\sum_{(x,y)\in\Zb_L^{2d}}\epsilon_\nf\cdot\chi_0(x_\nf-a_\nf)\chi_0(y_\nf-b_\nf)\cdot e\big[\zeta_\nf\cdot\delta L^{2\gamma}(t_\nf-t_{\nf'})\langle x_\nf,y_\nf\rangle+\langle \xi_\nf,x_\nf\rangle+\langle \xi_\nf',y_\nf\rangle\big],
\end{equation} and $\epsilon_\nf$ equals $1$, $0$ or $-1$ if $0$, $1$ or $2$ vectors in $\{x,y\}$ equal $0$. Note also that, by (\ref{timecouple}), we have
\begin{equation}\label{incohproof11}\Ic^*\subset\prod_{\nf\in\Nc^{\mathrm{ch}}}\Bc_\nf,\quad \Bc_\nf=[\Lf_{\nf},\Lf_{\nf}+1]\times[\Lf_{\nf'},\Lf_{\nf'}+1];
\end{equation} in particular \emph{for each $\nf$ such that $\Lf_{\nf}\neq \Lf_{\nf'}$,} there exists a fixed integer $q$ such that $t_{\nf}\leq q\leq t_{\nf'}$ or $t_{\nf'}\leq q\leq t_\nf$.

We now analyze $\Sc_n$. A trivial bound is $|\Sc_\nf|\lesssim_1 L^{2d}$. Next, let $\zeta_\nf\cdot\delta L^{2\gamma}(t_\nf-t_{\nf'}):=s_\nf$, so that $|s_\nf|\leq \Df\delta L^{2\gamma}$ and $\Sc_\nf$ depends only on $s_\nf$, then upon replacing $\epsilon_\nf$ with $1$, we get that
\begin{equation}\label{incohproof12}
\Sc_\nf=\sum_{(x,y)\in\Zb_L^{2d}}\chi_0(x_\nf-a_\nf)\chi_0(y_\nf-b_\nf)\cdot e\big[s_\nf\langle x_\nf,y_\nf\rangle+\langle \xi_\nf,x_\nf\rangle+\langle \xi_\nf',y_\nf\rangle\big]+\Oc_0(L^d).
\end{equation}If $|s_\nf|\geq L$, then Lemma \ref{lemlayer2} implies that
\begin{equation}\label{incohproof13}
\|\Sc_\nf\cdot\mathbf{1}_{|s_\nf|\geq L}\|_{L^1(\Bc_\nf)}\lesssim_2\delta^{-1}L^{2(d-\gamma)}\cdot L^{-2(d-1)(1-\gamma)+\eta};
\end{equation} if $|s_\nf|\leq L$, we can apply Poisson summation to (\ref{incohproof12}) to bound it by an expression similar to the left hand side of (\ref{layerregproof2}), then Lemma \ref{lemlayer1} implies that $|\Sc_\nf|\lesssim_0 L^{2d} \langle s\rangle^{-d}$, and hence
\begin{equation}\label{incohproof14}
\|\Sc_\nf\cdot\mathbf{1}_{|s_\nf|\leq L}\|_{L^1(\Bc_\nf)}\lesssim_0\delta^{-1}L^{2(d-\gamma)}.
\end{equation} Moreover, \emph{if $\Lf_{\nf}\neq \Lf_{\nf'}$}, then there exists a fixed integer $q$ such that $|t_\nf-q|\leq |t_\nf-t_{\nf'}|=\delta^{-1}L^{-2\gamma}|s_\nf|$, hence
\begin{equation}\label{incohproof15}
\|\Sc_\nf\cdot\mathbf{1}_{|s_\nf|\leq L}\|_{L^1(\Bc_\nf)}\lesssim_2 \delta^{-1}L^{2(d-\gamma)}\cdot L^{-2\gamma}
\end{equation} as $\delta^{-1}\leq C_2$. Putting together, and noticing that there are exactly $1\leq g\leq n$ nodes $\nf$ such that $\Lf_{\nf}\neq \Lf_{\nf'}$, we get that
\begin{equation}\label{incohproof16}
\int_{\Ic^*}\prod_{\nf\in\Nc^{\mathrm{ch}}}|\Sc_{\nf}(t_{\nf},t_{\nf'})|\,\mathrm{d}t[\Nc]\leq C_0^nC_2^g\cdot \delta^{-n}L^{2(d-\gamma)n}\cdot L^{-(\gamma_1-\eta)g}.
\end{equation}

Finally, we turn to the study of $\widetilde{\Kc}(t,s)$. By (\ref{incohproof9}) and the definition of $\Ic^*$, we can write
\begin{equation}\label{incohproof17}
\widetilde{\Kc}(t,s)=\int_q^t\int_{q'}^s\Sc(t',s')\,\mathrm{d}t'\mathrm{d}s'
\end{equation} for some function $\Sc$ derived from (\ref{incohproof9}), where $t'$ and $s'$ are the $t_\nf$ variables with $\nf$ being the roots of two trees of $\Qc$. Note that, in view of the possibilities of layerings, one or both of the integrations in (\ref{incohproof17}) may be taken on a unit interval independent of $t$ (or $s$) such as $[r,r+1]$ for some $r<q$ (or $r<q'$). In such cases, including the case when one of the trees of $\Qc$ is trivial, the function $\widetilde{\Kc}(t,s)$ will not depend on $t$ (or $s$) and the proof becomes much easier. Moreover, in (\ref{incohproof17}) we have
\begin{equation}\label{incohproof18}
\|\Sc\|_{L^1}\leq (C_0\delta)^nL^{-(\gamma_1-2\eta)g},\quad \|\Sc\|_{L^\infty}\leq (C_0\delta)^nL^{-(\gamma_1-2\eta)g}\cdot L^{2d},
\end{equation} where the first inequality follows from (\ref{incohproof16}), and the second inequality follows from (\ref{incohproof13})--(\ref{incohproof15}) and the pointwise bound for $\Sc_\nf$. Now by translation we may replace $q$ and $q'$ by $0$ and assume $t,s\in[0,1]$, then from (\ref{incohproof18}), Lemma \ref{timeintlem} and interpolation, we conclude that a canonical extension
\begin{equation}\label{incohproof19}
(\widetilde{\Kc})_{\mathrm{ext}}(t,s)=\int_0^t\int_0^s\chi(t)\chi(t')\chi(s)\chi(s')\cdot\Sc(t',s')\,\mathrm{d}t'\mathrm{d}s',
\end{equation} for some suitable cutoff functions $\chi$, satisfies that
\begin{equation}\label{incohproof20}
\|(\widetilde{\Kc})_{\mathrm{ext}}\|_{\Xf^\eta}\lesssim_1 (C_0\delta)^nL^{-(\gamma_1-100d\eta)g}.
\end{equation} By the discussion at the end of (1), this proves (\ref{layerregest1}) for $\Kc_\Qc^*$.

(4) \uwave{The $\Kc_\Tc^*$ case.} The proof of (\ref{layerregest1}) for $\Kc_\Tc^*$ is mostly similar to the $\Kc_\Qc^*$ case but with one important difference, namely that, due to the absence of the factor $F_{\Lf_\lf}(k_\lf)$ for the lone leaf $\lf=\lf_{\mathrm{lo}}$, the function $\Kc_\Tc^*(t,s,k)$ will not have any decay in $k$. This first leads to the choice of the $\Xf^{\eta,30d,0}$ norm without weight in (\ref{layerregest1}) (while the derivatives are treated in the same way as $\Kc_\Qc^*$); moreover the manipulation in (\ref{incohproof4}) is not possible, so we still need to use the original $\Xf^{\eta,30d,0}$ norm that takes $L_k^\infty$ before $L_{\lambda_j}^1$, but the first inequality in (\ref{estichi}) is not true for such norms because the Hausdorff-Young inequality is not available for Banach space valued functions. Therefore new arguments are needed to treat this norm.

By the same arguments as in the $\Kc_\Qc^*$ case, we can reduce the estimate of $\Kc_\Tc^*$ to that of
\begin{equation}\label{incohproof21}
\widetilde{\Kc}(t,s,k)=\bigg(\frac{\delta}{2L^{d-\gamma}}\bigg)^{2n}\zeta(\Tc)\int_{\Dc^*}\prod_{\nf\in\Nc^{\mathrm{ch}}}\Sc_{\nf}(t_{\nf},t_{\nf'})\,\mathrm{d}t[\Nc]
\end{equation} similar to (\ref{incohproof9}), where $\Dc^*$ is as in (\ref{timepairtree}), $\Sc_\nf$ is as in (\ref{incohproof10}) and satisfies (\ref{incohproof13})--(\ref{incohproof15}). This expression depends on the parameters $(a_\nf,b_\nf,\xi_\nf,\xi_\nf')$ as defined above, all of which are allowed to depend on $k$.

Consider the nodes $\nf\in\{\rf,\uf\}$ and the corresponding linked nodes $\nf'$, where $\rf$ is the root of $\Tc$ and $\uf:=\nf_{\mathrm{lo}}^{\mathrm{pr}}$ is the parent node of the lone leaf. By Propositions \ref{propstructure} and \ref{deflink} there are only two cases: when $\rf'=\uf$, or when $(\rf,\rf',\uf',\uf)$ are listed in this order with each being a descendant of the previous one. In the first case we have
\begin{equation}\label{incohproof22}
\widetilde{\Kc}(t,s,k)=\int_{s}^t\int_{s}^{t_1}\Sc_0(t_1,s_1,k)\Sc(t_1,s_1,k)\,\mathrm{d}s_1\mathrm{d}t_1,
\end{equation} and in the second case we have
\begin{equation}\label{incohproof23}
\widetilde{\Kc}(t,s,k)=\int_{s}^t\int_{s}^{t_1}\int_s^{t_2}\int_s^{s_2}\Sc_1(t_1,t_2,k)\Sc_2(s_2,s_1,k)\Sc(t_2,s_2,k)\,\mathrm{d}s_1\mathrm{d}s_2\mathrm{d}t_2\mathrm{d}t_1
\end{equation} for some functions $\Sc$ and $\Sc_j$ derived from (\ref{incohproof21}), where in (\ref{incohproof22}) we have $t_1=t_\rf$ and $s_1=t_{\uf}$, and in (\ref{incohproof23}) we have $(t_1,t_2,s_2,s_1)=(t_\rf,t_{\rf'},t_{\uf'},t_{\uf})$. Note that, in view of the possibilities of layerings, the upper and lower bounds $t_j$ or $s_j$ in some of the integrals in (\ref{incohproof22}) and (\ref{incohproof23}) may be replaced by some constant integer $r$ with $q\leq r\leq q'$, but the proof in these cases can be done similarly.

Now, with the formulas (\ref{incohproof22}) and (\ref{incohproof23}), we can define the canonical extension $(\widetilde{\Kc})_{\mathrm{ext}}$ by inserting all the cutoff functions $\chi(t)\chi(s)$ and $\chi(t_j)\chi(s_j)$ (after suitable translations) for all $j$ as in (\ref{incohproof19}), and we only need to prove that 
\begin{equation}\label{incohproof24}\|(\widetilde{\Kc})_{\mathrm{ext}}\|_{\Xf^{\eta,0,0}}\lesssim_1 (C_0\delta)^nL^{-(\gamma_1-100d\eta)g}
\end{equation} under the assumption that all input functions $F_{\Lf_\lf}(k_\lf)$ are bounded in $\Sf^{10d,10d}$. We may assume that all the $\Sc$ and $\Sc_j$ functions are supported in some fixed unit boxes; they satisfy the bound
\begin{equation}\label{incohproof25}
\begin{aligned}
\|\Sc_0(t_1,s_1)\|_{L^1(\Rb^2)}\cdot\|\Sc(t_1,s_1)\|_{L^\infty(\Rb^2)}&\leq (C_0\delta)^nL^{-(\gamma_1-2\eta)g}\qquad \mathrm{for\ }(\ref{incohproof22}),\\\|\Sc_1(t_1,t_2)\|_{L^1(\Rb^2)}\cdot\|\Sc_2(s_2,s_1)\|_{L^1(\Rb^2)}\cdot\|\Sc(t_2,s_2)\|_{L^\infty(\Rb^2)}&\leq (C_0\delta)^nL^{-(\gamma_1-2\eta)g}\qquad\mathrm{for\ }(\ref{incohproof23})
\end{aligned}
\end{equation} uniformly in $k$ (due to (\ref{incohproof13})--(\ref{incohproof15})), so Lemma \ref{timeintlem} already implies
\begin{equation}\label{incohproof26}\|(\widetilde{\Kc})_{\mathrm{ext}}\|_{\Xf^{-\eta,0,0}}\lesssim_1 (C_0\delta)^nL^{-(\gamma_1-2d\eta)g}.
\end{equation} Next, by expanding out $\Sc_j$ using (\ref{incohproof10}) and allowing to lose a power of at most $L^{4d}$, we can replace $\Sc_0$ in (\ref{incohproof22}) by $e^{i\Gamma_1(t_1-s_1)}$,  and replace $\Sc_1$ and $\Sc_2$ in (\ref{incohproof23}) by $e^{i\Gamma_1(t_1-t_2)}$ and $e^{i\Gamma_2(s_2-s_1)}$ in (\ref{incohproof23}), where $\Gamma_j$ are quantities depending on $k$. We also have
\begin{equation}\label{incohproof27}
\begin{aligned} \|\partial_{t_1}\partial_{s_1}\Sc(t_1,s_1)\|_{L^1(\Rb^2)}&\leq (C_0\delta)^nL^{-(\gamma_1-100d\eta)g}\cdot L^{4d}\qquad\mathrm{for\ }(\ref{incohproof22}),\\
\|\partial_{t_2}\partial_{s_2}\Sc(t_2,s_2)\|_{L^1(\Rb^2)}&\leq (C_0\delta)^nL^{-(\gamma_1-100d\eta)g}\cdot L^{4d}\qquad\mathrm{for\ }(\ref{incohproof23})
\end{aligned}
\end{equation} again uniformly in $k$, because any exponential factor depending on $(t_1,s_1)$ (or $(t_2,s_2)$ for (\ref{incohproof23})) are already included in the $\Sc_0$ and $(\Sc_1,\Sc_2)$ factors. Then, by Lemma \ref{timeintlem2}, we get
\begin{equation}\label{incohproof28}
\|(\widetilde{\Kc})_{\mathrm{ext}}\|_{\Xf^{1/4,0,0}}\lesssim_1(C_0\delta)^nL^{-(\gamma_1-100d\eta)g}\cdot L^{4d};
\end{equation} finally by interpolating with (\ref{incohproof26}) we get (\ref{incohproof24}).
\end{proof}
\section{Layered regular objects III: Asymptotics for coherent objects}\label{seclayer3}
In this section we consider the coherent case, as described in Proposition \ref{layerreg1}. Let the objects and notations (including $(q,q')$ and $\Bc$ etc.) be fixed as in the beginning of Section \ref{seclayer2}, where we now assume $g=0$. Note that the expressions $(\Kc_\Qc^*,\Kc_\Tc^*)$ and all our analysis below depend only on the values $(q,q')$, which are defined using the bigger garden $\Gc$ in Proposition \ref{layerreg1}, but \emph{not} on the rest of the garden $\Gc$.
\subsection{The coherent estimates}\label{seclayer3-1} The main estimates for $\Kc_\Qc^*$ and $\Kc_\Tc^*$ for coherent objects $\Qc$ and $\Tc$, including the asymptotic information, are stated as follows.
\begin{prop}\label{proplayer2} If $g=0$, then we can decompose
\begin{equation}\label{layerregest2}\Kc_\Qc^*=(\Kc_\Qc^*)_{\mathrm{app}}+\Rs_\Qc^*,\quad \Kc_\Tc^*=(\Kc_\Tc^*)_{\mathrm{app}}+\Rs_\Tc^*,
\end{equation} where $(\Kc_\Qc^*)_{\mathrm{app}}$ and $(\Kc_\Tc^*)_{\mathrm{app}}$ are nonzero only for \emph{dominant} couples $\Qc$ and trees $\Tc$ (see Proposition \ref{propstructure} for definition), and the remainders satisfy that
\begin{equation}\label{layerregest3}\|\Rs_\Qc^*\|_{\Xf^{\eta,30d,30d}(\Bc)}+\|\Rs_\Tc^*\|_{\Xf^{\eta,30d,0}(\Bc)}\lesssim_1 (C_1\delta)^n\cdot L^{-(\gamma_1-\sqrt{\eta})}.
\end{equation} Moreover, $(\Kc_\Qc^*)_{\mathrm{app}}$ can be written as the sum of at most $2^n$ terms of form
\begin{equation}\label{layerregest4}
(\Kc_\Qc^*)_{\mathrm{app}}(t,s,k)=\sum\delta^n\cdot\widetilde{\Jc_\Qc}(t,s)\cdot \widetilde{\Mc_\Qc}(k),
\end{equation} where each term satisfies
\begin{equation}\label{layerregest5}
\|\widetilde{\Jc_\Qc}\|_{\Xf^{1-\eta}(\Bc)}\lesssim_1 C_1^{n},\quad\|\widetilde{\Mc_\Qc}\|_{\Sf^{30d,30d}}\lesssim_1 C_1^{n}.
\end{equation} Similarly, $(\Kc_\Tc^*)_{\mathrm{app}}$ can be written as the sum of at most $2^n$ terms of form (\ref{layerregest4}) but with $\widetilde{\Jc_\Tc}$ and $\widetilde{\Mc_\Tc}$ instead of $\widetilde{\Jc_\Qc}$ and $\widetilde{\Mc_\Qc}$, where each term satisfies
\begin{equation}\label{layerregest6}
\|\widetilde{\Jc_\Tc}\|_{\Xf^{1-\eta}(\Bc)}\lesssim_1 C_1^{n},\quad\|\widetilde{\Mc_\Tc}\|_{\Sf^{30d,0}}\lesssim_1 C_1^{n}.
\end{equation}
\end{prop}
\begin{prop}\label{proplayer3} For $t\in[p,p+1]$ and $0\leq n\leq N_{p+1}^4$, we have
\begin{equation}\label{layerregest7}\sum_{\Qc}(\Kc_\Qc^*)_{\mathrm{app}}(t,t,k)=\Uc_n(t,k)+\widetilde{\Rs_n}(t,k),
\end{equation} where $\Uc_n$ is defined as in (\ref{wketaylor2}), and the sum is taken over all regular couples $\Qc$ of order $2n$ such that all its nodes are in layer $p$. The remainder satisfies, for all $t\in[p,p+1]$ and $k\in\Rb^d$, that
\begin{equation}\label{layerregest8}|\widetilde{\Rs_n}(t,k)|\lesssim_1 (C_1\delta)^n\langle k\rangle^{-30d}\cdot L^{-(\theta_{p+1}+4\theta_{p})/5}.
\end{equation}
\end{prop}
\begin{prop}\label{proplayer4} Let the conjugate of couple $\Qc$ and paired tree $\Tc$ be $\overline{\Qc}$ and $\overline{\Tc}$ as in Definitions \ref{deftree} and \ref{defgarden} (which are obviously extended to include layerings), then we have
\begin{equation}\label{layerregest9}\Kc_{\overline{\Qc}}^*(t,s,k)=\overline{\Kc_\Qc^*(s,t,k)},\quad \Kc_{\overline{\Tc}}^*(t,s,k)=\overline{\Kc_\Tc^*(t,s,k)}
\end{equation} (the same holds for the $(\cdots)_{\mathrm{app}}$ approximations in Proposition \ref{proplayer2}). Moreover, fix any $0\leq q'\leq q\leq p$ and $0\leq n\leq N_{p+1}^4$. Then for any $(t,s)\in\Bc$ (with $t>s$) and $k\in\Rb^d$, we have
\begin{equation}\label{layerregest10}
\sum_{(\Qc,\Tc)}(\Kc_\Qc^*)_{\mathrm{app}}(t,s,k)\cdot\overline{(\Kc_\Tc^*)_{\mathrm{app}}(t,s,k)}\in\Rb,
\end{equation} where the sum is taken over all the coherent regular couples $\Qc$ and regular trees $\Tc$ as described in Proposition \ref{layerreg1} corresponding to this fixed $(q,q')$, such that $\Tc$ has sign $+$, and $\Qc$ and $\Tc$ has \emph{total order} $n(\Qc)+n(\Tc)=2n$. In addition, if in (\ref{layerregest10}) we make the further restriction that the total number of layer $p$ branching nodes in $\Qc$ and $\Tc$ equals a given constant, then (\ref{layerregest10}) remains true.
\end{prop}
\subsection{Preparations for the proof} \label{seclayer3-2} The proof of Propositions \ref{proplayer2}--\ref{proplayer4} mainly rely on combinatorial arguments, plus the analysis of $\Kc_\Qc^*$ and $\Kc_\Tc^*$, where $\Qc$ and $\Tc$ are regular couples and regular trees such that all their nodes (except the lone leaf) are in the same layer. Such analysis is proved in earlier works \cite{DH21,DH23}, and will be used as black box.
\begin{df}[Enhanced dominant couples and equivalence \cite{DH21,DH23}]\label{defencpl} Recall the notion of type 1 and 2 regular couples, and dominant couples (and paired trees) in Proposition \ref{propstructure}. Let $\Qc$ be a layered dominant couple, and assume that all nodes of $\Qc$ are in the same layer. We define an \emph{enhanced dominant couple} to be a pair $\Qs=(\Qc,Z)$, such that $Z\subset\Nc^{\mathrm{ch}}$, and for each $\nf\in Z$ the branching node $\nf'$ linked with $\nf$ is a child node of $\nf$ (we call such $Z\subset\Nc^{\mathrm{ch}}$ a \emph{special} subset).

Using Proposition \ref{propstructure}, we can define an equivalence relation $\sim$ between enhanced dominant couples $\Qs=(\Qc,Z)$, as follows. First the enhanced trivial couple (with no branching node so $\Nc^{\mathrm{ch}}=Z=\varnothing$) is only equivalent to itself, and two enhanced dominant couples where $\Qc$ have different types or different layers are never equivalent. Next, if $\Qs=(\Qc,Z)$ and $\Qs'=(\Qc',Z')$, where $\Qc$ and $\Qc'$ have type 1 and the same layer, then we have $\Nc^{\mathrm{ch}}=\Nc_1^{\mathrm{ch}}\cup\Nc_2^{\mathrm{ch}}\cup\Nc_3^{\mathrm{ch}}\cup\{\rf\}$ as in Proposition \ref{deflink}, where $\Nc_j^{\mathrm{ch}}$ is the $\Nc^{\mathrm{ch}}$ set for $\Qc_j$ and $\rf$ is the root with $+$ sign, see Proposition \ref{propstructure}. In this case $Z$ is special if and only if $Z=Z_1\cup Z_2\cup Z_3$ (i.e. $\rf$ is \emph{not} in $Z$) where $Z_j\subset \Nc_j^{\mathrm{ch}}$ is special, and similarly for $\Qc'$. Let $\Qs_j=(\Qc_j,Z_j)$, we then define $\Qs\sim\Qs'$ if and only if $\Qs_j\sim\Qs_j'$ for $1\leq j\leq 3$.

Now let $\Qs$ and $\Qs'$ be as before, but suppose $\Qc$ and $\Qc'$ have type $2$ and the same layer. Let $\Qc_0$ and $\Xc^{\epsilon}$ be associated with $\Qc$ as in Proposition \ref{propstructure}, and similarly for $\Qc'$ (same for the other objects appearing below). We use the notation in Proposition \ref{propstructure} for $\Qc$, and correspondingly for $\Qc'$; note that $\nf_{2j-1}^\pm$ is paired with $\nf_{2j}^\pm$ for $1\leq j\leq m_\pm$ by definition of dominant couples. Recall also that 
\begin{equation}\label{unionch}\Nc^{\mathrm{ch}}=\bigg(\bigcup_{j,\epsilon,\iota}\Nc_{j,\epsilon,\iota}^{\mathrm{ch}}\bigg)\cup\Nc_{\mathrm{lp}}^{\mathrm{\mathrm{ch}}}\cup\big\{\nf_{2j-1}^+:1\leq j\leq m_+\big\}\cup\big\{\nf_{2j-1}^-:1\leq j\leq m_-\big\}
\end{equation} as in Proposition \ref{deflink}; then $Z$ is special if and only if
\begin{equation}\label{unionz}Z=\bigg(\bigcup_{j,\epsilon,\iota}Z_{j,\epsilon,\iota}\bigg)\cup Z_{\mathrm{lp}}\cup\big\{\nf_{2j-1}^+:j\in Z^+\big\}\cup\big\{\nf_{2j-1}^-:j\in Z^-\big\}
\end{equation} for some special subsets $Z_{j,\epsilon,\iota}\subset \Nc_{j,\epsilon,\iota}^{\mathrm{ch}}$ and $Z_{\mathrm{lp}}\subset\Nc_{\mathrm{lp}}^{\mathrm{ch}}$, and some subsets $Z^\epsilon\subset \{1,\cdots,m_\epsilon\}$. Similar representations are defined for $\Qs'$. For $\epsilon\in\{\pm\}$ and each $1\leq j\leq m_\epsilon$, consider the tuple $(\mathtt{I}_{j,\epsilon},\mathtt{c}_{j,\epsilon},\Xs_{j,\epsilon,1},\Xs_{j,\epsilon,2})$; here $\mathtt{I}_{j,\epsilon}=1$ if $j\in Z^\epsilon$ and $\mathtt{I}_{j,\epsilon}=0$ otherwise, $\mathtt{c}_{j,\epsilon}\in\{1,2,3\}$ is the \emph{first digit of} the code of the mini tree associated with the pair $\{2j-1,2j\}\in \Pc^\epsilon$ (see Definition \ref{defreg}),  and $\Xs_{j,\epsilon,\iota}$ is the equivalence class of the enhanced dominant couple $\Qs_{j,\epsilon,\iota}=(\Qc_{j,\epsilon,\iota},Z_{j,\epsilon,\iota})$ for $\iota\in\{1,2\}$. Define also $\Xs_{\mathrm{lp}}$ to be the equivalence class of the enhanced dominant couple $\Qs_{\mathrm{lp}}=(\Qc_{\mathrm{lp}},Z_{\mathrm{lp}})$.

We now define $\Qs\sim\Qs'$, if and only if (i) $m_++m_-=m_+'+m_-'$, and (ii) the tuples coming from $\Qc$ (there are total $m_++m_-$ of them) form \emph{a permutation of} the corresponding tuples coming from $\Qc'$ (there are total $m_+'+m_-'$ of them), and (iii) $\Xs_{\mathrm{lp}}=\Xs_{\mathrm{lp}}'$. Finally, note that if $\Qs=(\Qc,Z)$ and $\Qs'=(\Qc',Z')$ are equivalent then $n(\Qc)=n(\Qc')$ and $|Z|=|Z'|$.
\end{df}
\begin{df}[Enhanced dominant trees and equivalence \cite{DH23}]\label{defentree} Let $\Tc$ be a layered dominant tree, and assume that all its nodes (except the lone leaf) are in the same layer. We define an \emph{enhanced dominant tree} to be a pair $\Ts=(\Tc,Z)$ where $Z\subset\Nc^{\mathrm{ch}}$ is a \emph{special} subset in the sense that for each $\nf\in Z$ the branching node $\nf'$ linked with $\nf$ is a child node of $\nf$.

We then define the notion of equivalence between enhanced dominant trees in the same way as in Definition \ref{defencpl}: instead of (\ref{unionch})--(\ref{unionz}) we have
\begin{equation}\label{unionchz2}
\Nc^{\mathrm{ch}}=\bigg(\bigcup_{j,\epsilon,\iota}\Nc_{j,\epsilon,\iota}^{\mathrm{ch}}\bigg)\cup\big\{\nf_{2j-1}^0:1\leq j\leq m_0\big\},\quad Z=\bigg(\bigcup_{j,\epsilon,\iota}Z_{j,\epsilon,\iota}\bigg)\cup\big\{\nf_{2j-1}^0:j\in Z^0\big\}
\end{equation} in the notation of Proposition \ref{propstructure}. Consider the tuples $(\mathtt{I}_{j,0},\mathtt{c}_{j,0},\Xs_{j,0,1},\Xs_{j,0,2})$; here $\mathtt{I}_{j,0}=1$ if $j\in Z^0$ and $\mathtt{I}_{j,0}=0$ otherwise, $\mathtt{c}_{j,0}\in\{1,2,3\}$ is the first digit of the code of the mini tree associated with the pair $\{2j-1,2j\}\in \Pc^0$ (see Definition \ref{defreg}),  and $\Xs_{j,0,\iota}$ is the equivalence class of the enhanced dominant couple $\Qs_{j,0,\iota}=(\Qc_{j,0,\iota},Z_{j,0,\iota})$ for $\iota\in\{1,2\}$. Then, we define $\Ts=(\Tc,Z)$ and $\Ts'=(\Tc',Z')$ to be equivalent if they have the same layer, and the tuples coming from $\Tc$ form a permutation of the corresponding tuples coming from $\Tc'$.
\end{df}
We now state the results for $\Kc_\Qc^*$ and $\Kc_\Tc^*$ proved in \cite{DH21,DH23}. In Propositions \ref{regref1}--\ref{regref3} below, we assume $\Qc$ is a regular couple and $\Tc$ is a regular tree of order $2n$, such that all the nodes (except the lone leaf of $\Tc$) are in the same layer $r$ with $0\leq r\leq p$.
\begin{prop}\label{regref1} If $\Qc$ and $\Tc$ are \emph{not dominant}, then we have
\begin{equation}\label{regref1-1}
\|\Kc_\Qc^*\|_{\Xf^{\eta,30d,30d}(\Bc)}+\|\Kc_\Tc^*\|_{\Xf^{\eta,30d,0}(\Bc)}\lesssim_1(C_1\delta)^nL^{-(\gamma_1-\sqrt{\eta})}.
\end{equation} If $\Qc$ and $\Tc$ are dominant, then we have
\begin{equation}\label{regref1-2}\Kc_\Qc^*=(\Kc_\Qc^*)_{\mathrm{app}}+\Rs_\Qc^*,\quad \Kc_\Tc^*=(\Kc_\Tc^*)_{\mathrm{app}}+\Rs_\Tc^*
\end{equation} with $\Rs_\Qc^*$ and $\Rs_\Tc^*$ satisfying the same bound (\ref{regref1-1}).

Moreover, we have
\begin{equation}\label{regref1-3}
\begin{aligned}(\Kc_\Qc^*)_{\mathrm{app}}(t,s,k)&=2^{-2n}\delta^n\zeta(\Qc)\sum_{Z}\prod_{\nf\in Z}\frac{1}{\zeta_\nf\pi i}\Jc\Bc_{\Qc}^*(t,s)\cdot \Mc_\Qs^*(k)[F_{r}],\\
(\Kc_\Tc^*)_{\mathrm{app}}(t,s,k)&=2^{-2n}\delta^n\zeta(\Tc)\sum_{Z}\prod_{\nf\in Z}\frac{1}{\zeta_\nf\pi i}\Jc\Bc_{\Tc}^*(t,s)\cdot \Mc_\Ts^*(k)[F_{r}],
\end{aligned}
\end{equation} where the sum is taken over all special subsets $Z\subset\Nc^{\mathrm{ch}}$, and $\Qs=(\Qc,Z)$ and $\Ts=(\Tc,Z)$ are the corresponding enhanced dominant couple and dominant tree. The functions $\Jc\Bc_\Qc^*(t,s)$ and $\Jc\Bc_\Tc^*(t,s)$ depend only on $\Qc$ (and $\Tc$) and does not depend on $Z$. The functions $\Mc_\Qs^*(k)$ and $\Mc_\Ts^*(k)$ are constructed from the input functions $F_{r}(k_\lf)=\varphi(\delta {r},k_\lf)+\Rs_{r}(k_\lf)$, and also depends on the enhanced object $\Qs$ (and $\Ts$). All these functions are real valued, and for any $\Qs$ and $\Ts$, they satisfy the bounds
\begin{equation}\label{regref1-4}
\begin{aligned}
\|\Jc\Bc_\Qc^*\|_{\Xf^{1-\eta}(\Bc)}&\lesssim_1 C_1^{n},&\|\Mc_\Qs^*\|_{\Sf^{30d,30d}}&\lesssim_1 C_1^{n},\\
\|\Jc\Bc_\Tc^*\|_{\Xf^{1-\eta}(\Bc)}&\lesssim_1 C_1^n,&
\|\Mc_\Ts^*\|_{\Sf^{30d,0}}&\lesssim_1C_1^n.
\end{aligned}
\end{equation}

Finally, if ${r}=p$ and the input functions $F_p(k_\lf)=\varphi(\delta p,k_\lf)+\Rs_p(k_\lf)$ are replaced by $\varphi(\delta p,k_\lf)$, then we have
\begin{equation}\label{regref1-5}\sum_{n(\Qc)=2n}\Kc_\Qc^*(t,t,k)=\Uc_n(t,k),
\end{equation} where the sum is taken over all dominant couples $\Qc$ of order $2n$, and $\Uc_n(t,k)$ is defined as in (\ref{wketaylor2}).
\end{prop}
\begin{proof} When the scaling law $\gamma$ is close to $1$, this is proved in Propositions 6.7, 6.10, 7.5 and 7.11 of \cite{DH21} (the case of regular tree $\Tc$ is not included there, but follows in the same way as it forms a regular couple with the trivial tree). For general $\gamma$, the general decomposition leading to (\ref{regref1-1})--(\ref{regref1-2}) is proved in Proposition 6.1 of \cite{DH23}, while the exact expressions are independent of $\gamma$, so (\ref{regref1-3})--(\ref{regref1-5}) are the same as in \cite{DH21}.

Note that the current setting has two differences compared to \cite{DH21,DH23}: first, all the nodes are in layer ${r}$ instead of layer $0$, but this does not matter because $\Kc_\Qc^*$ and $\Kc_\Tc^*$ are clearly invariant after translating all the time variables by the same value ${r}$. Second, the input functions here are $F_{r}(k_\lf)=\varphi(\delta {r},k_\lf)+\Rs_{r}(k_\lf)$ instead of $\varphi_{\mathrm{in}(k_\lf)}$ in \cite{DH21,DH23}, but by our assumptions on $\Rs_{r}(k_\lf)$, these inputs are also real valued, and are bounded in the same regularity space that is more than enough for the proofs in \cite{DH21,DH23} to work. Thus the same proof in \cite{DH21,DH23} also applies here.
\end{proof}
\begin{prop}\label{regref2} Let the functions $\Jc\Bc_\Qc^*(t,s)$ and $\Jc\Bc_\Tc^*(t,s)$ be defined as in Proposition \ref{regref1}. Then for for trivial couple $\Qc$ we have $\Jc\Bc_\Qc^*(t,s)=1$. For type 1 dominant couple $\Qc$, let the notations be as in Proposition \ref{propstructure}, then for any $t,s\in[{r},{r}+1]$ we have
\begin{equation}\label{regref2-0}\Jc\Bc_\Qc^*(t,s)=\Jc\Bc_\Qc^*(\min(t,s))=2\int_{{r}}^{\min(t,s)}\prod_{j=1}^3\Jc\Bc_{\Qc_j}^*(\tau,\tau)\,\mathrm{d}\tau.
\end{equation}

If $\Qc$ is a dominant couple of type 2, let the notations be as in Proposition \ref{propstructure}, then for any $t,s\in[{r},{r}+1]$, we have (understanding $t_0=t$ and $s_0=s$) that
\begin{multline}\label{regref2-1}
\Jc\Bc_\Qc^*(t,s)=\int_{t>t_1>\cdots >t_{m_+}>{r}}\int_{s>s_1>\cdots >s_{m-}>{r}}\prod_{j=1}^{m_+}\prod_{\iota=1}^2\Jc\Bc_{\Qc_{j,+,\iota}}^*(t_j,t_j)\\
\times\prod_{j=1}^{m_-}\prod_{\iota=1}^2\Jc\Bc_{\Qc_{j,-,\iota}}^*(s_j,s_j)\cdot\Jc\Bc_{\Qc_{\mathrm{lp}}}^*(\min(t_{m_+},s_{m_-}))\prod_{j=1}^{m_+}\mathrm{d}t_j\prod_{j=1}^{m_-}\mathrm{d}s_j.
\end{multline}

If $\Tc$ is a dominant tree, let the notations be as in Proposition \ref{propstructure}, then for any $t,s\in[{r},{r}+1]$ and $t>s$, we have
\begin{equation}\label{regref2-2}
\Jc\Bc_{\Tc}^*(t,s)=\int_{t>t_1>\cdots>t_{m_0}>s}\prod_{j=1}^{m_0}\prod_{\iota=1}^2\Jc\Bc_{\Qc_{j,0,\iota}}^*(t_j,t_j)\prod_{j=1}^{m_0}\mathrm{d}t_j.
\end{equation}

Moreover, for any equivalence class $\Xs$ of enhanced dominant couples $\Qs=(\Qc,Z)$ \emph{with $Z\neq\varnothing$}, we have
\begin{equation}\label{regref2-3}
\sum_{\Qs=(\Qc,Z)\in\Xs}\bigg(\prod_{\nf\in Z}\frac{1}{\zeta_\nf \pi i}\bigg)\Jc\Bc_\Qc^*(t,t)=0.
\end{equation}
\end{prop}
\begin{proof} See Proposition 7.5 and 7.8 of \cite{DH21} (the case of dominant tree $\Tc$ is not included there, but follows in the same way as it forms a dominant couple with the trivial tree). These results are independent of the scaling law $\gamma$, so they also apply to the current case.
\end{proof}
\begin{prop}\label{regref3} Let the functions $\Mc_\Qs^*(k)$ and $\Mc_\Ts^*(k)$ be defined as in Proposition \ref{regref1}, then they only depend on the \emph{equivalence class} $\Xs$ (and $\Zs$) of the enhanced dominant couples $\Qs$ (and enhanced dominant trees $\Ts$), so we may denote $\Mc_\Qs^*(k)=\Mc_\Xs^*(k)$ and $\Mc_\Ts^*(k)=\Mc_\Zs^*(k)$. If $\Qc$ is trivial then we have $\Mc_\Qs^*(k)=F_{r}(k)$. If $\Qc$ is a type 1 dominant couple, let the notations be fixed as in Proposition \ref{propstructure} and let $\Xs_j$ be the equivalence class of $\Qs_j$ in Definition \ref{defencpl}, then
\begin{equation}\label{regref3-0}
\Mc_\Xs^*(k)=\int_{(\Rb^d)^3}\dirac(k_1-k_2+k_3-k)\dirac(|k_1|^2-|k_2|^2+|k_3|^2-|k|^2)\prod_{j=1}^3\Mc_{\Xs_j}^*(k_j)\,\mathrm{d}k_1\mathrm{d}k_2\mathrm{d}k_3.
\end{equation}

If $\Qc$ is a type 2 dominant couple, let the notations be fixed as in Proposition \ref{propstructure}, and also recall the tuples $(\mathtt{I}_{j,\epsilon},\mathtt{c}_{j,\epsilon},\Xs_{j,\epsilon,1},\Xs_{j,\epsilon,2})$ and $\Xs_{\mathrm{lp}}$, as in Definition \ref{defencpl}. Then we have
\begin{equation}\label{regref3-1}
\Mc_\Xs^*(k)=\Mc_{\Xs_{\mathrm{lp}}}^*(k)\cdot\prod_{\epsilon\in\{\pm\}}\prod_{j=1}^{m_\epsilon}\Nc_{(j,\epsilon)}^*(k).
\end{equation} If $\Tc$ is a dominant tree, let the notations be fixed as in Proposition \ref{propstructure}, and also recall the tuples $(\mathtt{I}_{j,0},\mathtt{c}_{j,0},\Xs_{j,0,1},\Xs_{j,0,2})$ as in Definition \ref{defentree}. Then we have
\begin{equation}\label{regref3-1-1}
\Mc_\Zs^*(k)=\prod_{j=1}^{m_0}\Nc_{(j,0)}^*(k).
\end{equation}
Here in (\ref{regref3-1}) and (\ref{regref3-1-1}) we have
\begin{equation}\label{regref3-2}
\Nc_{(j,\epsilon)}^*(k)=\int_{(\Rb^d)^3}\dirac(k_1-k_2+k_3-k)\cdot(\mathrm{Factor\ }1)\cdot(\mathrm{Factor\ }2)\,\mathrm{d}k_1\mathrm{d}k_2\mathrm{d}k_3,
\end{equation} and the same with $\epsilon$ replaced by $0$, where
\begin{equation}\label{regref3-3}(\mathrm{Factor\ }1)=
\left\{
\begin{aligned}
&\dirac(|k_1|^2-|k_2|^2+|k_3|^2-|k|^2),&\mathrm{if\ }\mathtt{I}_{j,\epsilon}&=0;\\
&\mathrm{p.v.}\,\frac{1}{|k_1|^2-|k_2|^2+|k_3|^2-|k|^2},&\mathrm{if\ }\mathtt{I}_{j,\epsilon}&=1,
\end{aligned}
\right.
\end{equation} and
\begin{equation}\label{regref3-4}
(\mathrm{Factor\ }2)=
\left\{
\begin{aligned}&\Mc_{\Xs_{j,\epsilon,1}}^*(k_2)\Mc_{\Xs_{j,\epsilon,2}}^*(k_3),&\mathrm{if\ }\mathtt{c}_{j,\epsilon}&=1;\\
&\Mc_{\Xs_{j,\epsilon,1}}^*(k_1)\Mc_{\Xs_{j,\epsilon,2}}^*(k_3),&\mathrm{if\ }\mathtt{c}_{j,\epsilon}&=2;\\
&\Mc_{\Xs_{j,\epsilon,1}}^*(k_1)\Mc_{\Xs_{j,\epsilon,2}}^*(k_2),&\mathrm{if\ }\mathtt{c}_{j,\epsilon}&=3.
\end{aligned}
\right.
\end{equation}
\end{prop}
\begin{proof} See Proposition 7.7 of \cite{DH21}. Again the result is independent of the scaling law $\gamma$, so it also applies to the current case.
\end{proof}
\subsection{Proof of Propositions \ref{proplayer2}--\ref{proplayer4}}\label{seclayer3-3} We now prove Propositions \ref{proplayer2}--\ref{proplayer4}.
\begin{proof}[Proof of Proposition \ref{proplayer2}] 
We will assume $q>q'$, since when $q=q'$ all the nodes in $\Qc$ must be in the same layer and the result follows from Proposition \ref{regref1}. Recall the structure of $\Qc$ described in Proposition \ref{layerreg1}. For each $q'\leq r\leq q$, consider all the nodes $\nf_a^+$ where $1\leq a\leq 2m_+$ such that $\Lf_{\nf_a^+}=r$. All these branching nodes are linked to each other, and they are precisely all the branching nodes of some regular tree $\Tc_{[r]}$. Denote the root of $\Tc_{[r]}$ by $\rf_{[r]}$, then the lone leaf of $\Tc_{[r]}$ is just $\rf_{[r-1]}$ or the root of a tree in $\Qc_{\mathrm{lp}}$, and all the nodes in $\Tc_{[r]}$ except the lone leaf are in layer $r$ (note that some of the $\Tc_{[r]}$ may be trivial, in which case some of the $\rf_{[r]}$ may coincide).

Next, note that the tree rooted at $\rf_{[q']}$ (which contains one of the trees in $\Qc_{\mathrm{lp}}$) forms a regular couple $\Qc_{[q']}$ with the other tree $\Tc^-$ of $\Qc$, and all nodes in $\Qc_{[q']}$ are in layer $q'$. Moreover, in the $k$-decoration we must have $k_{\rf_{[r]}}=k$ for each $q'\leq r\leq q$. For each $q'\leq r<q$, in the domain $\Ic^*$ defined in (\ref{timecouple}), the range for $t_{\rf_{[r]}}$ is $[r,r+1]$; for $r=q$ (so $\rf_{[q]}$ is just the root $\rf^+$ of the tree $\Tc^+$ of $\Qc$) the range should be replaced by $t_{\rf^+}\in[q,t]$, and similarly the range of $t_{\rf^-}$ is $t_{\rf^-}\in[q',s]$ (if $\rf_{[q']}$ or $\rf^-$ is a leaf then these should be suitably modified but this will not affect (\ref{prooflayer1-1}) below).

Therefore, by definition of $\Kc_\Qc^*$, we obtain that
\begin{equation}\label{prooflayer1-1}
\Kc_\Qc^*(t,s,k)=\prod_{q'<r< q}\Kc_{\Tc_{[r]}}^*(r+1,r,k)\cdot\Kc_{\Tc_{[q]}}^*(t,q,k)\cdot \Kc_{\Qc_{[q']}}^*(q'+1,s,k).
\end{equation} Since all the nodes of each individual $\Tc_{[r]}$ (except the lone leaf) and $\Qc_{[q']}$ are in the same layer, we may decompose each $\Kc_{\Tc_{[r]}}^*$ and $\Kc_{\Qc_{[q']}}^*$ as in (\ref{regref1-2}) and (\ref{regref1-3}), and putting together gives the decomposition (\ref{layerregest2}) and (\ref{layerregest4}) for $\Kc_\Qc^*$ (where any term with at least one $\Rs_{\Tc_{[r]}}^*$ or $\Rs_{\Qc_{[q']}}^*$ goes into $\Rs_\Qc^*$). The same argument applies for $\Tc$, where we have $\Tc_{[r]}$ for $q'\leq r\leq q$ but no $\Qc_{[q']}$, and similarly
\begin{equation}\label{prooflayer1-2}
\Kc_\Tc^*(t,s,k)=\prod_{q'<r< q}\Kc_{\Tc_{[r]}}^*(r+1,r,k)\cdot\Kc_{\Tc_{[q]}}^*(t,q,k)\cdot \Kc_{\Tc_{[q']}}^*(q'+1,s,k).
\end{equation} Finally, if $\Qc$ and $\Tc$ are not dominant, the at least one of the $\Tc_{[r]}$ or $\Qc_{[q']}$ will not be dominant, so the corresponding $(\cdots)_{\mathrm{app}}$ quantities will vanish, thus $(\Kc_{\Qc}^*)_{\mathrm{app}}$ and $(\Kc_{\Tc}^*)_{\mathrm{app}}$ will also vanish.
\end{proof}
\begin{proof}[Proof of Proposition \ref{proplayer3}] This basically follows from (\ref{regref1-5}). The input functions here are $F_p(k_\lf)=\varphi(\delta p,k_\lf)+\Rs_p(k_\lf)$, if we replace them by $\varphi(\delta p,k_\lf)$, the resulting $(\Kc_\Qc^*)_{\mathrm{app}}$ will be exactly $\Uc_n(t,k)$ by (\ref{regref1-5}); the difference, on the other hand, contains those terms with at least one input being $\Rs_p(k_\lf)$. The number of these terms is at most $2^n$ which can be absorbed by $C_1^n$, while each individual term includes an additional decay factor $L^{-\theta_p}$ due to the bound (\ref{ansatz2}) for $\Rs_p(k_\lf)$. This proves (\ref{layerregest8}).
\end{proof}
\begin{proof}[Proof of Proposition \ref{proplayer4}] First (\ref{layerregest9}) directly follows from the definitions of $\Kc_\Qc^*$ and $\Kc_\Tc^*$ in Definition \ref{defkg} and definition of conjugate in Definition \ref{defgarden} (the same holds for the $(\cdots)_{\mathrm{app}}$ approximations).

Now we focus on the proof of (\ref{layerregest10}). Let $\Qc$ and $\Tc$ be fixed such that $n(\Qc)+n(\Tc)=2n$. We may also assume $q>q'$, since otherwise all the nodes in $\Qc$ and $\Tc$ must be in the same layer, and the result follows from Proposition 6.3 of \cite{DH23}. In fact, the proof here can be viewed as an extended version of the proof of Proposition 6.3 of \cite{DH23}, taking into account of the layers.

Next, by Proposition \ref{proplayer2} we may restrict $\Qc$ and $\Tc$ to be dominant. Then, each $\Tc_{[r]}$ and $\Qc_{[q']}$ is dominant and has all its nodes (except the lone leaf for $\Tc_{[r]}$) in the same layer, so the corresponding $(\cdots)_{\mathrm{app}}$ quantities are given by (\ref{regref1-3}). Moreover each of these objects has its own inner structure (for example each $\Tc_{[r]}$ has the form in Proposition \ref{propstructure} with branching nodes of the regular chain being those $\nf_a^+$ where $\Lf_{\nf_a^+}=r$) that allows us to apply Propositions \ref{regref2} and \ref{regref3}. Now recall the notations in Proposition \ref{layerreg1}. Let $Z$ be the union of a special subset of $\Qc$ and a special subset of $\Tc$, then $Z$ is uniquely determined by the sets $Z_{j,\epsilon,\iota}$ and $Z_{j,0,\iota}$ (where $\epsilon\in\{\pm\}$, $1\leq j\leq m_\epsilon$ or $1\leq j\leq m_0$, and $\iota\in\{1,2\}$), $Z_{\mathrm{lp}}$, and the sets $(Z^\epsilon,Z^0)$, as in Definitions \ref{defencpl} and \ref{defentree}. Define also $\Qs_{\mathrm{lp}}$ and $\Xs_{\mathrm{lp}}$, and the tuples $(\mathtt{I}_{j,\epsilon},\mathtt{c}_{j,\epsilon},\Xs_{j,\epsilon,1},\Xs_{j,\epsilon,2})$ and $(\mathtt{I}_{j,0},\mathtt{c}_{j,0},\Xs_{j,0,1},\Xs_{j,0,2})$ etc., as in Definitions \ref{defencpl} and \ref{defentree}. Define $r_j^+=\Lf_{\nf_{2j-1}^+}$ for $1\leq j\leq m_+$ and $r_j^0=\Lf_{\nf_{2j-1}^0}$ for $1\leq j\leq m_0$, and $r_j^-=q'$.

Starting with the formulas (\ref{prooflayer1-1}) and (\ref{prooflayer1-2}), by applying (\ref{regref1-3}), (\ref{regref2-1})--(\ref{regref2-2}) and (\ref{regref3-1})--(\ref{regref3-1-1}), we get the formula
\begin{align}
(\ref{layerregest10})&=2^{-2n}\delta^{n}\sum_m\sum_{m_++m_-+m_0=m}\sum_{(\mathrm{tuples})}\sum_{(\Qs)}\zeta(\Qc)\zeta(\Tc)\prod_{\nf\in Z}\frac{ 1}{\widetilde{\zeta_\nf} \pi i}\int_{t>t_1>\cdots >t_{m_+}>q';\,r_j^+<t_j<r_j^++1}\nonumber\\&\times\int_{s>s_1>\cdots >s_{m_-}>q'}\int_{t>u_1>\cdots >u_{m_0}>s;\,r_j^0<u_j<r_j^0+1}\prod_{j=1}^{m_+}\mathrm{d}t_j\prod_{j=1}^{m_-}\mathrm{d}s_j\prod_{j=1}^{m_0}\mathrm{d}u_j\nonumber\\
&\times\prod_{j,\iota}\Jc\Bc_{\Qc_{j,+,\iota}}^*(t_j,t_j)\prod_{j,\iota}\Jc\Bc_{\Qc_{j,-,\iota}}^*(s_j,s_j)\prod_{j,\iota}\Jc\Bc_{\Qc_{j,0,\iota}}^*(u_j,u_j)\cdot \Jc\Bc_{\Qc_{\mathrm{lp}}}^*(s_{\mathrm{min}})\nonumber\\
\label{regrefproof3-1}&\times 
\Mc_{\Xs_{\mathrm{lp}}}^*(k)\cdot\prod_{\epsilon\in\{\pm\}}\prod_{j=1}^{m_\epsilon}\Nc_{(j,\epsilon)}^*(k)\prod_{j=1}^{m_0}\Nc_{(j,0)}^*(k).
\end{align}
 Here in (\ref{regrefproof3-1}) we have $s_{\mathrm{min}}=\min(t_{m_+},s_{m_-},s)$, $Z$ is the union of all the special sets $(Z_{j,\epsilon,\iota},Z_{j,0,\iota})$ and $Z_{\mathrm{lp}}$ and $(Z^\epsilon,Z^0)$, and $\Nc_{(j,\epsilon)}^*$ and $\Nc_{(j,0)}^*$ are defined in (\ref{regref3-2}). The symbol $\widetilde{\zeta_\nf}$ equals $-\zeta_\nf$ if $\nf$ belongs to $Z_{j,0,\iota}$ or $Z^0$, and equals $\zeta_\nf$ otherwise. Finally, the sum $\sum_{(\mathrm{tuples})}$ is taken over all the \emph{the extended tuples}
\begin{equation}\label{regrefproof3-2}\big\{(r_j^\epsilon,\mathtt{I}_{j,\epsilon},\mathtt{c}_{j,\epsilon},\Xs_{j,\epsilon,1},\Xs_{j,\epsilon,2}):\epsilon\in\{\pm\},1\leq j\leq m_\epsilon\big\}\cup \big\{(r_j^0,\mathtt{I}_{j,0},\mathtt{c}_{j,0},\Xs_{j,0,1},\Xs_{j,0,2}):1\leq j\leq m_0\big\},\end{equation} and the sum $\sum_{(\Qs)}$ is taken over all the enhanced regular couples $\Qs_{j,\epsilon,\iota}=(\Qc_{j,\epsilon,\iota},Z_{j,\epsilon,\iota})$ (and similarly for $(j,0,\iota)$ and the lone pair) that belong to the equivalence class $\Xs_{j,\epsilon,\iota}$. In these summations we assume $n(\Qc)+n(\Tc)=2n$; for the last statement of Proposition \ref{proplayer4} we make the further restriction that the total number $n_p$ of layer $p$ branching nodes in $\Qc$ and $\Tc$ equals a given constant.

Now continue with (\ref{regrefproof3-1}). If $Z=\varnothing$, then the whole expression is clearly real valued. If any of the sets $Z_{j,\epsilon,\iota}$ (or $Z_{j,0,\iota}$ or $Z_{\mathrm{lp}}$) is not empty, then the sum of $\prod_{\nf\in Z_{j,\epsilon,\iota}}(\widetilde{\zeta_\nf} \pi i)^{-1}\cdot\Jc\Bc_{\Qc_{j,\epsilon,\iota}}^*(\cdots)$ over all $\Qs_{j,\epsilon,\iota}$, which is a part of the summation $\sum_{(\Qs)}$, will give $0$ due to (\ref{regref2-3}); here note that all the other ingredients of (\ref{regrefproof3-1}) are independent of $\Qs_{j,\epsilon,\iota}$ (in particular we can check that varying $\Qs_{j,\epsilon,\iota}$ within a fixed equivalence class will not change the value of $\zeta(\Qc_{j,\epsilon,\iota})$). Therefore, below we may assume $Z=Z^+\cup Z^-\cup Z^0\neq\varnothing$ (and in particular $m=m_++m_-+m_0>0$).

In the sum $\sum_{(\mathrm{tuples})}$ in (\ref{regrefproof3-1}), we shall first fix $\Xs_{\mathrm{lp}}$ and the \emph{unordered collection} of all the extended tuples in (\ref{regrefproof3-2}), and replace $\sum_{(\mathrm{tuples})}$ by the summation over all choices of individual tuples, which are permutations of this fixed unordered collection; see Figure \ref{fig:equivregchain} for an illustration. We can do this because the value of $n(\Qc)+n(\Tc)$ is determined by the unordered collection. In the same way, if in (\ref{layerregest10}) we make the further restriction that $n_p$ equals a given constant, then we can again first fix the unordered collection of extended tuples, because the value of $n_p$ is also determined by the unordered collection. Now it will suffice to prove that the result is $0$ for each fixed unordered collection. For convenience, we shall denote the extended tuples in (\ref{regrefproof3-2}) by $\Us_j\,(1\leq j\leq m_+)$ and $\Vs_j\,(1\leq j\leq m_-)$ and $\Ws_j\,(1\leq j\leq m_0)$. Let the fixed unordered collection be $(\Ys_1,\cdots,\Ys_m)$ where $m=m_++m_-+m_0$ is also fixed.

We proceed by simplifying the expression (\ref{regrefproof3-1}). First $\Mc_{\Xs_{\mathrm{lp}}}^*(k)$ is fixed in the summation, moreover, using the definition in (\ref{regref3-2}), we can write the product
\[\prod_{\epsilon\in\{\pm\}}\prod_{j=1}^{m_\epsilon}\Nc_{(j,\epsilon)}^*(k)\prod_{j=1}^{m_0}\Nc_{(j,0)}^*(k)=\prod_{j=1}^{m_+}\Mf[\Us_j]\prod_{j=1}^{m_-}\Mf[\Vs_j]\prod_{j=1}^{m_0}\Mf[\Ws_j]=\prod_{j=1}^m\Mf[\Ys_j],\]
where $\Mf[\Ys]$ is an expression that depends only on the extended tuple $\Ys$. In particular this product does not change when we permute all the $\Ys_j$, so it can be extracted as a common factor. We can also verify that the value of $\zeta(\Qc)\zeta(\Tc)$ does not vary, so it can also be extracted.

Next, for any equivalence class $\Xs$ and extended tuple $\Ys=(r,\mathtt{I},\mathtt{c},\Xs_1,\Xs_2)$, define \begin{equation}\label{defBf}\Bf[\Xs](t):=\sum_{\Qs\in\Xs}\Jc\Bc_{\Qc}^*(t,t),\quad \Bf[\Ys](t):=\mathbf{1}_{r<t<r+1}\cdot\Bf[\Xs_1](t)\Bf[\Xs_2](t).\end{equation} Note also that the value of $\widetilde{\zeta_\nf}$ equals $1$ for $\nf\in Z^+$, and equals $-1$ for $\nf\in Z^-\cup Z^0$, and that the choices of the \emph{second digits} of the codes of mini trees gives another factor $2^m$ (which is another common factor). Then, after extracting the common factors, the remaining summation in (\ref{regrefproof3-1}) then reads
\begin{align}
&\sum_{m_++m_-+m_0=m}\sum_{(\Us_1,\cdots, \Us_{m_+},\Vs_1,\cdots,\Vs_{m_-}, \Ws_1, \cdots, \Ws_{m_0})}\int_{t>t_1>\cdots >t_{m_+}>q'}\int_{s>s_1>\cdots>s_{m_-}>q'}\int_{t>u_1>\cdots >u_{m_0}>s}\nonumber\\\label{regrefproof3-3}&\times \prod_{j=1}^{m_-}(-1)^{\mathtt{I}_{j,-}'} \prod_{j=1}^{m_0}(-1)^{\mathtt{I}_{j,0}'}\prod_{j=1}^{m_+}\Bf[\Us_j](t_j)\,\mathrm{d}t_j\prod_{j=1}^{m_-}\Bf[\Vs_j](s_j)\,\mathrm{d}s_j\prod_{j=1}^{m_0}\Bf[\Ws_j](u_j)\,\mathrm{d}u_j\cdot \Bf[\Xs_{\mathrm{lp}}](s_{\mathrm{min}}).
\end{align} Here $\mathtt{I}_{j,-}'$ and $\mathtt{I}_{j,0}'$ represent the ``$\mathtt{I}$" component of $\Vs_j$ and $\Ws_j$ respectively, and the sequence $(\Us_1,\cdots, \Us_{m_+},\Vs_1,\cdots,\Vs_{m_-}, \Ws_1, \cdots, \Ws_{m_0})$ runs over \emph{all permutations} of the sequence of extended tuples $(\Ys_1,\cdots,\Ys_m)$. Strictly speaking we should also require that the ``$r$" components of $\Us_j$ must be decreasing in $j$ (same for $\Vs_j$ and $\Ws_j$), but this restriction is redundant, since otherwise the product of the $\Bf$ factors will be zero due to the $\mathbf{1}_{r<t<r+1}$ factor in $\Bf[\Ys]$ as in (\ref{defBf}).

Now we rearrange the time variables $t_j\,(1\leq j\leq m_+)$, $s_j\,(1\leq j\leq m_-)$ and $u_j\,(1\leq j\leq m_0)$ into $t>v_1>\cdots >v_m>q'$, and let $A_+$ be the set of $j$ such that $v_j$ equals some $t_i\,(1\leq i\leq m_+)$, similarly define $A_-$ and $A_0$, then these subsets completely determine the rule of correspondence between the variables $(v_j)$ and $(t_i,s_i,u_i)$. Note that $s_{\mathrm{min}}=\min(v_m,s)$, and
\[\sum_{(\Us_1,\cdots, \Us_{m_+},\Vs_1,\cdots,\Vs_{m_-}, \Ws_1, \cdots, \Ws_{m_0})}\prod_{j=1}^{m_+}\Bf[\Us_j](t_j)\prod_{j=1}^{m_-}\Bf[\Vs_j](s_j)\prod_{j=1}^{m_0}\Bf[\Ws_j](u_j)
=\sum_{(\widetilde{\Ys_1},\cdots,\widetilde{\Ys_m})}\prod_{j=1}^m\Bf[\widetilde{\Ys_j}](v_j),\] where the latter summation is taken over all permutations $(\widetilde{\Ys_1},\cdots,\widetilde{\Ys_m})$ of $(\Ys_1,\cdots,\Ys_m)$. Note also that any $j$ for which $v_j<s$ can only be put in $A_+$ or $A_-$, and any $j$ for which $v_j>s$ can only be put in $A_+$ or $A_0$. Then we can reduce (\ref{regrefproof3-3}) to
\begin{multline}\int_{t>v_1>\cdots>v_m>q'} \sum_{(\widetilde{\Ys_1}, \cdots, \widetilde{\Ys_m})}\prod_{j=1}^m\Bf(\widetilde{\Ys_j})(v_j) \cdot\Bf_{\Xs_{\mathrm{lp}}}\left(\min(v_m, s) \right)\\
\times\sum_{A_+\cup A_-\cup A_0=\{1,\cdots,m\}} \prod_{j\in A_-}(-1)^{\mathtt{I}_{j}}\mathbf{1}_{v_j<s} \prod_{j\in A_0}(-1)^{\mathtt{I}_{j}}\mathbf{1}_{v_j>s}\,\mathrm{d}v_1\cdots\mathrm{d}v_{m},
\end{multline} where the first sum is taken over all permutations $(\widetilde{\Ys_1},\cdots,\widetilde{\Ys_m})$ of $(\Ys_1,\cdots,\Ys_m)$, the second sum is taken over all partitions of $\{1,\cdots,m\}$ into $(A_+,A_-,A_0)$, and $\mathtt{I}_j$ is the ``$\mathtt{I}$" component of $\widetilde{\Ys_j}$. Finally, since
$$
\sum_{(A_+,A_-,A_0)} \prod_{j\in A_-}(-1)^{\mathtt{I}_{j}}\mathbf{1}_{v_j<s} \prod_{j\in A_0}(-1)^{\mathtt{I}_{j}}\mathbf{1}_{v_j>s}=\prod_{j=1}^{m}\left(1+(-1)^{\mathtt{I}_{j}}\mathbf{1}_{v_j<s}+(-1)^{\mathtt{I}_{j}}\mathbf{1}_{v_j>s}\right)=\prod_{j=1}^{m}\left(1+(-1)^{\mathtt{I}_{j}}\right),
$$ we know that this expression must be $0$ as long as $Z\neq\varnothing$, since this implies that at least one $\mathtt{I}_j=1$. This completes the proof.
\begin{figure}[h!]
\includegraphics[scale=0.39]{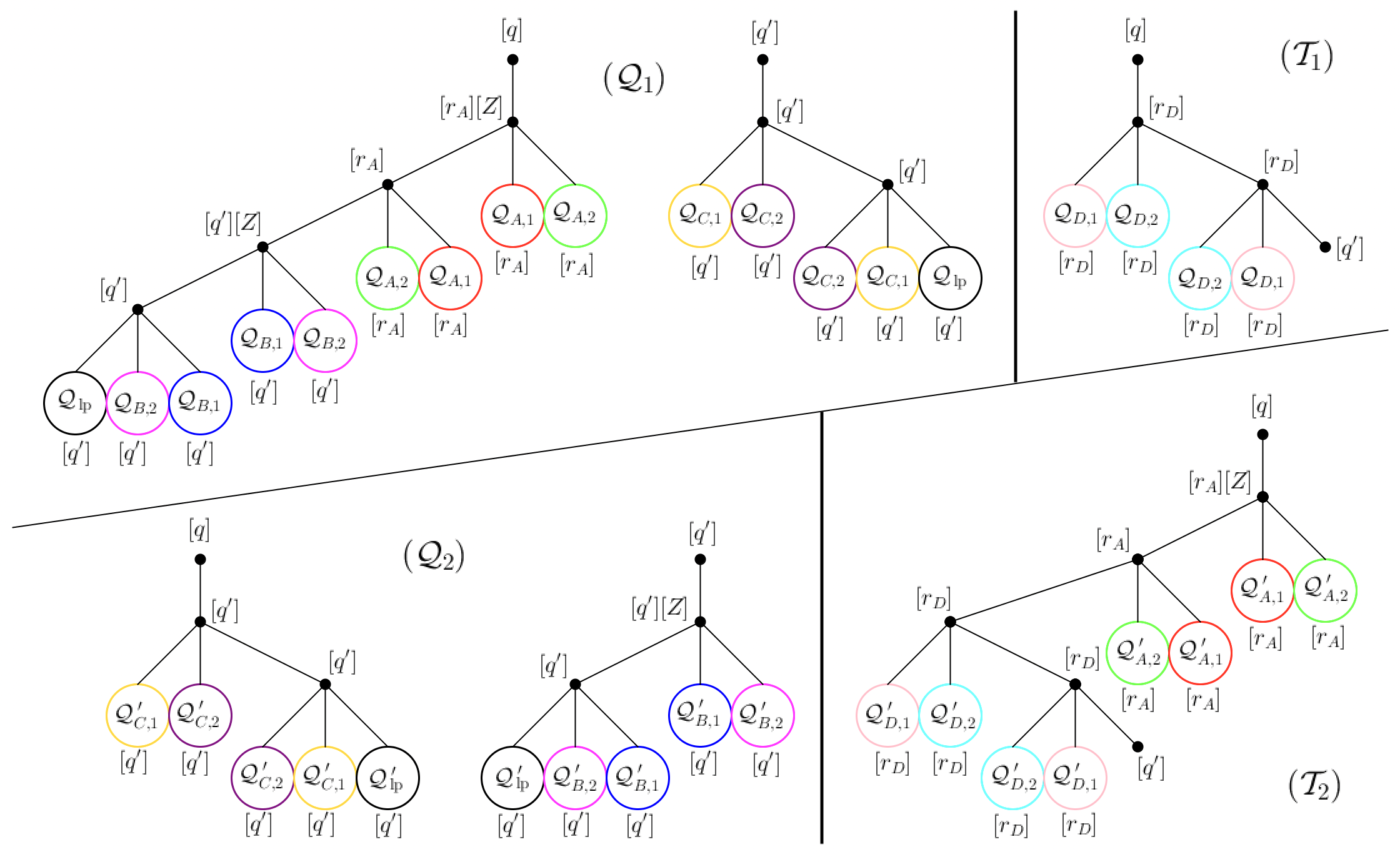}
\caption{An example of two terms, associated with $(\Qc_1,\Tc_1)$ and $(\Qc_2,\Tc_2)$, that occur in the summation with the same unordered collection of extended tuples. Here $r_A$ etc. are layers, $\Qc_{A,1}$ etc. (with $Z_{A,1}=\varnothing$ etc.) are equivalent to $\Qc_{A,1}'$ etc., and $[Z]$ indicates that the corresponding branching node belongs to $Z$.}
\label{fig:equivregchain}
\end{figure}
\end{proof}
\section{Molecules, vines and twists}\label{section7}
\subsection{Molecules} We recall the notion of molecules introduced in \cite{DH21,DH21-2}.
\begin{df}[Molecules \cite{DH21,DH21-2}]\label{defmol} A \emph{molecule} $\Mb$ is a directed graph, formed by vertices (called \emph{atoms}) and edges (called \emph{bonds}), where multiple and self-connecting bonds are allowed. We will write $v\in \Mb$ and $\ell\in\Mb$ for atoms $v$ and bonds $\ell$ in $\Mb$, and write $\ell\leftrightarrow v$ if $v$ is one of the two endpoints of $\ell$. The degree of an atom $v$ (including bonds of both directions) is denoted by $d(v)$. We further require that (i) each atom has at most $2$ outgoing bonds and at most $2$ incoming bonds (a self-connecting bond counts as outgoing once and incoming once), and that (ii) there is no saturated (connected) component, where connectedness is always understood in terms of undirected graphs, and a component is saturated if it contains only degree $4$ atoms. For a molecule $\Mb$ we define $V$ to be the number of atoms, $E$ the number of bonds and $F$ the number of components. Define $\chi:=E-V+F$ to be its \emph{circuit rank}. Finally, by an \emph{atomic group} we mean any subset $\Gb$ of atoms of $\Mb$, together with all the bonds between atoms in $\Gb$.
\end{df}
\begin{df}[Molecules from gardens \cite{DH21,DH21-2}]\label{defcplmol} Given a garden $\Gc$, define the molecule $\Mb=\Mb(\Gc)$ associated with $\Gc$, as follows. The atoms of $\Mb$ are the branching nodes $\nf\in\Nc$ of $\Gc$. For any two atoms $\nf_1$ and $\nf_2$, we connect them by a bond if either (i) $\nf_1$ is the parent of $\nf_2$, or (ii) a child of $\nf_1$ is paired to a child of $\nf_2$ as leaves. We fix the direction of each bond as follows: in case (i) the bond should go from $\nf_1$ to $\nf_2$ (or $\nf_2$ to $\nf_1$) if $\nf_2$ has sign $-$ (or sign $+$); in case (ii) the bond should go from the $\nf_j$ whose paired child has sign $-$ to the one whose paired child has sign $+$.

For any atom $v\in\Mb(\Gc)$, let $\nf=\nf(v)$ be the corresponding branching node in $\Gc$. We also introduce a labeling system to bonds in $\Mb(\Gc)$ \emph{in addition to its molecule structure}: for any bond between $v_1$ and $v_2$, if $\nf(v_1)$ is the parent of $\nf(v_2)$ then we label this bond by PC, place a label P at $v_1$, and place a label C at $v_2$; otherwise we label this bond by LP. Note that one atom $v$ may receive multiple P and C labels coming from different bonds $\ell\leftrightarrow v$. Finally, for any bond $\ell\leftrightarrow v$, define also $\mf=\mf(v,\ell)$ such that (i) if $\ell$ is PC with $v$ labeled C, then $\mf=\nf(v)$; (ii) if $\ell$ is PC with $v$ labeled P, then $\mf$ is the branching node corresponding to the other endpoint of $\ell$ (which is a child of $\nf(v)$); (iii) if $\ell$ is LP then $\mf$ is the leaf in the leaf pair defining $\ell$ that is a child of $\nf(v)$.

Note that, if $\Gc$ is a pre-layered garden, then this naturally leads to a layering of the corresponding molecule $\Mb(\Gc)$, where each atom $v$ is in layer $\Lf_v:=\Lf_{\nf(v)}$, which is the layer of the branching node $\nf(v)$ in $\Gc$.
\end{df}
\begin{prop}\label{propcplmol}
If $\Gc=\Qc$ is a nontrivial couple of order $n$, then $\Mb=\Mb(\Qc)$ has either two atoms of degree $3$ or one atom of degree $2$, and the other atoms all have degree $4$; in particular it has $n$ atoms and $2n-1$ bonds, and circular rank $\chi(\Mb)=n$. If $\Gc$ is an irreducible garden of width $2R$ and order $n$, then $\Mb=\Mb(\Gc)$ has $n$ atoms and $2n-R$ bonds, and circular rank $\chi(\Mb)=n-R+1$, in particular $n\geq R-1$. In both cases $\Mb$ is connected. Finally, the number of gardens $\Gc$ with fixed order $n$, fixed width $2R$, and fixed molecule $\Mb=\Mb(\Gc)$ (as a directed graph), is at most $(C_0R)!C_0^n$.
\end{prop}
\begin{proof} The count for atoms and bonds follows from Definition \ref{defcplmol} (some  discussion is needed when some trees of $\Gc$ is trivial), which also implies the statement about degrees when $\Gc=\Qc$ is a couple. Connectivity follows because all atoms (branching nodes) in the same tree are connected by PC bonds, and those from different trees are connected by LP bonds due to irreducibility of $\Gc$. Finally, the bound on the number of choices for $\Gc$ is proved in Proposition 6.4 of \cite{DH21-2}.
\end{proof}
\begin{df}[Decorations of molecules \cite{DH21,DH21-2}]\label{defdecmol} Let $\Mb$ be a molecule, suppose we also fix the vectors $c_v\in\Zb_L^d$ for each $v\in\Mb$ such that $c_v=0$ when $v$ has degree $4$. We then define a $(c_v)$-\emph{decoration} (or just a decoration) of $\Mb$ to be a set of vectors $(k_\ell)$ for all bonds $\ell\in\Mb$, such that $k_\ell\in\Zb_L^d$ and
\begin{equation}\label{decmole1}\sum_{\ell\leftrightarrow v}\zeta_{v,\ell}k_\ell=c_v
\end{equation} for each atom $v\in\Mb$. Here the sum is taken over all bonds $\ell\leftrightarrow v$, and $\zeta_{v,\ell}$ equals $1$ if $\ell$ is outgoing from $v$, and equals $-1$ otherwise. For each such decoration and each atom $v$, define also that
\begin{equation}\label{defomegadec}\Gamma_v=\sum_{\ell\leftrightarrow v}\zeta_{v,\ell}|k_\ell|^2.
\end{equation}

Suppose $\Mb=\Mb(\Gc)$ comes from a garden $\Gc=(\Tc_1,\cdots,\Tc_{2R})$ with signature $(\zeta_1,\cdots,\zeta_{2R})$. Let $\rf_j$ be the root of $\Tc_j$. By Definition \ref{defcplmol}, we have $d(v)<4$ for an atom $v$ if and only if $\nf(v)=\rf_i$ for some $i$ or a child of $\nf(v)$ is paired with $\rf_j$ as leaves for some $j$ (or both). Denote these two possibilities by $\Rf_i$ and $\Pf_j$. Now, for $k_j\in\Zb_L^d\,(1\leq j\leq 2R)$ satisfying $\zeta_1k_1+\cdots +\zeta_{2R}k_{2R}=0$, we define a $(k_1,\cdots,k_{2R})$-decoration of $\Mb$ to be a $(c_v)$-decoration with $(c_v)$ given by 
\begin{equation}\label{molecv}
c_v=-\bigg(\mathbf{1}_{\Rf_i}\cdot \zeta_{i}k_i+\sum_j\mathbf{1}_{\Pf_j}\cdot \zeta_jk_j\bigg),
\end{equation} where the summation is taking over all $j$ such that $\Pf_j$ holds. Note also that for \emph{couples $\Qc$ and $k$-decorations of $\Qc$}, the vectors involved (including $k$ itself) may be general $\Rb^d$ vectors as in Definition \ref{defdec}. Thus we shall define $k$-decorations of $\Mb=\Mb(\Qc)$ for any $k\in\Rb^d$ in the same way as above, but allow $k_\ell$ and $c_v$ to be general $\Rb^d$ vectors such that $k_\ell-k\in\Zb_L^d$.

Given any $(k_1,\cdots,k_{2R})$-decoration of $\Gc$ in the sense of Definition \ref{defdec}, define a $(k_1,\cdots,k_{2R})$-decoration of $\Mb(\Gc)$ such that $k_\ell=k_{\mf(v,\ell)}$ for an endpoint $v$ of $\ell$ (Definition \ref{defcplmol}). We can verify that this $k_\ell$ is well-defined  (i.e. does not depend on the choice of $v$), and gives a bijection between $(k_1,\cdots,k_{2R})$-decorations of $\Gc$ and $(k_1,\cdots,k_{2R})$-decoration of $\Mb(\Gc)$. For such decorations we have
\begin{equation}\label{molegammav}
\Gamma_v=-\zeta_{\nf(v)}\Omega_{\nf(v)}-\bigg(\mathbf{1}_{\Rf_i}\cdot \zeta_{i}|k_i|^2+\sum_j\mathbf{1}_{\Pf_j}\cdot \zeta_j|k_j|^2\bigg).
\end{equation} All the above definitions also work in the case of $k$-decorations of couples and molecules, which allows for general vectors $k\in\Rb^d$.

Finally, given $\beta_v\in\Rb$ for each $v\in\Mb$ and $k_\ell^0\in\Zb_L^d$ for each $\ell\in\Mb$, we define a decoration $(k_\ell)$ to be \emph{restricted by} $(\beta_v)$ and/or $(k_\ell^0)$, if $|\Gamma_v-\beta_v|\leq \delta^{-1}L^{-2\gamma}$ for each $v$ and/or $|k_\ell-k_\ell^0|\leq 1$ for each $\ell$.
\end{df}
\subsection{Blocks, vines and ladders} On the molecular level, the combinatorial structures of \emph{blocks} with the special case of \emph{vines}, as well as \emph{ladders}, were introduced in \cite{DH23} and also play a fundamental role in the current paper.
\subsubsection{Blocks} We start by recalling the definition of blocks.
\begin{df}[Blocks \cite{DH23}]\label{defblock}Given a molecule $\Mb$, an atomic group $\Bb\subset\Mb$ is called a \emph{block}, if all atoms in $\Bb$ have degree $4$ within $\Bb$, except for exactly two atoms $v_1$ and $v_2$ (called \emph{joints} of the block, the other atoms called \emph{interior atoms}), each of which having out-degree $1$ and in-degree $1$ (hence total degree $2$) within $\Bb$, see Figure \ref{fig:blocknew}. Define $\sigma(\Bb)$ as the number of bonds between $v_1$ and $v_2$. Note that $\sigma(\Bb)\in\{0,1,2\}$, and $\sigma(\Bb)=2$ if and only if $\Bb$ is a double bond. Moreover, we define a \emph{hyper-block} $\Hb$ to be the atomic group formed by adding one bond between the two joints $v_1$ and $v_2$ of a block $\Bb$ (we call $\Hb$ and $\Bb$ the \emph{adjoint} of each other), and define $\sigma(\Hb)=\sigma(\Bb)+1$.

If two blocks share one common joint and no other common atom, we define their \emph{concatenation} to be their union, which is either a block or a hyper-block (depending on whether the two other joints of the two blocks are connected by a bond), see Figure \ref{fig:blocknew}. Note that a hyper-block cannot be concatenated with another block or hyper-block in this way. In general any finitely many (at least two) blocks can be concatenated to form a new block $\Bb$, or a new hyper-block $\Hb$, in which case we must have $\sigma(\Bb)=0$ and $\sigma(\Hb)=1$.
\end{df}
  \begin{figure}[h!]
  \includegraphics[scale=.35]{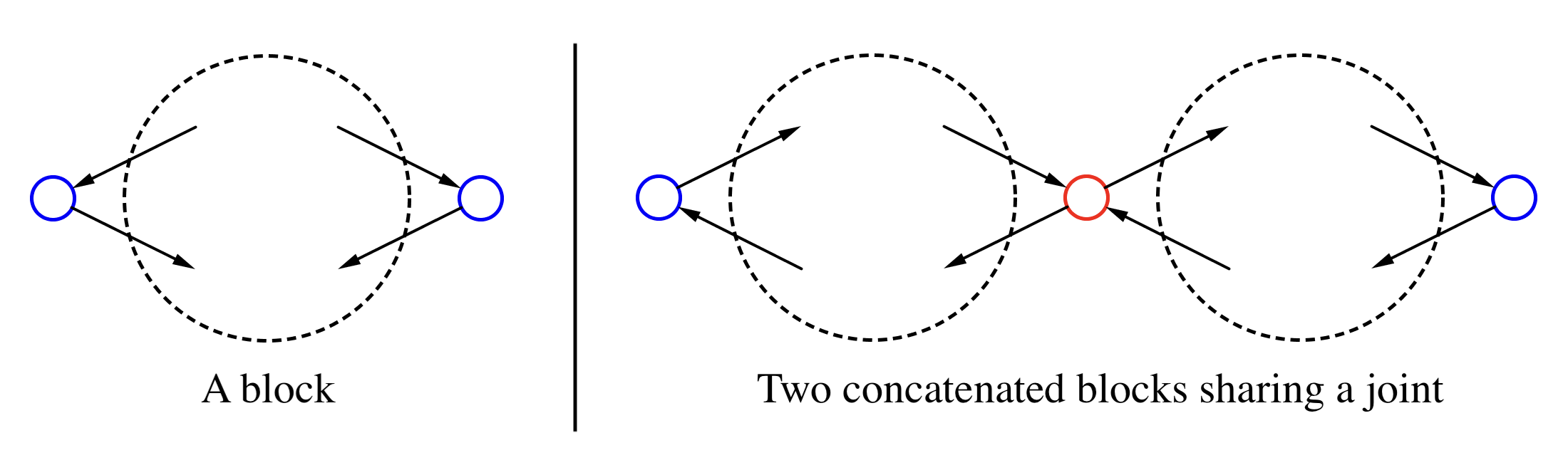}
  \caption[Chain of two blocks]{An illustration of a block, and two concatenated blocks sharing a joint, as in Definition \ref{defblock}. Here the two joint atoms at the end are colored blue, the one common joint atom colored red, and all other atoms have degree $4$.  Longer chains can be constructed similarly.}
  \label{fig:blocknew}
\end{figure} 
\begin{lem}\label{disjointlem} Let $\Mb$ be a molecule. Suppose $\Ab,\Bb\subset\Mb$, each of them is a block or a hyper-block, and $\Ab\not\subset\Bb$, $\Bb\not\subset\Ab$ and $\Ab\cap\Bb\neq\varnothing$.

Let $a_1$ and $a_2$ be the joints of $\Ab$, and $b_1$ and $b_2$ be the joints of $\Bb$. Suppose further that (i) $\Bb\backslash\{b_1,b_2\}$ is connected, and (ii) for any $v\in\Bb\backslash\{b_1,b_2\}$, the subset $\Bb\backslash\{v\}$ is either connected, or has two connected components containing $b_1$ and $b_2$ respectively, and (iii) the same holds for $\Ab$. 

Then $\Ab$ and $\Bb$ are both blocks, and exactly one of the three following scenarios happens: (a) $\Ab$ and $\Bb$ share two common joints and no other common atom, and $\sigma(\Ab)=\sigma(\Bb)=1$, (b) $\Ab$ and $\Bb$ share one common joint and no other common atom, and can be concatenated like in Definition \ref{defblock}; (c) $\Ab$ is formed by concatenating two blocks $\Cb_0$ and $\Cb_1$, and $\Bb$ is formed by concatenating $\Cb_1$ with another block $\Cb_2$ (where $\Cb_0\cap\Cb_2=\varnothing$).
\end{lem}
\begin{proof} See Lemma 4.11 of \cite{DH23}.\end{proof}
Next, we recall a result regarding the relative position of a block $\Bb\subset\Mb(\Gc)$ in a garden $\Gc$.
\begin{prop}[Structure of blocks in gardens]\label{block_clcn} Let $\Gc$ be a garden of width $2R$ and $\Bb\subset\Mb(\Gc)$ be a block with two joints $v_1$ and $v_2$, and let $\uf_j=\nf(v_j)$ with the notations in Definition \ref{defcplmol}.
\begin{enumerate}[{(1)}]
\item Then, up to symmetry, exactly one of the following two scenarios happens.
\begin{enumerate}[{(a)}]
\item There is a child $\uf_{11}$ of $\uf_1$ and two children $\uf_{21},\uf_{22}$ of $\uf_2$, such that (i) $\uf_{11}$ has the same sign as $\uf_1$, $\uf_{21}$ has sign $+$ and $\uf_{22}$ has sign $-$, (ii) $\uf_2$ is a descendant $\uf_1$ but not of $\uf_{11}$, and (iii) all the leaves in the set $\Gc[\Bb]$ are completely paired, where $\Gc[\Bb]$ denotes all nodes that are descendants of $\uf_1$ but not of $\uf_{11},\uf_{21}$ or $\uf_{22}$ (in particular $\uf_1\in\Gc[\Bb]$ and $\uf_{11},\uf_{21},\uf_{22}\not\in\Gc[\Bb]$). Here we say $\Bb$ is a \emph{(CL) block} and also denote $\uf_{23}$ as the child node of $\uf_2$ other than $\uf_{21}$ and $\uf_{22}$, see Figure \ref{fig:couples_cl}.
\item There is a child $\uf_{11}$ of $\uf_1$ and $\uf_{21}$ of $\uf_2$, such that (i) $\uf_{11}$ has the same sign as $\uf_1$ and $\uf_{21}$ has the same sign as $\uf_2$, (ii) $\uf_2$ is either a descendant of $\uf_{11}$ or not a descendant of $\uf_1$ (similar for $\uf_1$), and (iii) all the leaves in the set $\Gc[\Bb]$ are completely paired, where $\Gc[\Bb]$ denotes all the nodes that are descendants of $\uf_1$ but not of $\uf_{11}$, and all the nodes that are descendants of $\uf_2$ but not of $\uf_{21}$ (in particular $\uf_1,\uf_2\in\Gc[\Bb]$ and $\uf_{11},\uf_{21}\not\in\Gc[\Bb]$). Here we say $\Bb$ is a \emph{(CN) block}, see Figure \ref{fig:couples_cn}.
\end{enumerate}
\item For (CL) blocks $\Bb$ we can define a new garden $\Gc^{\mathrm{sp}}$ by removing all nodes $\mf\in\Gc[\Bb]\backslash\{ \uf_1\}$, and turning $\uf_{11},\uf_{21}$ and $\uf_{22}$ into the three new children of $\uf_1$ with corresponding subtrees attached; here the position of $\uf_{11}$ as a child of $\uf_1$ remains the same as in $\Gc$, and the positions of $\uf_{21}$ and $\uf_{22}$ as children of $\uf_1$ are determined by their signs. Then, the molecule $\Mb^{\mathrm{sp}}=\Mb(\Gc^{\mathrm{sp}})$ is formed from $\Mb$ by merging all the atoms in $\Bb$ (including two joints) into one single atom. We call this operation going from $\Gc$ to $\Gc^{\mathrm{sp}}$ \emph{splicing}.
\end{enumerate}
\end{prop}
\begin{figure}[h!]
\includegraphics[scale=0.4]{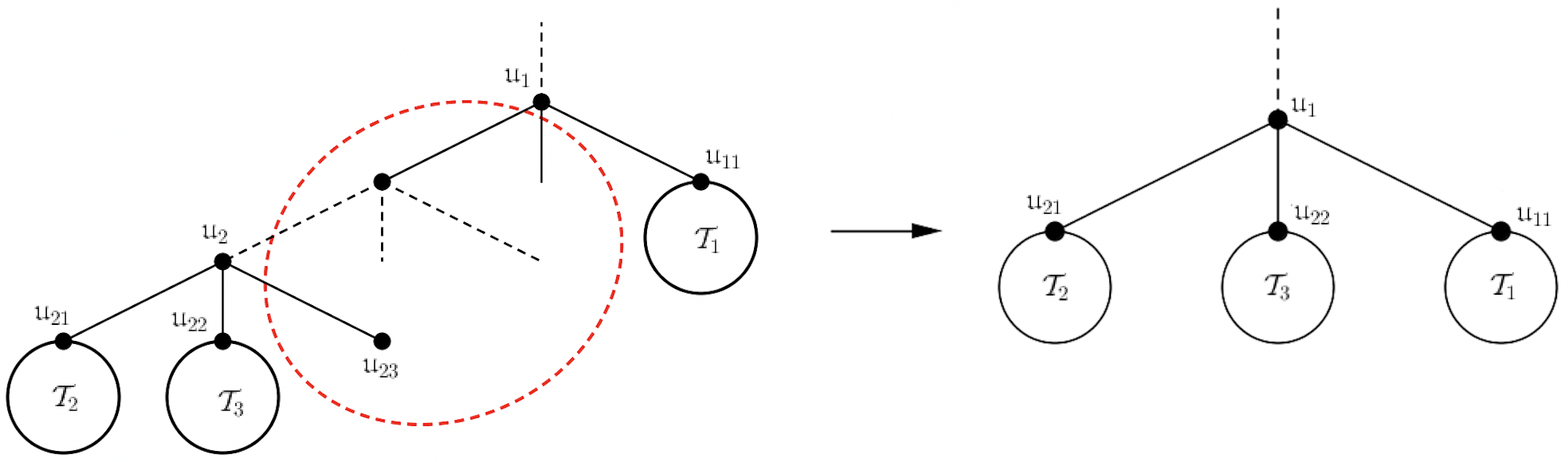}
\caption[Block of type (CL)]{A (CL) block (as in Proposition \ref{block_clcn}) viewed in the garden, together with the result of splicing. Here $\uf_{11}$ is the right (or left) child of $\uf_1$ (assume both have sign $+$) $\uf_{21}$ and $\uf_{22}$ have signs $+$ and $-$ respectively, and all the leaves in the red circle (including $\uf_{23}$) are completely paired.}
\label{fig:couples_cl}
\end{figure}
\begin{figure}[h!]
\includegraphics[scale=0.43]{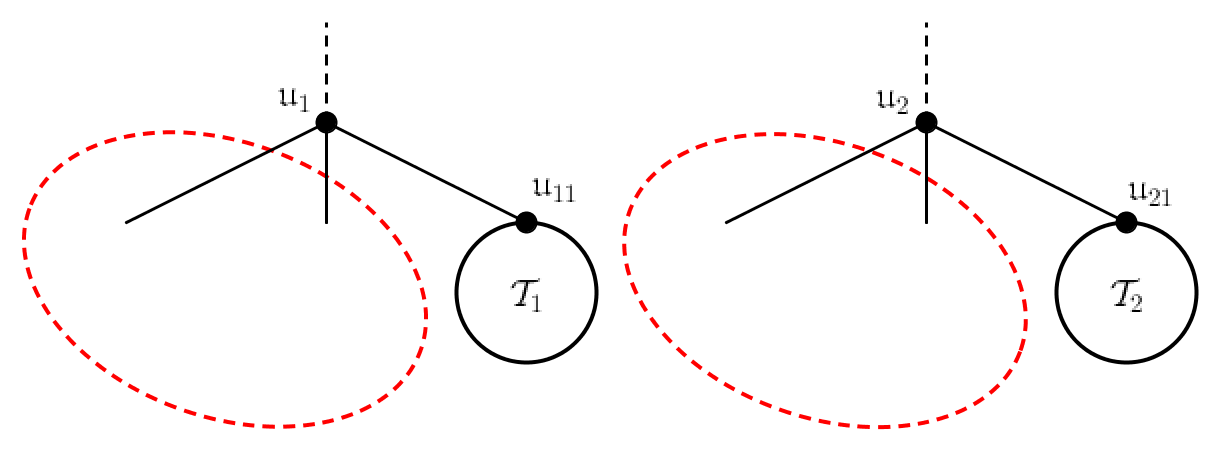}
\caption[Block of type (CN)]{A (CN) block, as in Proposition \ref{block_clcn}, viewed in the garden. Here $\uf_{11}$ is the right (or left) child of $\uf_1$, $\uf_{21}$ is the right (or left) child of $\uf_2$, and all the leaves in the two red circles are completely paired.}
\label{fig:couples_cn}
\end{figure}
\begin{proof} For couples $\Qc$, this is proved in Proposition 4.12 of \cite{DH23}. The general gardens $\Gc$ make no difference.
\end{proof}
\begin{prop}\label{cnblock}
Given a garden $\Gc$ of width $2R$. Consider any (CL) block in $\Mb(\Gc)$, we call it a \emph{root block} if the total degree of its two joints is at most $6$. Then, if we remove \emph{any set of disjoint (CN) and root (CL) blocks}, where by removing a block $\Bb$ we mean removing all non-joint atoms and all bonds $\ell\in\Bb$, then the resulting molecule \emph{has at most $2R$} connected components. If there are exactly $2R$ components, then one of them must be an isolated atom.
\end{prop}
\begin{proof} By Proposition \ref{propcplmol} (decomposing $\Gc$ into irreducible gardens if necessary) and counting degrees, we see that there are at most $R$ root (CL) blocks. For each root (CL) block, either (i) one of its joints has degree $2$, or both of its joints have degree $3$. In the latter case, due to a simple discussion using Definition \ref{defcplmol}, we see that there are only two possibilities, namely that (ii) $\uf_1$ is the root of one tree $\Tc_1$ in $\Gc$, and a child of $\uf_2$ (say $\uf_{21}$) is paired with the root of another tree $\Tc_2$ as leaves, or (iii) $\uf_{11}$ and a child of $\uf_2$ (say $\uf_{21}$) are paired with the roots of two trees in $\Gc$ as leaves. Moreover each root (CL) block in case (iii) corresponds to two trivial trees in $\Gc$, so there are at most $R-1$ of them (and there is none when $R=1$).

We may choose a root (CL) block $\Bb_0$ that is either in case (i) or case (ii) (if there are at most $R-1$ root (CL) blocks then we do not need to choose $\Bb_0$ and the proof works the same way). If $\Bb_0$ is case (i), then we first remove all the (CN) blocks. As in the proof of Proposition 4.12 in \cite{DH23}, we may define the molecule $\Mb(\Gc')$ for generalized gardens $\Gc'$ formed by $2R$ trees whose branching nodes have \emph{one or three} children nodes (plus that we only keep the pairing structure but ignore the signs of nodes and directions of bonds), similar to Definition \ref{defcplmol}. Then, the resulting molecule after removing all the (CN) blocks will be $\Mb(\Gc')$, where $\Gc'$ is a generalized garden such that for  each removed (CN) block, only nodes $\{\uf_1,\uf_2,\uf_{11},\uf_{21}\}$ from $\Gc[\Bb]$ remain in $\Gc'$, and that $\uf_{11}$ is the only child of $\uf_1$ and $\uf_{21}$ is the only child of $\uf_2$.

It is now easy to see that $\Mb(\Gc')$ has at most $R$ connected components, because each tree still has an odd number of leaves that cannot all be paired with each other. We next remove $\Bb_0$ which is in case (i), and generates one new connected component which is an isolated atom; subsequently, removing the remaining $R-1$ root (CL) blocks generates at most $R-1$ new connected components, so the result is true. Finally, if $\Bb_0$ is in case (ii), then we first remove $\Bb_0$, the resulting molecule $\Mb'$ will be $\Mb(\Gc')$ plus two separate extra single bonds, where $\Gc'$ is the new garden with the two trees $\Tc_1$ and $\Tc_2$ in $\Gc$ replaced by the two trees rooted at $\uf_{11}$ and 
$\uf_{22}$. Since all the (CN) blocks in $\Mb(\Gc)$ are are also (CN) blocks in $\Mb(\Gc')$, we can repeat the above argument and still get at most $2R-1$ connected components. This completes the proof.
\end{proof}
\begin{rem}\label{rem_realiz}  The set $\Gc[\Bb]$, as defined in Proposition \ref{block_clcn}, will be called the \emph{realization} of the block $\Bb\subset\Mb(\Gc)$ in $\Gc$. Denote also $\Nc[\Bb]$ and $\Lc[\Bb]$ to be the set of branching nodes and leaves in $\Gc[\Bb]$.
\end{rem}
\begin{cor}\label{blockchainprop} Let $\Gc$ be a garden and $\Bb\subset\Mb(\Gc)$ be a block or hyper-block that is concatenated by at least two blocks $\Bb_j\,(1\leq j\leq m)$ as in Definition \ref{defblock}, where $m\geq 2$. Then at most one $\Bb_j$ can be a (CN) block. If $\Bb$ is a block and all $\Bb_j$ are (CL) blocks, then $\Bb$ is a (CL) block. If $\Bb$ is a block and there is one (CN) block $\Bb_j$, then after doing splicing at all other (CL) blocks, this $\Bb$ becomes a single (CN) block $\Bb_j$.
\end{cor}
\begin{proof} See Corollary 4.13 of \cite{DH23}.
\end{proof}
\subsubsection{Vines and ladders} Next we recall the definition of vines and ladders in \cite{DH23}.
\begin{df}[Vines and ladders \cite{DH23}]\label{defvine}\emph{Vines} are defined as the blocks (I)--(VIII) illustrated in Figure \ref{fig:vines}. We also define the notion of \emph{ladders} as illustrated in Figure \ref{fig:vines}. We require each ladder to have $m+1\geq 2$ double bonds, and that each pair of two parallel single bonds must have opposite directions. Define the \emph{length} of a ladder to be $m\geq 1$. We refer to vines (I)--(II) as \emph{bad vines}, and vines (III)--(VIII) as \emph{normal vines}. Note that $\sigma(\Vb)=0$ for all vines $\Vb$ except vines (V) and vines (I) (see Definition \ref{defblock}), for which $\sigma(\Vb)=1$ and $\sigma(\Vb)=2$ respectively.

Define \emph{hyper-vines} (or HV for short) to be the hyper-blocks that are adjoints of vines, as in Definition \ref{defblock}. We also define \emph{vine-chains} (or VC's), resp. \emph{hyper-vine-chains} (or HVC's), to be the blocks, resp. hyper-blocks, that are formed by concatenating finitely many vines as in Definition \ref{defblock} (these vines are called \emph{ingredients}). Note that a single vine is viewed as a VC, but an HV is not viewed as an HVC. It is easy to verify that assumptions (i) and (ii) in Lemma \ref{disjointlem} hold for any HV, VC or HVC. For simplicity, we will refer to any HV, VC or HVC as \emph{vine-like objects}.

Note that, if the molecule $\Mb=\Mb(\Gc)$ comes from a garden, then any vine could be a (CL) or (CN) vine depending on whether it is a (CL) or (CN) block.
\end{df}
  \begin{figure}[h!]
  \includegraphics[scale=.12]{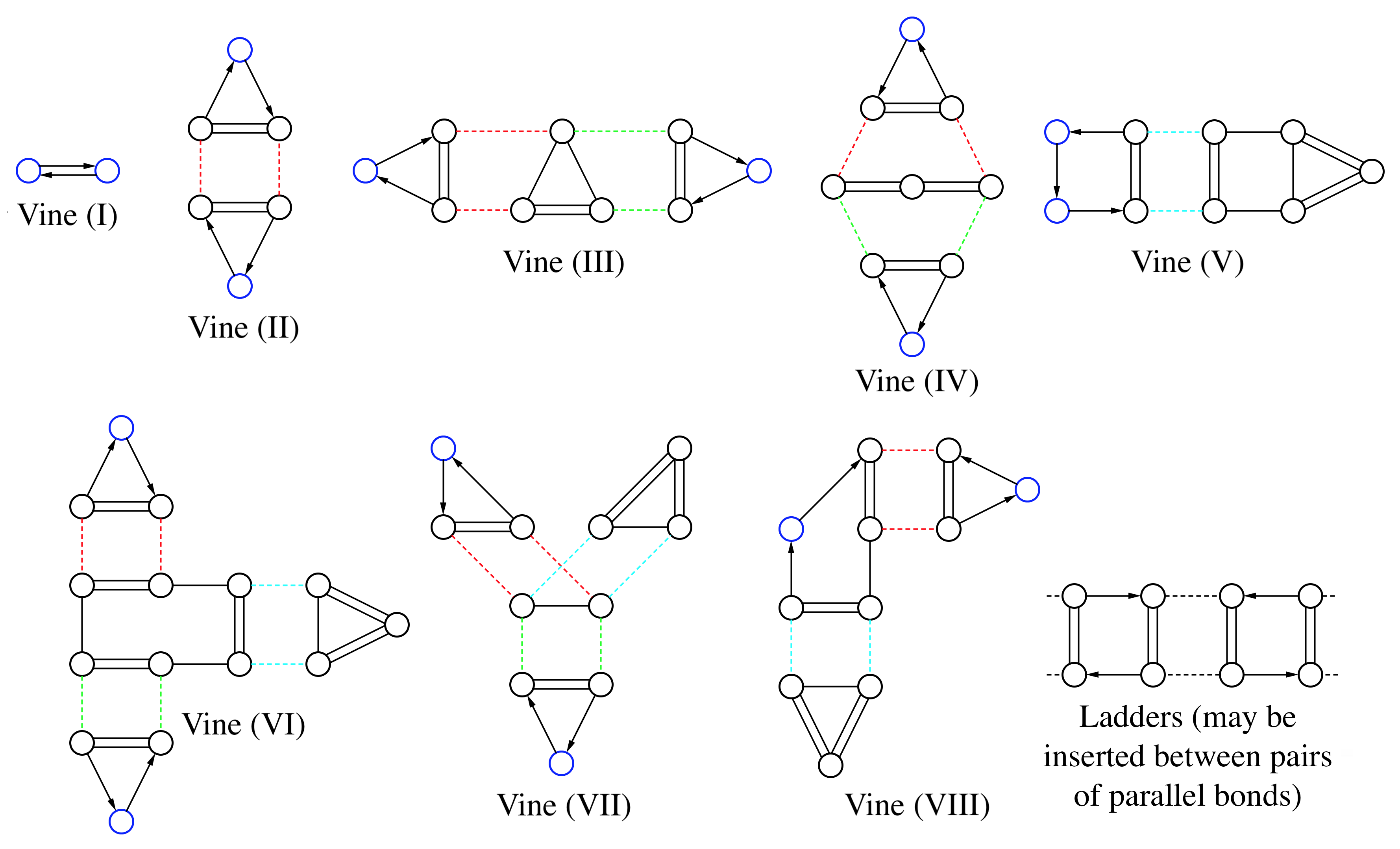}
  \caption[Block families (I)--(V)]{Vines (I)--(VIII). Here some bonds are drawn with directions, indicating that certain pairs of bonds must have opposite directions (see Definition \ref{defblock}). Moreover, a ladder may be inserted between each pair of parallel bonds that are drawn as dashed lines (distinguished by different colors).}
  \label{fig:vines}
\end{figure} 
The next proposition discussed the relative position of (CL) vines (I), and the part of (CL) vines (II) near one of its joints, in a garden.
\begin{prop}[Structure of vines in gardens]\label{molecpl} Consider a (CL) vine $\Vb\subset\Mb(\Gc)$ with joints $v_1$ and $v_2$. Let $\uf_j=\nf(v_j)$, by Proposition \ref{block_clcn} we may assume $\uf_2$ is a descendant of $\uf_1$, and also specify two children $\uf_{21}$ and $\uf_{22}$ of $\uf_2$ that have signs $+$ and $-$ respectively; let $\uf_{23}$ be the other child of $\uf_2$ (as in Proposition \ref{block_clcn}), note that $\uf_{23}$ has the same sign as $\uf_2$.
\begin{enumerate}[{(1)}]
\item If $\Vb$ is Vine (I), then exactly one of the following two scenarios happens. See Figure \ref{fig:block_mole}.
\begin{enumerate}[{(a)}]
\item Vine (I-a): $\uf_2$ is the left or right child of $\uf_1$, and $\uf_{23}$ is paired to the middle child $\uf_0$ of $\uf_1$ as leaves.
\item Vine (I-b): $\uf_2$ is the middle child of $\uf_1$, and $\uf_{23}$ is paired to the left or right child $\uf_0$ of $\uf_1$ as leaves.
\end{enumerate}
\item If $\Vb$ is Vine (II), then $v_2$ is connected to two atoms $v_3$ and $v_4$ by single bonds, while $v_3$ and $v_4$ are connected by a double bond. Let $\uf_j=\nf(v_j)$, then (up to symmetry) exactly one of the following five scenarios happens. See Figure \ref{fig:block_mole}.
\begin{enumerate}[{(a)}]
\item Vine (II-a): $\uf_2$ is a child of $\uf_4$, and $\uf_{23}$ is paired with one child $\uf_0$ of $\uf_3$ as leaves, and the other two children of $\uf_4$ are paired with the other two children of $\uf_3$ as leaves. Here neither $\uf_3$ nor $\uf_4$ is a descendant of the other, but they have a common ancestor, namely $\uf_1$.

\item Vine (II-b): $\uf_2$ is a child of $\uf_3$, and $\uf_{23}$ is paired with one child $\uf_0$ of $\uf_4$ as leaves, and the other two children of $\uf_3$ are paired with the other two children of $\uf_4$ as leaves. Here neither $\uf_4$ nor $\uf_3$ is a descendant of the other, but they have a common ancestor, namely $\uf_1$.

\item Vine (II-c): $\uf_4$ is a child of $\uf_3$ and $\uf_2$ is a child of $\uf_4$. One of the the other two children of $\uf_3$ is paired with one of the other children of $\uf_4$ as leaves, and the remaining child $\uf_0$ of $\uf_3$ is paired with $\uf_{23}$ as leaves. Here $\uf_3$ is a descendant of $\uf_1$.

\item Vine (II-d): $\uf_2$ and $\uf_4$ are two children of $\uf_3$, and $\uf_{23}$ is paired with one child $\uf_0$ of $\uf_4$ as leaves, and the remaining child of $\uf_3$ is paired with another child of $\uf_4$ as leaves. Here $\uf_3$ is a descendant of $\uf_1$.

\item Vine (II-e): $\uf_2$ is a child of $\uf_3$, and $\uf_4=\uf_{23}$. The other two children of $\uf_3$ are paired with two of the children of $\uf_4$ as leaves. Here $\uf_3$ is a descendant of $\uf_1$.
\end{enumerate}
\end{enumerate}
\begin{figure}[h!]
\includegraphics[scale=0.2]{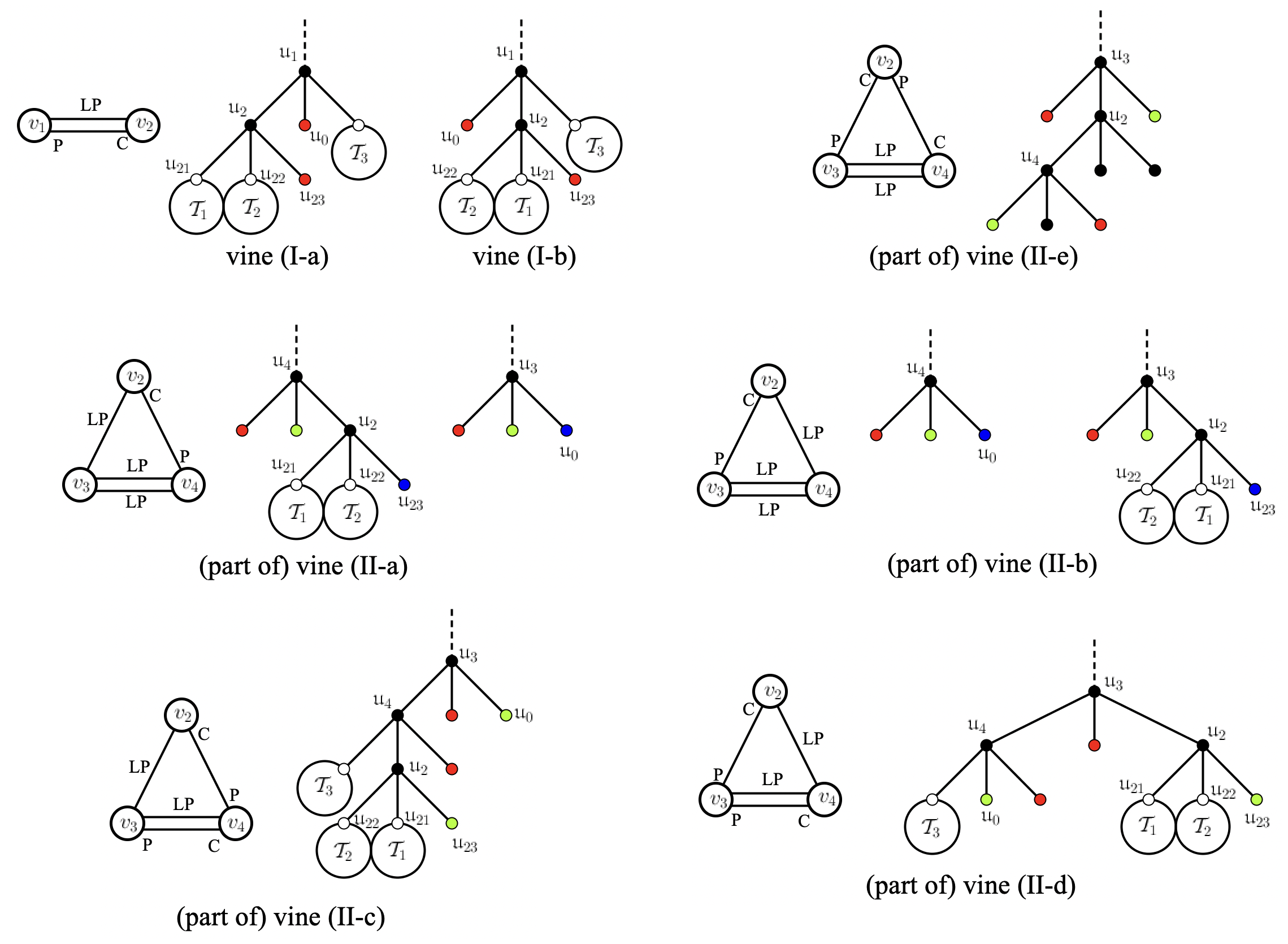}
\caption[Block families]{Vines (II-a)--(II-e) and (I-a)--(I-b) in gardens (the relevant positions of nodes may vary). Here the labels of the bonds in the molecule are indicated as in Definition \ref{defcplmol}. The nodes in white color are called \emph{free children}, with a tree $\Tc_j$ root at each of them; these notions will be used in Definition \ref{twist} below.}
\label{fig:block_mole}
\end{figure}
For simplicity, below we will call a (CL) vine $\Vb$ \emph{core} if it is bad and not Vine (II-e), and \emph{non-core} if it is normal or Vine (II-e).
\end{prop}
\begin{proof} See Proposition 5.3 of \cite{DH23}.
\end{proof}
\begin{rem} Note that the classification of vines (I)--(VIII) only involves the structure of the molecule $\Mb$ (as a directed graph), but the distinction between (CL) and (CN) vines, as well as classification of vines (II-a)--(II-e), is intrinsic to the structure of the \emph{garden} $\Gc$; for instance, it does not make sense to talk about (CL) vines or vines (II-e) if $\Mb$ does not have the form $\Mb(\Gc)$.
\end{rem}
\subsection{Twists}\label{sectwist}
By exploiting the structure of bad (CL) vines in a garden, as described in Proposition \ref{molecpl}, we can define the operation of \emph{twisting}, which captures the cancellation between such vines.
\begin{df}[Twists \cite{DH23}]\label{twist} Let $\Gc$ be a given garden with the corresponding molecule $\Mb(\Gc)$, and let $\Vb\subset\Mb(\Gc)$ be a core (CL) vine as in Proposition \ref{molecpl}. Let $v_j$ and $\uf_j$ for $1\leq j\leq 4$ be as in Proposition \ref{molecpl}, we shall define a new garden $\Gc'$, which we call a \emph{unit twist} of $\Gc$, as follows.

First, in $\Gc'$, let any possible parent-child relation, as well as any possible children pairings, between $\uf_3$ and $\uf_4$, be exactly the same as in $\Gc$. Next, let the structure of $\Gc'$ \emph{excluding the subtrees rooted at $\uf_3$ and $\uf_4$, or the subtree rooted at $\uf_1$ for Vine (I)}, be exactly the same as $\Gc$. Moreover, consider the \emph{free children} in Figure \ref{fig:block_mole}, i.e. the nodes in white color; we require that the positions of the two free children $\uf_{21}$ and $\uf_{22}$ (as children of $\uf_2$), as well as the positions of the two subtrees (namely $\Tc_1$ and $\Tc_2$) rooted at them, be \emph{switched}\footnote{They are switched because the sign of $\uf_2$ is changed (see Remark \ref{twistexplain}); if we locate $\uf_{21}$ as the child of $\uf_{2}$ other than $\uf_{23}$ that has sign $+$ (and same for $\uf_{22}$ but with sign $-$), then this remains the same for both gardens.} in $\Gc'$ compared to $\Gc$. For the other free child (if it exists), we require that its position (as a child of $\uf_1$ or $\uf_4$) and the subtree (namely $\Tc_3$) rooted at it, be exactly the same in $\Gc'$ as in $\Gc$. Then, it is easy to see that there are exactly two options to insert $\uf_2$, one as a child of $\uf_3$, and the other as a child of $\uf_4$ (for Vine (I), the two options are children of $\uf_1$ that has the same or opposite sign with $\uf_1$). One of these two choices leads to $\Gc$, and we define the garden given by the other choice as $\Gc'$. Clearly $\Gc'$ is prime iff $\Gc$ is, where we recall the notion of prime gardens in Proposition \ref{propstructure2}.

In general, suppose we start with a collection of (CL) vines $\Vb_j\subset\Mb(\Gc)\,(0\leq j\leq q-1)$, such that any two are either disjoint or only share one common joint and no other common atom (i.e. the union of all $\Vb_j$ equals the disjoint union of finitely many VC's and HVC's). Then, we call any garden $\Gc'$ a \emph{twist} of $\Gc$, if $\Gc'$ can be obtained from $\Gc$ by performing the unit twist operation at a subset of these blocks, which only contains core vines. In particular, for any given $\Gc$ and $\Vb_j$, the number of possible twists is a power of two, and at most $2^q$.
\end{df}
Since the notion of twisting is of vital importance in our proof (especially in Section \ref{cancelvine}), we will make several remarks below explaining Definition \ref{twist} in more detail.
\begin{rem}\label{explaintwist} We discuss an example of the (unit) twist operation in Definition \ref{twist}. Suppose $\Vb$ is Vine (II-c) or (II-d) in Figure \ref{fig:block_mole}. Then we have that:
\begin{itemize}
\item The node $\uf_4$ is the left child of $\uf_3$, and the middle child of $\uf_3$ is paired with the right child of $\uf_4$ as leaves.
\item The left child of $\uf_4$ is a free child with subtree $\Tc_3$. The left and middle children of $\uf_2$ are the two free children $\uf_{21}$ and $\uf_{22}$ (or $\uf_{22}$ and $\uf_{21}$), with the subtrees rooted at $\uf_{21}$ and $\uf_{22}$ being $\Tc_1$ and $\Tc_2$ respectively.
\item $\uf_2$ is a child of $\uf_3$ (or $\uf_4$), and the right child $\uf_{23}$ of $\uf_2$ is paired to a child of $\uf_4$ (or $\uf_3$) as leaves.
\end{itemize}

Now by Definition \ref{twist}, all these properties must hold in both $\Gc$ and $\Gc'$; also the structure of $\Gc$ and $\Gc'$, excluding the subtree rooted at $\uf_3$, must be the same. This leaves only two possibilities: either $\uf_2$ is \emph{the middle child} of $\uf_4$ and $\uf_{23}$ is paired to \emph{the right child} of $\uf_3$ as leaves, or $\uf_2$ is \emph{the right child} of $\uf_3$ and $\uf_{23}$ is paired to \emph{the middle child} of $\uf_4$ as leaves. These are exactly vines (II-c) and (II-d) in in Proposition \ref{molecpl}. Note that for vines (II-c), $\uf_{22}$ is the left child of $\uf_2$ and $\uf_{21}$ is the middle child, while for vines (II-d) $\uf_{21}$ is the left child and $\uf_{22}$ is the middle child, which is consistent with the description in Definition \ref{twist}.

In the same way, we can see that performing one unit twist operation \emph{exactly switches vines (I-a), (II-a), (II-c) vines with vines (I-b), (II-b), (II-d) vines}, respectively.
\end{rem}
\begin{rem}\label{twistexplain} Throughout the proof below, for any fixed (CL) vine $\Vb$, we always adopt the notations $(\uf_1,\uf_2,\uf_{11},\uf_{21},\uf_{22})$ as in Proposition \ref{block_clcn}; for bad (CL) vines we also adopt the notations $(\uf_3,\uf_4,\uf_{23},\uf_0)$ as in Proposition \ref{molecpl}, whenever applicable. The following useful facts are easily verified from Definition \ref{twist}. They are stated for unit twists but can be extended to general twists.
\begin{enumerate}[{(a)}]
\item Let $\Gc$ and $\Gc'$ be unit twists of each other at a bad (CL) vine $\Vb\subset\Mb(\Gc)$, then $\Mb(\Gc)$ and $\Mb(\Gc')$ are the same as directed graphs. They also have the same labelings of bonds, except at the atom $v_2$, where the labels of the two bonds connecting $v_2$ to atoms in $\Vb$ are switched (one label is PC with $v_2$ labeled C and the other label is LP), see Figure  \ref{fig:block_mole}.
\item The values of $\zeta_{\uf_j}$ for any branching node $\uf_j\,(j\neq 2)$ are the same for $\Gc$ and $\Gc'$, while the values of $\zeta_{\uf_2}$ are the opposite for $\Gc$ and $\Gc'$.
\item If we do splicing (as defined in Proposition \ref{block_clcn}) for $\Gc$ and the (CL) vine $\Vb$, or for $\Gc'$ and the same (CL) vine $\Vb$ (as shown in (a) above), then the two resulting \emph{gardens}, defined as $\Gc^{\mathrm{sp}}$ and $(\Gc')^{\mathrm{sp}}$, are the same.
\item The $k$-decorations of $\Gc$ are in bijection with $k$-decorations of $\Gc'$, where the values of $k_\mf$ for any branching node or leaf $\mf$ are the same in both cases, but we \emph{switch} the values of $k_{\uf_2}$ and $k_{\uf_{23}}$, see Figure \ref{fig:twist_dec}.
\item Moreover, for any $k$-decoration of $\Gc$ and the corresponding $k$-decoration of $\Gc'$ as in (d), the decorations of $\Gc^{\mathrm{sp}}=(\Gc')^{\mathrm{sp}}$ \emph{inherited} from the  are the same; here \emph{inheriting} means that the value of $k_\mf$ is kept the same for any $\mf$, whether it is viewed as a node of $\Gc$ or $\Gc^{\mathrm{sp}}$ (this notion will also be used in other similar settings).
\end{enumerate}
\end{rem}
\begin{figure}[h!]
\includegraphics[scale=0.2]{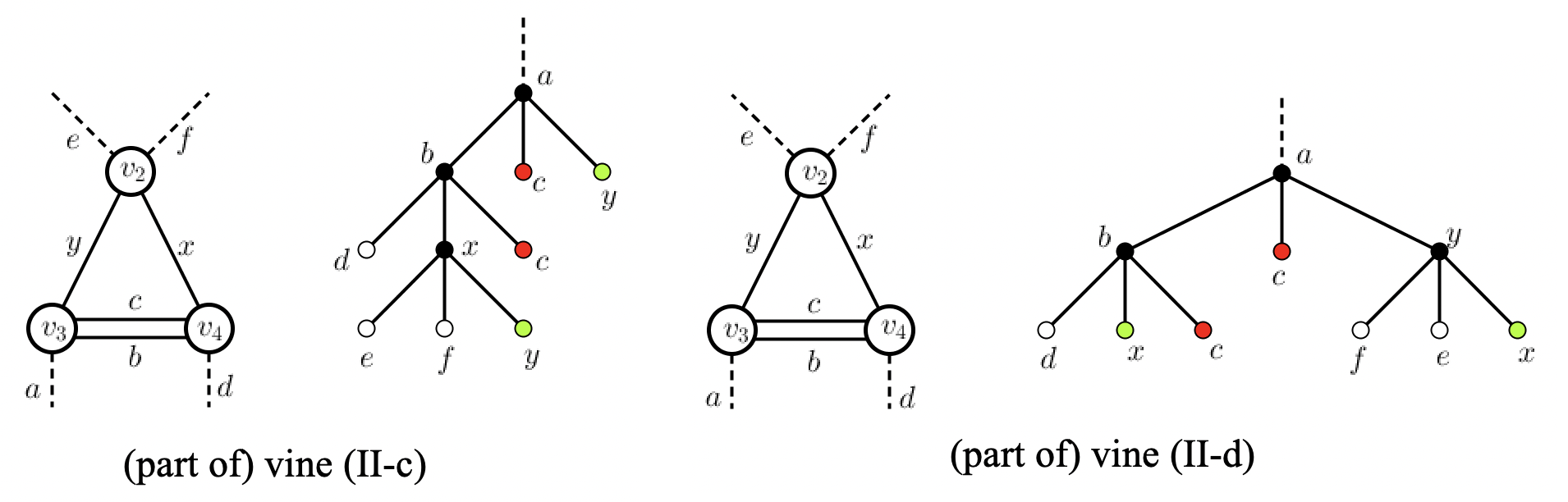}
\caption{Decorations for two gardens with Vine (II-c) and Vine (II-d) which are unit twists of each other, see Remark \ref{twistexplain}. Note that, in the context of Figure \ref{fig:block_mole}, the values of $k_{\mf}$ are the same for each $\mf$, except that the values of $k_{\uf_2}$ and $k_{\uf_{23}}$ are switched.}
\label{fig:twist_dec}
\end{figure}
\subsection{Layered full twists}\label{secfulltwist} Note that the definition and discussion of twists in Section \ref{sectwist} involves only the garden structure of $\Gc$ (and the molecule structure of $\Mb(\Gc)$), but not the \emph{layering structure}. Since we are dealing with layered gardens in Proposition \ref{layergarden}, in the proof below we need to apply twisting operations that take into account the layerings and the regular trees and regular couples in Proposition \ref{propstructure2}. These are called \emph{layered full twists} or \emph{LF twists}, which we will now discuss.

Before proceeding, we first make one observation: let $\Gc$ be a garden obtained from another garden $\Gc_1$ (which may be $\Gc_{\mathrm{sk}}$) by replacing each leaf pair with a regular couple and each branching node with a regular tree, in the sense of Proposition \ref{propstructure2}. Then, any layering of $\Gc$ naturally induces a \emph{pre-layering} of $\Gc_1$, where for each branching node $\mf$ of $\Gc_1$, let the corresponding regular tree in $\Gc$ be $\Tc_{(\mf)}$, then we assign the layer of $\mf$ in $\Gc_1$ to be the layer of the \emph{lone leaf} of $\Tc_{(\mf)}$ in $\Gc$.
\begin{rem}\label{sublayer} In general, a layering of $\Gc$ only induces a pre-layering of $\Gc_1$; however, by Proposition \ref{layerreg1}, if $\Gc$ is canonical, and if all the regular couples and regular trees involved in the process of obtaining $\Gc$ from $\Gc_1$ are \emph{coherent}, then the pre-layering of $\Gc_1$ can be uniquely extended to a canonical layering. In fact, the layer of each leaf pair $(\lf,\lf')$ is just $\min(\Lf_{\lf^{\mathrm{pr}}},\Lf_{(\lf')^{\mathrm{pr}}})$, where $\lf^{\mathrm{pr}}$ and $(\lf')^{\mathrm{pr}}$ are the parents of $\lf$ and $\lf'$.
\end{rem}
\begin{df}[Layered full twists]\label{lftwist} Let $\Gc\in\Gs_{p+1}$ (or $\Cs_{p+1}$) be a canonical layered object, and $\Gc_{\mathrm{sk}}$ be its skeleton as in Proposition \ref{propstructure2}. We then have a pre-layering of $\Gc_{\mathrm{sk}}$ as described above (and also a layering of atoms of the molecule $\Mb(\Gc_{\mathrm{sk}})$ following Definition \ref{defcplmol}). For any branching node $\mf$ of $\Gc_{\mathrm{sk}}$ and any leaf pair $(\lf,\lf')$ of $\Gc_{\mathrm{sk}}$, let the corresponding regular tree and regular couple be $\Tc_{(\mf)}$ and $\Qc_{(\lf,\lf')}$.

Now let $\Vb\subset\Mb(\Gc_{\mathrm{sk}})$ be a core (CL) vine as in Definition \ref{twist}, and let the notations \[(\uf_1,\uf_{11},\uf_{2},\uf_{21},\uf_{22},\uf_{23},\uf_0,\uf_3,\uf_4)\] be as in Propositions \ref{block_clcn} and \ref{molecpl} (note $(\uf_3,\uf_4)$ is absent for vines (I-a) and (I-b)). Assume further that $\Lf_{\uf_2}\leq\min(\Lf_{\uf_3},\Lf_{\uf_4})$ for the pre-layering of $\Gc_{\mathrm{sk}}$, and that the regular tree $\Tc_{(\uf_2)}$ and the regular couple $\Qc_{(\uf_{23},\uf_0)}$ in $\Gc$ are both \emph{coherent}.

Let $\Gc_{\mathrm{sk}}'$ be the unit twist of $\Gc_{\mathrm{sk}}$ at $\Vb$ as in Definition \ref{twist}, then the branching nodes and leaf pairs of $\Gc_{\mathrm{sk}}'$ are in one-to-one correspondence with those of $\Gc_{\mathrm{sk}}$: the branching node $\uf_2$ and the leaf pair $(\uf_{23},\uf_0)$ appear in both gardens, and the other branching nodes and leaf pairs are exactly the same. We then define a \emph{unit layered full twist} (or \emph{unit LF twist}) $\Gc'$ of $\Gc$, as follows.

Let the pre-layering of $\Gc_{\mathrm{sk}}'$ be given such that the layer $\Lf_\mf$ of each branching node $\mf$ of $\Gc_{\mathrm{sk}}'$ equals the layer of $\mf$ in $\Gc_{\mathrm{sk}}$. To form $\Gc'$, starting from either $\Gc_{\mathrm{sk}}$ or $\Gc_{\mathrm{sk}}'$, we replace each branching node $\mf\neq \uf_2$ by the \emph{same} regular tree $\Tc_{(\mf)}$ as in $\Gc$, and replace each leaf pair $(\lf,\lf')\neq (\uf_{23},\uf_0)$ by the \emph{same} regular couple $\Qc_{(\lf,\lf')}$, with the same layerings. Then we replace the coherent regular tree $\Tc_{(\uf_2)}$ and regular couple $\Qc_{(\uf_{23},\uf_0)}$ by regular tree $\Tc_{(\uf_2)}^{\,\prime}$ and regular couple $\Qc_{(\uf_{23},\uf_0)}^{\,\prime}$ respectively, which satisfy the followings:
\begin{itemize}
\item The layer of the lone leaf of $\Tc_{(\uf_2)}^{\,\prime}$ equals the layer $\Lf_{\uf_2}$ of $\uf_2$ in $\Gc_{\mathrm{sk}}'$.
\item The objects $\Tc_{(\uf_2)}^{\,\prime}$ and $\Qc_{(\uf_{23},\uf_0)}^{\,\prime}$ are coherent, and have the exact structure described in Proposition \ref{layerreg1}, where $(q,q')$ in Proposition \ref{layerreg1} are given by the pre-layering of $\Gc_{\mathrm{sk}}'$.
\item We have $n(\Tc_{(\uf_2)}^{\,\prime})+n(\Qc_{(\uf_{23},\uf_0)}^{\,\prime})=n(\Tc_{(\uf_2)})+n(\Qc_{(\uf_{23},\uf_0)})$, and the total number of \emph{layer $p$} branching nodes in $\Tc_{(\uf_2)}^{\,\prime}$ and $\Qc_{(\uf_{23},\uf_0)}^{\,\prime}$ equals that of $\Tc_{(\uf_2)}$ and $\Qc_{(\uf_{23},\uf_0)}$.
\end{itemize}

In general, suppose we start with a collection of (CL) vines $\Vb_j\subset\Mb(\Gc_{\mathrm{sk}})\,(0\leq j\leq q-1)$, such that any two are either disjoint or only share one common joint and no other common atom. Then, we call any garden $\Gc'$ a \emph{layered full twist} (or \emph{LF twist}) of $\Gc$, if $\Gc'$ can be obtained from $\Gc$ by performing some unit LF twist operation at each $\Vb_j$ that is core and satisfies the above assumptions (i.e. $\Lf_{\uf_2}\leq\min(\Lf_{\uf_3},\Lf_{\uf_4})$ and that $\Tc_{(\uf_2)}$ and $\Qc_{(\uf_{23},\uf_0)}$ are both coherent).
\end{df}
\begin{prop}\label{lftwistprop} Let $\Gc\in \Gs_{p+1}$ and let $\Gc'$ be an LF twist of $\Gc$ as in Definition \ref{lftwist}, then we have $\Gc'\in\Gs_{p+1}$ and $n(\Gc')=n(\Gc)$. Recall also the definition of $\Gs_{p+1}^{\mathrm{tr}}$ in Definition \ref{defgtr}. If $n(\Gc)\leq N_{p+1}^4$, then we have $\Gc\in\Gs_{p+1}^{\mathrm{tr}}$ if and only if $\Gc'\in \Gs_{p+1}^{\mathrm{tr}}$. The same holds for $\Cs_{p+1}$ and $\Cs_{p+1}^{\mathrm{tr}}$.
\end{prop}
\begin{proof} We only need to consider $\Gs_{p+1}$ and unit LF twists. The fact that $n(\Gc)=n(\Gc')$ is obvious from definition. Moreover, if $n(\Gc)\leq N_{p+1}^4$, then the requirement (ii) in the definition of $\Gs_{p+1}^{\mathrm{tr}}$ and $\Cs_{p+1}^{\mathrm{tr}}$ in Definition \ref{defgtr} is redundant; also the requirement (i) is just that the number of \emph{layer $p$} branching nodes in each of the trees in $\Gc$ is at most $N_{p+1}$. But for each tree, this number is the same for $\Gc$ and $\Gc'$ due to the requirements about layer $p$ branching nodes in Definition \ref{lftwist}, so we only need to prove $\Gc'\in\Gs_{p+1}$ (same for $\Cs_{p+1}$), then it will follow that $\Gc\in\Gs_{p+1}^{\mathrm{tr}}\Leftrightarrow\Gc'\in\Gs_{p+1}^{\mathrm{tr}}$.

Now we prove $\Gc'\in\Gs_{p+1}$. We may assume that $\Gc'$ is constructed starting from $\Gc_{\mathrm{sk}}'$ (the unit twist of $\Gc_{\mathrm{sk}}$), since the case when $\Gc'$ is constructed from $\Gc_{\mathrm{sk}}$ is easier. Next, in $\Gc$ and $\Gc'$, we shall replace both coherent regular trees $\Tc_{(\uf_2)}$ and $\Tc_{(\uf_2)}^{\,\prime}$ by the single node $\uf_2$, and both coherent regular couples $\Qc_{(\uf_{23},\uf_0)}$ and $\Qc_{(\uf_{23},\uf_0)}^{\,\prime}$ by the leaf pair $(\uf_{23},\uf_0)$. By Proposition \ref{layerreg1} and Remark \ref{sublayer}, we know that these operations do not change the canonicity of $\Gc$ and $\Gc'$. Suppose the layered gardens $\Gc$ and $\Gc'$ become $\Hc$ and $\Hc'$ after these operations, then $\Hc'$ is a unit LF twist of $\Hc$ as in Definition \ref{twist} with each node layered in the same way as the corresponding node in $\Hc$. we only need to prove that if $\Hc\in\Gs_{p+1}$ then $\Hc'\in\Gs_{p+1}$.

To show $\Hc'\in\Gs_{p+1}$, we only need to verify that $\Lf_{\nf'}\leq\Lf_\nf$ for any branching node $\nf$ and its child node $\nf'$, and also verify the statement of Proposition \ref{canonequiv} for any $(\nf,\nf')$ with all their descendant leaves completely paired. The former is true because the layering of $\Hc'$ is the same as that of $\Hc$, and all the parent-child relations in $\Hc'$ and $\Hc$ are also the same (except that $\uf_2$ is a child of $\uf_4$ in vines (II-a) and (II-c), while it is a child of $\uf_3$ in vines (II-b) and (II-d), but this does not matter because $\Lf_{\uf_2}\leq\min(\Lf_{\uf_3},\Lf_{\uf_4})$ in any case). The latter can be verified by examining the different cases in Figure \ref{fig:block_mole}; for example, consider when $\Hc$ is Vine (II-a) and $\Hc'$ is Vine (II-b). Let $(\nf,\nf')$ be arbitrary nodes in $\Hc'$ as in Proposition \ref{canonequiv}. If neither of $(\nf,\nf')$ is an ancestor node of $\uf_{23}$ or $\uf_0$, then $(\nf,\nf')$ is not affected by the twisting, hence the result for $\Hc'$ follows from that of $\Hc$. If one of $(\nf,\nf')$ is an ancestor of $\uf_3$ (or $\uf_4$), then either $\nf$ or $\nf'$ must also be an ancestor of $\uf_4$ (or $\uf_3$) due to complete pairing of leaves, in which case all the descendant leaves of $\nf$ and $\nf'$ are also completely paired in $\Hc$, so the result again follows. The remaining cases are when $(\nf,\nf')=(\uf_0,\uf_{23})$ or $(\nf,\nf')=(\uf_0,\uf_2)$, which are again easily checked. The other cases in Figure \ref{fig:block_mole} follow from similar discussions.
\end{proof}
\section{Layered vines and ladders: Combinatorics}\label{section8} In this section we first introduce the notion of \emph{coherent} vines and ladders, which plays a similar role as that of coherent regular couples and regular trees in Definition \ref{defcoh}, but works on the molecular level.\begin{df}[Coherent vines and ladders]\label{cohmol} Let $\Mb=\Mb(\Gc)$ be the molecule coming from a garden $\Gc$, and $\Vb\subset\Mb$ is a vine as illustrated in Figure \ref{fig:vines}. We may split the set of non-joint atoms of $\Vb$ into groups of $2$ or $3$ elements, as follows:
\begin{itemize}
\item For Vine (I) there is no non-joint atom.
\item For Vine (II) we put each pair of atoms connected by a double bond into one group.
\item For vines (III) and (V)--(VIII) there is a unique triangle that contains no joint, and for Vine (IV) there is a unique three-atom subset with two double bonds between them; we put these three atoms into one group.
\item For vines (III)--(VIII), after fixing this three-atom group, we put each pair of remaining atoms connected by a double bond into one group.
\item For Vine (VII) there are two atoms left after the above, and we put them in the same group.
\end{itemize}

Now, suppose $\Gc$ is \emph{pre-layered}, so each atom $v\in\Mb$ has a layer $\Lf_v$ as in Definition \ref{defcplmol}. Then, we define the vine $\Vb$ to be \emph{coherent} if the atoms in each group are in the same layer. Similarly, we define a ladder to be coherent, if any two atoms connected by a double bond are in the same layer. In both cases, for non-coherent objects, we define the \emph{incoherency index} to be the number $g$ of groups such not all atoms in this group are in the same layer.
\end{df}
\subsection{Number of coherent layerings} We now establish an upper bound for the number of coherent layerings of a given vine or ladder, in the same spirit as Proposition \ref{layerreg2}.
\begin{prop}\label{layerlad} Let $\Lb$ be a ladder of length $m$ in any molecule $\Mb=\Mb(\Gc)$ coming from a garden $\Gc$. Consider a pre-layering of $\Gc$ which leads to a layering of atoms of $\Mb$; assume that $\Lb$ has fixed incoherency index $g$ under this layering, where $0\leq g\leq m+1$. Then we have\begin{equation}\label{sumofexp0}
\sum_{(\Lf_v)}\prod_{\ell}C_0^{|\Lf_v-\Lf_{v'}|}\leq C_0^mC_2^{g+1},
\end{equation} with possibly different values of $C_0$ (see the convention in Section \ref{setupparam}), where the summation is taken over all possible layerings of atoms in $\Lb$ with fixed $g$, and the product is taken over all bonds $\ell$ connecting atoms $v,v'\in\Lb$.
\end{prop}
\begin{proof} Given $g$, there are at most $\binom{m+1}{g}\leq 2^{m+1}$ choices for the locations of the double bonds where incoherency happens. With these locations fixed, we may consider the ladder between any two consecutive locations which is then coherent, so we only need to prove the result when $g=0$. By disregarding the atoms at both ends of $\Lb$, we may also assume that all atoms in the ladder have degree $4$. Let the atoms in $\Lb$ be $u_j$ and $v_j$ (which are connected by a double bond) where $0\leq j\leq m$, and define $\Lf_{u_j}=\Lf_{v_j}:=q_j$.

We claim that, for any $1\leq j\leq m-1$, either $q_j\leq q_{j-1}\leq\cdots\leq q_0$, or $q_j\leq q_{j+1}\leq\cdots\leq q_m$. In fact, suppose not, then there exist $1\leq j'\leq j\leq j''\leq m-1$, such that
\[q_j\leq q_{j-1}\leq\cdots \leq q_{j'},\,\,q_{j'}>q_{j'-1};\quad q_j\leq q_{j+1}\leq\cdots \leq q_{j''},\,\,q_{j''}>q_{j''+1}.\] By symmetry, we may also assume $q_{j'}\geq q_{j''}$. Starting from the atom $w_1:=u_{j'}$, we successively choose $w_i$ such that $\nf(w_i)$ is the parent node of $\nf(w_{i-1})$ in the notation of Definition \ref{defcplmol}, then $w_i$ is connected to $w_{i-1}$ by a bond, so if $w_i\in\{u_{j_i},v_{j_i}\}$ then $j_i\in\{j_{i-1},j_{i-1}+1,j_{i-1}-1\}$. On the other hand $\Lf_{w_i}$ is nondecreasing in $i$, so $j_i$ cannot equal $j'-1$ or $j''+1$, thus we have $j'\leq j_i\leq j''$ for each $i$, which means the process continues indefinitely instead of reaching either end of the ladder. This is clearly impossible as all the $w_i$ must be different, which proves the claim.

Now, by the above claim, we know that the whole sequence $(q_0,\cdots,q_m)$ can be divided into two monotonic subsequences $q_0\geq\cdots\geq q_j$ and $q_{j+1}\leq\cdots \leq q_m$ (just look at the place where the first inequality in the claim changes to the second inequality when $j$ changes to $j+1$). It is well known that there are at most $2^{m+\Df+1}$ ways to choose a nondecreasing sequence of $m+1$ terms in $\{0,\cdots,\Df\}$, so the number of choices for $(q_0,\cdots,q_m)$ is at most $2^{2(m+\Df+1)}=C_0^mC_2$. In addition, due to the monotonicity of the subsequences of $(q_j)$, the product in (\ref{sumofexp0}) is also bounded by
\[C_0^{2(q_0-q_j)}\cdot C_0^{2(q_m-q_{j+1})}\leq C_0^{2\Df}\leq C_2.\] Putting together, this proves (\ref{sumofexp0}).
\end{proof}
\begin{prop}\label{layervine} Let $\Gc\in\Gs_{p+1}$ (or $\Cs_{p+1}$) be a canonical layered object, and let the skeleton $\Gc_{\mathrm{sk}}$, the molecule $\Mb=\Mb(\Gc_{\mathrm{sk}})$ and the regular objects $\Tc_{(\mf)}$ and $\Qc_{(\lf,\lf')}$ be as in Definition \ref{lftwist}. Let $\Vb\subset\Mb$ be any (CL) vine, note that the layering of $\Gc$ induces a layering of atoms of $\Vb$ by Section \ref{secfulltwist} and Definition \ref{defcplmol}. 

Assume $\Vb$ contains $2m+2$ (for bad vine) or $2m+5$ (for normal vine) atoms, and has incoherency index $0\leq g\leq m$ (for bad vine) or $0\leq g\leq m+1$ (for normal vine). Let the two joints of $\Vb$ be $v_1$ and $v_2$, such that $\uf_1=\nf(v_1)$ is an ancestor of $\uf_2=\nf(v_2)$, and assume that $\Lf_{v_1}=q$ and $\Lf_{v_2}=q'$ are given with $q\geq q'$. 

Then the same bound (\ref{sumofexp0}) holds, where the summation is taken over all possible layerings of atoms in $\Vb$ with fixed $g$, and the product is taken over all bonds $\ell$ connecting atoms $v,v'\in\Vb$. Moreover, if $g=0$ and the regular objects $\Tc_{(\mf)}$ and $\Qc_{(\lf,\lf')}$ are \emph{coherent} for all branching nodes $\mf$ and leaf pairs $(\lf,\lf')$ \emph{that belong to} $\Gc_{\mathrm{sk}}[\Vb]\backslash\{\uf_1\}$, then the right hand side of (\ref{sumofexp0}) can be improved to $C_0^{m+|q-q'|}$. Here $\Gc_{\mathrm{sk}}[\Vb]$ is the realization of $\Vb$ in the garden $\Gc_{\mathrm{sk}}$ as defined in Proposition \ref{block_clcn} (see Remark \ref{rem_realiz}).
\end{prop}
\begin{proof} First, (\ref{sumofexp0}) follows from Proposition \ref{layerlad}, because $\Vb$ consists of at most three ladders plus less than $20$ extra atoms, so we can apply Proposition \ref{layerlad} to the part of (\ref{sumofexp0}) involving atoms and bonds in one of these ladders, while the part involving the extra atoms contribute at most a factor of $C_0^{100\Df}\leq C_2$.

Now, we only need to consider the coherent case ($g=0$) and assume all the $\Tc_{(\mf)}$ and $\Qc_{(\lf,\lf')}$ for $\mf,\lf,\lf'\in\Gc_{\mathrm{sk}}[\Vb]\backslash\{\uf_1\}$ are coherent. Using Remark \ref{sublayer}, we may replace each such $\Tc_{(\mf)}$ by a single node and each such $\Qc_{(\lf,\lf')}$ by a leaf pair, without affecting the canonicity of $\Gc$. Therefore, we may replace $\Gc$ by the resulting garden $\Hc\in\Gs_{p+1}$ (whose skeleton is also $\Gc_{\mathrm{sk}}$) such that the same (CL) vine $\Vb$ is contained in $\Mb(\Hc)$. Now we look at the possible layerings of $\Vb$ and the product in (\ref{sumofexp0}).

The case of Vine (I) is obvious. For other vines except Vine (V), by the grouping in Definition \ref{cohmol}, we can always find groups $\Gb_i\,(1\leq i\leq r)$ such that $v_2$ is connected to $\Gb_1$, each $\Gb_i$ is connected to $\Gb_{i+1}$, and $\Gb_r$ is connected to $v_1$. Moreover, for each $i$, after removing atoms in $\Gb_i$, then vine becomes disconnected with $v_2$ and $\Gb_j\,(j<i)$ in one component, and $v_1$ and $\Gb_j\,(j>i)$ in another. These $\Gb_i$ exhaust all groups for vines (II)--(IV), while for vines (VI)--(VIII), all the groups other than $\Gb_i$ form exactly the shape which is Vine (V) minus two joints. Let the layer of atoms in $\Gb_i$ be $\Lf_i$. If we start from $v_2$ and each time move from any atom $v$ to $v^+$ where $\nf(v^+)$ is the parent node of $\nf(v)$, just as in the proof of Proposition \ref{layerlad}, then eventually we will arrive at $v_1$ since $\Vb$ is a (CL) vine, due to Proposition \ref{block_clcn}. In this process we must pass through atoms in $\Gb_1,\cdots,\Gb_r$ in this order, due to the above properties of $\Gb_i$, which implies that $q'\leq \Lf_1\leq\cdots\Lf_r\leq q$. This implies that, both the number of choices for $\Lf_i\,(1\leq i\leq r)$, and the product in (\ref{sumofexp0}) with the bond $\ell$ restricted to those connecting two atoms in $\Gb_i$ and $\Gb_j$, are bounded by $C_0^{m+|q-q'|}$.

This already completes the proof of the improved bound for vines (II)--(IV). As for vines (V)--(VIII), we only need to look at the remaining atoms and bonds not in the groups $\Gb_i$, which form a Vine (V) minus two joints. We shall prove that all these atoms must be in the same layer $\Lf$, and that $q\geq \Lf\geq q'$. This clearly allows to control the number of possible layerings, as well as the product in (\ref{sumofexp0}) involving the remaining bonds, by $C_0^{|q-q'|}$.

To prove this last claim, let $\Hb$ be the set of all remaining atoms not in the groups $\Gb_i$, then every atom in $\Hb$ has degree $4$, and there are only two bonds connecting atoms in $\Hb$ to atoms not in $\Hb$, namely two single bonds between atoms $v_1\not\in\Hb$ and $v_2\in\Hb$, and between $v_3\not\in \Hb$ and $v_4\in\Hb$, see Figure \ref{fig:vines}; note also that $v_2$ and $v_4$ belong to the same group with the grouping in Definition \ref{cohmol}, so $\Lf_{v_2}=\Lf_{v_4}$. We only need to prove that $\Lf_{v_2}$ must be between $\Lf_{v_1}$ and $\Lf_{v_3}$, because then we can remove the atoms $(v_2,v_4)$ and apply the same argument to the resulting set $\Hb'$, and proceed with induction.

If we start from any atom $v\in \Hb$ and keep going from $v$ to $v^+$ as above, then eventually we will reach an atom not in $\Hb$, which is either $v_1$ or $v_3$. Therefore, by symmetry we may assume $\mf_1$ is the parent node of $\mf_2$, where $\mf_j=\nf(v_j)$. If $\mf_3$ is also the parent node of $\mf_4$, then the branching nodes in the trees rooted at $\mf_2$ and $\mf_4$ must all correspond to atoms in $\Hb$, and all the leaves in these two trees must be completely paired (otherwise we will have extra bonds between atoms in $\Hb$ and atoms not in $\Hb$), so by Proposition \ref{canonequiv} we get $\Lf_{\mf_2}=\Lf_{\mf_4}=\min(\Lf_{\mf_1},\Lf_{\mf_3})$, which implies what we need.

Finally, if $\mf_3$ is not the parent node of $\mf_4$, then $\mf_2$ must be an ancestor of $\mf_4$. There are now two cases. The first case is that $\mf_4$ is the parent node of $\mf_3$, in which case we clearly have $\Lf_{\mf_1}\geq \Lf_{\mf_2}=\Lf_{\mf_4}\geq \Lf_{\mf_3}$ by the ancestor-descendant relations, as desired. The second case is that $\mf_4$ has a child node $\mf_6$ paired with a child node $\mf_5$ of $\mf_3$ as leaves. In this case, similar arguments as above imply that the branching nodes in the tree rooted at $\mf_2$ must all correspond to atoms in $\Hb$, and all the leaves in this tree except $\mf_6$ must be completely paired. Since also $\mf_6$ is paired with $\mf_5$, by applying Proposition \ref{canonequiv} to $(\mf_2,\mf_5)$ and $(\mf_5,\mf_6)$ we get that $\Lf_{\mf_5}=\Lf_{\mf_6}=\min(\Lf_{\mf_3},\Lf_{\mf_4})$ and $\max(\Lf_{\mf_2},\Lf_{\mf_5})\geq \min(\Lf_{\mf_1},\Lf_{\mf_3})$. Since also $\Lf_{\mf_2}=\Lf_{\mf_4}$, we get that $\Lf_{\mf_2}\geq \min(\Lf_{\mf_1},\Lf_{\mf_3})$, which implies what we need because also $\Lf_{\mf_2}\leq\Lf_{\mf_1}$. This completes the proof.
\end{proof}
\section{Analysis of vines: Role of cancellation and incoherency}\label{cancelvine} 
\subsection{A general estimate for exponential sums} We will state and prove an estimate for general expressions involving exponential sums (and oscillatory integrals), namely Proposition \ref{sumint}; for this we need some number theory lemmas, which are listed below.
\subsubsection{Number theory lemmas}\label{ntlemma} The number theory estimates we need involve summation over $1$ or $3$ vectors, and are similar to those involving $2$ vectors in Lemmas \ref{lemlayer1} and \ref{lemlayer2}.
\begin{lem}\label{lemlin} We use the notation $e(z)=e^{2\pi iz}$ (same with Lemmas \ref{lemcubic0}--\ref{lemcubic} below). Fix $0\neq r\in\Zb_L^d$ with $L^{-1}\leq|r|\leq L^{-\gamma+\eta}$, and recall the cutoff function $\chi_0$ defined in Section \ref{setupnotat}, then uniformly in $(a,\xi)\in(\Rb^d)^2$ and $\tau\in\Rb$, we have
\begin{equation}\label{lemlin2}\int_{[\tau,\tau+\delta L^{2\gamma}]}\bigg|\sum_{x\in \Zb_L^d}\chi_0(x-a)\cdot e(s\langle r,x\rangle+\langle \xi,x\rangle)\bigg|\,\mathrm{d}s\lesssim_0\min(\delta L^{d+2\gamma},L^d|r|^{-1}).
\end{equation}
\end{lem}
\begin{proof} The bound $\delta L^{d+2\gamma}$ is trivial as the sum over $x$ is bounded by $L^d$. For the other bound, we apply Poisson summation to get
\[\sum_{x\in \Zb_L^d}\chi_0(x-a)\cdot e(s\langle r,x\rangle+\langle \xi,x\rangle)=L^d\sum_{g\in\Zb^d}e(\langle sr+\xi-Lg,a\rangle)\cdot\widehat{\chi_0}(sr+\xi-Lg),\] so the left hand side if (\ref{lemlin2}) is bounded by
\begin{equation}\label{lemlin3}\int_{[\tau,\tau+\delta L^{2\gamma}]}\sum_{g\in\Zb^d}\langle sr+\xi-Lg \rangle^{-10d}\,\mathrm{d}s.\end{equation} Assume the first coordinate of $r$ is $r^1$ with $|r^1|\sim_0|r|$; we fix the first coordinate $g^1$, then the result of integrating in $s$ and summing in $g^j\,(j\geq 2)$ is bounded by $|r^1|^{-1}$, and is moreover bounded by $(|g^1|+L)^{-5d}$ for all but one choice of $g^1$ (as $sr^1$ belongs to a fixed interval of length $\delta L^{2\gamma}|r^1|\leq \delta L^{\gamma+\eta}\ll_0 L$). This allows to bound (\ref{lemlin3}) by $|r|^{-1}$, which then proves (\ref{lemlin2}).
\end{proof}
\begin{lem}\label{lemcubic0} For any $|s|,|t|\leq L$, and uniformly in $(a,b,c,\xi,\xi',\xi'')\in(\Rb^d)^6$, we have
\begin{multline}\label{lemcub01}
\sum_{(g,h,q)\in\Zb^{3d}}\bigg|\int_{\Rb^{3d}}\chi_0(x-a)\chi_0(y-b)\chi_0(z-c)\cdot e[\langle Lg+\xi,x\rangle+\langle Lh+\xi',y\rangle+\langle Lq+\xi'',z\rangle]\\\times e(t\langle x,y\rangle+s\Lambda)\,\mathrm{d}x\mathrm{d}y\mathrm{d}z\bigg|\,\lesssim_0 (1+|s|+|t|)^{-d},
\end{multline}
where $\Lambda\in\{\langle x,z\rangle,\langle z,x+y-z\rangle\}$.
\end{lem}
\begin{proof} This is similar to the proof of Lemma \ref{lemlayer1}. For fixed $(a,b,c,\xi,\xi',\xi'')$ and $|t|,|s|\leq L$, we have
\[\nabla_{x,y,z}(t\langle x,y\rangle+s\Lambda)=t\ell_1+s\ell_2+\Oc_0(L),\] where $\ell_j=(\ell_j^1,\ell_j^2,\ell_j^3)$ is a vector that is a fixed linear combination of $(a,b,c)$. The case where
\[\max(|Lg+\xi+t\ell_1^1+s\ell_2^1|,|Lh+\xi'+t\ell_1^2+s\ell_2^2|,|Lq+\xi''+t\ell_1^3+s\ell_2^3|)\geq 10dL\] can be treated as in the proof of Lemma \ref{lemlayer1}, by integrating by parts in one of the variables $u:=x-a$, $v:=y-b$ or $w:=z-c$. If the above maximum if $\leq 10dL$, then for fixed $(t,s)$ there are $\leq C_0$ choices fo $(g,h,q)$. For fixed $(g,h,q)$, we may fix $z$ and apply stationary phase arguments in $(x,y)$, or fix $y$ and apply stationary phase arguments in $(x,z)$, to bound the integral in $(x,y,z)$ by $\min((1+|s|)^{-d},(1+|t|)^{-d})$.
\end{proof}
\begin{lem}\label{lemcubic} Let $\Lambda\in\{\langle x,z\rangle,\langle z,x+y-z\rangle\}$, then uniformly in $(a,b,c,\xi,\xi',\xi'')\in(\Rb^d)^6$, we have
\begin{multline}\label{lemcub2}
\int_{L\leq \max(|t|,|s|)\leq\Df\delta L^{2\gamma}}\bigg|\sum_{x,y,z\in\Zb_L^d}\chi_0(x-a)\chi_0(y-b)\chi_0(z-c)\\\times e\big(t\langle x,y\rangle+s\Lambda+\langle \xi,x\rangle+\langle \xi',y\rangle+\langle \xi'',z\rangle\big)\bigg|\,\mathrm{d}t\mathrm{d}s\lesssim_2 L^{3d+\gamma-2\gamma_0}\log L.
\end{multline}
\end{lem}
\begin{proof} First consider $\Lambda=\langle z,x+y-z\rangle$, then we have $4\langle x,y\rangle=|u|^2-|v|^2$ and $4\Lambda=|u|^2-|w|^2$, where $(u,v,w)=(x+y,x-y,x+y-2z)$, so
\[t\langle x,y\rangle+s\langle z,x+y-z\rangle=\frac{t+s}{4}|u|^2-\frac{t}{4}|v|^2-\frac{s}{4}|w|^2.\] In this case the proof is similar to Lemma 6.6 of \cite{DH23}: by changing to the $(u,v,w)$ variables and applying standard reductions (decomposing various cutoff functions into Fourier integral and suitably shifting $(\xi,\xi',\xi'')$ etc.), we can reduce to
\begin{equation}\label{lemcub3}
\mathrm{LHS\ of\ }(\ref{lemcub2})\lesssim_0 L^{4}\int_{L^{-1}\leq \max(|t'|,|s'|)\leq\Df\delta L^{2\gamma-2}}|G(t')|^d|G(s')|^d|G(t'+s')|^d\,\mathrm{d}t'\mathrm{d}s',
\end{equation} where $(t',s')=L^{-2}\cdot(t,s)$ and
\begin{equation}\label{wyel}G(t')=\sup_{h\in\Zb,\rho\in\Rb}\bigg|\sum_{x\in \Zb}\chi_0\bigg(\frac{x-h}{L}\bigg)e(t'x^2+\rho x)\bigg|\lesssim_0\frac{L}{\sqrt{q}\left(1+L\big|t'-\frac{a}{q}\big|^{1/2}\right)},\end{equation}where the last inequality is due to the Weyl bound, provided that $|a|<q\leq L$, $\gcd(a,q)=1$ and $\big|t'-\frac{a}{q}\big|\leq\frac{1}{qL}$.  By (\ref{lemcub3}) and Young's inequality, we have
\begin{equation}
\mathrm{LHS\ of\ }(\ref{lemcub2})\lesssim_0 L^4\|G\|_{L^{3d/2}([0,1])}^{2d}\cdot\|G\|_{L^{3d/2}([L^{-1},\Df\delta L^{2\gamma-2}])}^d.
\end{equation} Since $3d/2>4$, it is well-known that $\|G\|_{L^{3d/2}([0,1])}^{3d/2}\lesssim_0 L^{3d/2-2}$. To bound the localized norm, since $|t'|\geq L^{-1}$, we must have $a\neq 0$ in the Dirichlet approximation, which implies that $|a|\leq C_2 L^{2\gamma-2}|q|$ and $|q|\geq C_2^{-1}L^{2-2\gamma}$ since also $|t'|\leq \Df\delta L^{2\gamma-2}$. Let $|q|\sim_0 Q\leq L$ and $\big|t'-\frac{a}{q}\big|\sim_0 R$ (or $\big|t'-\frac{a}{q}\big|\lesssim_0 R$ for $R=L^{-2}$) with $L^{-2}\leq R\leq (QL)^{-1}$, then using (\ref{wyel}) we can bound
\begin{equation*}\int_{L^{-1}\leq |t'|\leq \Df\delta L^{2\gamma-2}}|G(t')|^{3d/2}\,\mathrm{d}t'\lesssim_2\sum_{\substack{C_2^{-1}L^{2-2\gamma}\leq Q\leq L\\L^{-2}\leq R\leq (QL)^{-1}}}\sum_{\substack{|q|\leq Q\\|a|\leq C_2L^{2\gamma-2}Q}}(QR)^{-3d/4}R\lesssim_2 L^{3d/2-2+(2\gamma-2)};
\end{equation*} putting together we get
\begin{equation}
\mathrm{LHS\ of\ }(\ref{lemcub2})\lesssim_2 L^4 L^{2d-8/3}\cdot L^{d-4/3-4\gamma_0/3}\lesssim_2 L^{3d-\gamma_0},
\end{equation} which is better than we need.

Now consider $\Lambda=\langle x,z\rangle$. We can apply Poisson summation in $(y,z)$ in (\ref{lemcub2}), to reduce the summation to
\begin{multline}\label{lemcub4}L^{2d}\sum_{x\in\Zb_L^d}\sum_{g,h\in\Zb^d}\chi_0(x-a)\cdot e\big(\langle \xi,x\rangle+\langle tx+\xi'-Lg,b\rangle+\langle sx+\xi''-Lh,c\rangle\big)\\\times\widehat{\chi_0}(tx+\xi'-Lg)\widehat{\chi_0}(sx+\xi''-Lh).
\end{multline} Let $\{z\}$ denote the distance of $z$ to the closest integer. For convenience we will replace $\widehat{\chi_0}$ by a compactly supported function (the proof can be easily adapted), so by (\ref{lemcub4}) we have
\begin{equation}
\mathrm{LHS\ of\ }(\ref{lemcub2})\lesssim_0 L^{2d+4}\int_{L^{-1}\leq\max(|t'|,|s'|)\leq \Df\delta L^{2\gamma-2}}H(t',s')^d\,\mathrm{d}t'\mathrm{d}s',
\end{equation} where $(t',s')=L^{-2}(t,s)$, and
\begin{equation}
H(t',s')=\sup_{r\in\Zb,\sigma_1,\sigma_2\in\Rb}\sum_{x'\in\Zb}\mathbf{1}_{|x'-r|\leq L}\mathbf{1}_{\{t'x'+\sigma_1\}\leq L^{-1}}\mathbf{1}_{\{s'x'+\sigma_2\}\leq L^{-1}},
\end{equation} where $x'=Lx^j$ for some coordinate $x^j$, and correspondingly $r=La^j$ and $(\sigma_1,\sigma_2)=L^{-1}((\xi')^j,(\xi'')^j)$ etc. Assume $|t'|\geq L^{-1}$, clearly $H(t',s')\leq\min(H_0(t'),H_0(s'))$ where
\begin{equation}
H_0(t')=\sup_{r\in\Zb,\sigma\in\Rb}\sum_{x'\in\Zb}\mathbf{1}_{|x'-r|\leq L}\mathbf{1}_{\{t'x'+\sigma\}\leq L^{-1}},
\end{equation} which implies that
\begin{multline}\label{lemcub5}\mathrm{LHS\ of\ }(\ref{lemcub2})\lesssim_0 L^{2d+4}\sum_{A\leq L} A^d\cdot|\{L^{-1}\leq |t'|\leq \Df\delta L^{2\gamma-2}:H_0(t')\geq A\}|\\\times |\{|s'|\leq \Df\delta L^{2\gamma-2}:H_0(s')\geq A\}|:=L^{2d+4}\sum_{A\leq L}A^d\mu_1(A)\mu_2(A),
\end{multline} where $A$ is dyadic. We may assume $A\geq 1$, as the rest of the integral gives $L^{2d+4}\cdot (C_2L^{2\gamma-2})^2$ which is enough for (\ref{lemcub2}). By Cauchy-Schwartz we may replace $\mu_1(A)\mu_2(A)$ by $\mu_1(A)^2$ and $\mu_2(A)^2$ and take geometric average.

Consider $\mu_2(A)^2$; assume $s'$ has Dirichlet approximation $R\sim_0\big|s'-\frac{a}{q}\big|\leq\frac{1}{qL}$ where $|a|<q\sim_0 Q\leq L$ and $L^{-2}\leq R\leq (QL)^{-1}$ (again we assume $\big|s'-\frac{a}{q}\big|\lesssim_0 R$ when $R=L^{-2}$). Then for any $\sigma$, the inequality
\[L^{-1}>\{s'x'+\sigma\}=\bigg\{\frac{ax'}{q}+\sigma+\eta x'\bigg\},\qquad (|\eta|\sim_0 R)\] and the fact that $\gcd(a,q)=1$ imply that $x'$ has less then $4$ choices modulo $q$, and moreover for each choice, $x'$ belongs to a fixed interval of length $(LR)^{-1}$. This shows $H_0(s')\lesssim_0\min(LQ^{-1},(LQR)^{-1})$, while the volume for $s'$ with $(Q,R)$ fixed is bounded by $QR(1+C_2QL^{2\gamma-2})$ upon summing over $(a,q)$ and discussing whether $a=0$ or not. With fixed $A$ such that $H_0(s')\geq A$, by further summing over $(Q,R)$, we get $\mu_2(A)\lesssim_2 L^{-1}A^{-1}\log L+L^{2\gamma-2}A^{-2}$, so by summing over $A\geq 1$ we can bound $(\ref{lemcub5})\lesssim_2 L^{2d+4}[L^{d-4}\log L+L^{4\gamma-4}\max(1,L^{d-4})]$ with $\mu_2(A)^2$ instead of $\mu_1(A)\mu_2(A)$.

The case of $\mu_1(A)^2$ is similar, but now we have $a\neq 0$ due to $|t'|\geq L^{-1}$, so the volume for $t'$ is bounded by $C_2Q^2RL^{2\gamma-2}$, and we have $(\ref{lemcub5})\lesssim_2 L^{2d+4}\cdot L^{4\gamma-4}\max(1,L^{d-4})$ with $\mu_1(A)^2$ instead of $\mu_1(A)\mu_2(A)$.  Putting together using Cauchy-Schwartz, this then proves (\ref{lemcub2}).
\end{proof}
\subsubsection{Exponential sums} We now state the main exponential sum estimate.
\begin{prop}\label{sumint} Consider an expression $\Ic$ defined as follows. It is a function of the time variables $(t,t',t'')$ and vector variables $(e,f,g,h)$ called \emph{output variables}. Assume each of $(t,t',t'')$ belongs to a fixed unit interval of form $[q,q+1]$, where $q$ may be different for different variables (same below), and $t>\max(t',t'')$; denote this set of $(t,t',t'')$ by $\Bc$. Assume each of $(e,f,g,h)$ belongs to a fixed unit ball, and $e-f=h-g:=r\neq 0$ with $|r|\leq L^{-\gamma+\eta}$. 

Assume that $\Ic$ involves \emph{input variables}, including the vector variables $x_0^i,x_j^i,y_j^i\in\Zb_L^d$ and time variables $t_0^i,t_j^i,s_j^i$, as well as \emph{parameters}  including $\lambda_0^i,\lambda_j^i,\mu_j^i$, for $1\leq i\leq I$ and $1\leq j\leq m_i$. Fix $(\mathrm{Nor},\mathrm{Bad})$ as a partition of $\{1,\cdots,I\}$, for $i\in \mathrm{Nor}$, we also include the extra input vector variables $u_1^i,u_2^i,u_3^i\in \Zb_L^d$, time input variables $\tau_1^i,\tau_2^i,\tau_3^i$ and parameters $\sigma_1^i,\sigma_2^i,\sigma_3^i$. Denote ${\textbf{x}}=(x_0^i,x_j^i,y_j^i,u_j^i)$ to be the collection of all input vector variables, and $M$ be the the total number of these variables, which equals the sum of $n_i:=2m_i+1+3\cdot\mathbf{1}_{\mathrm{Nor}}(i)$.

Suppose there exist \emph{alternative variables} $k_j^i,\ell_j^i$ where $1\leq j\leq n_i$ with $n_i$ defined as above, and we write ${\textbf{k}}=(k_j^i)$ and $\boldsymbol{\ell}=(\ell_j^i)$, such that any vector in $(k_j^i,\ell_j^i,e,f,g,h)$ belongs to $\underline{w}+\Zb_L^d$ for a fixed vector $\underline{w}$. Assume also that (i) we have ${\textbf{k}}=T_1{\textbf{x}}+{\textbf{h}}_1$ and $\boldsymbol{\ell}=T_2{\textbf{x}}+{\textbf{h}}_2$ for some constant matrices $T_j$ and some vectors ${\textbf{h}}_j$ depending only on $(e,f,g,h)$, such that all coefficients of $T_1,\,T_1^{-1},\,T_2$ are $0$ or $\pm 1$; (ii) for each $(i,j)$ there exists $(i',j')<(i,j)$ in lexicographic order such that $\ell_{j}^i-\ell_{j'}^{i'}$ is an algebraic sum of at most $\Lambda=\Lambda_j^i$ distinct terms in $(k_\alpha^\beta,e,f,g,h)$, and $\prod_{(i,j)}\Lambda_j^i\leq C_0^M$; (iii) any input variable (such as $x_j^i$ or $u_j^i$ etc.) is the difference of two variables in $(k_\alpha^\beta,\ell_\alpha^\beta,e,f,g,h)$.

The expression $\Ic$ is defined as\begin{align}\label{sumintI1}\Ic=\Ic(t,t',t'',e,f,g,h)&:=\sum_{\textbf{x}}\prod_{i=1}^{I}\prod_{j=1}^{n_i}\Kc_j^i(k_j^i)\cdot\Lc_j^i(\ell_j^i)\cdot \int_{\Dc} \bigg[\prod_{i=1}^I\bigg(\prod_{j=1}^{m_i} e^{\pi i \cdot\delta L^{2\gamma}(t_j^i-s_j^i)\langle x_j^i, y_j^i\rangle}\nonumber\\&\times \prod_{j=0}^{m_i}e^{\pi i\cdot\delta L^{2\gamma}(t_j^i-t)\langle r,\zeta_j^i\rangle}\prod_{j=0}^{m_i}e^{\pi i\lambda_j^i t_j^i}\,\mathrm{d}t_j^i\prod_{j=1}^{m_i} e^{\pi i\mu_j^is_j^i}\,\mathrm{d}s_j^i\bigg)\bigg]
\nonumber\\&\times\bigg[ \prod_{i\in \mathrm{Nor}}\bigg(e^{\pi i\cdot\delta L^{2\gamma}(\tau_1^i-\tau_3^i)\langle u_1^i,u_2^i\rangle}e^{\pi i\delta L^{2\gamma}(\tau_2^i-\tau_3^i)\Lambda_i}e^{\pi i\delta L^{2\gamma}(\tau_3^i-t)\langle r,\xi_i\rangle}\prod_{j=1}^3e^{\pi i\sigma_j^i\tau_j^i}\,\mathrm{d}\tau_j^i\bigg)\bigg]
\end{align} Here in (\ref{sumintI1}), the sum is taken over all input vector variables in $\textbf{x}$, and $\Dc$ is the domain of all input time variables $(t_j^i,s_j^i,\tau_j^i)$ defined by the following conditions: (i) each variable belongs to a fixed unit time interval of form $[q,q+1]\,(q\in\Zb)$, (ii) all variables are \emph{less than $t$}, and $t_0^1>\max(t',t'')$; (iii) any fixed collection of inequalities of form $v<w$ where $v$ and $w$ are variables from $(t_j^i,s_j^i,\tau_j^i)$. For each $1\leq i\leq I$ and $1\leq j\leq n_i$, the functions $\Kc_j^i$ and $\Lc_j^i$ satisfy that 
\begin{equation}
\label{propertyKj}\|\Kc_j^i\|_{\Sf^{30d,30d}}\leq L^{-(\gamma_1-\sqrt{\eta})g_{ij}},\quad \|\Lc_j^i\|_{\Sf^{30d,0}}\leq L^{-(\gamma_1-\sqrt{\eta})g_{ij}'},
\end{equation} 
where $g_{ij},g_{ij}'\geq 0$ are integers. The expressions $\zeta_j^i=a_j^ix_j^i+b_j^iy_j^i+c_j^ir$ for some $a_j^i,b_j^i,c_j^i\in\{0,\pm1\}$ when $j\geq 1$, and $\zeta_0^i=x_0^i$ when $j=0$; moreover $\Lambda_i\in\{\langle u_1^i,u_3^i\rangle,\langle u_3^i,u_1^i+u_2^i-u_3^i\rangle\}$ and $\xi_i=a_iu_1^i+b_iu_2^i+c_iu_3^i+d_ir$ with $a_i,b_i,c_i,d_i\in\{0,\pm1\}$.

In addition, assume for \emph{each} $i\in\mathrm{Bad}$ that one of the followings holds: \begin{enumerate}[{(a)}]
\item The conditions in the definition of $\Dc$ include either $s_1^i<t_0^i<t_1^i$ or $t_1^i<t_0^i<s_1^i$;
\item The summation and integration in (\ref{sumintI1}) contains an extra factor of $|r|$;
\item The summation and integration in (\ref{sumintI1}) contains an extra factor of $|t_1^i-s_1^i|^{1-3\eta}$;
\item For at least one $1\leq j\leq n_i$ we have $\max(g_{ij},g_{ij}')\geq 1$.
\end{enumerate}
Moreover, consider the unit interval that each input time variable belongs to; let $M_0$ be the number of $(i,j)$ such that $t_j^i$ and $s_j^i$ are \emph{not} in the same unit interval, plus the number of $i\in A$ such that $\tau_1^i,\tau_2^i,\tau_3^i$ are \emph{not} in the same unit interval. Also let $M_1$ be the sum of all $g_{ij}$ and $g_{ij}'$ in (\ref{propertyKj}).

Then, uniformly in all the parameters (including $\underline{w}$), in $\Bc$ and in the choice of the unit balls containing $(e,f,g,h)$ described above, we have
\begin{equation}\label{FourierboundI}\|\Ic\|_{\Xf^{2\eta^8,0,0}(\Bc)}\lesssim_1\big[\max_{(i,j)}(\langle\lambda_j^i\rangle,\langle \mu_j^i\rangle,\langle\sigma_j^i\rangle)\big]^{\eta^6}\cdot(C_1\delta^{-1/2})^ML^{M(d-\gamma){-\eta^2(|\mathrm{Bad}|+M_1)-\eta^{8}M_0+\eta^7}},
\end{equation} where the norm $\Xf^{2\eta^8,0,0}(\Bc)$ is defined as in Section \ref{setupnorms}.
\end{prop}
\begin{proof} \uwave{Step 1: Reductions.} Using the $\langle k_j^i\rangle^{-30d}$ decay of $\Kc_j^i(k_j^i)$ in (\ref{propertyKj}), we can localize $k^i_j$ into sets of the form $|k^i_j-\alpha^i_j|\leq 1$ where $\alpha^i_j\in \Zb^d$, and sum over $\alpha^i_j$ at the end. With all these $\alpha^i_j$ fixed, note that there are $O((\Lambda^1_1)^d)$ choices for the integer parts of the coordinates of $\ell_1^1$. Similarly, once we fix all those integer parts, there are $O((\Lambda^1_2)^d)$ choices for the integer parts of the coordinates of $\ell_2^1$, and so on. As a result, at the expense of a multiplicative factor of $(\prod_{(i, j)}\Lambda_j^i)^d\leq C_0^M$, we may fix points $\beta^i_j\in \Zb^d$ so that $|\ell^i_j-\beta^i_j|\leq 1$. Consequently, since each of $(x^i_j,y^i_j,u_j^i)$ is the difference of some $k^i_j$ and $\ell^i_j$, we may fix some $(a^i_j, b^i_j,c_j^i)$ (which may depend on the unit balls containing $(e,f,g,h)$) such that $\max(|x^i_j-a^i_j|,|y^i_j-b^i_j|,|u^i_j-c^i_j|)\leq 1$ . Finally, note that due to the properties of the transformation $T_1$, the change of variables $\boldsymbol{x} \leftrightarrow \boldsymbol k$ is volume preserving.

Now set \[W(\boldsymbol x):=\prod_{i=1}^{I}\prod_{j=1}^{n_i}\Kc_j^i(k_j^i)\cdot\Lc_j^i(\ell_j^i).\] Under the above support restrictions, we then have $
\|\widehat W\|_{L^1}\leq C_1^ML^{-(\gamma_1-\sqrt{\eta})M_1}.
$ In fact, this bound holds if the Fourier transform is taken in the variables $(k_j^i)_{1\leq j \leq n_i}$ by \eqref{propertyKj}, thus it also holds for Fourier transform taken in $\boldsymbol x$, due to the properties of the affine linear transform. Thus, by expanding $W$ in terms of its Fourier transform, we can replace this function by the product of modulation factors $\prod_{i}\left(\prod_j e^{\pi i \nu^i_j \cdot x^i_j} e^{\pi i \rho^i_j \cdot y^i_j}\prod_{j=1}^3e^{\pi i \varpi^i_j u^i_j}\right)$, and only need to prove uniform estimates in $(\nu^i_j, \rho^i_j,\varpi_j^i)$. Note that $10\%$ of the gain $L^{-(\gamma_1-\sqrt{\eta})M_1}$ already provides the decay in (\ref{FourierboundI}) involving $M_1$, and we can save the other $90\%$ for the decay involving bad vines case (d).

\uwave{Step 2: An $L^\infty$ estimate.} We shall first prove the estimate \eqref{FourierboundI} for $\|\Ic\|_{L^\infty}$. In this case, we shall move the time integral over $\Dc$ out of the expression in \eqref{sumintI1} and take absolute values inside this integral. Denote the integrand (without absolute value) by $\Kc$. This gets rid of all the exponential factors involving $\lambda_j^i, \mu_j^i, \sigma_j^i$ and allows to pointwise bound $\Ic$ by
\begin{align}\label{Iboundstep2} 
|\Ic|&\leq \prod_{i=1}^I \int_{[q_0^i, q_0^i+1]}\mathrm{d}t_{0}^i\cdot\bigg|\sum_{x_0^i}e^{\pi i\cdot\delta L^{2\gamma}(t_0^i-t)\langle r,\zeta_0^i\rangle+\pi i (\nu^i_0 \cdot x^i_0)} \chi_0(x_0^i -a_0^i)\bigg| \cdot\Nc^i\cdot  \prod_{j=1}^{m_i}\mathcal M_{j}^i(t_0^i),\\
\mathcal M_{j}^i&:= \int_{\Dc_j^i(t_0^i)} \mathrm{d}t_j^i \, \mathrm{d}s_j^i\cdot\bigg| \sum_{(x_j^i, y_j^i)}\chi_0(x_j^i -a_j^i)\chi_0(y_j^i -b_j^i)e^{\pi i \cdot\delta L^{2\gamma}[(t_j^i-s_j^i)\langle x_j^i, y_j^i\rangle+(t_j^i-t)\langle r,\zeta_j^i\rangle]+\pi i (\nu^i_j \cdot x^i_j+ \cdot \rho^i_j \cdot y^i_j)}\bigg|\nonumber.
\end{align}
Here $\Nc^i=1$ if $i\in \textrm{Bad}$, and if $i\in \mathrm{Nor}$ we define
$$
\Nc^i:=\int_{\widetilde \Dc_i} \prod_{j=1}^3\,\mathrm{d}\tau_j^i \cdot\bigg|\sum_{(u_1, u_2, u_3)}\prod_{k=1}^3 e^{\pi i \varpi^i_k u^i_k}\chi_0(u_k^i-c_k^i) e^{\pi i\cdot\delta L^{2\gamma}[(\tau_1^i-\tau_3^i)\langle u_1^i,u_2^i\rangle+(\tau_2^i-\tau_3^i)\Lambda_i+(\tau_3^i-t)\langle r,\xi_i\rangle]}\bigg|.
$$
In the definition of $\mathcal M_j^i$ above, we denote by $\Dc_j^i(t_0^i)$ the integration domain for $(t_j^i, s_j^i)$ given by the unit interval restrictions of $t_j^i$ and $s_j^i$, and possibly the conditions in (a) if $j=1$ (in particular, we discard all other inequalities involving $(t_j^i, s_j^i)$). Similarly, in the definition of $\Nc^i$, $\widetilde {\Dc}_i$ is the domain of $(\tau_1^i, \tau_2^i, \tau_3^i)$ given by the unit interval restrictions of $(\tau_1^i, \tau_2^i, \tau_3^i)$.

We start with the estimate for $\mathcal M_j^i$. We first change variables $(t_j^i, s_j^i)$ into $(t_j^i, v_j^i)$ where $v_j^i=\delta L^{2\gamma} (t_j^i-s_j^i)$, and then split the integration region into $|v_j^i|\leq L$ and $|v_j^i|>L$. Denoting by $\Mc_j^{i, 1}$ and $\Mc_j^{i, 2}$ the contribution of those two respective regions, we can bound using Lemma \ref{lemlayer2} that
\[
\begin{aligned}
\Mc_j^{i, 2} &\leq (\delta L^{2\gamma})^{-1} \int_{I_j^i} \mathrm{d}t_j^i \int_{|v_j^i|\geq L}\bigg| \sum_{(x_j^i, y_j^i)}\chi(x_j^i -a_j^i)\chi(y_j^i -b_j^i)
\\&\qquad\qquad\qquad\qquad\times e^{\pi i \cdot v_j^i \langle x_j^i, y_j^i\rangle}e^{\pi i\cdot\delta L^{2\gamma}(t_j^i-t)\langle r,\zeta_j^i\rangle+{\pi i (\nu^i_j \cdot x^i_j+\rho^i_j \cdot y^i_j)}}\bigg|\mathrm{d}v_j^i\\
&\lesssim_2\delta^{-1}L^{2(d-\gamma)-2(d-1)(1-\gamma)+\eta},
\end{aligned}
\]
where we integrate trivially in $t_j^i$ over its unit interval domain $I_j^i$.

To bound $\Mc_j^{i,1}$ we apply Poisson summation in $(x_j^i, y_j^i)$ followed by Lemma \ref{lemlayer1} to bound 
$$
\Mc_j^{i,1} \lesssim_0 (\delta L^{2\gamma})^{-1} L^{2d} \int_{I_j^i} \mathrm{d}t_j^i \int_{|v_j^i|\leq L} \langle v_j^i\rangle^{-d}\mathrm{d}v_j^i
\lesssim_2\delta^{-1}L^{2(d-\gamma)},
$$
if we have no restrictions of type (a)--(d) between $t_j^i$ and $v_j^i$ (in the case when $i \in \textrm{Bad}$). If $i \in \textrm{Bad}$ and we are in case (a), we have that $|t_1^i-t_0^i|\leq |t_1^i-s_1^i|=(\delta L^{2\gamma})^{-1}|v_1^i|$, so integrating in $t_1^i$ first would gain a factor of $(\delta L^{2\gamma})^{-1}$ in the above estimate (since $d\geq 3$). The same argument holds if $t_j^i$ and $s_j^i$ belong to different unit intervals, for which we have 
\begin{equation}\label{timeintgap}
|t_j^i-q_j^i| \leq |t_j^i-s_j^i|=(\delta L^{2\gamma})^{-1}|v_j^i|
\end{equation}
where $q_j^i$ is one of the endpoints of the fixed unit interval containing $t_j^i$. If we are in case (c) then we have an extra factor of $|t_1^i-s_1^i|^{1-3\eta}=(\delta L^{2\gamma})^{-(1-3\eta)}|v_1^i|^{1-3\eta}$ which again gains $(\delta L^{2\gamma})^{-(1-3\eta)}$. If we are in case (b) or case (d), then we clearly gain a factor of $|r|$ or $L^{-0.9(\gamma_1-\sqrt\eta)}$ compared to the above estimate (we only exploit $90\%$ of the $L^{-\gamma_1+\sqrt\eta}$ gain). As a consequence of all this, we have
$$
\prod_{j=1}^{m_i}\Mc_j^{i}\lesssim_1 \big(\delta^{-1} L^{2(d-\gamma)}\big)^{m_i} \cdot(L^{-1.8\gamma}+L^{-0.9(\gamma_1-\sqrt\eta)}+|r|)^{\mathbf 1_{\textrm{Bad}(i)}} \cdot\prod_{j=1}^{m_i}L^{-0.1\gamma\mathbf 1_{\Mf_0}(i,j)},
$$
where $\Mf_0$ is the set of $(i, j)$ such that $t_j^i$ and $s_j^i$ belong to different unit intervals. The estimate for $\Nc^i$ follows in exactly the same way as $\Mc^i_j$ but with the variables $w_1^i=(\delta L^{2\gamma})(\tau^i_1-\tau^i_3)$ and $w_2^i=(\delta L^{2\gamma})(\tau^i_2-\tau^i_3)$. We split into two regions depending on whether $\max(|w_1^i|, |w_2^i|)$ is $\leq L$ or $\geq L$, use Lemmas \ref{lemcubic0} and \ref{lemcubic} in place of Lemmas \ref{lemlayer1} and \ref{lemlayer2}, and integrate trivially in $\tau_3^i$. The contribution of the first region is bounded by $(\delta L^{2\gamma})^{-2}L^{3d}$ by Lemma \ref{lemcubic0}, whereas that of the second region is bounded by $(\delta L^{2\gamma})^{-2}L^{3d+\gamma-2\gamma_0}\log L$ by Lemma \ref{lemcubic}. Note that if one of the $\tau_k^i$ belongs to a different unit interval than the others then we get that $|\tau_3^i-q^i|\leq (\delta L^{2\gamma})^{-1}\max(|w_1^i|, |w_2^i|)$ which allows to gain an extra factor of $(\delta L^{2\gamma})^{-0.9}$. As such, we have that
\begin{equation}\label{Nibound}
\Nc^{i}\lesssim_1 \delta^{-2} L^{3(d-\gamma)} L^{-\gamma_0}\cdot L^{-0.1\gamma_01_{\Mf_1}(i)},
\end{equation}
where $\Mf_1$ is the set of $i \in \mathrm{Nor}$ such that $\tau_1^i, \tau_2^i, \tau_3^i$ do not belong to the same unit interval. Finally, using the above two estimates for $\Mc_j^{i}$ and $\Nc^i$ (which are uniform in $t_0^i$), we can go back to \eqref{Iboundstep2}, and use Lemma \ref{lemlin} (after rescaling $t_0^i\to \delta L^{2\gamma}t_0^i$) to obtain the uniform bound
\begin{align}
|\Ic|\leq& \prod_{i\in \mathrm{Nor}} C_1^{m_i} \delta^{-(m_i+2)}L^{(2m_i+4)(d-\gamma)}L^{-0.1\gamma_0M_0^i}\prod_{i\in \mathrm{Bad}}C_1^{m_i} \delta^{-(m_i+1)}L^{(2m_i+1)(d-\gamma)} L^{-0.1\gamma_0(1+M_0^i)}\nonumber\\
\label{IL^1bound}\leq& C_1^M (\delta^{-1/2})^M L^{M(d-\gamma)}L^{-0.1\gamma_0(|\mathrm{Bad}|+M_0)},
\end{align}
where we denote by $M_0^i$ the number of $j$ such that $t_j^i$ and $s_j^i$ are not in the same unit interval (plus 1 if $i \in \mathrm{Nor}$ and $(\tau_1^i, \tau_2^i, \tau_3^i)$ are not in the same unit interval). Note that in (\ref{IL^1bound}) when $i\in\mathrm{Nor}$, we use the first bound in Lemma \ref{lemlin} if $\gamma\leq 1/2$ and the second bound if $\gamma>1/2$ (when $\gamma>1/2$, the power $L^{-\gamma_0}$ on the right hand side of (\ref{Nibound}) can be replaced by $L^{-\gamma_0}\cdot L^{-\min(\gamma-\gamma_0,0.1\gamma_0)}$ using Lemmas \ref{lemcubic0} and \ref{lemcubic}). Also the decay factor involving $M_1$ in (\ref{FourierboundI}) is already taken care of before (same comment applies in Step 3 below).

\medskip
 
\uwave{Step 3: The $\Xf^{2\eta^8,0,0}$ estimate.} 
Now we prove the needed estimate for $\Ic$ in $\Xf^{2\eta^8, 0, 0}$. First define $ \Ic^*(t, t', t'')=\Ic(t,t'+t, t''+t)$, then we only need to consider $\Ic^*$ because the $\Xf^{2\eta^8,0,0}$ norm for both functions are comparable. Now we have
$$
 \Ic^*(t, t', t'')=e^{\pi i \theta t}\int_{{\Dc^*}} \Kc(\boldsymbol t) \mathrm{d}\boldsymbol t 
$$
for some integrand $\Kc$, where $\theta$ is an algebraic sum of the $(\lambda_j^i, \mu_j^i,\sigma_j^i)$ parameters, and we denote by $\boldsymbol t$ the vector of all the $(t_j^i, s_j^i, \tau_1^i, \tau_2^i, \tau_3^i)$ for $1\leq i \leq I, 1\leq j \leq n_i$. The domain of integration $ \Dc^*$ is the translation by $t$ of the domain $\Dc$. Set $\beta=\eta^{7.5}$. Since $t, t', t''$ each belongs to a unit interval, we can use a partition of unity to write $\Ic^*$ as a sum of $L^{3\beta}$ terms where each term is supported on an interval of size $L^{-\beta}$ in each of the variables $t, t',t''$. As such, up to paying a factor of $L^{3\beta}$, we shall fix $t, t',t''$ to  fixed intervals of size $L^{-\beta}$ and only estimate the $\Xf^{2\eta^8,0,0}$ of the resulting function. Since each cutoff factor $\chi_0(L^\beta t-c)$ amplifies the $\Xf^{2\eta^8,0,0}$ norm by at most $L^\beta$, by paying another $L^{3\beta}$ factor, we can also remove these cutoff functions if necessary.

Note that after the translation, the integrand $\Kc$ has no dependence on $(t, t', t'')$. The domain $\Dc^*$ does depend on $(t, t', t'')$, but we can find a larger domain ${\Qc^*}$ independent of $(t, t',t'')$ (obtained by enlarging the unit interval restrictions for $(t_j^i, s_j^i, \tau_1^i, \tau_2^i, \tau_3^i)$ into intervals of size $1+2L^{-\beta}$ using the fixed $L^{-\beta}$ interval restrictions of $t,t',t''$). We may then assume $\Kc$ is supported in $\Qc^*$, and we can repeat essentially the same argument as in Step 2 to obtain the weaker bound 
\begin{equation}\label{Kcmodifiedarg}
\| \Kc \|_{L_{(e,f, g, h)}^\infty L_{\boldsymbol t}^1( \Qc^*)} \leq C_1^M (\delta^{-1/2})^M L^{M(d-\gamma)}L^{-0.1\gamma_0|\mathrm{Bad}|-\beta M_0+6\beta}.
\end{equation}
The only modification needed for the argument in Step 2 is when dealing with the case where $t_j^i$ and $s_j^i$ belong to different unit intervals before translating by $t$, because now the enlarged intervals of length $(1+2L^{-\beta})$ in $\Qc^*$ might intersect. However, in this case, we simply replace \eqref{timeintgap} with
$$
|t_j^i-q_j^i| \leq2L^{-\beta}+ |t_1^i-s_1^i|
$$
for some other fixed value $q_j^i$, which gives a gain of $C_0L^{-\beta}$ for each such incoherency, thus explaining the bound \eqref{Kcmodifiedarg}. 

Taking the Fourier transforms of $\Ic^*$ in $(t, t',t'')$ and omitting the time cutoff functions as said above, we obtain that
\begin{equation}\label{FouriertildeI}
\Fc\Ic^*(\xi, \xi', \xi'', e,f,g,h)=\int_{\Qc^*}\Kc(\boldsymbol t)\,\mathrm{d}\boldsymbol t \int_{\Rb^3}\mathbf 1_{ \Dc^*}(t,t',t'',\boldsymbol t) e^{\pi i \left((\xi+\theta) t+\xi't'+\xi''t''\right)}\cdot\psi\,\mathrm dt\mathrm dt'\mathrm dt'',
\end{equation}
where $\psi=\psi(t,t',t'')$ is some other suitable unit support cutoff function. It is clear that, due to the definition of $\Dc^*$, the above integral can be written as the product of three integrals in $t$, $t'$ and $t''$ respectively, and each integral is taken on an interval depending on $\boldsymbol t$. Thus the second integral above is bounded (uniformly in $\boldsymbol t$) by $\langle \xi +\theta\rangle^{-1}\langle \xi' \rangle^{-1}\langle \xi'' \rangle^{-1}$ (here we may assume $|\xi+\theta|\geq 1$ etc., since the other cases are easily dealt with), which implies that 
$$
\|\Ic^*\|_{\Xf^{-\eta^8,0,0}}\lesssim_1 \| \Kc \|_{L_{(e,f, g, h)}^\infty L_{\boldsymbol t}^1( \Qc^*)}\lesssim_1 C_1^M (\delta^{-1/2})^M L^{M(d-\gamma)}L^{-0.1\gamma_0|\mathrm{Bad}|-\beta M_0+6\beta}.
$$

We shall interpolate the above estimate with a lossy one for the $\Xf^{1/4,0,0}$ norm to obtain the result. For this, note by definition of $\Dc^*$ that, the $(t,t', t'')$ integral in \eqref{FouriertildeI} can essentially be written as a linear combination (up to unimodular factors) of
\begin{equation}\label{FourierofD*}
\langle \xi+\theta\rangle^{-1}\langle \xi'\rangle^{-1}\langle\xi''\rangle^{-1}e^{\pi i \left((\xi+\theta) t_1+\xi't_2+\xi'' t_3\right)}
\end{equation}
for some variables $t_1, t_2, t_3$ from the vector $\boldsymbol t$ which may or may not not be distinct; note that the exact choice of $(t_1,t_2,t_3)$ may depend on inequality relations between variables in $\boldsymbol t$, but we may fix such a relation at the cost of another $C_0^M$ factor. First consider the case when the three variables $t_1, t_2, t_3$ are distinct. In this case, we can write $\Fc \Ic^*$ essentially as a linear combination of 
$$
\langle \xi+\theta\rangle^{-1}\langle \xi'\rangle^{-1}\langle\xi''\rangle^{-1}\cdot\Fc_{(t_1, t_2, t_3)}\bigg(\int \mathbf 1_{\Qc^*}\Kc (\cdot, \boldsymbol{t'}) \mathrm{d}\boldsymbol{t'}\bigg)(\xi+\theta, \xi', \xi'')
$$
where $\boldsymbol{t'}$ are the variables in $\boldsymbol t$ other than $(t_1, t_2, t_3)$. In this case, we can estimate $\|\Ic^*\|_{\Xf^{1/4,0,0}}$ by first bounding $L^\infty_{(e, f, g, h)}$ by $L_{(e, f, g, h)}^1$, then interchanging the $L^1$ norms in $(\xi, \xi', \xi'')$ and $(e,f, g, h)$, and finally bounding $L_{(e, f, g ,h)}^1$ by $L^\infty_{(e, f, g, h)}$ with a loss of $L^{4d}$ (recall that there only $L^{4d}$ choices of $(e,f, g,h)$ with fixed $\underline{w}$ and fixed unit balls). This allows us to fix the values of $(e, f, g, h)$ and compute the weighted $L^1$ norm in $(\xi, \xi', \xi'')$. This reduces to estimating a \emph{scalar} function for which Plancherel is available. Now we apply Cauchy Schwartz in $(\xi, \xi', \xi'')$ followed by Plancherel to get that
\begin{equation}\label{Xpositivebound}
\|\Ic^*\|_{\Xf^{1/4,0,0}}\leq \langle \theta\rangle^{1/4}\cdot C_0^M L^{4d}\sup_{e, f, g, h}\bigg\|\int \mathbf 1_{\Qc^*}\Kc (\boldsymbol{t})\mathrm{d}\boldsymbol{t'}\bigg\|_{L^2_{(t_1, t_2, t_3)}}.
\end{equation}
The latter can be estimated by fixing $(t_1, t_2, t_3)$ and estimating the integral in $\boldsymbol{t'}$ as in Step 2. The effect of fixing $(t_1, t_2, t_3)$ is to replace at most three bounds on $\Mc_j^i$ or $\Nc^i$ by the trivial bound $L^{2d}$ or $L^{3d}$. So up to an extra factor of $L^{14d}$, we have the same bound for \eqref{Xpositivebound} as \eqref{IL^1bound}. Similarly, if in \eqref{FourierofD*}, two of the variables are the same, say $t_1\equiv t_2$, then $\Fc\Ic^*$ can essentially be bounded by 
$$
\langle \xi+\theta\rangle^{-1}\langle \xi'\rangle^{-1}\langle\xi''\rangle^{-1}\Fc_{(t_1, t_3)} \bigg(\int \mathbf 1_{\Qc^*}\Kc (\cdot, \boldsymbol{t'}) \mathrm{d}\boldsymbol{t'}\bigg)(\xi+\theta+ \xi', \xi''),
$$
where $\boldsymbol{t'}$ are the variables in $\boldsymbol t$ other than $(t_1,t_3)$. We argue in the same way as above, except that we replace the use of Cauchy-Schwarz in $(\xi, \xi', \xi'')$ above by a combination of Cauchy-Schwarz in $(\xi+\xi',\xi'')$ and Young's inequality. The case $t_1\equiv t_2\equiv t_3$ is again simialr. As a result, we can estimate $\|\Ic^*\|_{\Xf^{1/4,0,0}}$ by $L^{14d}\langle \theta\rangle^{1/4}$ times the right hand side of \eqref{Kcmodifiedarg}, which implies (\ref{FourierboundI}) upon interpolation (recall that $\theta$ is an algebraic sum of the $(\lambda_j^i, \mu_j^i,\sigma_j^i)$ parameters).
\end{proof}
\subsection{Application to vine chains} We now apply Proposition \ref{sumint} to expressions associated with vines in molecules coming from gardens. First we discuss the reduction of $\Kc_\Gc$ (Definition \ref{defkg}, (\ref{defkg0})) after collapsing some of the embedded regular couples and regular trees in $\Gc$.
\subsubsection{Reduction and twisting of $\Kc_\Gc$}\label{regred} Let $\Gc$ be a garden, and suppose it is obtained from its skeleton $\Gc_{\mathrm{sk}}$ by replacing each leaf pair $(\lf,\lf')$ with a regular couple $\Qc_{(\lf,\lf')}$, and each branching node $\mf$ with a regular tree $\Tc_{(\mf)}$, as in Definition \ref{lftwist}. Assume we have fixed a canonical layering of $\Gc$, which induces a layering of each $\Qc_{(\lf,\lf')}$ and $\Tc_{(\mf)}$, as well as a pre-layering of $\Gc_{\mathrm{sk}}$. Then, by Definition \ref{defkg} and (\ref{defkg0}), it is easy to verify that
\begin{multline}\label{redkg}
\Kc_\Gc(p+1,k_1,\cdots,k_{2R})=\bigg(\frac{\delta}{2L^{d-\gamma}}\bigg)^{n_{\mathrm{sk}}}\zeta(\Gc_{\mathrm{sk}})\sum_{\Is_{\mathrm{sk}}}\epsilon_{\Is_{\mathrm{sk}}}\int_{\Ic_{\mathrm{sk}}}\prod_{\nf\in\Nc_{\mathrm{sk}}}e^{\pi i\zeta_\nf\cdot\delta L^{2\gamma}\Omega_\nf t_\nf}\,\mathrm{d}t_\nf\\\times\prod_{\lf\in\Lc_{\mathrm{sk}}}^{(+)}\Kc_{\Qc_{(\lf,\lf')}}^*(t_{\lf^{\mathrm{pr}}},t_{(\lf')^{\mathrm{pr}}},k_\lf)\prod_{\mf\in\Nc_{\mathrm{sk}}}\Kc_{\Tc_{(\mf)}}^*(t_{\mf^{\mathrm{pr}}},t_\mf,k_\mf).
\end{multline}
Here $n_{\mathrm{sk}}$ is the order of $\Gc_{\mathrm{sk}}$, $\Ic_{\mathrm{sk}}$ is the domain defined in (\ref{timegarden}) for $\Gc_{\mathrm{sk}}$ with $t$ replaced by $p+1$, $\Is_{\mathrm{sk}}$ is a $(k_1,\cdots,k_{2R})$-decoration of $\Gc_{\mathrm{sk}}$, and the other objects are as before but for the garden $\Gc_{\mathrm{sk}}$; the product $\prod_{\lf\in \Lc_{\mathrm{sk}}}^{(+)}$ is taken over all leaf pairs $(\lf,\lf')$ where $\lf$ has sign $+$, and $\Kc_{\Qc_{(\lf,\lf')}}^*$ and $\Kc_{\Tc_{(\mf)}}^*$ are defined in Definition \ref{defkg} (also, if $\mf$ is the root of a tree then we replace $t_{\mf^{\mathrm{pr}}}$ by $p+1$; this will be assumed throughout).

Now suppose $\Ub$ is a \emph{(CL) vine chain} in the molecule $\Mb(\Gc_{\mathrm{sk}})$, and let $\Gc'$ be an LF twist of $\Gc$ at $\Ub$ (viewed as a collection of (CL) vines, i.e. $\Gc'$ is formed from $\Gc$ by performing some unit LF twist at each vine in $\Ub$ that is core and satisfies the assumptions in Definition \ref{lftwist}). If we write both $\Kc_\Gc$ and $\Kc_{\Gc'}$ in the form of (\ref{redkg}), then the differences between these expressions \emph{only occur} in the part of (\ref{redkg}) \emph{involving the realization} $\Gc_{\mathrm{sk}}[\Ub]$ of $\Ub$ in $\Gc_{\mathrm{sk}}$ (see Remark \ref{rem_realiz} for definition).

More precisely, let the notions $\uf_1,\uf_2,\uf_{11},\uf_{21},\uf_{22}$ etc. be associated to $\Ub$ as in Proposition \ref{block_clcn}. Let $n_{(1)}$ be the number of branching nodes in $\Gc_{\mathrm{sk}}[\Ub]\backslash\{\uf_1\}$, and let $n_{(2)}$ be $n_{(1)}$ plus the total order of regular couples $\Qc_{(\lf,\lf')}$ and regular trees $\Tc_{(\mf)}$ for $\lf,\lf',\mf\in\Gc_{\mathrm{sk}}[\Ub]\backslash\{\uf_1\}$. We may then define the expression
\[\Kc_{(\Gc)}^{\Ub}=\Kc_{(\Gc)}^{\Ub}(t_{\uf_1},t_{\uf_{21}},t_{\uf_{22}},k_{\uf_1},k_{\uf_{11}},k_{\uf_{21}},k_{\uf_{22}})\] similar to (\ref{redkg}), but with the following changes:
\begin{itemize}
\item We replace the power $(\delta/(2L^{d-\gamma}))^{n_{\mathrm{sk}}}$ by $(\delta/(2L^{d-\gamma}))^{n_{(1)}}$, and the factor $\zeta(\Gc_{\mathrm{sk}})$ by the product of $i\zeta_\nf$ where $\nf$ runs over all branching nodes in $\Gc_{\mathrm{sk}}[\Ub]\backslash\{\uf_1\}$.
\item In the summation $\sum_{\Is_{\mathrm{sk}}}(\cdots)$, we only sum over the variables $k_\nf$ for $\nf\in\Gc_{\mathrm{sk}}[\Ub]\backslash\{\uf_1\}$ (including branching nodes and leaves), and treat the other $k_\nf$ variables as fixed. We also replace $\epsilon_{\Is_{\mathrm{sk}}}$ by the product of factors on the right hand side of (\ref{defepscoef}), but only for $\nf\in \Gc_{\mathrm{sk}}[\Ub]$.
\item In the integral $\int_{\Ic_{\mathrm{sk}}}(\cdots)$, we only integrate over the variables $t_\nf$ for all branching nodes $\nf\in\Gc_{\mathrm{sk}}[\Ub]\backslash\{\uf_1\}$, and treat the other $t_\nf$ variables as fixed.
\item In the first product $\prod_{\nf\in\Nc_{\mathrm{sk}}}(\cdots)$, we only include those factors where $\nf\in\Gc_{\mathrm{sk}}[\Ub]$; in the products $\prod_{\lf\in\Lc_{\mathrm{sk}}}^{(+)}(\cdots)$ and $\prod_{\mf\in\Nc_{\mathrm{sk}}}(\cdots)$, we only include those factors where $\lf,\mf\in\Gc_{\mathrm{sk}}[\Ub]\backslash\{\uf_1\}$.
\end{itemize}
It is easy to check that the function $\Kc_{(\Gc)}^{\Ub}$ defined this way indeed only depends on the time variables $t_{\uf_1},t_{\uf_{21}},t_{\uf_{22}}$ and vector variables $k_{\uf_1},k_{\uf_{11}},k_{\uf_{21}},k_{\uf_{22}}$.
\subsubsection{(CL) vine chains} We now prove the main result for (CL) vine chains, which exhibits the key cancellation between $\Kc_\Gc$ and $\Kc_{\Gc'}$ for LF twists $\Gc'$ of $\Gc$.
\begin{prop}\label{vineest} In the same setting as in Section \ref{regred} including all notations, let $\Ub$ be a (CL) vine chain in the molecule $\Mb(\Gc_{\mathrm{sk}})$. Define 
\begin{equation}\label{vineest1}
\widetilde{\Kc}^\Ub(t_{\uf_1},t_{\uf_{21}},t_{\uf_{22}},k_{\uf_1},k_{\uf_{11}},k_{\uf_{21}},k_{\uf_{22}}):=e^{-\pi i\cdot\delta L^{2\gamma}t_{\uf_1}\Gamma}\sum_{\Gc'}\Kc_{(\Gc')}^{\Ub}(t_{\uf_1},t_{\uf_{21}},t_{\uf_{22}},k_{\uf_1},k_{\uf_{11}},k_{\uf_{21}},k_{\uf_{22}}),
\end{equation} where the summation is taken over all the $\Gc'$ which are LF twists of $\Gc$ at $\Ub$, and $\Kc_{(\Gc')}^{\Ub}$ is defined as in Section \ref{regred}, and $\Gamma$ is defined as $\Gamma:=\zeta_{\uf_{11}}|k_{u_{11}}|^2+\zeta_{\uf_{21}}|k_{u_{21}}|^2+\zeta_{\uf_{22}}|k_{u_{22}}|^2-\zeta_{\uf_{1}}|k_{u_{1}}|^2$. Same as Proposition \ref{sumint}, we also fix $\underline{w}$ such that $k_\nf-\underline{w}\in\Zb_L^d$ for each $\nf$.

Let $\Hc_{\uf_{1}},\cdots,\Hc_{\uf_{22}}$ be unit balls, and $\Bc$ be the intersection of the set $\{t_{\uf_1}>\max(t_{\uf_{21}},t_{\uf_{22}})\}$ with the product of three unit intervals, where the unit intervals are $t_{\uf_1}\in[\Lf_{\uf_1},\Lf_{\uf_1}+1]$ and the same for $t_{\uf_{21}}$ and $t_{\uf_{22}}$, with $\Lf_{\uf_1}$ being the layer of $\uf_1$ etc. Then, uniformly in $\Hc_{\uf_{1}},\cdots,\Hc_{\uf_{22}}$ and $\Bc$, we have
\begin{equation}\label{vineest2}
\big\|\mathbf{1}_{\Hc_{\uf_1}}(k_{\uf_1})\cdots \mathbf{1}_{\Hc_{\uf_{22}}}(k_{\uf_{22}})\cdot \widetilde{\Kc}^\Ub\big\|_{\Xf^{2\eta^8,0,0}(\Bc)}\lesssim_1 (C_1\sqrt{\delta})^{n_{(2)}}L^{-\eta^{8}n_{(3)}-\eta^2b+\eta^7}.
\end{equation} Here the norm $\Xf^{2\eta^8,0,0}(\Bc)$ is defined as in Section \ref{setupnorms}, $n_{(1)}$ and $n_{(2)}$ are defined in the same way as in Section \ref{regred}, $b$ denotes the number of \emph{bad} vines in $\Ub$, and $n_{(3)}$ denotes the sum of all incoherency indices for all the vines in $\Ub$ (see Definition \ref{cohmol}) and all the regular couples $\Qc_{(\lf,\lf')}$ and regular trees $\Tc_{(\mf)}$ for $\lf,\lf',\mf\in\Gc_{\mathrm{sk}}[\Ub]\backslash\{\uf_1\}$ (see Definition \ref{defcoh}).
\end{prop}
\begin{proof} 
The main objective of the proof is to show that \eqref{vineest1} can be reduced to an expression in the form \eqref{sumintI1}, so that the bound (\ref{vineest2}) follows from \eqref{FourierboundI}. This involves a parametrization of the vine chain, which was done for a single vine in \cite{DH23}, Proposition 7.5 (for bad (CL) vine) and Proposition 7.6 (for normal (CL) vine). The proof here is similar to \cite{DH23} with a few additions associated with vine chains.

\uwave{Step 1: Notations and reductions.} Start by considering the expression for $\widetilde \Kc_{(\Gc)}^{\Ub}:=e^{-\pi i\cdot\delta L^{2\gamma} t_{\uf_1}\Gamma}\cdot\Kc_{(\Gc)}^{\Ub}$ which is given by
\begin{multline*}
\widetilde \Kc_{(\Gc)}^{\Ub}=\bigg(\frac{\delta}{2L^{d-\gamma}}\bigg)^{n_{(1)}}\zeta[\Ub] \sum_{\Is [\Ub]}e^{-\pi i\cdot\delta L^{2\gamma} t_{\uf_1}\Gamma} \cdot\epsilon_{\Is[\Ub]}\cdot\int_{ \Ic[\Ub]}  \prod_{\nf\in \Nc[\Ub]} e^{\zeta_\nf\pi i\cdot\delta L^{2\gamma}\Omega_\nf t_\nf}\,\mathrm{d}t_\nf\\
\times\,{\prod_{\lf\in
\Lc[\Ub]}^{(+)}\Kc_{\Qc_{(\lf,\lf')}}^*(t_{\lf^{\mathrm{pr}}},t_{(\lf')^{\mathrm{pr}}},k_\lf)}\prod_{\mf\in\Nc[\Ub]\backslash\{\uf_1\}}\Kc_{\Tc_{(\mf)}}^*(t_{\mf^{\mathrm{pr}}},t_{\mf},k_\mf).
\end{multline*}
Here $\zeta [\Ub]$ is the product of $i\zeta_\nf$ where $\nf$ runs over all branching nodes in $\Gc_{\mathrm{sk}}[\Ub]\backslash\{\uf_1\}$, as in the definition of $\Kc_{(\Gc)}^{\Ub}$ in Section \ref{regred} above; similarly $\Is[\Ub]$,  $\epsilon_{ \Is[\Ub]}$ and $\Ic[\Ub]$ are also defined as above, and $\Nc[\Ub]$ and $\Lc[\Ub]$ are respectively the sets of branching nodes and leaves in $\Gc_{\mathrm{sk}}[\Ub]$. Note that $$\delta L^{2\gamma}\cdot \Gamma=\delta L^{2\gamma}\sum_{\nf\in\Nc[\Ub]} \zeta_\nf \Omega_\nf,
$$
since each factor $\zeta_\nf |k_{\nf}|^2$ with $\nf\in\Gc_{\mathrm{sk}}[\Ub] \setminus\{\uf_1\}$ appears twice in the above sum with opposite signs. We also define $(t,t',t''):=(t_{\uf_1},t_{\uf_{21}},t_{\uf_{22}})$ and $(e,f,g,h)=(k_{\uf_1},k_{\uf_{11}},k_{\uf_{21}},k_{\uf_{22}})$ \emph{but possibly with $e$ and $f$ switched} depending on the signs $\zeta_{\uf_{1}}=\zeta_{\uf_{11}}$. Thus, we can rewrite 
\begin{multline}\label{defKtilde}
\widetilde \Kc_{(\Gc)}^{\Ub}= \bigg(\frac{\delta}{2L^{d-\gamma}}\bigg)^{n_{(1)}}\zeta[\Ub] \sum_{\Is [\Ub]} \epsilon_{\Is[\Ub]}\cdot\int_{ \Ic[\Ub]}  \prod_{\nf\in \Nc[\Ub]} e^{\zeta_\nf\pi i\cdot\delta L^{2\gamma}\Omega_\nf (t_\nf-t)}\,\mathrm{d}t_\nf\\
\times\,{\prod_{\lf\in
\Lc[\Ub]}^{(+)}\Kc_{\Qc_{(\lf,\lf')}}^*(t_{\lf^{\mathrm{pr}}},t_{(\lf')^{\mathrm{pr}}},k_\lf)}\prod_{\mf\in\Nc[\Ub]\backslash\{\uf_1\}}\Kc_{\Tc_{(\mf)}}^*(t_{\mf^{\mathrm{pr}}},t_{\mf},k_\mf).
\end{multline}
We shall number the (CL) vine ingredients of $\Ub$ as $\Ub_i\,(1\leq i\leq I)$ and associate to each $\Ub^i\,(1\leq i\leq I)$ the nodes $\uf_1^i, \uf_{2}^i, \uf_{21}^i$ etc. as in  Proposition \ref{block_clcn}. We order these $\Ub^i$ from bottom to top, i.e. assume that $\uf_1^i=\uf_2^{i+1}$; from this we also have that $\uf_{11}^i=\uf_{23}^{i+1}$, and $(\uf_1, \uf_{11})=(\uf_1^I, \uf_{11}^I)$ and $(\uf_2, \uf_{21}, \uf_{22})=(\uf_2^1, \uf_{21}^1, \uf_{22}^1)$ etc. Also define the two subsets $\mathrm{Bad}$ and $\mathrm{Nor}$ of $\{1, \cdots I\}$ corresponding to bad and normal vines $\Ub_i$. At the level of the molecule, we denote by $v_1^i$ and $v_{2}^i$ the joints of the vine $\Ub^i$ for $1\leq i \leq I$ such that $v_j^i$ corresponds to the branching node $\uf_j^i$, so we have $v_1^i=v_2^{i+1}$ for $1\leq i \leq I-1$. Define $r:=g-h=k_{\uf_{21}}-k_{\uf_{22}}$, we may also assume $r=e-f=k_{\uf_{1}}-k_{\uf_{11}}$ (the opposite sign case being the same), then for each $1\leq i \leq I$ we have that
\[r=k_{\uf_{21}^i}-k_{\uf_{22}^i}=\pm (k_{\uf_1^i}-k_{\uf_{11}^i})\neq 0.\]

\uwave{Step 2: Re-parametrization of decorations.} We now perform re-parametrizations for decorations $\Is[\Ub]$. This will be done for each $\Ub^i$ individually, and depends on the type (bad or normal) of $\Ub^i$.

\smallskip
(A) Suppose $i\in\mathrm{Bad}$ and $\Ub^i$ is Vine (II), then it can be annotated as in Figure \ref{fig:vineAnnotated}. Here we denote by $(v_1^i, v_2^i)$ the joints and $(v_{2j+1}^i, v_{2j+2}^i)\,(1\leq j\leq m_i)$ the interior atoms such that $v^i_{2j+1}$ is connected by a double bond to $v^i_{2j+2}$. Then $\Gc_{\mathrm{sk}}[\Ub^i]\backslash\{\uf_1^i\}$ contains $n_i:=2m_i+1$ branching nodes (which we denote by $\uf_j^i\,(2\leq j\leq 2m_i+2)$ corresponding to $v_j^i$), and $2m_i+1$ leaf pairs.

For any decoration, denote the values of $k_\lf$ for leaves $\lf\in \Gc_{\mathrm{sk}}[\Ub^i]\backslash\{\uf_1^i\}$ by $(k_j^i: 1\leq j \leq 2m_i+1)$, and denote the values of $k_\nf$ for branching nodes $\nf\in\Gc_{\mathrm{sk}}[\Ub^i]\backslash\{\uf_1^i\}$ by $(\ell_j^i: 1\leq j \leq 2m_i+1)$. This decoration leads to decoration of bonds of $\Ub^i$ as in Definition \ref{defdecmol}; for $1\leq j\leq m_i-1$ we we denote by $(a_j^i, b_j^i)$ the decoration of the double bond between $v^i_{2j+1}$ and $v^i_{2j+2}$, and denote by $(c_j^i,d_j^i)$ the decoration of the single bonds between $\{v_{2j+1}^i,v_{2j+2}^i\}$ and $\{v_{2j+3}^i,v_{2j+4}^i\}$. For $j\in\{0,m_i\}$ we make obvious modifications as shown in Figure \ref{fig:vineAnnotated} (A). Finally, we also define $\iota^i_j=+1$ if bond decorated by $c^i_j$ is outgoing from $v^i_{2j+1}$ for $1\leq j \leq m_i$ or from $v^i_2$ for $j=0$, and $\iota^i_j=-1$ otherwise. Note that specifying $\iota^i_j$ completely determines the directions of all the bonds in $\Ub^i$ except for the vertical double bonds in Figure \ref{fig:vineAnnotated} (A).
\begin{figure}[h!]
  \includegraphics[scale=.4]{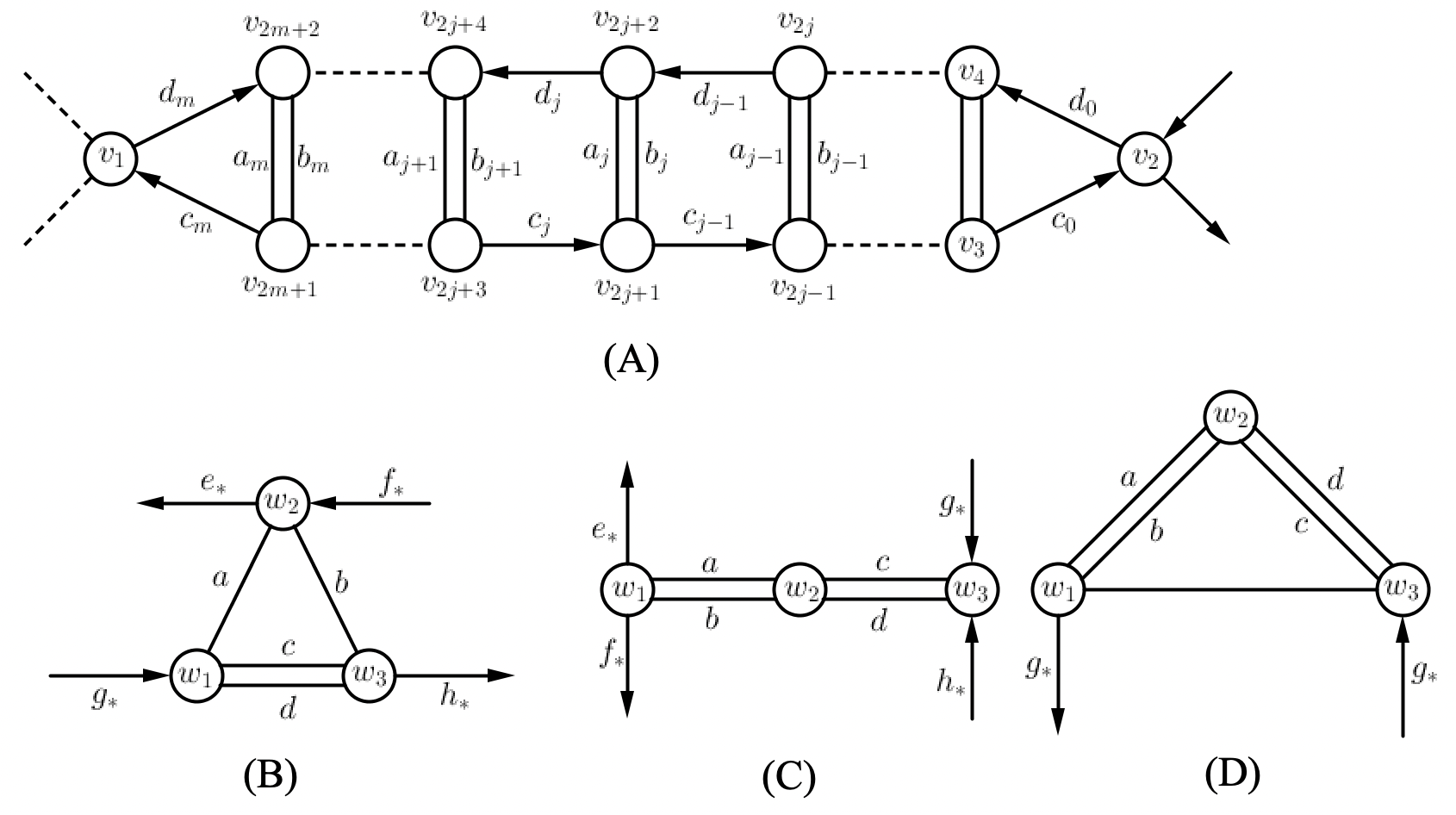}
  \caption{(A) is the annotation of the decoration of a Vine (II), whereas (B)--(D) are the annotations of parts of Vine (III), (IV) and (V)--(VIII) respectively as referred to in Step 2 of the proof of Proposition \ref{vineest}. The $i$ superscripts of $v_j^i$ and $w_j^i$ are omitted in the above diagram to simplify the notation.}
  \label{fig:vineAnnotated}
\end{figure}

From this decoration of the bonds and using Definition \ref{defdecmol}, we have that 
$$
\iota^i_j (c^i_j -d^i_j)=\iota^i_{j-1}(c^i_{j-1}-d^i_{j-1})=\pm r, \quad 1\leq j \leq m_i;
$$
we may assume it is $+r$ without loss of generality. Moreover, we have from \eqref{molegammav} that 
\begin{equation}\label{zetaOmegasum}
\iota_{\uf_{2j+1}^i}\Omega_{\uf_{2j+1}^i}+\iota_{\uf_{2j+2}^i}\Omega_{\uf_{2j+2}^i}=-\Gamma_{v_{2j+1}^i}-\Gamma_{v^i_{2j+2}}=\pm 2r\cdot(c_{j-1}^i-c_j^i\mathrm{\ or\ }c_{j-1}^i-d_j^i).
\end{equation}
As a result of this, we can define new variables $x_0^i, x_j^i , y_j^i$ ($1\leq j \leq m_i$) such that (a) each of $(x_j^i, y_j^i)$ is the difference of two vectors among $(a_j^i, b_j^i, c_j^i, c^i_{j-1})$ for $1\leq j \leq m_i$, (b) for $1\leq j\leq m_i$ we have \begin{equation}\label{Omegax_0ladder}\zeta_{\uf^i_{2j+2}}\Omega_{\uf^i_{2j+2}}=-\Gamma_{v^i_{2j+2}}= 2x_j^i \cdot y_j^i,\qquad\zeta_{\uf^i_{2j+1}}\Omega_{\uf^i_{2j+1}}=-2x^i_j\cdot y^i_j +2r\cdot \zeta^i_j\end{equation} where $\zeta^i_j=\alpha^i_j x^i_j+\beta^i_j y_j+\theta^i_j r$ for some $\alpha^i_j, \beta^i_j, \theta^i_j \in \{0, \pm 1\}$ with $(\alpha^i_j)^2+(\beta^i_j)^2\neq 0$, and (c) for $j=0$, we choose $x_0^i$ so that 
\begin{equation}\label{Omegax_0}
\zeta_{\uf^i_{2}}\Omega_{\uf^i_{2}}=2r\cdot x_0^i.
\end{equation} This choice of $x_0^i$ depends on $c_0^i, d_0^i$ and the decoration of $\Ub^{i-1}$ (if $i=1$, then $x_0^i$ is a linear combination of $(c_0^1, d_0^1, e,f,g,h)$).

We denote $\boldsymbol{x}^i=(x_0^i, x^i_j, y^i_j)_{1\leq j \leq m_i}$, ${\boldsymbol{k}^i}=(k_j^i:1\leq j\leq 2m_i+1)$ and $\boldsymbol{\ell}^i=(\ell_j^i:1\leq j\leq 2m_i+1)$. By the above discussions, we can use $\boldsymbol{x}^i$ as a substitute or re-parametrization for the variables $k_{\uf}$ for $\uf\in \Gc_{\mathrm{sk}}[\Ub^i]\setminus\{\uf_1^i\}$ (which are just ${\boldsymbol{k}^i}$ and $\boldsymbol{\ell}^i$). Note that (A-i) we have ${\boldsymbol{k}^i}=T_1^i{\boldsymbol{x}^i}+{\boldsymbol{h}}_1^i$ and $\boldsymbol{\ell}^i=T_2^i{\boldsymbol{x}^i}+{\boldsymbol{h}}_2^i$ for some matrices $T_k^i$ and some vectors ${\boldsymbol{h}}_k^i$ depending only on $(k^i_{\uf_1},k^i_{\uf_{11}}, k^i_{\uf_{21}}, k_{\uf^i_{22}})$, such that all coefficients of $T_1^i,\,(T_1^i)^{-1},\,T_2^i$ are $0$ or $\pm 1$; and (A-ii) any variable in $\boldsymbol{x}^i$ is the difference of two variables, such that each of them is either a variable in ${\boldsymbol{k}^i}$ or $\boldsymbol{\ell}^i$, or in $(k^i_{\uf_1},k^i_{\uf_{11}}, k^i_{\uf_{21}}, k^i_{\uf_{22}})$. Moreover, we have that (\ref{Omegax_0ladder}) holds for the ladder part and (\ref{Omegax_0}) holds for the joint $v_2^i$. These follow from simple calculations based on the above parametrization.

\smallskip
(B) Suppose $i\in\mathrm{Bad}$ and $\Ub^i$ is Vine (I) (which is just a double bond), then this parametrization is much easier. Indeed, it corresponds to the case $m_i=0$ above, for which we only need to choose $x_0^i$ as above so that \eqref{Omegax_0} holds. Clearly (A-i) and (A-ii) also hold in this simple case.

\smallskip
(C) Now suppose $i\in\mathrm{Nor}$, so $\Ub^i$ is a normal vine. By examining Figure \ref{fig:vines}, we see it contains $m_i$ pairs of atoms connected by double bounds that belong to the ladders represented by colored dashed lines (for Vine (VII) we also include the pair of atoms where all three ladders intersect). Moreover there are in total $n_i:=2m_i+4$ atoms excluding the joint atom $v^i_1$ (hence $2m_i+4$ branching nodes and $2m_i+4$ leaf pairs in $\Gc_{\mathrm{sk}}[\Ub^i]\backslash\{\uf_1^i\}$). These atoms can be divided into three groups: (a) the joint $v_2^i$ which corresponds to the branching node $\uf_2^i$; (b) the $m_i$ pairs of atoms in the ladders; and (c) the three remaining atoms. We denote the atoms in (b) by $v_j^i\,(3\leq j\leq 2m_i+2)$ and the corresponding branching nodes in $\Gc_{\mathrm{sk}}[\Ub^i]\backslash\{\uf_1^i\}$ by $\uf_j^i$ similar to (A), and denote the atoms in (c) by $(w^i_1, w^i_2, w^i_3)$ and the corresponding branching nodes by $(\wf_1^i,\wf_2^i,\wf_3^i)$.

For any decoration, denote the values of $k_\lf$ for leaves $\lf\in \Gc_{\mathrm{sk}}[\Ub^i]\backslash\{\uf_1^i\}$ by $(k_j^i: 1\leq j \leq 2m_i+4)$, and denote the values of $k_\nf$ for branching nodes $\nf\in\Gc_{\mathrm{sk}}[\Ub^i]\backslash\{\uf_1^i\}$ by $(\ell_j^i: 1\leq j \leq 2m_i+4)$. This decoration also leads to decoration of bonds of $\Ub^i$.

Now we want to define new variables $(x^i_0, x^i_j, y^i_j, u^i_1, u^i_2, u^i_3)_{1\leq j \leq m_i}$. First, by the same argument for Vine (II) above, we have that $\Omega_{\uf_2}=2r\cdot x^i_0$ where $x_0^i$ can be defined similar to (A). Next, for each pair of atoms $(v_{2j+1}^i,v_{2j+2}^i)$ connected by a double bond in a ladder and the corresponding branching nodes $\uf_{2j+1}^i,\uf_{2j+2}^i\in\Gc_{\mathrm{sk}}\backslash\{\uf_1^i\}$, the same argument for Vine (II) above shows that $$
\zeta_{\uf_{2j+1}^i} \Omega_{\uf_{2j+1}^i}+\zeta_{\uf_{2j+2}^i}\Omega_{\uf_{2j+2}^i}=2\widetilde r \cdot \mu.
$$ Here $\mu$ is the difference of two $k_\nf$ vectors corresponding to two of the four single bonds at this pair of atoms, and $\widetilde r\in \{0, r\}$ is the the same for all pairs of atoms in the same ladder and is equal to $r$ for the ladders attached to the joints and zero otherwise\footnote{In Figure \ref{fig:vines}, the $\widetilde r$ values for those ladders with cyan color are $0$; we shall call such ladders zero-gap ladders.}. Note that for the pair of atoms where all three ladders intersect in Vine (VII), the argument needs to be slightly adjusted but the result remains the same, with $\widetilde{r}=r$ in this case. In any case, we can define $(x_j^i, y_j^i)\in (\Zb_L^{d})^2$ similar to (A), such that (\ref{Omegax_0ladder}) holds with $r$ replaced by $\widetilde{r}$.

It now remains to define the variables $(u_1^i, u_2^i, u_3^i)$. Here we only discuss Vine (III) in detail below, as arguments in other cases are similar.

(C-i) For Vine (III), using the notation in Figure \ref{fig:vineAnnotated} (B), we may denote $\omega_j^i=\zeta_{\wf_j}\Omega_{\wf_j}$ where $\wf_j^i$ is the node in $\Gc_{\mathrm{sk}}[\Ub^i]$ corresponding to the atom $w^i_j$. In Figure \ref{fig:vineAnnotated} (B) and Figure \ref{fig:vines}, note that $e_*-g_*=f_*-h_*=\pm r$, and $(e_*,g_*)$ is determined by $x^i_0$ and some of the $(x^i_j,y^i_j)$ variables. Now, if the bonds decorated by $c$ and $d$ have opposite directions (say $c$ goes from $w^i_1$ to $w^i_3$, and $d$ goes from $w^i_3$ to $w^i_1$), we may define $(u^i_1,u^i_2,u^i_3)=(c-d,d-a,d-b)$. If the bonds decorated by $c$ and $d$ have the same direction (which has to go from $w^i_1$ to $w^i_3$), then we may define $(u^i_1,u^i_2,u^i_3)=(c-a,a-d,c-b)$. Then we have $\omega^i_1=2u_1^i\cdot u_2^i$, while $\omega_2^i$ equals $2u_1^i\cdot u_3^i$ in the first case, and equals $2u_3^i\cdot (u_1^i+u_2^i-u_3^i)$ in the second case. Moreover we have $\omega_1+\omega_2+\omega_3=\pm(|e_*|^2-|f_*|^2-|g_*|^2+|h_*|^2)=2r\cdot\xi^i$ where $\xi^i=a^iu_1^i+b^iu_2^i+c^i u_3^i+dr$ with $a^i,b^i,c^i,d^i\in\{0,\pm1\}$. In any case, the variables $(u_1^i,u_2^i,u_3^i)$ determine $(a,b,c,d,f_*,h_*)$ of Figure \ref{fig:vineAnnotated}, and allow us to proceed with parametrizing the next ladder starting from $(f_*,h_*)$ by the rest of $(x_j^i,y_j^i)$ variables. The argument for Vine (IV) is similar, see Figure \ref{fig:vineAnnotated} (C).

(C-ii) For Vines (V)--(VIII), the argument is again similar, and in fact much easier. In Figure \ref{fig:vineAnnotated} (D) and Figure \ref{fig:vines}, note that the two bonds going in and out the triangle are both decorated by $g_*$ (which is determined by the $(x_0^i,x_j^i,y_j^i)$ variables), which means that the vector $r$ for Vine (III) above is replaced by $\widetilde{r}=0$. In particular, we have $\omega_1^i+\omega_2^i+\omega_3^i=0$ where $\omega^i_j=\zeta_{\wf_j^i}\Omega_{\wf_j^i}$. Then we argue as above, with $(u_1^i,u_2^i,u_3^i)=(a-e,b-e,a-c)$ if bonds decorated by $a$ and $b$ have the same direction, and $(u_1^i,u_2^i,u_3^i)=(e-g_*,b-g_*,d-g_*)$ if they have opposite directions, then the same results will hold.

As a result, we can always define the variables $(u_1^i, u_2^i, u_3^i)$, such that when put together with the other variables defined above, the variables $\boldsymbol{x}^i=(x_0^i, x_j^i, y_j^i, u_1^i, u_2^i, u_3^i)_{1\leq j \leq m_i}$ can be used as a substitute or re-parametrization for the variables $k_\uf$ for $\uf\in\Gc_{\mathrm{sk}}[\Ub^i]\backslash\{\uf_1^i\}$. These $k_\uf$ variables are formed by $\boldsymbol{k}^i=(k_j^i:1\leq j\leq 2m_i+4)$ and $\boldsymbol{\ell}^i=(\ell_j^i:1\leq j\leq 2m_i+4)$; the variables $\boldsymbol{x}^i$, $\boldsymbol{k}^i$ and $\boldsymbol{\ell}^i$ verify the same properties (A-i) and (A-ii) as stated at the end of (A). Moreover, we have that (\ref{Omegax_0ladder}) holds for each ladder (with $r$ replaced by $\widetilde{r}$), and (\ref{Omegax_0}) holds for the joint $v_2^i$. For the remaining three branching nodes in $\Gc_{\mathrm{sk}}[\Ub^i]$, their corresponding resonance factors are given by $2u_1^i \cdot u_2^i$ and $\Lambda^i\in\{2u_1^i\cdot u_3^i,2u_3^i\cdot (u_1^i+u_2^i-u_3^i)\}$, and $-2u_1^i\cdot u_2^i-\Lambda^i+\widetilde r \cdot\xi^i$ where $\xi^i=au^i_1+bu^i_2+cu^i_3+dr$ with $a,b,c,d\in\{0,\pm1\}$ and $\widetilde{r}\in\{r,0\}$.

\smallskip
We now obtain the new set of variables $\boldsymbol{x}=(\boldsymbol{x}^i:1\leq i\leq I)$, where $\boldsymbol{x}^i=(x_0^i, x_j^i, y_j^i:1\leq j\leq m_i)$ for $i\in\mathrm{Bad}$ and $\boldsymbol{x}^i=(x_0^i, x_j^i, y_j^i, w_1^i, w_2^i, w_3^i:1\leq j\leq m_i)$ for $i\in\mathrm{Nor}$. Let $\boldsymbol{k}=(\boldsymbol{k}^i:1\leq i \leq I)$ and similarly for $\boldsymbol{\ell}$, then these variables satisfy the properties (i)--(iii) stated in Proposition \ref{sumint}. In fact (i) and (iii) follow from putting together for each $i$ the properties (A-i) and (A-ii) stated at the end of (A) (which are true also for (B) and (C)).  Property (ii), after a possible reordering of $\boldsymbol{k}$ and $\boldsymbol{\ell}$ variables, is implied by the following statement which is a consequence of Lemma 6.6 of \cite{DH21}:

$\bullet$ For each vector $k_\uf$ in $\boldsymbol{\ell}$ (where $\uf\in\Nc[\Ub]\backslash\{\uf_1\}$) there exists $\uf'\in\Nc[\Ub]\backslash\{\uf_1\}$ which is a child of $\uf$ such that $k_{\uf}\pm k_{\uf'}$ is an integer linear combination of $(k_\lf:\lf\in\Lc[\Ub])$ and $(k_{\uf_1},k_{\uf_{11}}, k_{\uf_{21}}, k_{\uf_{22}})$ with absolute value sum of coefficients at most $\Lambda_\uf$, and moreover we have $\prod_\uf\Lambda_\uf\leq C_0^{n_{(1)}}$. 

In addition, from the above discussion we know that (\ref{Omegax_0ladder}) holds for each ladder (possibly with $r$ replaced by $\widetilde{r}$), (\ref{Omegax_0}) holds for each joint $v_2^i$, and the remaining three resonance factors at each normal vine $\Ub^i$ is given by expressions involving $u_j^i\,(1\leq j\leq 3)$ as stated above. Also let $n_i:=2m_i+1+3\cdot\mathbf{1}_{\mathrm{Nor}}(i)$, so the sum of these $n_i$ equals $n_{(1)}$ which is also the total number of these variables.

We shall rename the time variables as follows: for each $i$ let $t_0^i:=t_{\uf_2^i}$. For each $1\leq j\leq m_i$ let $(t_j^i,s_j^i):=(t_{\uf_{2j+2}^i},t_{\uf_{2j+1}^i})$, and for each $i\in\mathrm{Nor}$ and $1\leq j\leq 3$ let $\tau_j^i:=t_{\wf_j^i}$. Then, with all the above preparations, we can write
\begin{align}
\widetilde \Kc_{(\Gc')}^{\Ub}&= \bigg(\frac{\delta}{2L^{d-\gamma}}\bigg)^{n_{(1)}}\zeta[\Ub] \sum_{\boldsymbol x} \epsilon_{\boldsymbol x}\cdot\int_{ \Ic[\Ub]} \bigg[\prod_{i=1}^I\bigg(\prod_{j=1}^{m_i} e^{\pi i \cdot\delta L^{2\gamma}(t_j^i-s_j^i)\langle x_j^i, y_j^i\rangle}
\times \prod_{j=0}^{m_i}e^{\pi i\cdot\delta L^{2\gamma}(t_j^i-t)\langle r,\zeta_j^i\rangle}\,\bigg]  \nonumber\\
&\times\bigg[ \prod_{i\in \mathrm{Nor}}e^{\pi i\cdot\delta L^{2\gamma}(\tau_1^i-\tau_3^i)\langle u_1^i,u_2^i\rangle}e^{\pi i\delta L^{2\gamma}(\tau_2^i-\tau_3^i)\Lambda_i}e^{\pi i\delta L^{2\gamma}(\tau_3^i-t)\langle r,\xi_i\rangle}\bigg]\nonumber\\
\label{tildeK2}&\times{\prod_{\lf\in
\Lc[\Ub]}^{(+)}\Kc_{\Qc_{(\lf,\lf')}}^*(t_{\lf^{\mathrm{pr}}},t_{(\lf')^{\mathrm{pr}}},k_\lf)}\prod_{\mf\in\Nc[\Ub]\backslash\{\uf_1\}}\Kc_{\Tc_{(\mf)}}^*(t_{\mf^{\mathrm{pr}}},t_{\mf},k_\mf)\,\mathrm{d}t_{\mf},
\end{align}
where the relationship between $\boldsymbol x$ and $(k_\mf)$ is as described above, and $\epsilon_{\boldsymbol{x}}$ is just the $\epsilon_{\Is[\Ub]}$ in (\ref{defKtilde}).

\uwave{Step 3: The case of non-core vines.} We now consider the case when each $\Ub^i$ is either non-core or does not satisfy the assumptions in Definition \ref{lftwist} (here, if $\Ub^i$ is core, these assumptions are that $\Lf_{\uf_2^i}\leq\min(\Lf_{\uf_3^i},\Lf_{\uf_4^i})$ and that both $\Tc_{\uf_2^i}$ and $\Qc_{(\uf_{23}^i,\uf_0^i)}$ are coherent, where the notations are as in Proposition \ref{molecpl} with superscript $i$). In this case, there is no twist and we just need to study the single term $\widetilde \Kc_{(\Gc)}^{\Ub}$. To this end, we use Propositions \ref{proplayer1} and \ref{proplayer2} to replace the factor $\Kc_{\Qc_{(\lf,\lf')}}^*(t_{\lf^{\mathrm{pr}}},t_{(\lf')^{\mathrm{pr}}},k_\lf)$ on the last line in \eqref{tildeK2} by a linear combination of
$$(\sqrt{\delta})^{n(\Qc_{(\lf,\lf')})}e^{\pi i (\lambda t_{\lf^{\mathrm{pr}}}+\lambda't_{(\lf')^{\mathrm{pr}}})}\cdot\Kc_{(\lf,\lf')}(k_\lf),$$ with coefficient being in $L_{\lambda, \lambda'}^1$ with $(\langle \lambda\rangle+\langle \lambda'\rangle )^\eta$ weight. Here each $\Kc_{(\lf,\lf')}$ is bounded in $\Sf^{30d,30d}$ by $L^{-(\gamma_1-\sqrt{\eta})g_{\lf,\lf'}}$ uniformly in $(\lambda,\lambda')$, where $g_{\lf,\lf'}$ is the incoherency index of $\Qc_{(\lf,\lf')}$. The same can be done for $\Kc_{\Tc_{(\mf)}}^*$, but with the $\Sf^{30d,30d}$ norm replaced by $\Sf^{30d,0}$, and $g_{\lf,\lf'}$ replaced by $g_{\mf}'$ which is the incoherency index of $\Tc_{(\mf)}$. Therefore, with all the $(\lambda,\lambda')$ variables as above fixed, we can reduce $\widetilde \Kc_{(\Gc)}^{\Ub}$ to a term of form
\begin{equation*}
(C_1\sqrt{\delta})^{n_{(0)}} \bigg(\frac{\delta}{2L^{d-\gamma}}\bigg)^{n_{(1)}}\zeta[\Ub] \cdot\Ic(t,t',t'',e,f,g,h).
\end{equation*}
Here $(t,t',t'',e,f,g,h)=(t_{\uf_1},t_{\uf_{21}},t_{\uf_{22}},k_{\uf_1},k_{\uf_{11}},k_{\uf_{21}},k_{\uf_{22}})$ as before, $n_{(0)}$ is the total order of all $\Qc_{(\lf,\lf')}$ and $\Tc_{(\mf)}$, and $\Ic$ is an expression as in \eqref{sumintI1}. Here note that once we specify $r=k_{\uf_{21}}-k_{\uf_{22}}=k_{\uf_1}-k_{\uf_{11}}\neq 0$, we may replace $\epsilon_{\boldsymbol x}$ in \eqref{tildeK2} by $1$, since if any $\epsilon$ factor is not $1$, this would imply some restrictions on the relevant variables (such as $x_j^i=0$ or $y_j^i=0$, or $u_1^i+u_2^i-u_3^i=0$ etc.), which allows us to sum over these variables trivially and gain an extra power $L^{-0.1\gamma_0}$. Moreover, for any bad vine $\Ub^i$, if it is not core, then it must be Vine (II-e) as in Proposition \ref{molecpl}, so we have $t_{\uf_4^i}<t_{\uf_2^i}<t_{\uf_3^i}$ from Figure \ref{fig:block_mole}, which translate to either $t^i_1<t_0^i<s_1^i$ or $s^i_1<t_0^i<t_1^i$ in the new time variables, hence assumption (a) in Proposition \ref{sumint} holds. If $\Ub^i$ is core, then by definition, either $\Lf_{\uf_2^i}>\min(\Lf_{\uf_3^i},\Lf_{\uf_4^i})$ or one of $\Tc_{\uf_2^i}$ and $\Qc_{(\uf_{23}^i,\uf_0^i)}$ is not coherent. In the second case assumption (d) in Proposition \ref{sumint} holds; in the first case, since $\uf_2^i$ is also a child node of either $\uf_3^i$ or $\uf_4^i$, we know that either $t_{\uf_3^i}<t_{\uf_2^i}<t_{\uf_4^i}$ or $t_{\uf_4^i}<t_{\uf_2^i}<t_{\uf_3^i}$, hence assumption (a) in Proposition \ref{sumint} holds.

This justifies that $\Ic$ indeed verifies all assumptions in Proposition \ref{sumint}, thus (\ref{vineest2}) follows from \eqref{FourierboundI}, if we notice that in the notation of Proposition \ref{sumint} we have $M=n_{(1)}$ and $n_{(2)}=n_{(1)}+n_{(0)}$, and $|\mathrm{Bad}|=b$ and $M_0+M_1=n_{(3)}$. Note also that the $\eta^6$ power of the $(\lambda_j^i,\mu_j^i,\sigma_j^i)$ variables in (\ref{FourierboundI}) can be absorbed by exploiting the $(\langle \lambda\rangle+\langle\lambda'\rangle)^\eta$ weight above, as each $\lambda_j^i$ etc. is an algebraic sum of these $\lambda$ and $\lambda'$.

\uwave{Step 4: Cancellation at the core vines.} We now consider the general case $\widetilde{\Kc}^{\Ub}=\sum_{\Gc'} \widetilde \Kc^{\Ub}_{(\Gc')}$. We will prove that it can also be written as a linear combination of expressions of form 
\begin{equation}\label{linearcombI}
(\sqrt{\delta})^{n_{(0)}} \bigg(\frac{\delta}{2L^{d-\gamma}}\bigg)^{n_{(1)}}\zeta[\Ub]\cdot\Ic(t,t',t'',e,f,g,h),
\end{equation} with coefficients in a weighted $L^1$ space, as in Step 3; here $\Ic$ is as in (\ref{sumintI1}). Once this is done, then (\ref{vineest2}) again follows from \eqref{FourierboundI} by the same arguments as above.

We shall argue by induction on the number of \emph{core vines $\Ub^i$ satisfying the assumptions in Definition \ref{lftwist}}. The base case is when there is none, which is treated in Step 3. Suppose there is at least one such vine, let $\Ub^{p}$ be the one with $p$ being largest. Denote by $\Ub^{<p}$ the union of the vines $\Ub^1,\cdots,\Ub^{p-1}$, and define
$$\widetilde{\Kc}^{\Ub^{<p}}(t_{\uf^{p-1}_{1}},t_{\uf_{21}},t_{\uf_{22}},k_{\uf_1^{p-1}},k_{\uf_{11}^{p-1}},k_{\uf_{21}},k_{\uf_{22}})$$ 
to be the expression in \eqref{vineest1} (where we sum over all LF twists) but with $\Ub$ replaced by $\Ub^{<p}$.

Recall that $\uf_1^{p-1}=\uf^{p}_2$ and $\uf^{p-1}_{11}=\uf^{p}_{23}$. Since no LF twisting is needed for $\Ub^i\,(i>p)$, we can use \eqref{tildeK2} to get that
$$
\sum_{\Gc'}\widetilde \Kc^{\Ub}_{(\Gc')}=\sum_{\Gc_p}(\widetilde \Jc_{\Gc_p})^{\Ub^{p}},
$$
where the last sum is taken over all $\Gc_p$ which are (unit) LF twists of $\Gc$ at $\Ub^{p}$ only; for each such twist, we have
\begin{align}
(\widetilde \Jc_{\Gc_p})^{\Ub^{p}}&= \bigg(\frac{\delta}{2L^{d-\gamma}}\bigg)^{n_{(1)}^{\geq p}}\zeta[\Ub^{\geq p}] \sum_{\boldsymbol x_{\geq p}} \int_{ \Ic[\Ub^{\geq p}]} \bigg[\prod_{i=p}^I\bigg(\prod_{j=1}^{m_i} e^{\pi i \cdot\delta L^{2\gamma}(t_j^i-s_j^i)\langle x_j^i, y_j^i\rangle}
\times \prod_{j=0}^{m_i}e^{\pi i\cdot\delta L^{2\gamma}(t_j^i-t_{\uf_1})\langle r,\zeta_j^i\rangle}\,\bigg)  \nonumber\\
&\times \bigg[ \prod_{i\geq p+1}e^{\pi i\cdot\delta L^{2\gamma}(\tau_1^i-\tau_3^i)\langle u_1^i,u_2^i\rangle}e^{\pi i\delta L^{2\gamma}(\tau_2^i-\tau_3^i)\Lambda_i}e^{\pi i\delta L^{2\gamma}(\tau_3^i-t_{\uf_1})\langle r,\xi_i\rangle}\bigg]\nonumber\\
&\times \prod_{\lf\in
\Lc[\Ub^{\geq p}]}^{(+)}\Kc_{\Qc_{(\lf,\lf')}}^*(t_{\lf^{\mathrm{pr}}},t_{(\lf')^{\mathrm{pr}}},k_\lf)\prod_{\mf\in\Nc[\Ub^{\geq p}]\backslash\{\uf_1\}}\Kc_{\Tc_{(\mf)}}^*(t_{\mf^{\mathrm{pr}}},t_{\mf},k_\mf)\,\mathrm{d}t_{\mf}\nonumber\\
\label{tildeJ2}&\times\widetilde{\Kc}^{\Ub^{<p}}(t_{\uf^{p}_{2}},t_{\uf_{21}},t_{\uf_{22}},k_{\uf_2^{p}},k_{\uf_{23}^{p}},k_{\uf_{21}},k_{\uf_{22}}).
\end{align}
Here, $\Ub^{\geq p}$ is the vine chain formed by $\Ub^{i}\,(i\geq p)$, which we identify as the result of splicing/merging all the vines $\Ub^i\,(i<p)$ as in Proposition \ref{block_clcn}, and $n_{(1)}^{\geq p}$ is the same as $n_{(1)}$ but defined for $\Ub^{\geq p}$, and the other objects are all associated with $\Ub^{\geq p}$. We have also replaced $\epsilon_{\boldsymbol x_{\geq p}}$ by $1$ since otherwise we would obtain better estimates.

The last factor in (\ref{tildeJ2}) is
\[\widetilde{\Kc}^{\Ub^{<p}}(t_{p},t',t'',e_{p},f_{p},g,h),\] where $(t_p,t',t'')=(t_{\uf_2^p},t_{\uf_{21}},t_{\uf_{22}})$ and $(g,h)=(k_{\uf_{21}},k_{\uf_{22}})$, and $(e_p,f_p)$ is either $(k_{\uf_2^{p}},k_{\uf_{23}^{p}})$ or $(k_{\uf_{23}^{p}},k_{\uf_2^{p}})$ similar to Step 1. Note that, if we perform a unit LF twist at $\Ub^p$ which corresponds to a unit twist of skeleton, then after this operation, the quantity $\Kc^{\Ub^{<p}}$ is transformed to another quantity with exactly the role of $e_p$ and $f_p$ switched (or equivalently, the roles of $k_{\uf_2^{p}}$ and $k_{\uf_{23}^{p}}$ switched). However, by Remark \ref{twistexplain} (d), we know that the roles of $k_{\uf_2^{p}}$ and $k_{\uf_{23}^{p}}$ in corresponding decorations of $\Ub^{\geq p}$ are switched again due to the unit twist. Thus the quantity $\widetilde{\Kc}^{\Ub^{<p}}$ can be extracted as a common factor for all unit LF twists at $\Ub^{p}$.

Now consider $\Gc_p$ which is a unit LF twist of $\Gc$. By Remark \ref{twistexplain} and Definition \ref{lftwist}, for this $\Gc_p$ we have $\zeta[\Ub^{\geq p}]= \sigma \zeta_{\uf_2^{p}}$ with $\sigma$ independent of the choice of $\Gc_p$. Moreover, the regular couples $\Qc_{(\lf, \lf')}$ and regular trees $\Tc_{(\mf)}$ are also independent of the choice of $\Gc_p$, except for $\Qc_{(\uf_0^{p}, \uf_{23}^{p})}$ and $\Tc_{(\uf_2^{p})}$, which range over all regular couples and trees described at the end of Definition \ref{lftwist} (we denote this set by $\Bs$). Therefore, the sum over all $\Gc_p$ equals the sum over $(\Gc_p)_{\mathrm{sk}}\in\{\Gc_{\mathrm{sk}},\Gc_{\mathrm{sk}}'\}$ where $\Gc_{\mathrm{sk}}'$ is the unit twist of $\Gc_{\mathrm{sk}}$, together with the sum over all $(\Qc_{(\uf_0^{p}, \uf_{23}^{p})},\Tc_{(\uf_2^{p})})\in \Bs$ (where the set $\Bs$ may also depend on the choice of $(\Gc_p)_{\mathrm{sk}}$). Recall from Definition \ref{twist} that the unit twist exchanges Vine (I-a), (II-a), (II-c) with Vine (I-b), (II-b), (II-d) respectively. Below we will only consider Vines (II-a) and (II-b), as Vines (II-c) and (II-d) are similar, and Vines (I-a) and (I-b) are much simpler\footnote{Compared to Vines (II), the case of Vines (I) will always involve an extra factor or $|r|$ corresponding to assumption (b) of Proposition \ref{sumint}.}.

Next consider all the input functions $\Kc_{\Qc_{(\lf, \lf')}}^*$ and $\Kc_{\Tc_{(\mf)}}^*$ in (\ref{tildeJ2}). For $(\lf,\lf')\neq (\uf_0^{p}, \uf_{23}^{p})$ and $\mf\neq \uf_2^{p}$, we simply expand them as in Step 3; this results in $e^{i\lambda t_{\lf^{\mathrm{pr}}}}$ etc. modulation factors and functions of $k_\lf$ with ($\Sf^{30d,30d}$ or $\Sf^{30d,0}$) bounds of $L^{-(\gamma_1-\sqrt{\eta})g}$ where $g$ is the corresponding incoherency index. As for $\Kc_{\Qc_{(\uf_0^{p}, \uf_{23}^{p})}}^*$ and $\Kc_{\Tc_{(\uf_2^{p})}}^*$, we know they must be coherent by Definition \ref{lftwist}. By Propositions \ref{proplayer2} and \ref{proplayer4}, we can replace them by either remainder factors which satisfy the same bound as above with $g\geq 1$ (this $g$ now is not incoherency index but merely indicates a remainder term), or the approximants $(\Kc_{\Qc_{(\uf_0^{p}, \uf_{23}^{p})}}^*)_{\mathrm{app}}$ and $(\Kc_{\Tc_{(\uf_2^{p})}}^*)_{\mathrm{app}}$ which satisfy (\ref{layerregest4})--(\ref{layerregest6}). The remainder term can be treated in the same way as in Step 3, since cancellation is not needed here due to the power gain $L^{-(\gamma_1-\sqrt{\eta}})$, so we will only consider the $(\cdots)_{\mathrm{app}}$ terms.

We consider first the case $\Lf_{\uf_3^p}=\Lf_{\uf_4^p}$, in which case the set $\Bs$ does not depend on the choice of $(\Gc_p)_{\mathrm{sk}}\in\{\Gc_{\mathrm{sk}},\Gc_{\mathrm{sk}}'\}$. By the induction hypothesis, we can write $\widetilde{\Kc}_\Gc^{\Ub^{<p}}$ as a linear combination of expressions of form \eqref{linearcombI} but adapted to $\Ub^{<p}$ (where $\Ic$ is as in (\ref{sumintI1})). We henceforth perform the above reductions for the $\Kc_{\Qc_{(\lf, \lf')}}^*$ and $\Kc_{\Tc_{(\mf)}}^*$ factors, and make the change of vector variables and rename the time variables as in Step 2. Putting all these together, we can write $\sum_{\Gc_p}(\widetilde{\Jc}_{\Gc_p})^{\Ub^p}$ as a linear combination of expressions $\Jc$ that has the same form as (\ref{linearcombI}) now adapted to $\Ub$ (again with $\Ic$ as in (\ref{sumintI1})), but with only one difference. Namely, the factors
\begin{equation}\label{replaceterm0}e^{\pi i\lambda_0^pt_0^p}\cdot e^{\pi i\lambda_1^pt_1^p}\cdot e^{\pi i\mu_1^ps_1^p}\cdot \Kc_1^p(k_1^p)\cdot\Lc_1^p(\ell_1^p)\end{equation}
 that occur in the expression of $\Ic$ in (\ref{sumintI1}) corresponding to $i=p$, should be replaced by
 \begin{multline}\label{replaceterm1}
 (C_1\sqrt{\delta})^{-(n(\Qc)+n(\Tc))}\sum_{(\Qc,\Tc)\in\Bs}^{(+)}\bigg\{\mathbf{1}_{t_0^p<t_1^p}\cdot(\Kc_{\overline{\Qc}}^*)_{\mathrm{app}}(t_0^p,s_1^p,y_*)\cdot(\Kc_\Tc^*)_{\mathrm{app}}(t_1^p,t_0^p,x_*)\\
 -\mathbf{1}_{t_0^p<s_1^p}\cdot(\Kc_{\Qc}^*)_{\mathrm{app}}(t_1^p,t_0^p,x_*)\cdot(\Kc_{\overline{\Tc}}^*)_{\mathrm{app}}(s_1^p,t_0^p,y_*)\bigg\},
 \end{multline} where the summation is taken over all $(\Qc,\Tc)\in\Bs$ such that $\Tc$ has sign $+$.
 
 We explain (\ref{replaceterm1}) as follows: the garden $(\Gc_p)_{\mathrm{sk}}$ has two possibilities which are unit twists of each other, one being Vine (II-a) and the other being Vine (II-b). In either case, by the renaming of time variables above, we always have $(t_0^p,s_1^p,t_1^p)=(t_{\uf_2^p},t_{\uf_3^p},t_{\uf_4^p})$ and the unit interval that each of them belongs to does not depend on the twisting; also in relation with (\ref{replaceterm0}), we assume $k_1^p=k_{\uf_{23}^p}=k_{\uf_0^p}$ and $\ell_1^p=k_{\uf_2^p}$.

For Vine (II-a), see Figure \ref{fig:block_mole}, we may assume $\zeta_{\uf_2^p}=\zeta_{\uf_{23}^p}=+$ and $\zeta_{\uf_0^p}=-$. We then choose $\Qc_{(\uf_0^p,\uf_{23}^p)}=\overline{\Qc}$ and $\Tc_{(\uf_2^p)}=\Tc$. In (\ref{replaceterm1}) we set $x_*=k_{\uf_2^p}$ and $y_*=k_{\uf_{23}^p}=k_{\uf_0^p}$. Finally the structure of the couple forces $t_{\uf_2^p}<t_{\uf_4^p}$, which is $t_0^p<t_1^p$.

For Vine (II-b), we have $\zeta_{\uf_2^p}=\zeta_{\uf_{23}^p}=-$ and $\zeta_{\uf_0^p}=+$. We then choose $\Qc_{(\uf_0^p,\uf_{23}^p)}=\Qc$ and $\Tc_{(\uf_2^p)}=\overline{\Tc}$. In (\ref{replaceterm1}) we set $x_*=k_{\uf_{23}^p}=k_{\uf_0^p}$ and $y_*=k_{\uf_2^p}$ (note that the roles of $k_{\uf_2}^p$ and $k_{\uf_{23}^p}$ in decorations are switched after unit twisting, see Remark \ref{twistexplain} (d)). Finally the structure of the couple forces $t_{\uf_2^p}<t_{\uf_3^p}$, which is $t_0^p<s_1^p$.

In the above discussions, note also that the $x_*-y_*$ is fixed and equals $\pm r$. These arguments then lead to the expression (\ref{replaceterm1}), where the negative sign before the second term is due to the change of $\zeta_{\uf_2^p}$ after the unit twisting. Of course these arguments are for Vines (II-a) and (II-b), but the case of Vines (II-c) and (II-d) is similar, and the case of Vines (I) is much easier.

Now it remains to analyze (\ref{replaceterm1}). Recall the identity (\ref{layerregest9}). We apply (\ref{layerregest4})--(\ref{layerregest6}) to reduce the expression in (\ref{replaceterm1}) to one of the followings:
\begin{enumerate}[{(a)}]
\item Terms containing $\mathbf{1}_{t_1^p< t_0^p<s_1^p}$ or $\mathbf{1}_{s_1^p< t_0^p<t_1^p}$. This leads to assumption (a) in Proposition \ref{sumint}.
\item Terms containing $(\Kc_{\Qc}^*)_{\mathrm{app}}(\cdot,\cdot,y_*)-(\Kc_{\Qc}^*)_{\mathrm{app}}(\cdot,\cdot,x_*)$ or $(\Kc_{\Tc}^*)_{\mathrm{app}}(\cdot,\cdot,x_{*})-(\Kc_{\Tc}^*)_{\mathrm{app}}(\cdot,\cdot,y_{*})$,with the time variables being the same in both functions. Since $x_*-y_*=\pm r$, this is bounded by $|r|$ and leads to assumption (b) in Proposition \ref{sumint}.
\item Terms containing $\widetilde{\Jc_\Qc}(t_1^p,t_0^p)-\widetilde{\Jc_\Qc}(s_1^p,t_0^p)$ or $\widetilde{\Jc_\Tc}(t_1^p,t_0^p)-\widetilde{\Jc_\Tc}(s_1^p,t_0^p)$. These come from the decomposition (\ref{layerregest4}); due to the $\Xf^{1-\eta}$ bounds in (\ref{layerregest5})--(\ref{layerregest6}), we can extract a factor of $|t_1^p-s_1^p|^{1-3\eta}$ and still preserve the same weighted Fourier $L^1$ estimate as above. This leads to assumption (c) in Proposition \ref{sumint}.
\item The leading factor, which we can use (\ref{layerregest9}) to reduce to
$$
\left[ \overline{(\Kc_{\Qc}^*})_{\mathrm{app}}(t_1^p,t_0^p,x_*)(\Kc_{\Tc}^*)_{\mathrm{app}}(t_1^p,t_0^p,x_*)-(\Kc_{\Qc}^*)_{\mathrm{app}}(t_1^p,t_0^p,x_*)\overline{(\Kc_{\Tc}^*)_{\mathrm{app}}}(t_1^p,t_0^p,x_*)\right].
$$ Thanks to (\ref{layerregest10}), this is \emph{exactly zero} after summing over all $(\Qc,\Tc)\in\Bs$.
\end{enumerate}
Here note that, in cases (a)--(c) above, the number of choices for $(\Qc,\Tc)$ is at most $C_2\cdot C_0^{n(\Qc)+n(\Tc)}$ by Proposition \ref{layerreg2}, which is acceptable because this only occurs once for $\Ub^p$ and can be covered by the $L^{-\eta^2}$ gain for this bad vine obtained below. Therefore, in the case $\Lf_{\uf_3^p}=\Lf_{\uf_4^p}$, we may fix one choice of $(\Qc,\Tc)$ which will satisfy one of the assumptions (a)--(d) Proposition \ref{layerreg2}.

Finally, consider the case $\Lf_{\uf_3^p}>\Lf_{\uf_4^p}$ (which is true for Vines (II-c) and (II-d), and can be assumed by symmetry for Vines (II-a) and (II-b)). In this case the set $\Bs$ of regular couples and regular trees defined above, is not the same for $(\Gc_q)_{\mathrm{sk}}=\Gc_{\mathrm{sk}}$ and for $(\Gc_q)_{\mathrm{sk}}=\Gc_{\mathrm{sk}}'$. More precisely, 
\begin{itemize}
\item The set $\Bs$ for Vine (II-a) consists of $(\Qc,\Tc)$ described at the end of Definition \ref{lftwist}, which have the exact structure described in Proposition \ref{layerreg1}, with $(q,q')$ in Proposition \ref{layerreg1} replaced by $(\Lf_{\uf_3^p},\Lf_{\uf_2^p})$ for $\Qc$, and by $(\Lf_{\uf_4^p},\Lf_{\uf_2^p})$ for $\Tc$;
\item The other set $\Bs'$ for Vine (II-b) consists of the same $(\Qc,\Tc)$, but with $(q,q')$ in Proposition \ref{layerreg1} replaced by $(\Lf_{\uf_4^p},\Lf_{\uf_2^p})$ for $\Qc$, and by $(\Lf_{\uf_3^p},\Lf_{\uf_2^p})$ for $\Tc$.
\end{itemize}

To treat this issue we apply the following argument. First, for any $\Qc$ we can extend the definition of $\widetilde{\Jc_\Qc}(t,s)$ (and hence $(\Kc_\Qc^*)_{\mathrm{app}}(t,s,k)$) in Proposition \ref{proplayer2} to include all $t\in\Rb$. In fact, currently it is defined for $t\in[q,q+1]$
(for notations see Proposition \ref{layerreg1} and the beginning of Section \ref{seclayer2}), where $q\geq \Lf_\rf$ with $\rf$ being the positive root of $\Qc$. Moreover, by Definition \ref{defkg}, see also the proof of Proposition \ref{proplayer2}, it is easy to see that (i) if $q>\Lf_{\rf}$ then $\widetilde{\Jc_\Qc}(t,s)$ is constant in $t$ for $t\in[q,q+1]$, and (ii)  if $q=\Lf_{\rf}$ then $\widetilde{\Jc_\Qc}(t,s)$ equals $0$ if $t=\Lf_\rf$. Now we can define $\widetilde{\Jc_\Qc}(t,s)$ for all $t\in\mathbb{R}$, by keeping the current definition for $t\in[\Lf_\rf,\Lf_\rf+1]$ (when $q=\Lf_\rf$) and extending it continuously as constant (in $t$) functions for $t\in(-\infty,\Lf_\rf]$ and $t\in[\Lf_\rf+1,+\infty)$. The same extensions can be defined also for $\widetilde{\Jc_\Tc}$ for any $\Tc$. These extensions are consistent with the current definition for any $q\geq \Lf_\rf$, and equals $0$ for $t\leq \Lf_\rf$; moreover, an elementary argument shows that the norm bounds (\ref{layerregest5})--(\ref{layerregest6}) for $\widetilde{\Jc_\Qc}$ and $\widetilde{\Jc_\Tc}$, which hold for $t\in[q,q+1]$, are also preserved (with $\Xf^{1-\eta}$ weakened by $\Xf^{1-2\eta}$) by such continuous constant extensions on any interval of length $O(1)$.

Now we consider the case $\Lf_{\uf_3^p}>\Lf_{\uf_4^p}$. Recall that $(t_0^p,s_1^p,t_1^p)=(t_{\uf_2^p},t_{\uf_3^p},t_{\uf_4^p})$; if $\Lf_{\uf_3^p}-\Lf_{\uf_4^p}\geq 2$, then we must have $|t_1^p-s_1^p|\geq 1$, so we can trivially insert a factor $|t_1^p-s_1^p|^{1-3\eta}$ without any loss, which leads to assumption (c) in Proposition \ref{sumint}. Thus we may assume $\Lf_{\uf_3^p}-\Lf_{\uf_4^p}=1$, in which case $|t_1^p-s_1^p|\leq 2$. Here, instead of (\ref{replaceterm1}), we have the difference of two sums of products of $\Kc_\Qc^*$ and $\Kc_\Tc^*$ factors, where the first sum is taken over $(\Qc,\Tc)\in\Bs$ and the second sum is taken over $(\Qc,\Tc)\in\Bs'$. However, for any fixed $(\Qc,\Tc)$ we can use the above-defined extensions (which satisfy (\ref{layerregest5})--(\ref{layerregest6}) on any interval of length at most $4$) to replace any time variable $s_1^p$ in any $(\Kc_\Qc^*)_{\mathrm{app}}$ and $(\Kc_\Tc^*)_{\mathrm{app}}$ by $t_1^p$; for the difference created in this process, we can still extract a factor of $|t_1^p-s_1^p|^{1-3\eta}$ and preserve the same weighted Fourier $L^1$ estimates. This again leads to assumption (c) in Proposition \ref{sumint}, just as in case (c) described above.

Finally, once all the time variables $s_1^p$ are replaced by $t_1^p$, we see that the contribution of $(\Qc,\Tc)$ (which is basically $(\Kc_\Qc^*)_{\mathrm{app}}(t_1^p,t_0^p,\cdot)\cdot(\Kc_\Tc^*)_{\mathrm{app}}(t_1^p,t_0^p,\cdot)$) will be zero unless $(\Qc,\Tc)\in\Bs''\subset\Bs\cap\Bs'$, which is defined in the same way as $\Bs$ and $\Bs'$ above, but with $(q,q')$ in Proposition \ref{layerreg1} replaced by $(\Lf_{\uf_4^p},\Lf_{\uf_2^p})$ for \emph{both $\Qc$ and $\Tc$}. This is because $t_1^p\in[\Lf_{\uf_4^p},\Lf_{\uf_4^p}+1]$, so for any $\Qc$ or $\Tc$ with root $\rf$ satisfying $\Lf_{\rf}\geq \Lf_{\uf_4^p}+1$ we must have $\widetilde{\Jc_\Qc}(t_1^p,t_0^p)=\widetilde{\Jc_\Tc}(t_1^p,t_0^p)=0$. Therefore, we can now reduce to a sum of form (\ref{replaceterm1}) but with $\Bs$ replaced by $\Bs''$, and (\ref{layerregest10}) implies that the leading term as described in case (d) above is again $0$. This completes the case $\Lf_{\uf_3^p}>\Lf_{\uf_4^p}$. 

With all the above discussions, and by applying again the same arguments in Step 3 exploiting the weighted Fourier $L^1$ bounds, we can finally reduce the target quantity $\sum_{\Gc_p}(\widetilde{\Jc}_{\Gc_p})^{\Ub^p}$ to a linear combination of expressions of form (\ref{linearcombI}) (with $\Ic$ as in (\ref{sumintI1})) such that at least one of assumptions (a)--(d) in Proposition \ref{sumint} holds for the bad vine $\Ub^p$ (and consequently for every bad vine). Here note that the requirements (i)--(ii) in Proposition \ref{sumint} for the relations between the variables $(\boldsymbol{x},\boldsymbol{k},\boldsymbol{\ell})$ are verified by using the induction hypothesis, while requirement (iii) follows by again applying Lemma 6.6 of \cite{DH21} as in Step 2 above. This completes the proof of (\ref{vineest2}).
\end{proof}
\section{Stage 1 reduction}\label{stage1red} We now describe the first reduction step in the proof of Proposition \ref{layergarden}. In this step we reduce any garden $\Gc$ first to its skeleton $\Gc_{\mathrm{sk}}$ as in Section \ref{regred}, and then to a smaller garden $\Gc_{\mathrm{sb}}$ by splicing a suitably chosen set of (CL) vine chains. This set in particular contains all the bad vines where the cancellation argument in \cite{DH23} is needed.
\subsection{Vine collections $\Vs$ and $\Vs_0$}\label{redsplice0} Recall we are estimating the sum (\ref{expgarden}) over all $\Gc\in\Gs_{p+1}^{\mathrm{tr}}$ (or $\Qc\in\Cs_{p+1}^{\mathrm{tr}}$) of fixed signature $(\zeta_1,\cdots,\zeta_{2R})$ and order $n$. The starting point is to divide this sum into subset sums, where each subset sum is taken over an \emph{equivalence class} of gardens $\Gc\in\Gs_{p+1}^{\mathrm{tr}}$; these equivalence classes are defined via LF twists for a suitable collection of (CL) vines $\Vb_j\subset\Mb(\Gc_{\mathrm{sk}})$ as in Definition \ref{lftwist}. We describe this collection of (CL) vines as follows.

Recall the notions of vines and vine-chains (VC's), hyper-vines (HV's) and hyper-vine-chains (HVC's), and ladders in Definition \ref{defvine}.
\begin{lem}\label{vinechainlem} Define a \emph{double-Vine (V)}, or DV for short, to be the union of two vines (V), see Figure \ref{fig:vines}, that share two common joints and no other common atoms (note this is not a vine-like object in Definition \ref{defvine}). Then, for any molecule $\Mb$, there is a unique collection $\Cs$ of disjoint atomic groups, such that each atomic group in $\Cs$ is an HV, VC, HVC or DV, and any vine-like object in $\Vb$ is a subset of some atomic group in $\Cs$.
\end{lem}
\begin{proof} See Lemma 8.1 of \cite{DH23}.
\end{proof}
\begin{df}[The collection $\Vs$ \cite{DH23}]\label{defcong}Let $\Gc$ be a garden with skeleton $\Gc_{\mathrm{sk}}$,  and let $\Cs$ be defined for the molecule $\Mb(\Gc_{\mathrm{sk}})$ by Lemma \ref{vinechainlem}. Define a collection $\Vs$ of (CL) vines as follows: for each \emph{non-root VC} in $\Cs$ (i.e. a VC that is not a root block, see Proposition \ref{cnblock}) we add to $\Vs$ all its vine ingredients that are (CL) vines. For each HVC  and root VC in $\Cs$, we add to $\Vs$ \emph{all but one} of its vine ingredients, such that (a) if there is a (CN) vine then we only skip this one, and (b) if all vines are (CL) vines then we only skip the ``top" vine whose $\uf_1$ node is the ancestor of all other $\uf_1$ nodes (as branching nodes of $\Gc_{\mathrm{sk}}$, in the notation of Proposition \ref{block_clcn}). We do not add anything that comes from any HV or DV in $\Cs$.

Note that the vines in $\Vs$ can be arranged into finitely many disjoint VC's, as is consistent with Definition \ref{lftwist}. We then define any couple $\Gc'$ to be \emph{congruent} to $\Gc$, if $\Gc'$ is an LF twist of $\Gc$ at vines in $\Vs$ (i.e. we perform a unit LF twist at each vine in $\Vs$ that is core and satisfies the assumptions in Definition \ref{lftwist}). Clearly, performing a full twist does not affect the molecule $\Mb(\Gc_{\mathrm{sk}})$ nor the choice of vines in $\Vs$, and thus congruence is an equivalence relation.
\end{df}
\begin{df}[Large and small gaps \cite{DH23}]\label{defdiff} Given any molecule $\Mb$ and block $\Bb\subset\Mb$, let the four bonds in $\Bb$ at the two joints be $\ell_1,\ell_2\in\Bb$ at one joint, and $\ell_3,\ell_4\in\Bb$ at the other. Then for any decoration $(k_\ell)$ we have $k_{\ell_1}-k_{\ell_2}=\pm(k_{\ell_3}-k_{\ell_4}):=r$; we call this vector the \emph{gap} of $\Bb$ relative to this decoration. Note that once the parameters $(c_v)$ of a decoration are fixed as in Definition \ref{defdecmol}, then this $r$ can be expressed as a function of the vectors $k_{\ell_j^*}$, where $\ell_j^*$ runs over all bonds connecting a given joint of $\Bb$ to atoms \emph{not} in $\Bb$. If $\Bb$ is concatenated by blocks $\Bb_j$, then all $\Bb_j$ must have the same gap as $\Bb$. We also define the gap of a hyper-block to be the gap of its adjoint block.

More generally, if $v$ is an atom and $\ell_1,\ell_2\leftrightarrow v$ are two bonds with opposite directions, then we define the \emph{gap} of the triple $(v,\ell_1,\ell_2)$ relative to a given decoration as $r:=k_{\ell_1}-k_{\ell_2}$. In particular the gap of any block or hyper-block equals a suitable gap at either of its joints. Next, for any ladder of length $\geq 1$ (see Definition \ref{defvine} and Figure \ref{fig:vines}), the difference $k_\ell-k_{\ell'}$ for any pair of parallel single bonds $(\ell,\ell')$ must be equal (up to a sign change), which we also define to be the \emph{gap} of the ladder. In particular, if $\Vb$ is a vine (or VC) with gap $r$, then for any ladder contained in $\Vb$ (inserted between parallel dashed bonds of the same color in Figure \ref{fig:vines}), the gap of this ladder is either $\pm r$ or $0$. Finally, for all the gaps defined above, we say it is \emph{small gap, large gap or zero gap} (writing SG, LG and ZG for short) if $0<|r|\leq L^{-\gamma+\eta}$, $|r|> L^{-\gamma+\eta}$, or $r=0$.
\end{df}
\subsection{Reduction to $\Gc_{\mathrm{sb}}$ and $\Mb(\Gc_{\mathrm{sb}})$}\label{redsplice} By Definition \ref{defcong} and Proposition \ref{lftwistprop} with the assumption $n\leq N_{p+1}^4$ in the statement of Proposition \ref{layergarden}, we know that the sum (\ref{expgarden}) can be divided into subset sums, where each subset sum is taken over one single congruence class for $\Gc$.

Let us fix a single subset sum for a certain congruence class. Note that we are summing $\Kc_\Gc$ over $\Gc$, while each $\Kc_\Gc$ is a sum of form (\ref{redkg}) over decorations of $\Gc_{\mathrm{sk}}$, which can be identified with decorations of $\Mb(\Gc_{\mathrm{sk}})$ via Definition \ref{defdecmol}; moreover the decorations of congruent gardens are in one-to-one correspondence with each other and preserves the gap of each vine-like object. Thus, let $\Cs$ be as in Lemma \ref{vinechainlem}, then for each vine-like object in $\Cs$, we can specify whether it has SG, LG or ZG under the decoration (each such specification defines \emph{an extra restriction} on the decoration, which will be summarized in the factor $\Zc_{\Is_{\mathrm{sb}}}$ in the proof below). Then, we can further divide the given subset sum into at most $3^n$ terms, such that each vine-like object in $\Cs$ is specified to be either SG, LG or ZG in each term. Let $\Vs_0$ be the collection of VC's in $\Vs$ that are SG (this collection is not affected by any LF twist at vines in $\Vs$). Note that $\Gc$ runs over a congruence class, which is defined by LF twists at vines in $\Vs$; if we strengthen the equivalence relation by allowing only LF twists at vines in $\Vs_0$, then we can further divide the sum into at most $2^n$ terms, each of which is a sum over a new equivalence class. As such, we only need to consider one of these new terms, which we shall refer to as $\Ks_{p+1,n}^{\mathrm{eqc}}$ for below.

Let all the VC's in the collection $\Vs_0$ be $\Ub_1,\cdots,\Ub_{q_{\mathrm{sk}}}$. For each $\Ub_j$, denote the corresponding nodes $\uf_1,\uf_{22}$ etc. (as in Proposition \ref{block_clcn}) by $\uf_1^j,\uf_{22}^j$ etc.; also denote $\Gc_{\mathrm{sk}}[\Ub_j]\backslash\{\uf_1^j\}:=\Uc_j$. The term $\Ks_{p+1,n}^{\mathrm{eqc}}$ is the sum of $\Kc_\Gc$ (with the extra SG, LG or ZG restrictions on the decoration defined above) over all gardens $\Gc$ in a new equivalence class, which is defined by LF twists at all the vines in $\Ub_j$. Now for each $1\leq j\leq q_{\mathrm{sk}}$, we consider the expression (\ref{redkg}) of $\Kc_\Gc$ (with extra restrictions), and the part of it \emph{involving the realization} $\Gc_{\mathrm{sk}}[\Ub_j]$, which is precisely defined as $\Kc_{(\Gc)}^{\Ub_j}$ in Section \ref{regred}. We make three observations, which are easily verified using relevant definitions and structure of the molecule $\Mb(\Gc_{\mathrm{sk}})$:
\begin{itemize}
\item The expression $\Kc_{(\Gc)}^{\Ub_j}$ is not affected by LF twists at vines in $\Ub_i\,(i\neq j)$.
\item For any $(i,j)$, the variables \[(t_{\uf_1^i},t_{\uf_{21}^i},t_{\uf_{22}^i},k_{\uf_1^i},k_{\uf_{11}^i},k_{\uf_{21}^i},k_{\uf_{22}^i})\] which the function $\Gc_{\mathrm{sk}}[\Ub_i]$ depends on, do not overlap with the variables $k_\nf$ for $\nf\in\Uc_j$, nor the variables $t_\nf$ for branching nodes $\nf\in\Uc_j$.
\item Similarly, the extra SG, LG and ZG restrictions defined above can also be expressed using the variables that do not overlap with any $k_\nf$ nor $t_\nf$ for $\nf\in\Uc_j$.
\end{itemize}

Let $\Gc_{\mathrm{sb}}$ be the result of splicing all vines in $\Vs_0$ from $\Gc_{\mathrm{sk}}$ as in Proposition \ref{block_clcn}; note that $\Gc_{\mathrm{sb}}$ is a garden and is not affected by any LF twist for $\Gc$ (see Remark \ref{twistexplain} (c)). It has a pre-layering inherited from $\Gc_{\mathrm{sk}}$ (i.e. when splicing each vine or VC as in Proposition \ref{block_clcn}, all the nodes in the post-splicing garden, including $(\uf_1,\uf_{11},\uf_{21},\uf_{22})$ in Proposition \ref{block_clcn}, are in the same layer as in the pre-splicing garden). By the above observations, and note that for each $j$ we are also summing $\Kc_{(\Gc)}^{\Ub_j}$ for $\Gc$ being an LF twist at $\Ub_j$, we can rewrite
\begin{multline}\label{redvine}
\Ks_{p+1,n}^{\mathrm{eqc}}(k_1,\cdots,k_{2R})=\bigg(\frac{\delta}{2L^{d-\gamma}}\bigg)^{n_{\mathrm{sb}}}\zeta(\Gc_{\mathrm{sb}})\sum_{\Is_{\mathrm{sb}}}\widetilde{\epsilon}_{\Is_{\mathrm{sb}}}\cdot \Zc_{\Is_{\mathrm{sb}}}\cdot\int_{\Ic_{\mathrm{sb}}}\prod_{\nf\in\Nc_{\mathrm{sb}}}e^{\pi i\zeta_\nf\cdot\delta L^{2\gamma}\Omega_\nf t_\nf}\,\mathrm{d}t_\nf\\
\times\prod_{\lf\in\Lc_{\mathrm{sb}}}^{(+)}\Kc_{\Qc_{(\lf,\lf')}}^*(t_{\lf^{\mathrm{pr}}},t_{(\lf')^{\mathrm{pr}}},k_\lf)\prod_{\mf\in\Nc_{\mathrm{sb}}}\Kc_{\Tc_{(\mf)}}^*(t_{\mf^{\mathrm{pr}}},t_\mf,k_\mf)\prod_{j=1}^{q_{\mathrm{sk}}}\widetilde{\Kc}^{\Ub_j}(t_{\uf_1^j},t_{\uf_{21}^j},t_{\uf_{22}^j},k_{\uf_1^j},k_{\uf_{11}^j},k_{\uf_{21}^j},k_{\uf_{22}^j}).
\end{multline} Here in (\ref{redvine}), all the notions $n_{\mathrm{sb}}$, $\Is_{\mathrm{sb}}$ and $\Ic_{\mathrm{sb}}$ etc. are just like in (\ref{redkg}) but for the garden $\Gc_{\mathrm{sb}}$, except that (i) the factor $\widetilde{\epsilon}_{\Is_{\mathrm{sb}}}$ contains the product of factors on the right hand side of (\ref{defepscoef}) \emph{only for} $\nf\not\in\{\uf_1^j:1\leq j\leq q_{\mathrm{sk}}\}$, and (ii) the extra factors $\widetilde{\Kc}^{\Ub_j}$ are the ones defined in (\ref{vineest1}) with $\Ub$ replaced by $\Ub_j$, and $\Zc_{\Is_{\mathrm{sb}}}$ is the indicator function of all the extra SG, LG and ZG restrictions defined above (this $\Zc$ depends only on $k[\Gc_{\mathrm{sb}}]$). Note that the factor $e^{-\pi i\cdot\delta L^{2\gamma}t_{\uf_1}\Gamma}$ in (\ref{vineest1}) exactly cancels the factor $e^{\pi i\zeta_\nf\cdot\delta L^{2\gamma}\Omega_\nf t_\nf}$ in (\ref{redvine}) for $\nf=\uf_1^j$.

We perform some further reductions to (\ref{redvine}). First, by the decay properties of $\Kc_{\Qc_{(\lf,\lf')}}^*$ in Propositions \ref{proplayer1} and \ref{proplayer2}, we may decompose (\ref{redvine}) by restricting each $k_\lf$ to a unit ball $|k_\lf-k_\lf^0|\leq 1$; we only need to treat a single term of this form, as summability is provided by the above decay properties. By Lemma \ref{treelemma}, with a loss of at most $C_0^n$, we may also restrict $k_\nf$ to a unit ball $|k_\nf-k_\nf^0|\leq 1$ for each \emph{node} $\nf$.

Second, by the weighted Fourier $L^1$ bounds for $\Kc_{\Qc_{(\lf,\lf')}}^*$ and $\Kc_{\Tc_{(\mf)}}^*$ (see Propositions \ref{proplayer1} and \ref{proplayer2}), and the same bounds for $\widetilde{\Kc}^{\Ub_j}$ with each vector variable $k_{\uf_1^j}$ etc. restricted to a unit ball (see Proposition \ref{vineest}), we can decompose these functions using time Fourier transform, and reduce (\ref{redvine}) to the following expression (with some fixed parameters $\xi$ and $\xi_\nf$)
\begin{align}\label{redvine2}
\widetilde{\Ks}_{p+1,n}^{\mathrm{eqc}}&:=\zeta(\Gc_{\mathrm{sb}})\cdot e^{\pi i\xi (p+1)}\cdot \big(\max_{\nf\in\Nc_{\mathrm{sb}}}(\langle \xi\rangle+\langle \xi_\nf\rangle)\big)^{-2\eta^8}\cdot(C_1\sqrt{\delta})^{n-n_{\mathrm{sb}}}\cdot L^{-\eta^{8}g_{\mathrm{sk}}-\eta^2b_{\mathrm{sk}}+\eta^7q_{\mathrm{sk}}}\cdot \Ks_{\Gc_{\mathrm{sb}}},\nonumber\\
\Ks_{\Gc_{\mathrm{sb}}}&:=\bigg(\frac{\delta}{2L^{d-\gamma}}\bigg)^{n_{\mathrm{sb}}}\sum_{\Is_{\mathrm{sb}}}\widetilde{\epsilon}_{\Is_{\mathrm{sb}}}\cdot \Zc_{\Is_{\mathrm{sb}}}\cdot\prod_{\mf}\Kc_{(\mf)}(k_\mf)\cdot\Kc^{\dagger}(k[\Wc])\int_{\Ic_{\mathrm{sb}}}\prod_{\nf\in\Nc_{\mathrm{sb}}}e^{\pi i\zeta_\nf\cdot\delta L^{2\gamma}\Omega_\nf t_\nf} \cdot e^{\pi i\xi_\nf t_\nf}\,\mathrm{d}t_\nf.
\end{align} Here $\Kc_{(\mf)}$ is a function of $k_\mf$ for each node $\mf$ (branching node or leaf), which is supported in the unit ball $|k_\mf-k_\mf^0|\leq 1$, and has all derivatives up to order $30d$ bounded by $1$. The function $\Zc_{\Is_{\mathrm{sb}}}$ is the same as above, and $\Kc^\dagger$ is a function of $k[\Wc]$ where $\Wc=\{\uf_1^j,\uf_{11}^j,\uf_{21}^j,\uf_{22}^j:1\leq j\leq q_{\mathrm{sk}}\}$. This function $\Kc^\dagger$ is supported in the unit balls $|k_{\uf_1^j}-k_{\uf_1^j}^0|\leq 1$ etc., and is bounded by $1$ (without derivatives). The quantity $g_{\mathrm{sk}}$ represents the total incoherency index for all the regular couples $\Qc_{(\lf,\lf')}$ and regular trees $\Tc_{(\mf)}$ for $\lf,\lf'\in\Lc_{\mathrm{sk}}$ and $\mf\in\Nc_{\mathrm{sk}}$, plus the total incoherency index for all the vines in $\Vs_0$. The quantity $b_{\mathrm{sk}}$ represents the number of \emph{bad} vines in $\Vs_0$, and $q_{\mathrm{sk}}$ is the number of VC's in $\Vs_0$ as above.
\subsection{Analysis of ladders: $L^1$ bound and extra decay} For simplicity, we denote $\Mb(\Gc_{\mathrm{sb}}):=\Mb_{\mathrm{sb}}$ below. To estimate $\Ks_{\Gc_{\mathrm{sb}}}$ in (\ref{redvine2}), we need to deal with the ladders in the molecule $\Mb_{\mathrm{sb}}$, and for this we need to interpolate between two arguments, which will be presented in Section \ref{redcount} below. Both arguments involve the process of representing (\ref{redvine2}) from the point of view of the molecule $\Mb_{\mathrm{sb}}$, which we now describe.

Recall that we identify $(k_1,\cdots,k_{2R})$-decorations $\Is_{\mathrm{sb}}$ of $\Gc_{\mathrm{sb}}$ with $(k_1,\cdots, k_{2R})$-decorations $\Os_{\mathrm{sb}}$ of $\Mb_{\mathrm{sb}}$ via Definition \ref{defdecmol}. For $\nf=\nf(v)$ the corresponding values of $\zeta_\nf\Omega_\nf$ and $-\Gamma_v$ differ by a constant $\Gamma_v^0$ which depends only on $(k_1,\cdots,k_{2R})$, see (\ref{defomegadec}) and (\ref{molegammav}). Also by definition, the time domain $\Ic_{\mathrm{sb}}$ as in (\ref{timegarden}) can be written as the set $\Oc_{\mathrm{sb}}$ which contains the tuples $(t_v)$ for atoms $v\in\Mb_{\mathrm{sb}}$ satisfying that (i) $t_v\in[\Lf_v,\Lf_v+1]$ where $\Lf_v$ is the layer of $v$ as in Definition \ref{defcplmol}, and (ii) $t_{v_1}>t_{v_2}$ if there is a PC bond between $v_1$ and $v_2$ with $v_2$ labeled C (i.e. $\nf(v_1)$ is the parent of $\nf(v_2)$, see Definition \ref{defcplmol}).

Next, The parameters $\xi_\nf$ are renamed as $\xi_v$, the functions $\Kc_{(\mf)}(k_\mf)$ can be written as $\Kc_{(\ell)}(k_\ell)$ for all bonds $\ell$ which satisfy the same assumptions (such as support in $|k_\ell-k_\ell^0|\leq 1$ with fixed values of $k_\ell^0$ etc.), and the function $\Kc^\dagger$ can be written as a uniformly bounded function of the variables $k[\Wb]$ (which we still denote by $\Kc^\dagger$), where $\Wb$ is the set of bonds $\ell\leftrightarrow v$ for all atoms $v\in\Mb_{\mathrm{sb}}$ such that $\nf(v)\in\{\uf_1^j:1\leq j\leq q_{\mathrm{sk}}\}$. We may also define $\widetilde{\epsilon}_{\Os_{\mathrm{sb}}}$ to be equal to $\widetilde{\epsilon}_{\Is_{\mathrm{sb}}}$ in (\ref{redvine2}); in the molecule setting it can be written as a product of $\epsilon_v$ for all atoms $v\in\Mb_{\mathrm{sb}}$ such that $\nf(v)\not\in\{\uf_1^j:1\leq j\leq q_{\mathrm{sk}}\}$, where each $\epsilon_v$ is a suitable function of $(k_\ell:\ell\leftrightarrow v)$ valued in $\{-1,0,1\}$. Similarly we define $\Zc_{\Os_{\mathrm{sb}}}$ to be equal to $\Zc_{\Is_{\mathrm{sb}}}$ in (\ref{redvine2}), which can be written as a product of $\Zc_v$ for some of the atoms $v\in\Mb_{\mathrm{sb}}$. Here each $\Zc_v$ is a suitable function of $(k_{\ell_1},k_{\ell_2})$ for some opposite-direction bonds $\ell_1,\ell_2\leftrightarrow v$, which has value in $\{0,1\}$ depending on the SG, LG and ZG properties of $v$.

With all the above interpretations, we can rewrite (\ref{redvine2}) as
\begin{equation}\label{redvine3}
\Ks_{\Gc_{\mathrm{sb}}}:=\bigg(\frac{\delta}{2L^{d-\gamma}}\bigg)^{n_{\mathrm{sb}}}\sum_{\Os_{\mathrm{sb}}}\widetilde{\epsilon}_{\Os_{\mathrm{sb}}}\cdot\Zc_{\Os_{\mathrm{sb}}}\cdot\prod_{\ell}\Kc_{(\ell)}(k_\ell)\cdot\Kc^{\dagger}(k[\Wb])\int_{\Oc_{\mathrm{sb}}}\prod_{v\in\Mb_{\mathrm{sb}}}e^{\pi i\delta L^{2\gamma}(-\Gamma_v+\Gamma_v^0)t_v}e^{\pi i\xi_v t_v}\,\mathrm{d}t_v.
\end{equation}

Our first argument depends on an estimate which is in the same spirit as Proposition \ref{vineest} (but technically much easier as it does not involve Fourier $L^1$ norms).
\begin{prop}\label{ladderl1new} Let the molecule $\Mb_{\mathrm{sb}}$ be as in Section \ref{redsplice}. Let $\Lb$ be a ladder in $\Mb_{\mathrm{sb}}$, which has length $m$ and incoherency index $g$, such that $\nf(v)\not\in\{\uf_1^j:1\leq j\leq q_{\mathrm{sk}}\}$ for any atom $v\in\Lb$. Let $v_j$ be the (at most $4$) atoms that are not in $\Lb$ but connected to atoms in $\Lb$ by a bond, and let $\ell_j$ be the (at most $4$) bonds that connect atoms in $\Lb$ to atoms not in $\Lb$. Define the expression \[\Ks^\Lb=\Ks^\Lb(t_{v_1},t_{v_2},t_{v_3},t_{v_4},k_{\ell_1},k_{\ell_2},k_{\ell_3},k_{\ell_4})\] similar to (\ref{redvine3}), but with the following changes:
\begin{itemize}
\item We replace the power $(\delta/(2L^{d-\gamma}))^{n_{\mathrm{sb}}}$ by $(\delta/(2L^{d-\gamma}))^{2m+1}$ and remove the $\Kc^\dagger$ factor.
\item In the summation $\sum_{\Os_{\mathrm{sb}}}(\cdots)$, we only sum over the variables $k_\ell$ for $\ell\in\Lb$ (i.e. $\ell$ connecting two atoms in $\Lb$), and treat the other $k_\ell$ variables as fixed. We also replace $\widetilde{\epsilon}_{\Os_{\mathrm{sb}}}$ and $\Zc_{\Os_{\mathrm{sb}}}$ by the product of $\epsilon_v$ and $\Zc_v$ (defined above) over all $v\in \Lb$.
\item In the integration $\int_{\Oc_{\mathrm{sb}}}(\cdots)$, we only integrate over the variables $t_v$ for $v\in\Lb$, and treat the other $t_v$ variables as fixed.
\item In the product $\prod_{v\in\Mb_{\mathrm{sb}}}(\cdots)$, we only include those factors where $v\in\Lb$; in the product $\prod_{\ell}$, we only include those factors where $\ell\in\Lb$.
\end{itemize}
It is easy to check that the function $\Ks^\Lb$ defined this way indeed only depends on the time variables $t_{v_1},\cdots,t_{v_4}$ and vector variables $k_{\ell_1},\cdots,k_{\ell_4}$. Then, uniformly in all $t_{v_j}$ and $k_{\ell_j}$ variables, we have \begin{equation}\label{redladder2}
|\Ks^{\Lb}|\lesssim_1 (C_1\delta)^{m}L^{-(\gamma_1-\sqrt{\eta})g+20d}.
\end{equation}
\end{prop}
\begin{proof} We shall only sketch the proof here, as the argument is a combination of the simpler parts of the proofs of Propositions \ref{sumint} and \ref{vineest}. First, note that the gap associated with the ladder is fixed by the difference $r:=k_{\ell_1} -k_{\ell_2}$. Next, we argue as in Step 2 of the proof of Proposition \ref{vineest} to re-parametrize the sum over the decoration as we did for the ladder component of Vine (II), see Figure \ref{fig:vineAnnotated} (A). This leads to the new set of variables $(x_j,y_j)_{1\leq j\leq m}$, where each $(x_j,y_j)$ corresponds to a double bond in the ladder connecting two atoms $v$ and $w$, such that $\Gamma_v=2x_j\cdot y_j$ and $\Gamma_w=-2x_j \cdot y_j+2r\cdot \zeta_j$. Here $\zeta_j=\alpha_jx_j+\beta_j y_j+\theta_j r$ for some $\alpha_j, \beta_j, \theta_j\in \{0, \pm 1\}$ with $\alpha_j^2+\beta_j^2\neq 0$.

Recall that in Step 2 of the proof of Proposition \ref{sumint} we estimated each quantity $\Mc_j^i$ pointwise, leading to the final $L^\infty$ bound of $\Ic$. Here we can do exactly the same thing, and estimate the sum in $(x_j, y_j)$ by either $\delta^{-1}L^{2(d-\gamma)}$ (if the corresponding atoms $v$ and $w$ are in the same layer) or $\delta^{-1}L^{2(d-\gamma)}L^{-(\gamma_1-\sqrt{\eta})}$ (if $v$ and $w$ are in different layers). In the end we are left with the last remaining sum, which is taken over the decoration of the last double bond in the ladder, and can be bounded trivially by $L^{2d}$. Putting together, this then implies (\ref{redladder2}).
\end{proof}
Our second argument depends on the following $L^1$ estimate for time integral, which is basically the same as in \cite{DH21,DH23} but with some differences due to layerings.
\begin{prop}\label{ladderl1old} Let the molecule $\Mb_{\mathrm{sb}}$ be as in Section \ref{redsplice}. Consider all the maximal ladders in $\Mb_{\mathrm{sb}}$, which we denote by $\Lb_j\,(1\leq j\leq q_{\mathrm{sb}})$ with length $z_j$. Let $\rho_{\mathrm{sb}}:=q_{\mathrm{sb}}+m_{\mathrm{sb}}$, where $m_{\mathrm{sb}}$ is the number of atoms not in any of these maximal ladders. Note the decoration $\Os_{\mathrm{sb}}$ in $\Ks_{\Gc_{\mathrm{sb}}}$ in (\ref{redvine3}) satisfies $|k_\ell-k_\ell^0|\leq 1$ and $\widetilde{\epsilon}_{\Os_{\mathrm{sb}}}\cdot\Zc_{\Os_{\mathrm{sb}}}\neq 0$; we fix a dyadic number $P_j\in [L^{-1},1]\cup\{0\}$ for each $1\leq j\leq q_{\mathrm{sb}}$, and further restrict $\Os_{\mathrm{sb}}$ such that the gap $r_j$ of the ladder $\Lb_j$ satisfies $|r_j|\sim_0 P_j$ (or $|r_j|\gtrsim_01$ if $P_j=1$).

Consider the integral $\int_{\Oc_{\mathrm{sb}}}(\cdots)$ in (\ref{redvine3}), which is a function of $\alpha_v:=\delta L^{2\gamma}(-\Gamma_v+\Gamma_v^0)+\xi_v$ where $v\in\Mb_{\mathrm{sb}}$, and denote it by $\Bb(\alpha[\Mb_{\mathrm{sb}}])$. Let $\lfloor\alpha_v\rfloor:=\sigma_v$, and we restrict $\Gamma_v$ (and hence $\alpha_v$ and $\sigma_v$) to those that \emph{may occur} for \emph{some} decoration $\Os_{\mathrm{sb}}$ satisfying all the above restrictions, including the further restrictions on gaps for $\Lb_j$. Then, under these assumptions, we have
\begin{equation}\label{ladderl1}
\sum_{\sigma[\Mb_{\mathrm{sb}}]}\sup_{\alpha[\Mb_{\mathrm{sb}}]:|\alpha_v-\sigma_v|\leq 1}|\Bc(\alpha[\Mb_{\mathrm{sb}}])|\lesssim_1(C_1\delta^{-1/2-5\nu})^{n_{\mathrm{sb}}}\cdot L^{\delta^{5\nu}\rho_{\mathrm{sb}}} \prod_{j=1}^{q_{\mathrm{sb}}}\Xf_j^{z_j},
\end{equation} where $\Xf_j:=\min((\log L)^2,1+\delta L^{2\gamma}P_j)$.
\end{prop}
\begin{proof} The proof is divided into several steps.

\smallskip
\uwave{Step 1: Reduction and localization.} First recall the following elementary inequality (see also (10.5) of \cite{DH21}): for any unit box $[0,1]^n$, any function $F:[0,1]^{\Mb_{\mathrm{sb}}}\to\Rb$ and any $\Lf_v\in\Rb$ we have
\begin{equation}\label{elementarylem}\sup_{\alpha\in[0,1]^{\Mb_{\mathrm{sb}}}}|F(\alpha)|\leq\sum_{A\subset\Mb_{\mathrm{sb}}}\int_{[0,1]^{\Mb_{\mathrm{sb}}}}\bigg|\prod_{v\in A}(\partial_{\alpha_v}-\pi i\Lf_v)F(\alpha)\bigg|\,\mathrm{d}\alpha\end{equation} For the proof, simply get rid of $\Lf_v$ by multiplying by the $e^{-\pi i\Lf_v\alpha_v}$ factor, and then proceed by induction in $n$.

Using (\ref{elementarylem}), with an acceptable loss of $C_0^{n_{\mathrm{sb}}}$, we can reduce the left hand side of (\ref{ladderl1}) to the integral of $|\Bc(\alpha[\Mb_{\mathrm{sb}}])|$ on the union of the unit boxes $\lfloor\alpha_v\rfloor=\sigma_v$ for all possible $(\sigma_v)$. The derivatives $\partial_{\alpha_v}-\pi i\Lf_v$ turn into factors $\pi i(t_v-\Lf_v)$ in the $\int_{\Oc_{\mathrm{sb}}}(\cdots)$ integral in (\ref{redvine3}), which are harmless due to the restriction $t_v\in[\Lf_v,\Lf_v+1]$, so we will ignore them below. Denote by $\Wc$ the above union of unit boxes, next we will specify some localization properties of $\Wc$.

Suppose that $\alpha=(\alpha_v)$ is defined from some decoration $\Os_{\mathrm{sb}}$ satisfying all the restrictions stated in the assumptions, and fix one ladder $\Lb_j$ satisfying $\Xf_j=1+\delta L^{2\gamma}P_j \leq (\log L)^2$, and any atoms $v$ and $w$ connected by a double bond, we have $\Gamma_v+\Gamma_w=2r_j\cdot y_w$ where $r_j$ is the gap of $\Lb_j$ and $y_{w}$ is the difference of the two vectors $k_\ell$ for two different single bonds at $v$ and $w$. Using that $|k_\ell-k_\ell^0|\leq 1$, we see that $\Gamma_v+\Gamma_w=2r_j\cdot z_w^0+\Oc_0(P_j)$ where $z_w^0$ is a fixed vector (note however that $r_j$ is not a fixed vector), and thus $\alpha_v+\alpha_w=r_j\cdot y_w^0+c_w+\Oc_0(\Xf_j)$ for some  fixed vector $y_w^0$ and constant $c_{w}$ that are independent of the decoration. The integer part of this $\Oc_0(\Xf_j)$ has at most $\Oc_0(\Xf_j)$ choices, so we may fix one choice for each $j$ and each $(v,w)$, at an additional cost of \begin{equation}\label{ladderloss}C_0^{n_{\mathrm{sb}}}\cdot \prod_{j:\,\Xf_j=1+\delta L^{2\gamma}P_j}\Xf_j^{z_j+1};\end{equation} this can be viewed as dividing $\Wc$ into subsets, where the total number of subsets is given by (\ref{ladderloss}). Once each integer part is fixed, we can then absorb it into $c_w$ and write $\alpha_v+\alpha_w=r_j\cdot y_w^0+c_w+\Oc_0(1)$. Note also that the $z_j+1$ in (\ref{ladderloss}) may be replaced by $z_j$, as the loss caused by this is at most $(\log L)^{2q_{\mathrm{sb}}}$ which is acceptable. 

Now, by Lemma \ref{lincomblem}, for each $j$ we can find $q=q_j\leq d$ and atoms $w_1^j,\cdots,w_q^j\in\Lb_j$, such that for any $w\in \Lb_j$ we can write
\[y_w^0=\sum_{i=1}^q\lambda_iy_{w_i^j}^0,\quad |\lambda_i|\leq C_0.\] This implies that for each pair of atoms $(v,w)$ in $\Lb_j$ as above, we have
\begin{equation}\label{localize1}\alpha_v+\alpha_w=\sum_{i=1}^q\lambda_i\big(\alpha_{v_i^j}+\alpha_{w_i^j}\big)+d_w+\Oc_0(1),\end{equation} where $d_w$ is another constant that is independent of the decoration. Clearly (\ref{localize1}) does not change if each $\alpha_v$ is shifted by $\Oc_0(1)$, so it is satisfied by any point in $\Wc$. In addition, the difference of any two vectors in the decoration $\Os_{\mathrm{sb}}$ is a vector that belongs to $\Zb_L^d$ and a fixed unit ball, so for each $v$ the quantity $\Gamma_v$ may take at most $L^{20d}$ values for all possible decorations. This means that if $\alpha=(\alpha_v)\in \Wc$, then for each $v$ we must have $\alpha_v\in \Sc_v$ where $\Sc_v$ is the union of at most $L^{20d}$ unit intervals that depends only on $v$.

\smallskip
\uwave{Step 2: Monotonicity.} Fix a ladder $\Lb_j$ with $\Xf_j=1+\delta L^{2\gamma}P_j$, and let $w_i^j\,(1\leq i\leq q=q_j)$ be as above. Consider the sequence $\{\max(t_v,t_w)\}$ where $(v,w)$ runs over all pairs of atoms in $\Lb_j$ connected by a double bond. We claim that this sequence can be divided into two subsequences, each of which is \emph{monotnonic} when viewed from left to right as in Figure \ref{fig:vines}.

In fact, fix any pair $(v,w)$, and assume (say) $t_v>t_w$. Note that $\Mb_{\mathrm{sb}}=\Mb(\Gc_{\mathrm{sb}})$; define $v^+$ to be the unique atom such that $\nf(v^+)$ is the parent node of $\nf(v)$ in $\Gc_{\mathrm{sb}}$. Then $v^+\neq w$ and is adjacent to $v$, so it must be the atom to the left or right of $v$ in Figure \ref{fig:vines}. Let $w^+$ be the atom connected to $v^+$ by a double bond, then $\max(t_{v^+},t_{w^+})\geq t_{v^+}>t_v=\max(t_v,t_w)$. This means that the sequence $\{\max(t_v,t_w)\}$ has \emph{no local maximum}, thus it is formed by at most two monotonic subsequences when viewed from let to right as in Figure \ref{fig:vines}.

Now, for each pair $(v,w)$ in $\Lb_j$ such that $w\not\in\{w_i^j:1\leq i\leq q=q_j\}$, we may specify whether $t_v>t_w$ or $t_v<t_w$ in addition to the existing inequality relations between $(t_v)$ forced by parent-child relations. With an acceptable loss of $C_0^{n_{\mathrm{sb}}}$, we may fix such a set of inequalities; then for each such $(v,w)$, we choose the atom corresponding to the bigger time variable. Define the collection of these atoms by $\Xb$, and let $\Yb=\Mb_{\mathrm{sb}}\backslash\Xb$. Note that $\Yb$ consists of (i) the atoms in $(v,w)$ corresponding to the smaller time variable, (ii) the atoms $w_i^j$ and corresponding $v_i^j$ for $1\leq i\leq q$, (iii) all atoms in ladders $\Lb_j$ where $\Xf_j=(\log L)^2$, and (iv) all atoms that are not in any ladder $\Lb_j$.

\smallskip
\uwave{Step 3: Partial time integration.} Now we analyze the target expression, which is
\begin{equation}\label{targettimeint}\int_{\Wc}\bigg|\int_{\widetilde{\Oc}_{\mathrm{sb}}}\prod_{v\in\Xb\cup\Yb}e^{\pi i\alpha_vt_v}\,\mathrm{d}t_v\bigg|\,\mathrm{d}\alpha[\Xb]\mathrm{d}\alpha[\Yb]
\end{equation} thanks to the reduction in Step 1. Here $\widetilde{\Oc}_{\mathrm{sb}}$ is the set $\Oc_{\mathrm{sb}}$ but with the extra inequalities in Step 2 added. In this integral, we will first integrate in the $t[\Yb]$ variables; clearly we may only integrate the product for $v\in \Yb$, and the result may depend on $\alpha[\Yb]$ and $t[\Xb]$, so we denote it by \begin{equation}\label{targettimeint1}\Hc=\Hc(t[\Xb],\alpha[\Yb])=\int_{(t[\Xb],t[\Yb])\in\widetilde{\Oc}_{\mathrm{sb}}}\prod_{v\in\Yb}e^{\pi i\alpha_vt_v}\,\mathrm{d}t_v.\end{equation}

Note that $\alpha_v\in \Sc_v$ for any $\alpha[\Mb_{\mathrm{sb}}]\in\Wc$ as shown above, moreover we have
\begin{equation}\label{targettimeint2}\int_{\alpha_v\in\Sc_v\,\!(\forall v\in \Yb)}|\Hc(t[\Xb],\alpha[\Yb])|\,\mathrm{d}\alpha[\Yb]\leq (C_0\log L)^{|\Yb|}
\end{equation} uniformly in $t[\Xb]$. This can be proved by integrating in the $t[\Yb]$ variables in a certain order in (\ref{targettimeint1}) and proceeding by induction, as in Lemma 10.2 of \cite{DH21} or Proposition 2.3 of \cite{DH19}. The only caution is that each time we should integrate in the $t_v$ variable where $v$ is a \emph{minimal element} under the ancestor-descendant partial ordering (i.e. $v_1\succeq v_2$ if and only if $\nf(v_1)$ is an ancestor of $\nf(v_2)$ in $\Gc_{\mathrm{sb}}$). Note that the lower bound of this integration in $t_v$ is always a constant, and the upper bound is either a constant or another variable $t_{v^+}$ where $\nf(v^+)$ is the parent of $\nf(v)$. After integrating in $t_v$, we obtain a factor that is bounded by $\langle \alpha_v+\beta\rangle^{-1}$, where $\beta$ is either $0$ or some quantity determined by the $\alpha_{v_*}$ variables where $t_{v_*}$ has been integrated out in previous steps (this bound is true even when $|\alpha_v+\beta|\leq 1$ because the interval of integration in $t_v$ has length at most $1$). This factor can be bounded upon integrating in $\alpha_v$, using the elementary inequality
\[\sup_{\beta\in\Rb}\int_{\Sc_v}\frac{1}{\langle \alpha_v+\beta\rangle}\,\mathrm{d}\alpha_v\leq C_0\log L.\] We are then left with an expression that is the sum of at most four integrals with similar forms (possibly depending on the size of $t_{v^+}$) but one fewer integrated variable; next we simply integrate in $t_{v'}$ where $v'$ is a minimal element in the same partial ordering after removing $v$, and so on.

\smallskip
\uwave{Step 4: Completing the proof.} With (\ref{targettimeint2}) we can now complete the proof of (\ref{ladderl1}). First we perform the reduction in Step 1 and reduce to estimating (\ref{targettimeint}) with a loss of $\prod_{j=1}^{q_{\mathrm{sb}}}\Xf_j^{z_j}$. Then we apply (\ref{targettimeint2}) and also apply the fact from Step 2 that, all the variables $t[\Xb\cap\Lb_j]$ belong to $[0,\Df]$ and can be arranged into \emph{at most two monotonic sequences}. Note also that, by definition of $\Xb$ and $\Yb$, once $\alpha[\Yb]$ is fixed, each $\alpha_v\,(v\in \Xb)$ must belong to an interval of length $\Oc_0(1)$ thanks to (\ref{localize1}). Putting together, this leads to the bound
\[(\ref{targettimeint})\leq \int_{(\mathrm{monotonicity})}\,\mathrm{d}t[\Xb]\int_{\Wc}|\Hc(t[\Xb],\alpha[\Yb])|\,\mathrm{d}\alpha[\Xb]\mathrm{d}\alpha[\Yb]\leq C_0^{n_{\mathrm{sb}}}(\log L)^{|\Yb|}\prod_{j:\,\Xf_j=1+\delta L^{2\gamma}P_j}\frac{\Df^{z_j}}{(z_j-d)!},\] where the integral in $t[\Xb]$ is taken under the above monotonicity conditions. Note that
\[(\log L)^{|\Yb|}\leq \prod_{j:\,\Xf_j=(\log L)^2}\Xf_j^{z_j}\cdot\prod_{j:\,\Xf_j=1+\delta L^{2\gamma}P_j}(\log L)^{z_j}\cdot (\log L)^{10d\rho_{\mathrm{sb}}}\] due to the definition of $\Yb$; note also that in view of the $(\delta^{-1/2-5\nu)^{n_{\mathrm{sb}}}}$ factor in (\ref{ladderl1}), we have
\[\delta^{(1/2+5\nu)n_{\mathrm{sb}}}\leq\prod_{j:\,\Xf_j=1+\delta L^{2\gamma}P_j}\delta^{(1+10\nu)z_j}.\] Therefore, the desired bound (\ref{ladderl1}) follows from the elementary inequality
\[\frac{(\log L)^{z_j}\Df^{z_j}\delta^{(1+10\nu)z_j}}{(z_j-d)!}\leq(\log L)^{2d}\exp\big(\Df\delta^{1+10\nu}\log L\big)\leq L^{\delta^{8\nu}}\] for each $j$, where we recall that $\Df\cdot\delta=\tau_*\ll \delta^{-\nu}$ from (\ref{defD}).
\end{proof}
\subsection{Properties of $\Mb_{\mathrm{sb}}$} We state some properties of the molecule $\Mb_{\mathrm{sb}}=\Mb(\Gc_{\mathrm{sb}})$, which are essentially the same as in \cite{DH23}.
\begin{prop}\label{subpro} Consider the garden $\Gc_{\mathrm{sb}}$ and the corresponding molecule $\Mb_{\mathrm{sb}}=\Mb(\Gc_{\mathrm{sb}})$. Consider also a $(k_1,\cdots,k_{2R})$-decoration $\Is_{\mathrm{sb}}$ of $\Gc_{\mathrm{sb}}$ such that $\widetilde{\epsilon}_{\Is_{\mathrm{sb}}}\cdot\Zc_{\Is_{\mathrm{sb}}}\neq 0$, which we identify with a $(k_1,\cdots,k_{2R})$-decoration $\Os_{\mathrm{sb}}$ of $\Mb_{\mathrm{sb}}$ satisfying $\widetilde{\epsilon}_{\Os_{\mathrm{sb}}}\cdot\Zc_{\Os_{\mathrm{sb}}}\neq 0$ via Definition \ref{defdecmol} (this can be viewed as inherited from certain decoration of $\Mb(\Gc_{\mathrm{sk}})$ or $\Gc_{\mathrm{sk}}$). Then we have the followings:
\begin{enumerate}[{(1)}]
\item The molecule $\Mb_{\mathrm{sb}}$ is connected, has $n_{\mathrm{sb}}$ atoms and $2n_{\mathrm{sb}}-R$ bonds.
\item $\Mb_{\mathrm{sb}}$ contains a collection $\Cs_1$ of disjoint SG vine-like objects, such that each of them is either an HV, or a (CN) vine or a root (CL) vine. We can also specify a total of $q_{\mathrm{sk}}$ atoms which we call \emph{hinge atoms}; any triple bond must have an endpoint that is a hinge atom.
\item We say an atom $v$ is \emph{degenerate} if there are two bonds $(\ell_1,\ell_2)$ at $v$ of opposite directions and $k_{\ell_1}=k_{\ell_2}$ (note this is equivalent to being ZG as in Definition \ref{defdiff}), and that it is \emph{tame} if the values of $k_\ell$ for all $\ell\leftrightarrow v$ are equal, and equals a fixed value in $(k_1,\cdots,k_{2R})$ if $d(v)<4$. Then any degenerate/ZG atom that is not a hinge atom must be tame.
\item Any SG vine-like object which is not a subset of an SG vine-like object in $\Cs_1$, and also \emph{does not contain} a hinge atom, must be a subset of a DV or an LG or ZG vine-like object\footnote{Note that a ZG vine-like object means that the joints are degenerate, hence tame by (3).} in $\Mb_{\mathrm{sb}}$, which is disjoint with the vine-like objects in $\Cs_1$ and does not contain any hinge atom. Note that here we must be in one of the cases defined in Lemma \ref{subsetvc}.
\item For any hinge atom $v$, we can find two bonds $(\ell_1,\ell_2)$ at $v$, such that the triple $(v,\ell_1,\ell_2)$ always has SG. If $v$ is also a joint of an SG vine (or HV) in $\Cs_1$, then $(\ell_1,\ell_2)$ is chosen to belong to this vine (or the adjoint vine of this HV).
\item Any SG vine-like object $\Ub$ which is not a subset of an SG vine-like object in $\Cs_1$, and also \emph{contains} a hinge atom, \emph{must} contain a hinge atom $v$, such that either $v$ is an interior atom of $\Ub$, or $v$ is a joint of $\Ub$ and exactly one bond in $(\ell_1,\ell_2)$ belongs to $\Ub$, where $(\ell_1,\ell_2)$ is defined in (5) above.
\end{enumerate}
\end{prop}
\begin{proof} Recall the reduction process from $\Mb(\Gc_{\mathrm{sk}})$ to $\Mb_{\mathrm{sb}}$ by merging all VC's in $\Vs_0$. Now each VC in $\Vs_0$ is reduced to an atom, and we define them to be the hinge atoms (in fact they are just those atoms $v$ where $\nf(v)\in\{\uf_1^j:1\leq j\leq q_{\mathrm{sk}}\}$). Moreover, by Corollary \ref{blockchainprop}, we know that each SGVC, SGHV and SGHVC in $\Cs$ is reduced to either a single atom, or an SGHV, or an SG vine that is also a (CN) vine, or an SG vine that is also a root (CL) vine; we then define $\Cs_1$ to be the collection of all such vine-like objects (other than single atoms). For each hinge atom $v$, it is merged from an SGVC in $\Cs$ (say $\Ub$), then we choose $(\ell_1,\ell_2)$ in (5) such that in the molecule $\Mb(\Gc_{\mathrm{sk}})$ they are the two bonds at one joint of $\Ub$ that do not belong to $\Ub$ (this must exist because $\Ub$ is not a root VC). Note that if $d(v)=4$ then the other two bonds $(\ell_3,\ell_4)$ also have the same properties. If $v$ is also a joint of an SG vine (or HV) in $\Cs_1$, then we specify $(\ell_1,\ell_2)$ to belong to this vine (or the joint vine of this HV); if not, we can choose freely between $(\ell_1,\ell_2)$ and $(\ell_3,\ell_4)$.

The above definitions are the same as in Proposition 9.1 of \cite{DH23} (with the trivial difference that the DV's are not included in the collection $\Cs_1$ defined in (2), so the case of an SG vine-like object that is a subset of a DV in $\Mb_{\mathrm{sb}}$ is now included in (4)). This means that conclusions (2)--(6) also follow directly from the proof of Proposition 9.1 of \cite{DH23}. Note that here we are dealing with gardens, but the proof in \cite{DH23} does not distinguish between gardens and couples. As for conclusion (1), it simply follows from Proposition \ref{propcplmol} and the fact that $\Gc_{\mathrm{sb}}$ (just as $\Gc_{\mathrm{sk}}$) is an irreducible garden.
\end{proof}
\section{Stage 2 reduction and proof of Proposition \ref{layergarden}}\label{redcount}
\subsection{Reduction to counting estimates}\label{redcount0} To prove Proposition \ref{layergarden}, we need to consider the expression $\Ks_{p+1,n}^{\mathrm{tr}}$ defined in (\ref{expgarden}), and the difference term $\Rs_{p+1,n}$ defined in (\ref{estcouple}). These involve four types of terms:
\begin{enumerate}[{(a)}]
\item The term \begin{equation}\label{finalterma}\Ks_{p+1,n}^{\mathrm{tr}}(\zeta_1,\cdots,\zeta_{2R},k_1,\cdots,k_{2R})\end{equation} as in (\ref{expgarden}) for $R\geq 2$;
\item The term \begin{equation}\label{finaltermb}\Ks_{p+1,n}^{\mathrm{tr,nr}}(+,-,k,k)\end{equation} as in (\ref{expgarden}) for $R=1$, where we only sum over \emph{non-regular} couples;
\item The term \begin{equation}\label{finaltermc}\Ks_{p+1,n}^{\mathrm{tr,nc}}(+,-,k,k)\end{equation} as in (\ref{expgarden}) for $R=1$, where we only sum over \emph{regular, non-coherent} layered couples;
\item The term \begin{equation}\label{finaltermd}\Ks_{p+1,n}^{\mathrm{tr,co}}(+,-,k,k)-\Uc_{n/2}(p+1,k)\end{equation} as in (\ref{expgarden}) for $R=1$, where we only sum over the \emph{regular, coherent} layered couples, and $n$ is even, $\Uc_{n/2}$ is as in (\ref{estcouple}). 
\end{enumerate}

The terms (\ref{finaltermc}) and (\ref{finaltermd}) will be estimated using Propositions \ref{layerreg1}, \ref{proplayer1}, \ref{proplayer2} and \ref{proplayer3}. The terms (\ref{finalterma}) and (\ref{finaltermb}) will be estimated in this section. In both cases, note that the estimates (\ref{estgarden}) and (\ref{estcouple2}) involves the weights $(\langle k_1\rangle\cdots\langle k_{2R}\rangle)^{4\Lambda_{p+1}}$ for both gardens and couples, and the derivatives $\partial_k^\rho\,(|\rho|\leq 4\Lambda_{p+1})$ for couples, which we now discuss.

First, for each $1\leq j\leq 2R$ we have \[\langle k_j\rangle^{4\Lambda_{p+1}}\leq (3n_j)^{4\Lambda_{p+1}}\cdot\langle k_{\lf_j}\rangle^{4\Lambda_{p+1}}\] for some leaf $\lf_j\in\Tc_j$, where $n_j$ is the order of the tree $\Tc_j$, since $k_j$ is an algebraic sum of all such $k_{\lf_j}$. Since $4\Lambda_{p+1}\leq C_2$, we have $(3n_j)^{4\Lambda_{p+1}}\leq 2^{n_j}\cdot C_2$, and putting together all $1\leq j\leq 2R$ yields the factors $2^n$ and $C_2^R$. Moreover, if we include the extra factor $\langle k_{\lf_j}\rangle^{4\Lambda_{p+1}}$ to the input function $F_{\Lf_{\lf_j}}(k_{\lf_j})$ in (\ref{defkg0}), then this function is still bounded in $\Sf^{40d,40d}$ because of $F_q(k)=\varphi(\delta q,k)+\Rs_q(k)$, the bound (\ref{ansatz2}) and $\Lambda_{p}=40\Lambda_{p+1}$, but with norm being $C_2$ instead of $C_1$ for the $\varphi(\delta q,k)$ term. However this happens for at most $R$ input factors, so again the resulting loss is at most $C_2^R$. These losses can be covered by proving slightly stronger bounds with $L^{-(R-1)/4}$ replaced by $L^{-(R-1)(1/4+1/100)}$ in (\ref{estgarden}), and $L^{-(4\theta_{p+1}+\theta_p)/5}$ replaced by $L^{-(3\theta_{p+1}+2\theta_p)/5}$ in (\ref{estcouple2}), and $\delta^{2\nu n}$ replaced by $\delta^{3\nu n}$ in both, but these stronger bounds will follow from same arguments below.

Similarly, note that for couples ($R=1$), the function $\Ks_{p+1,n}^{\mathrm{tr}}$ is defined for all $k\in\Rb^d$ if we extend the input functions $F_q(k)=\varphi(\delta q,k)+\Rs_q(k)$ for $q\leq p$ by using (\ref{ansatz2}). If we translate $k$ and all $k_\lf$ by the same vector the whole expression is invariant, so taking any $\partial_k$ derivative is equivalent to taking exactly one $\partial_{k_\lf}$ derivative for some leaf $\lf\in\Gc$. If we take at most $4\Lambda_{p+1}\leq C_2$ derivatives, then we get at most $n^{C_2}$ terms, where at most $C_2$ input factors are affected for each term, and each affected factor is still bounded in $\Sf^{40d,40d}$ but by $C_2$ instead of $C_1$, for the same reason as above. This again leads to loss of at most $C_2n^{C_2}\leq C_22^n$, which is acceptable provided we prove the slightly stronger bounds above.

Therefore, from now on we will focus on the proof of (\ref{estgarden}) and (\ref{estcouple2}) with $4\Lambda_{p+1}$ replaced by $0$, i.e. the supremum norm in $k$ (or $k_j$). In the proof below we will fix the values of $k$ (or $k_j$), and all our estimates will be uniform in them; the input factors $F_q(k)=\varphi(\delta q,k)+\Rs_q(k)$ are then bounded in $\Sf^{40d,40d}$ by $C_1$.

\smallskip
With the above reductions, we now start with the term $\Ks_{p+1,n}^{\mathrm{tr}}$ in (\ref{finalterma}) or $\Ks_{p+1,n}^{\mathrm{tr,nr}}$ in (\ref{finaltermb}), which is a sum of $\Kc_\Gc$ as in (\ref{expgarden}). We then perform the stage 1 reduction in Section \ref{stage1red} and consider the sum over a single congruence class as defined in Definition \ref{defcong}. Note that we still need to count the number of congruence classes, which will be left to Section \ref{secrigid} below; for now we will fix this single sum and reduce it to (\ref{redvine3}) as done in Section \ref{stage1red}.
\subsubsection{Ladders and $L^1$ bounds}\label{l1reduction} Now start from the expression $\Ks_{\Gc_{\mathrm{sb}}}$ in (\ref{redvine3}). Recall the various parameters defined in Proposition \ref{ladderl1old}: $n_{\mathrm{sb}}$ is the order of $\Gc_{\mathrm{sb}}$, $q_{\mathrm{sb}}$ is the number of maximal ladders in $\Mb_{\mathrm{sb}}=\Mb(\Gc_{\mathrm{sb}})$, $z_j\,(1\leq j\leq q_{\mathrm{sb}})$ are the lengths of these ladders, $m_{\mathrm{sb}}$ is the number of atoms not in these ladders, and $\rho_{\mathrm{sb}}=q_{\mathrm{sb}}+m_{\mathrm{sb}}$. Also recall the parameters coming from the stage 1 reduction process in Sections \ref{redsplice0}--\ref{redsplice}: $g_{\mathrm{sk}}$ is the total incoherency index for all the regular couples $\Qc_{(\lf,\lf')}$ and regular trees $\Tc_{(\mf)}$ for $\lf,\lf'\in\Lc_{\mathrm{sk}}$ and $\mf\in\Nc_{\mathrm{sk}}$, plus the total incoherency index for all the vines in $\Vs_0$, $b_{\mathrm{sk}}$ represents the number of bad vines in $\Vs_0$, and $q_{\mathrm{sk}}$ is the number of VC's in $\Vs_0$. Note that $q_{\mathrm{sk}}$ also equals the number of hinge atoms in $\Mb_{\mathrm{sb}}$ (Proposition \ref{subpro} (2)). In addition, we also define $g_{\mathrm{sb}}$ to be the total incoherency index for all the maximal ladders in Proposition \ref{ladderl1old}.

The goal until the end of Section \ref{secrigid} is to prove the following two estimates:
\begin{equation}\label{finalsbest1}|\Ks_{\Gc_{\mathrm{sb}}}|\lesssim_2(C_1\sqrt{\delta})^{n_{\mathrm{sb}}}L^{-\eta \cdot g_{\mathrm{sb}}+80d\cdot(\rho_{\mathrm{sb}}+q_{\mathrm{sk}})},
\end{equation}
\begin{equation}\label{finalsbest2}|\Ks_{\Gc_{\mathrm{sb}}}|\lesssim_1(C_1\delta^{-10\nu})^{n_{\mathrm{sb}}}L^{-\eta^6\cdot(\rho_{\mathrm{sb}}+q_{\mathrm{sk}})}\cdot L^{-(R-1)(1/4+1/50)},
\end{equation} which we then interpolate to prove Proposition \ref{layergarden}. The proof of (\ref{finalsbest1}) is very easy with Proposition \ref{ladderl1new}, which we describe below.

Starting with (\ref{redvine3}), we choose all the maximal ladders in Proposition \ref{ladderl1old}; if any of them contains hinge atoms, then we replace it by the shorter ladders separated by these hinge atoms. This leads to a new set of $q'\leq q_{\mathrm{sb}}+q_{\mathrm{sk}}$ maximal ladders not containing hinge atom, such that the total incoherency index equals $g'\geq g_{\mathrm{sb}}-2q_{\mathrm{sk}}$. For each new ladder $\Lb$, we apply Proposition \ref{ladderl1new}, noting that the function $\Kc^\dagger(k[\Wb])$ does not depend on any of the $k_\ell$ variables for $\ell\in\Lb$; the resulting functions $\Ks^\Lb$ is then estimated using (\ref{redladder2}). The rest of (\ref{redvine3}), namely summation and integration over the $k_\ell$ and $t_v$ variables for $\ell$ and $v$ not in any ladder, is estimated trivially, using the fact that each $k_\ell$ belongs to a fixed unit ball and that each $t_v$ belongs to a fixed unit interval. Note that the number of atoms not in these new ladders is $m'\leq m_{\mathrm{sb}}+2q_{\mathrm{sk}}$, so we obtain that
\begin{equation}|\Ks_{\Gc_{\mathrm{sb}}}|\leq (C_1\sqrt{\delta})^{n_{\mathrm{sb}}-m'}\delta^{m'}L^{-(\gamma_1-\sqrt{\delta})g'+20dq'}\cdot L^{20dm'},
\end{equation} which easily implies (\ref{finalsbest1}) due to the upper bounds for $(q',m')$ and lower bound for $g'$.

From now on we will focus on the proof of (\ref{finalsbest2}). Starting with (\ref{redvine3}), which we may rewrite as
\begin{equation}\label{redvine4}
\Ks_{\Gc_{\mathrm{sb}}}=\bigg(\frac{\delta}{2L^{d-\gamma}}\bigg)^{n_{\mathrm{sb}}}\sum_{\Os_{\mathrm{sb}}}\widetilde{\epsilon}_{\Os_{\mathrm{sb}}}\cdot \Zc_{\Os_{\mathrm{sb}}}\cdot\prod_{\ell}\Kc_{(\ell)}(k_\ell)\cdot\Kc^{\dagger}(k[\Wb])\cdot\Bc(\alpha[\Mb_{\mathrm{sb}}])
\end{equation} using the definition in Proposition \ref{ladderl1old}, with the notions $\alpha_v:=\delta L^{2\gamma}(-\Gamma_v+\Gamma_v^0)+\xi_v$ etc. Then we may fix the values of $P_j\,(1\leq j\leq q_{\mathrm{sb}})$ as in Proposition \ref{ladderl1old}; this loses a factor $(\log L)^{q_{\mathrm{sb}}}$, which can be covered in view of the gain in (\ref{finalsbest2}) (say if we prove the stronger bound with $\eta^6$ replaced by $2\eta^6$, which follows from same arguments).

Next, we shall fix the values of $\sigma_v=\lfloor\alpha_v\rfloor$, and bound (\ref{redvine4}) by
\begin{multline}\label{redvine5}
|\Ks_{\Gc_{\mathrm{sb}}}|\leq\bigg(\frac{\delta}{2L^{d-\gamma}}\bigg)^{n_{\mathrm{sb}}}\cdot\sum_{\sigma[\Mb_{\mathrm{sb}}]}\sup_{\alpha[\Mb_{\mathrm{sb}}]:|\alpha_v-\sigma_v|\leq 1}|\Bc(\alpha[\Mb_{\mathrm{sb}}])|\\
\times\sup_{\sigma[\Mb_{\mathrm{sb}}]}\sum_{\Os_{\mathrm{sb}}:|\alpha_v-\sigma_v|\leq 1}\big|\widetilde{\epsilon}_{\Os_{\mathrm{sb}}}\cdot \Zc_{\Os_{\mathrm{sb}}}\cdot\prod_{\ell}\Kc_{(\ell)}(k_\ell)\cdot\Kc^{\dagger}(k[\Wb])\big|.
\end{multline} Note that the $\sum_{\sigma[\Mb_{\mathrm{sb}}]}(\cdots)$ factor in (\ref{redvine5}) is bounded by (\ref{ladderl1}), so to prove (\ref{finalsbest2}) we only need to bound the $\sup_{\sigma[\Mb_{\mathrm{sb}}]}(\cdots)$ factor uniformly in $(\sigma_v)$. This corresponds to solving a counting problem for all the $(k_1,\cdots,k_{2R})$-decorations $\Os_{\mathrm{sb}}$ of the molecule $\Mb_{\mathrm{sb}}$. Moreover, the decoration $\Os_{\mathrm{sb}}$ must satisfy the following requirements:
\begin{enumerate}[{(a)}]
\item It is restricted by some fixed $(k_\ell^0)$ and $(\beta_v)$ in the sense of Definition \ref{defdecmol}, where $(k_\ell^0)$ and $(\beta_v)$ are fixed parameters (the $(\beta_v)$ is determined by $(\sigma_v,\Gamma_v^0,\xi_v)$ etc.).
\item It is in the support of $\widetilde{\epsilon}_{\Os_{\mathrm{sb}}}\cdot\Zc_{\Os_{\mathrm{sb}}}$, in particular it must satisfy all the requirements in Proposition \ref{subpro}.
\item The gap $r_j$ or the ladder $\Lb_j$ is such that $|r_j|\sim_0 P_j$.
\end{enumerate}

For later use, we shall impose some further requirements on this decoration $\Os_{\mathrm{sb}}$, which we describe below. Consider all the SG vine-like objects in $\Mb_{\mathrm{sb}}$ that are subsets of DV or LG or ZG vine-like objects (Proposition \ref{subpro} (4), note that here we must be in one of the cases defined in Lemma \ref{subsetvc}), and select the ones that are maximal (i.e. those that are not subsets of larger SG vine-like objects with same properties). Here for any DV, if both vines (V) forming this DV are SG, then we select only one of them, plus all the maximal SG vine-like proper subsets of the other. Define by $\Cs_2$ the collection of SG vine-like objects formed in this way. Then:
\begin{enumerate}[resume*]
\item We fix each atom in $\Mb_{\mathrm{sb}}$ as either SG, LG or ZG; here we say an atom $v$ is SG if there is at least one choice of $(\ell_1,\ell_2)$ such that $(v,\ell_1,\ell_2)$ is SG. Note that an atom may be SG in more than one way, in which case we will further specify it below.
\item For each SG atom $v$, we specify the two bonds $(\ell_1,\ell_2)$ at $v$ such that $(v,\ell_1,\ell_2)$ is SG. Here we require that (i) for any hinge atom $v$ we must specify $(\ell_1,\ell_2)$ as in Proposition \ref{subpro} (5), and (ii) for any joint $v$ of any SG vine-like object in $\Cs_1$ or $\Cs_2$, we must specify $(\ell_1,\ell_2)$ to belong to this vine-like object\footnote{This is consistent with (i) due to Proposition \ref{subpro} (5). Also, if this vine-like object is an HV or HVC, then $(\ell_1,\ell_2)$ is chosen to belong to the corresponding adjoint VC.}; for any other atom $v$, we can specify $(\ell_1,\ell_2)$ arbitrarily.
\item Finally, for each SG atom $v$ which is not an interior atom of an SG vine-like object in $\Cs_1$ or $\Cs_2$, we fix its gap as $|r|\sim_0 Q_v$, where $Q_v\in[L^{-1},L^{-\gamma+\eta}]$ is a dyadic number.
\end{enumerate}
Note that putting all the extra requirements in (d)--(f) leads to a loss of at most $C_0^{n_{\mathrm{sb}}}(\log L)^{p_{\mathrm{sb}}}$, where $p_{\mathrm{sb}}$ is the total number of SG atoms satisfying the condition in (f).

Now, define $\Ef_{\mathrm{sb}}$ to be the number of $(k_1,\cdots,k_{2R})$-decorations $\Os_{\mathrm{sb}}$ of $\Mb_{\mathrm{sb}}$ that satisfy all the requirements (a)--(f) defined above, and define
\begin{equation}\label{defcounta}
\Af_{\mathrm{sb}}:=\Ef_{\mathrm{sb}}\cdot L^{-(d-\gamma)\chi(\Mb_{\mathrm{sb}})}\cdot (C_1^{-1}\delta^{1/2})^{\chi(\Mb_{\mathrm{sb}})}\prod_{j=1}^{q_{\mathrm{sb}}}\Xf_j^{z_j},
\end{equation}
where we recall the circular rank $\chi$ defined in Proposition \ref{propcplmol}. Then, using (\ref{redvine5}) and (\ref{ladderl1}), and that $\chi(\Mb_{\mathrm{sb}})=n_{\mathrm{sb}}-R+1$, we know that (\ref{finalsbest2}) would follow if we can prove that
\begin{equation}\label{finalcount1}
\Af_{\mathrm{sb}}\leq L^{-4\eta^6(\rho_{\mathrm{sb}}+p_{\mathrm{sb}})}\cdot L^{(R-1)(d-\gamma-1/4-1/40)}.
\end{equation} Here note that each hinge atom is SG by Proposition \ref{subpro} (5), and cannot be an interior atom of any SG vine-like object in $\Cs_1$ or $\Cs_2$, so we always have $q_{\mathrm{sk}}\leq p_{\mathrm{sb}}$.
\subsection{Reduction to large gap molecules}\label{redlg} Our goal now is to prove (\ref{finalcount1}). To bound the number of decorations $\Ef_{\mathrm{sb}}$ in the counting problem, we will reduce $\Mb_{\mathrm{sb}}$ by performing suitable \emph{cutting} operations introduced in \cite{DH23}, at all the SG vine-like objects and SG atoms. The proof is essentially the same as in Sections 9--10 of \cite{DH23}, with only minor differences due to the fact that here we are dealing with general gardens instead of couples. We start by recalling the definition of cuts.
\begin{df}[The cutting operation \cite{DH23}]\label{defcut}
Given a molecule $\Mb$ and an atom $v$. Suppose $v$ has two bonds $\ell_1$ and $\ell_2$ of opposite directions, then we may \emph{cut} the atom $v$ along  the bonds $\ell_1$ and $\ell_2$, by replacing $v$ with two atoms $v_1$ and $v_2$, such that the endpoint $v$ for the bonds $\ell_1$ and $\ell_2$ is moved to $v_1$, and the endpoint $v$ for the other bond(s) is moved to $v_2$, see Figure \ref{fig:cut}. We say this cut is an $\alpha$- (resp. $\beta$-) cut, if it does not (resp. does) generate a new connected component, and accordingly we say the resulting atoms $v_1$ and $v_2$ are $\alpha$- or $\beta$- atoms. If we are also given a decoration $(k_\ell)$ of $\Mb$, then we may define the gap of this cut to be $r:=k_{\ell_1}-k_{\ell_2}$.
\begin{figure}[h!]
\includegraphics[scale=0.44]{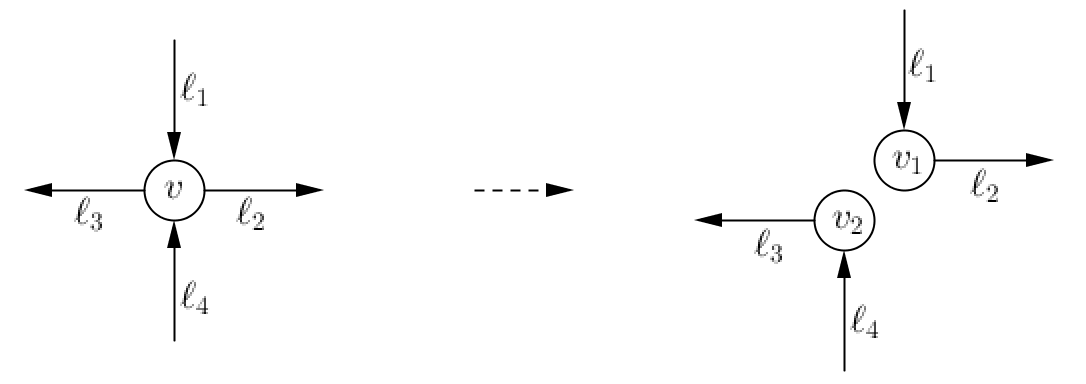}
\caption{A cutting operation executed at a degree $4$ atom $v$, see Definition \ref{defcut}.}
\label{fig:cut}
\end{figure}
\end{df}
\begin{rem} Compared to \cite{DH23}, here we also allow to perform cutting at a degree $2$ atom $v$. This case is in fact trivial: after the cut $v_2$ becomes an isolated atom (without any bond) which forms a single connected component by itself.
\end{rem}

With Definition \ref{defcut}, we shall now reduce $\Mb_{\mathrm{sb}}$ by performing various cutting operations, and possibly removing some connected components created in this process, until reaching a final molecule $\Mb_{\mathrm{fin}}$. At each step, let the molecule before and after the operation be $\Mb_{\mathrm{pre}}$ and $\Mb_{\mathrm{pos}}$, then any decoration $(k_\ell)$ of $\Mb_{\mathrm{pre}}$ can be naturally inherited to obtain a decoration of $\Mb_{\mathrm{pos}}$.

For any molecule $\Mb$ in this process (which could be $\Mb_{\mathrm{pre}}$ or $\Mb_{\mathrm{pos}}$), consider the possible $(c_v)$-decorations of $\Mb$, also restricted by some $(\beta_v)$ and $(k_\ell^0)$; we also assume that this decoration is inherited from a $(k_1,\cdots,k_{2R})$-decoration of $\Mb_{\mathrm{sb}}$ that satisfies all assumptions (a)--(f) in Section \ref{l1reduction}. Then we consider the number of such restricted decorations, take supremum over the parameters $(c_v,\beta_v,k_\ell^0)$, and define it to be $\Ef$. Similar to (\ref{defcounta}) we also define
\begin{equation}\label{defcounta2}\Af:=\Ef\cdot L^{-(d-\gamma)\chi(\Mb)}(C_1^{-1}\delta^{1/2})^{\chi(\Mb)}\prod_{j=1}^q\Xf_j^{z_j},
\end{equation} where $\chi(\Mb)$ is the circular rank of $\Mb$, and $z_j\,(1\leq j\leq q)$ are the lengths of the maximal ladders $\Lb_j$ in $\Mb$, and $\Xf_j$ are defined from $P_j$ as in Proposition \ref{ladderl1old}, while $P_j$ are fixed such that the gap $r_j$ of $\Lb_j$ is such that $|r_j|\sim_0P_j$. Define also $m$ to be the number of atoms not in the maximal ladders, and $\rho=q+m$. The notations $(\Af,\Ef)$ and $(\rho,q,m)$ will apply to the molecules appearing below, with possible subscripts matching those of $\Mb$ (so $\Af_{\mathrm{pre}}$ is defined for the molecule $\Mb_{\mathrm{pre}}$ etc.)

We will prove, for each operation, an inequality of form \begin{equation}\label{defdev}\Af_{\mathrm{pre}}\leq \If\cdot \Af_{\mathrm{pos}}\end{equation} for some quantity $\If$, which we define to be the \emph{deviation} of this operation. Clearly, if we can bound the deviation in each operation step, as well as $\Af_{\mathrm{fin}}$ for the final molecule $\Mb_{\mathrm{fin}}$, then we can deduce from these an upper bound for $\Af_{\mathrm{sb}}$ using (\ref{defdev}).

We will define the needed cutting operations, and prove (\ref{defdev}) for each operation, in Section \ref{redcut}; then in Section \ref{secrigid} we obtain an upper bound for $\Af_{\mathrm{fin}}$ and complete the proof of (\ref{finalcount1}).
\subsubsection{Cutting operations and estimates}\label{redcut} Recall the notion of $V,E,F$ and $\chi$ as in Definition \ref{defmol}; define also $V_\alpha$ and $V_\beta$ to be the number of $\alpha$- and $\beta$-atoms (see Definition \ref{defcut}, and also the description of operations below), and use $\Delta$ to denote increments (such as $\Delta F=F_{\mathrm{pos}}-F_{\mathrm{pre}}$).

\uwave{Operation 1: Removing SG vine-like objects.} In Operation 1, we collect all the SG vine-like objects in $\Cs_1$ and $\Cs_2$, as defined in Section \ref{l1reduction}; note that these also include any possible triple bonds in $\Mb_{\mathrm{sb}}$ (see Proposition \ref{subpro} (2)). For each of these objects $\Ub$, we cut the molecule at each of its joints $v$, along the two bonds $(\ell_1,\ell_2)$ fixed as in requirement (e) defined in Section \ref{l1reduction} (i.e. the two bonds that belong to $\Ub$). This disconnects a VC, denoted by $\Ub'$, from the rest of the molecule, and then we remove $\Ub'$. The two joints then become two atoms of degree at most $2$ (it is possible that they have degree $0$, i.e. isolated atoms); we define them as $\alpha$-atoms if no new connected component is created, and $\beta$-atoms otherwise. If they are $\alpha$-atoms, we also label each of them by the dyadic number $Q\in[L^{-1},L^{-\gamma+\eta}]$ such that $|r|\sim_0 Q$ for the gap $r$ of $\Ub'$ (note $r\neq 0$).

\uwave{Operation 2: Removing triangles.} In Operation 2, we consider the possible triangles $v_1v_2v_3$ in the molecule, such that there are bonds $\ell_j$ connecting $v_{j+1}$ and $v_{j+2}$ (where $v_4=v_1$ and $\ell_4=\ell_1$ etc.), and $(v_j,\ell_{j+1},\ell_{j+2})$ is an SG triple for $j\in\{1,2\}$. Let $|k_{\ell_{j+1}}-k_{\ell_{j+2}}|\sim_0 Q_{j}$ for $j\in\{1,2\}$ with $Q_j\in[L^{-1},L^{-\gamma+\eta}]$. We then cut the molecule at each $v_j$ along the bonds $(\ell_{j+1},\ell_{j+2})$, which disconnects the triangle formed by $v_j$ and $\ell_j$ from the rest of the molecule, and remove the triangle. This leaves $3$ atoms $v_j$ of degree at most $2$. We define $v_3$ as a $\beta$-atom, and define $v_j\,(1\leq j\leq 2)$ as an $\alpha$- (resp. $\beta$-) atom if it belongs to the same (resp. different) component with $v_3$, except when $v_1$ and $v_2$ are in the same component different from $v_3$, in which case we define $v_1$ as an $\alpha$-atom and define $v_2$ as a $\beta$-atom. For any $\alpha$-atom $v_j$ we label it by the corresponding $Q_j$.

\uwave{Operation 3: Remaining SG cuts.} In Operation 3, we select each of the remaining SG atoms $v$, and cut them along the designated bonds $(\ell_1,\ell_2)$ as in requirement (e) in Section \ref{l1reduction}. Note that each cut may be $\alpha$- or $\beta$-cut; we define the resulting atoms as $\alpha$- or $\beta$-atoms, and label any $\alpha$-atom by the dyadic number $Q\in[L^{-1},L^{-\gamma+\eta}]$, as in Definition \ref{defcut}.

\uwave{Operation 4: Remaining $\beta$-cuts.} In Operation 4, we look for all the possible atoms where a $\beta$-cut is possible, and perform the corresponding $\beta$-cut (defining the resulting atoms as $\beta$-atoms), until this can no longer be done. 

\uwave{Order of these operations.} We first perform Operation 1 for all SG vine-like objects in $\Cs_1$ (first for the (CN) vines and root (CL) vines, and then for the HV's), then perform Operation 1 for all SG vine-like objects in $\Cs_2$. Next we perform Operation 2 for all eligible triangles, and then perform Operation 3 for each of the remaining SG atoms. Finally we perform Operation 4 for all remaining atoms where a $\beta$-cut is possible.
\begin{rem}\label{remiota}
After all the cutting operations, the resulting graph will contain some $\alpha$-atoms. For each $\alpha$-atom $v$ and a given decoration we define an auxiliary number $\iota_v\in\{0,1\}$, such that $\iota_v=1$ if a cutting operation happens before the cutting at the atom $v$, such that the gaps $r,r'$ of these cuttings satisfy $0<|r\pm r'|\leq L^{-50\eta}Q_v$; if no such cutting operation exists then define $\iota_v=0$. Note that the number of choices for all $(\iota_v)$ is at most $2^{n_{\mathrm{sb}}}$  which can be safely ignored; therefore we may assume a choice of $(\iota_v)$ is fixed in the proof below.
\end{rem}
\begin{rem}\label{remtame} Note that the joints of the SG vine-like object in Operation 1, the atoms $v_1$ and $v_2$ in Operation 2 and the atom $v$ in Operation 3 are all SG with fixed $(\ell_1,\ell_2)$. We will also assume that no atom $v_3$ in Operation 2 or atom $v$ in Operation 4 is ZG; in fact, if any such atom $v$ is ZG, then it must be tame by Proposition \ref{subpro} (3), and thus can be treated by the arguments in Section 10.5 of \cite{DH23}. Namely, we consider the connected components after removing the atom $v$. If (i) $d(v)<4$, or if some component contains exactly one atom adjacent to $v$, then the values of $k_\ell$ for $\ell\leftrightarrow v$ must all be fixed in the given decoration, so we always have enough gain whether the resulting $\alpha$- or $\beta$-atoms have zero gap. We then cut at $v$ and proceed normally. If (ii) $d(v)=4$ and no component contains exactly one atom adjacent to $v$, then removing $v$ will yield $\Delta\chi\leq -2$ by direct calculation, which leads to a power gain $\If\leq L^{-(d-2\gamma-\eta)}$, as $(k_\ell)$ for $\ell\leftrightarrow v$ have at most $L^d$ choices. We then remove $v$ without assigning any new $\alpha$- or $\beta$-atoms and proceed with the rest of the molecule. The big gain $L^{-(d-2\gamma-\eta)}$ is enough to cover any loss that may occur, see Remark \ref{remtame2}.
\end{rem}
\begin{prop}\label{excessprop1} The followings hold regarding Operations 1--4, and the process of applying them as described above:
\begin{enumerate}[{(1)}]
\item After each Operation 1--4, \emph{except for an Operation 1 that creates a new connected component}, we have that \begin{equation}\label{excesseqn2}\If\leq L^{-\eta^3}\cdot\prod_{v}^{(\alpha)}L^{(\gamma+3\eta)/2-\kappa_v\eta}Q_v\cdot L^{(2\gamma_0+5\eta^2)\Delta F-(\gamma_0+2\eta^2)\Delta V_\beta}
\end{equation} for any $Q_v\in[L^{-1},L^{-\gamma+\eta}]$, where the product is taken over all the newly created $\alpha$-atoms $v$, and $Q_v$ denotes the label of $v$. The parameter $\kappa_v$ equals $0$ if $\iota_v=0$, and equals $50$ if $\iota_v\neq 0$.
\item If an Operation 1 creates a new connected component, then we have that
\begin{equation}\label{excesseqn2+}
\If\leq L^{\gamma_0-\eta^3}\cdot\prod_{v}^{(\alpha)}L^{(\gamma+3\eta)/2-\kappa_v\eta}Q_v\cdot L^{(2\gamma_0+5\eta^2)\Delta F-(\gamma_0+2\eta^2)\Delta V_\beta},
\end{equation} where the notations are the same as in (\ref{excesseqn2}).
\item In the whole process, the number of Operations 1 that indeed create new connected components, is at most $2R-1$. If it is exactly $2R-1$, then one of the new connected components created must be an isolated $\beta$-atom.
\end{enumerate}
\end{prop}
\begin{proof} (1) This is exactly Proposition 9.4 of \cite{DH23}, and follows from exactly the same proof. Note that in the couple setting in Section 9 of \cite{DH23} (i.e. $R=1$), no Operation 1 can create a new component, and the proof of Proposition 9.4 of \cite{DH23} relies on this fact.

(2) Note that all the SG vine-like objects in $\Cs_1$ and $\Cs_2$ are disjoint (and the DV's and LG and ZG vine-like objects containing those objects in $\Cs_2$ are also disjoint with the objects in $\Cs_1$), so removing any of them will not affect the others. By Lemma \ref{subsetvc}, any Operations 1 concerning $\Cs_2$ will \emph{not} create a new connected component, because the corresponding DV or LG or ZG vine-like object will still remain connected.

Therefore, we only need to consider Operations 1 concerning $\Cs_1$, which is always removing a single vine $\Ub$. By definition we have $\Delta F=1$ and $\Delta V_\beta=2$, and no new $\alpha$-atom is created. If $\Ub$ contains $m$ atoms (including joints). Consider the collection $(k_\ell)$ for all bonds $\ell\in\Ub$ in the decoration of $\Mb_{\mathrm{pre}}$, and let the number of choices for them be $\Ff$, then we have
\begin{equation}\label{dev1}\If\lesssim_1 \Ff\cdot L^{-(m-1)(d-\gamma)}\cdot (C_1^{-1}\delta^{1/2})^{m-1}\prod_{j:\Lb_j\subset\Ub}\Xf_j^{z_j}
\end{equation}
where the product is taken over all $j$ such that the maximal ladder $\Lb_j$ (Proposition \ref{ladderl1old}) is a subset of $\Ub$ (so there are at most three such $j$, see Figure \ref{fig:vines}), and $z_j$ is the length of $\Lb_j$.

Now we calculate $\Ff$. Consider the two bonds $\ell_1$ and $\ell_2$ at one joint of $\Vb\Cb$ that belong to $\Ub$. Note that removing $\Ub$ will create a new connected component; thus in the $(c_v)$-decoration of $\Mb_{\mathrm{pre}}$, we must have that $k_{\ell_1}\pm k_{\ell_2}$ equals a fixed constant and that $|k_{\ell_1}|^2\pm|k_{\ell_2}|^2$ belongs to a fixed interval of length at most $n_{\mathrm{sb}}\cdot\delta^{-1}L^{-2\gamma}$, both depending only on the parameters $(c_v)$ and $(\beta_v)$. By losing another factor $n_{\mathrm{sb}}\leq (\log L)^{C_2}$ which can be ignored in view of the $L^{-\eta^3}$ factor in (\ref{excesseqn2+}), we may assume that $|k_{\ell_1}|^2\pm|k_{\ell_2}|^2$ belongs to a fixed interval of length $\delta^{-1}L^{-2\gamma}$. By (\ref{atomcountbd}) in Lemma \ref{atomcountlem}, the number of choices for $(k_{\ell_1},k_{\ell_2})$ is at most $L^{d-\gamma+\gamma_0+\eta^3}$.

Now, once $(k_{\ell_1},k_{\ell_2})$ is fixed, we can invoke the re-parametrization introduced in the proof of Proposition \ref{vineest} (note that we are dealing with a single vine now), and define the new variables $(x_j,y_j)\,(1\leq j\leq \widetilde{m})$ if $\Ub$ is bad vine, or $(x_j,y_j)\,(1\leq j\leq \widetilde{m})$ and $(u_1,u_2,u_3)$ if $\Ub$ is normal vine, where $m=2\widetilde{m}+2$ for bad vine and $m=2\widetilde{m}+5$ for normal vine. In either case, since each $\Gamma_v$ belongs to a fixed interval of length $\delta^{-1}L^{-2\gamma}$, we know that each $(x_j,y_j)$ satisfies a system of form (\ref{basiccount01}), and $(u_1,u_2,u_3)$ satisfies a system of form (\ref{basiccount02}). Moreover for each $j$, if the atoms associated with $(x_j,y_j)$ (see the proof of Proposition \ref{vineest}) belongs to a certain ladder $\Lb$, then the value of $r$ in (\ref{basiccount01}) must equal the gap of $\Lb$. Therefore, by Lemma \ref{countlem}, we get in either case that
\[\Ff\lesssim_1 L^{d-\gamma+\gamma_0+\eta^3}L^{(m-2)(d-\gamma)}\cdot L^{\eta^3}(C_1\delta^{-1/2})^{m-1}\prod_{j:\Lb_j\subset\Ub}\Xf_j^{-z_j},\] hence $\If\lesssim_1 L^{\gamma_0+2\eta^3}$, so (\ref{excesseqn2+}) is true.

(3) By definition, we are doing Operations 1 concerning $\Cs_1$ before those concerning $\Cs_2$. Moreover, those Operations 1 concerning $\Cs_2$ and HV's in $\Cs_1$ will not create any new connected component, so we only need to consider the Operations 1 concerning (CN) and root (CL) vines in $\Mb_{\mathrm{sb}}=\Mb(\Gc_{\mathrm{sb}})$, which are done before all the others. By Proposition \ref{cnblock}, after removing all these (CN) and root (CL) vines, the number of connected components increases to at most $2R$, so in this process the number of Operations 1 that create new connected components is at most $2R-1$; if equality holds, then one of the new connected components created must be an isolated atom, which also has to be $\beta$-atom by definition.
\end{proof}
\subsection{The rigidity theorem}\label{secrigid} Starting from $\Mb_{\mathrm{sb}}$, we perform all the Operations 1--4 as described in Section \ref{redcut}, and obtain a final molecule $\Mb_{\mathrm{fin}}$. This molecule may not be connected, and let $\Mb$ be any of its connected components. Consider also a decoration of $\Mb_{\mathrm{fin}}$ (and $\Mb$) that is inherited from a decoration of $\Mb_{\mathrm{sb}}$ as in Section \ref{redlg}. Then, they satisfy certain crucial properties, which we summarize in the following proposition.
\begin{prop}\label{finalmol}
In any (connected) component $\Mb$ of $\Mb_{\mathrm{fin}}$, each $\alpha$-atom $v$ is labeled by a dyadic number $Q_v\in [L^{-1},L^{-\gamma+\eta}]$, such that if $v$ has two bonds $(\ell_1,\ell_2)$ then $|k_{\ell_1}-k_{\ell_2}|\sim_0 Q_v$ in the decoration; recall also $\iota_v$ and $\kappa_v$ introduced in Remark \ref{remiota} and Proposition \ref{excessprop1}. Define any atom that is not an $\alpha$- nor $\beta$-atom to be an \emph{$\varepsilon$-atom}. Then, any $\varepsilon$-atom in $\Mb_{\mathrm{fin}}$ \emph{must} be LG or ZG in the decoration, and must be tame if it is ZG. Moreover $\Mb_{\mathrm{fin}}$ contains no triple bond.

We say a component is \emph{perfect} if all $\alpha$- and $\beta$- atoms have degree $2$, and all $\varepsilon$-atoms have degree $4$. Otherwise we say $\Mb$ is \emph{imperfect}. Then, there are at most $R$ imperfect components.  Moreover, any perfect component has at least one $\beta$-atom. If a perfect component $\Mb$ is a cycle (i.e. all atoms have degree $2$), then it is either a double bond (which is also Vine (I)), or a cycle of length at least $4$, or a triangle with at most one $\alpha$-atom. If $\Mb$ is not a cycle, then all its $\alpha$- and $\beta$-atoms form several disjoint chains, such that the two ends of each chain are connected to two distinct $\varepsilon$-atoms. Finally, no $\alpha$- or $\beta$-atom is ZG with the assumptions in Remark \ref{remtame}, and if any (perfect or imperfect) component $\Mb$ is a \emph{vine} with two joints having degree $2$, then it \emph{must} be LG in the decoration.
\end{prop}
\begin{proof} We only need to prove that there are at most $R$ imperfect components. All the other statements follow from exactly the same proof as in Proposition 9.5 of \cite{DH23} (note also that any ZG $\varepsilon$-atom cannot be a hinge atom, thus must be tame by Proposition \ref{subpro} (3)).

The molecule $\Mb_{\mathrm{sb}}$ has $n=n_{\mathrm{sb}}$ atoms and $2n-R$ bonds, as in Proposition \ref{subpro} (1). It is easy to check that in each Operation 1--4, the value of $4V-2E-2(V_\alpha+V_\beta)$ is preserved. This value is $2R$ for $\Mb_{\mathrm{sb}}$, and has to be the same for $\Mb_{\mathrm{fin}}$. But this value also equals the sum of $4-d(v)$ over all $\varepsilon$-atoms $v$, plus the sum of $2-d(v)$ over all $\alpha$- and $\beta$- atoms $v$. If we add up these quantities over each component, then by definition, we get $0$ for a perfect component, and at least $2$ for an imperfect component. This implies that the number of imperfect components is at most $R$ (in fact, it is exactly $R$ if $R=1$).
\end{proof}
With Proposition \ref{finalmol}, we can reduce the proof of (\ref{finalcount1}) to the estimates for quantities $\Af$ and $\Ef$ associated with the counting problems for each component $\Mb$ of $\Mb_{\mathrm{fin}}$. These estimates are stated as a \emph{rigidity theorem} in Proposition 9.6 of \cite{DH23}; its proof is a major technical component of \cite{DH23}. Fortunately, here we can use it (almost) as a black box.
\begin{prop}
\label{lgmolect} For each component $\Mb$ of $\Mb_{\mathrm{fin}}$, define $(\Af,\Ef)$ and $(\rho,q,m)$ as in Section \ref{redlg} and (\ref{defcounta2}), but associated with $\Mb$. Let also each $\alpha$-atom $v$ be labeled by the dyadic number $Q_v$, then we have
\begin{equation}\label{finalcount}\Af\leq \prod_v^{(\alpha)}L^{-(\gamma+3\eta)/2+\kappa_v\eta}Q_v^{-1}\cdot L^{(\gamma_0+2\eta^2)V_\beta-(2\gamma_0+5\eta^2)G}\cdot L^{-\eta^5\rho+\eta Z},
\end{equation} where the product is taken over all $\alpha$-atoms $v$, and $V_\beta$ is the number of $\beta$-atoms in $\Mb$, and $G$ is $0$ or $1$ depending on whether $\Mb$ is imperfect or perfect component. Finally $Z$ equals $0$ if $R=1$; if $R\geq 2$ then $Z$ equals the number of $\varepsilon$-atoms \emph{not} of degree $4$ in $\Mb$.
\end{prop}
\begin{proof} For perfect components, or imperfect components when $R=1$ (which are called \emph{odd} components in \cite{DH23}), this is exactly the same as Proposition 9.6 of \cite{DH23}; see also Section 10.5 of \cite{DH23} that addresses tame atoms. Now consider an imperfect component when $R\geq 2$. By repeating the same proof in Proposition 9.6 of \cite{DH23} we can obtain that $\Af\lesssim_11$ (the proof in \cite{DH23} trivially extends to general imperfect components). On the other hand, we may also apply Lemma 10.2 of \cite{DH23} to get \[\Af\lesssim_1 L^{-\eta^2\rho+C_0(V_\alpha+V_\beta+Z)},\] noticing that $V_\alpha+V_\beta+Z\geq 1$. Since for any $Q_v\leq L^{-\gamma+\eta}$ we have
\[L^{-(\gamma+3\eta)/2+\kappa_v\eta}Q_v^{-1}\geq L^{(\gamma-5\eta)/2}\geq L^{\gamma/4},\] and also $G=0$, by interpolation we obtain that
\[\Af\lesssim_1 L^{-\eta^4\rho+C_0\eta^2(V_\alpha+V_\beta+Z)}\ll_1\mathrm{RHS\ of\ }(\ref{finalcount}).\qedhere\]
\end{proof}
With Proposition \ref{lgmolect}, we are finally ready to prove (\ref{finalcount1}).
\begin{proof}[Proof of the bound (\ref{finalcount1})] By iterating (\ref{defdev}) for each of the Operations 1--4 in reducing $\Mb_{\mathrm{sb}}$ to $\Mb_{\mathrm{fin}}$, and using the fact that $\Mb_{\mathrm{fin}}$ is formed by all the components $\Mb$, we have
\begin{equation}\label{together}
\Af_{\mathrm{sb}}\leq\Af_{\mathrm{fin}}\cdot\prod_{\mathrm{(all\ operations)}}\If,\qquad \Af_{\mathrm{fin}}\leq\prod_{\mathrm{(all\ components)}}\Af,
\end{equation} where the products are taken over all Operations 1--4 and all components $\Mb$ of $\Mb_{\mathrm{fin}}$. Putting together (\ref{excesseqn2}), (\ref{excesseqn2+}) and (\ref{finalcount}), we get that
\begin{equation}\label{together2}
\Af_{\mathrm{sb}}\leq L^{(2\gamma_0+5\eta^2)(F_{\mathrm{im}}-1)}\cdot L^{\gamma_0\cdot p_{\mathrm{op}}^{+}}\cdot L^{\eta Z_{\mathrm{tot}}}\cdot L^{-\eta^5\rho_{\mathrm{fin}}-\eta^3p_{\mathrm{op}}}.
\end{equation}Here in (\ref{together2}):
\begin{itemize}
\item $\rho_{\mathrm{fin}}$ is the $\rho$ value for $\Mb_{\mathrm{fin}}$ defined as in Section \ref{redlg}, and $F_{\mathrm{im}}$ is the number of imperfect components in $\Mb_{\mathrm{fin}}$, so $F_{\mathrm{im}}\leq R$ by Proposition \ref{finalmol}.
\item $p_{\mathrm{op}}$ is the total number of Operations 1--4, and $p_{\mathrm{op}}^+$ is the total number of Operations 1 that create new components, so $p_{\mathrm{op}}^+\leq 2R-1$ by Proposition \ref{excessprop1} (3). If equality holds, then one of the components must be an isolated $\beta$-atom, so we gain an extra power $L^{-(\gamma_0+\eta^2)}$ in (\ref{finalcount}) (where $V_\beta=\rho=1$ and $G=Z=0$); thus in practice we can assume $p_{\mathrm{op}}^+\leq 2R-2$.
\item $Z_{\mathrm{tot}}$ is the sum of $Z$, defined in Proposition \ref{lgmolect}, over all components $\Mb$, so $Z_{\mathrm{tot}}=0$ if $R=1$; if $R\geq 2$ then $Z_{\mathrm{tot}}$ is the total number of $\varepsilon$-atoms not of degree $4$, so by Proposition \ref{subpro} (1) we have $Z_{\mathrm{tot}}\leq 2R$. In any case we have $Z_{\mathrm{tot}}\leq 4(R-1)$.
\end{itemize} Note also that if an SG atom $v$ satisfies requirement (f) in Section \ref{l1reduction} (i.e. is not an interior atom of an SG vine-like object in $\Cs_1$ or $\Cs_2$), then we must be cutting at $v$ in one of the Operations 1--4, so we have $p_{\mathrm{sb}}\leq p_{\mathrm{op}}$. Moreover each of Operations 1--4 changes the value of $\rho$ by at most $C_0$, since each vine contains at most $3$ ladders and no more than $20$ atoms apart from these ladders, thus $\rho_{\mathrm{fin}}\geq \rho_{\mathrm{sb}}-C_0\cdot p_{\mathrm{op}}$. Putting together we get that
\begin{equation}\label{together3}
\Af_{\mathrm{sb}}\leq L^{(R-1)(4\gamma_0+4\eta+5\eta^2)}\cdot L^{-\eta^5(\rho_{\mathrm{sb}}-C_0\cdot p_{\mathrm{op}})-\eta^3 p_{\mathrm{op}}}\leq L^{-\eta^5(\rho_{\mathrm{sb}}+p_{\mathrm{sb}})}\cdot L^{(R-1)(4\gamma_0+1/10)},
\end{equation} which implies (\ref{finalcount1}) because $d-\gamma-4\gamma_0\geq 1/2$. This completes the proof.
\end{proof}
\begin{rem}\label{remtame2} Recall the discussion of ZG/tame atoms in Remark \ref{remtame}. In case (i) the proof is modified trivially. In case (ii), by the arguments in Remark \ref{remtame}, we know that this operation satisfies (\ref{excesseqn2}) with an extra power gain $L^{-(d-2\gamma-2\eta)-(2\gamma_0+5\eta^2)}$ (assuming $\Delta F=1$, otherwise $\Delta\chi=-3$ and the estimate is much better). On the other hand, this increases the number of imperfect components by $2$, leading to a loss $L^{2(2\gamma_0+5\eta^2)}$ in (\ref{together2}), but this is acceptable as $d-2\gamma-2\gamma_0\geq 1$.
\end{rem}
\subsubsection{Proof of Proposition \ref{layergarden}} By interpolating (\ref{finalsbest1}) with (\ref{finalsbest2}), which follows from (\ref{finalcount1}), we get that
\begin{equation}\label{finalproof1}
|\Ks_{\Gc_{\mathrm{sb}}}|\lesssim_2(C_1\delta^{10\nu})^{n_{\mathrm{sb}}}\cdot L^{-40\nu\eta\cdot g_{\mathrm{sb}}}\cdot L^{-\eta^6(\rho_{\mathrm{sb}}+q_{\mathrm{sk}})/2}\cdot L^{-(R-1)(1/4+1/60)}.
\end{equation} Combining with (\ref{redvine2}) and arguments in Section \ref{redsplice}, we get that
\begin{align}|\Ks_{p+1,n}^{\mathrm{eqc}}|&\lesssim_2(C_1\sqrt{\delta})^{n-n_{\mathrm{sb}}}(C_1\delta^{10\nu})^{n_{\mathrm{sb}}}\cdot L^{-\eta^{8}g_{\mathrm{sk}}-\eta^2b_{\mathrm{sk}}+\eta^7q_{\mathrm{sk}}}\nonumber\\&\qquad\qquad\qquad\quad\times L^{-40\nu\eta\cdot g_{\mathrm{sb}}}\cdot L^{-\eta^6(\rho_{\mathrm{sb}}+q_{\mathrm{sk}})/2}\cdot L^{-(R-1)(1/4+1/60)}\nonumber\\
\label{finalproof2}&\lesssim_2 (C_1\delta^{10\nu})^n \cdot L^{-\nu\eta(g_{\mathrm{sk}}+g_{\mathrm{sb}})}\cdot L^{-\eta^6(\rho_{\mathrm{sb}}+q_{\mathrm{sk}})/4}\cdot L^{-(R-1)(1/4+1/60)},
\end{align} where $g_{\mathrm{sk}}$ and $b_{\mathrm{sk}}$ are as in Section \ref{redsplice}.

With (\ref{finalproof2}) we can now finish the proof of Proposition \ref{layergarden}.
\begin{proof}[Proof of Proposition \ref{layergarden}] We only need to consider the terms (a)--(d) in (\ref{finalterma})--(\ref{finaltermd}). For each term, we can treat the $(\langle k_1\rangle\cdots \langle k_{2R}\rangle)^{4\Lambda_{p+1}}$ weights and $\partial_k^\rho$ derivative for couples as described in Section \ref{redcount}, so we just need to bound it pointwise for each fixed $(k_1,\cdots, k_{2R})$. By translation invariance, for (\ref{finaltermb})--(\ref{finaltermd}) we will assume $k\in\Zb_L^d$ (the general $k\in\Rb^d$ case is the same).

First consider (\ref{finaltermc}) and (\ref{finaltermd}). If in (\ref{finaltermc}) we restrict to couples of incoherency index $g\geq 1$, then the contribution of each individual couple is bounded by $(C_1\delta)^n\cdot L^{-\gamma_1g/2}$ by (\ref{layerregest1}). For fixed $g$, the number of choices for the regular couple (without layering) is at most $C_0^n$ by Proposition \ref{propstructure2}, and the number of possible layerings is at most $C_0^nC_2^g$ by Proposition \ref{layerreg2}. This implies that
\begin{equation}|(\ref{finaltermc})|\lesssim_2\sum_{g=1}^{n}(C_1\delta)^n\cdot L^{-\gamma_1g/2}\cdot C_0^nC_2^g\leq L^{-\gamma_1/4}\cdot\delta^{n/2}.
\end{equation} As for (\ref{finaltermd}), let the (canonical and) regular, coherent layered couple be $\Qc$. By Proposition \ref{layerreg1}, since both $q$ and $q'$ are now replaced by $p$, we know that all nodes in $\Qc$ must be in layer $p$, in particular the number of choices for $\Qc$ is at most $C_0^n$. Then $\Kc_\Qc$ is simply a restriction of $\Kc_\Qc^*$ defined in Definition \ref{defkg}, so by (\ref{layerregest2}), (\ref{layerregest3}), (\ref{layerregest7}) and (\ref{layerregest8}) we get
\begin{equation}|(\ref{finaltermd})|\leq L^{-(2\theta_{p+1}+3\theta_{p})/5}\cdot\delta^{n/2}.
\end{equation}

Now we consider (\ref{finalterma}) and (\ref{finaltermb}). By reductions in Sections \ref{cancelvine} and \ref{stage1red}, each of them can be reduced to a summation of terms of form $\Ks_{p+1,n}^{\mathrm{eqc}}$ which satisfy (\ref{finalproof2}). Let $\bar{g}:=g_{\mathrm{sk}}+g_{\mathrm{sb}}$ and $\bar{\rho}:=\rho_{\mathrm{sb}}+q_{\mathrm{sk}}$ with the relevant parameters defined as before; clearly these values are not changed under congruence relation defined in Definition \ref{defcong}. Also when $R=1$ and for any non regular couple $\Qc$, the skeleton $\Qc_{\mathrm{sk}}$ is nontrivial and hence $\bar{\rho}\geq 1$.

With $\bar{g}$ and $\bar{\rho}$ fixed, we consider the number of choices for the canonical layered garden $\Gc$. First consider the molecule after removing all maximal ladders from $\Mb_{\mathrm{sb}}=\Mb(\Gc_{\mathrm{sb}})$, which is a molecule of at most $C_0\bar{\rho}$ atoms and thus has at most $(C_0\bar{\rho})!$ choices; then, by adding back at most $C_0\bar{\rho}$ disjoint maximal ladders of total length at most $n$, we get at most $C_0^n(C_0\bar{\rho})!$ choices for $\Mb(\Gc_{\mathrm{sb}})$. Similarly by adding back at most $C_0\bar{\rho}$ disjoint VC's in $\Vs_0$ in Section \ref{redsplice0}, we get at most $C_0^n(C_0\bar{\rho})!$ choices for $\Mb(\Gc_{\mathrm{sk}})$. Then by Proposition \ref{propcplmol} we get at most $C_0^n(C_0\bar{\rho})!(C_0R)!$ choices for the garden $\Gc_{\mathrm{sk}}$, and by Proposition \ref{propstructure2} we get at most $C_0^n(C_0\bar{\rho})!(C_0R)!$ choices for the garden $\Gc$ if we do not consider layering.

Now we consider the possible canonical layerings of $\Gc$. By losing another factor $2^n$, we may fix the exact positions of all the $\bar{g}$ incoherencies. If we fix the pre-layering of $\Gc_{\mathrm{sk}}$ (which is equivalent to the layering of atoms of the molecule $\Mb(\Gc_{\mathrm{sk}})$), then by Propositions \ref{layerreg1} and \ref{layerreg2}, we know that the number of layerings for all the regular couples $\Qc_{(\lf,\lf')}$ and $\Tc_{(\mf)}$ attached to $\Gc_{\mathrm{sk}}$ is bounded by
\begin{equation}C_0^nC_2^{\bar{g}}\cdot \prod_{\ell}C_0^{|\Lf_v-\Lf_{v'}|},\end{equation} where the product is taken over all bonds $\ell$ of $\Mb(\Gc_{\mathrm{sk}})$ connecting atoms $v,v'\in\Mb(\Gc_{\mathrm{sk}})$ similar to  (\ref{sumofexp0}). Therefore, we only need to bound the summation
\begin{equation}\label{sumofexp1}\sum_{(\Lf_v)} \prod_{\ell}C_0^{|\Lf_v-\Lf_{v'}|},
\end{equation} where the summation is taken over all possible layerings of atoms of $\Mb(\Gc_{\mathrm{sk}})$.

Next, consider all the VC's in $\Vs_0$ defined in Definition \ref{defcong} which are merged in reducing $\Mb(\Gc_{\mathrm{sk}})$ to $\Mb(\Gc_{\mathrm{sb}})$ in Sections \ref{redsplice0}--\ref{redsplice}. Since they are formed by (CL) vines only, by Proposition \ref{block_clcn} we know that for each of the VC in $\Vs_0$, the layers of the joint atoms of its vine ingredients form a monotonic sequence. Thus, by Proposition \ref{layervine}, we know that the number of possible layerings of atoms in these VC's, multiplied by the product in (\ref{sumofexp1}) involving bonds in these VC's, is bounded by $C_0^nC_2^{\bar{g}}\cdot C_2^{q_{\mathrm{sk}}}$ (where $q_{\mathrm{sk}}$ is the number of all these VC's).

Finally, consider all the maximal ladders in $\Mb_{\mathrm{sb}}=\Mb(\Gc_{\mathrm{sb}})$. By repeating the above arguments but now using Proposition \ref{layerlad}, we can bound the the number of possible layerings of atoms in these ladders, multiplied by the product in (\ref{sumofexp1}) involving bonds in these ladders, by $C_0^nC_2^{\bar{g}}\cdot C_2^{q_{\mathrm{sb}}}$ (where $q_{\mathrm{sb}}$ is the number of all these maximal ladders). The number of atoms not in these maximal ladders is $m_{\mathrm{sb}}$, and the number of layerings for them, multiplied by the product in (\ref{sumofexp1}) involving bonds at these atoms, is trivially bounded by $C_2^{m_{\mathrm{sb}}}$.

Putting altogether, we know that with fixed values of $n$, $R$, $\bar{g}$ and $\bar{\rho}$, the number of possible canonical layered gardens $\Gc$ is bounded by
\[C_0^n(C_0\bar{\rho})!(C_0R)!C_2^{\bar{g}+\bar{\rho}}.\] The number of congruence classes, which is the number of terms of form $\Ks_{p+1,n}^{\mathrm{eqc}}$ contained in (\ref{finalterma}) and (\ref{finaltermb}), is also bounded by this number. Combining with (\ref{finalproof2}), we get that
\begin{equation}\label{finalprooffin}|(\ref{finalterma})|\mathrm{\ or\ }|(\ref{finaltermb})|\lesssim_2(C_1\delta^{10\nu})^nL^{-\eta^7} \cdot L^{-\nu\eta(\bar{\rho}+\bar{g})}C_2^{\bar{g}+\bar{\rho}}\cdot L^{-(R-1)(1/4+1/60)}(C_0R)!\cdot (C_0\bar{\rho})!.
\end{equation} Since $\max(R,\bar{\rho},\bar{g})\leq n\leq N_{p+1}^4\leq (\log L)^{C_2}$ by the assumption in Proposition \ref{layergarden}, from (\ref{finalprooffin}) we easily deduce that
\begin{equation}\label{finalprooffin2}|(\ref{finalterma})|\mathrm{\ or\ }|(\ref{finaltermb})|\lesssim_2\delta^{8\nu n}L^{-\eta^7} \cdot  L^{-(R-1)(1/4+1/70)}.
\end{equation}  This completes the proof of Proposition \ref{layergarden}.
\end{proof}
\section{Linearization and proof of Proposition \ref{fixedpoint}}\label{endgame}
\subsection{Flower trees and construction of parametrix}\label{flower} We turn to the proof of Proposition \ref{fixedpoint}. Like in \cite{DH21,DH23}, the key point is to construct a suitable parametrix of $1-\Ls_1$ for the $\Rb$-linear operator $\Ls_1$ defined in (\ref{eqnbk1.5}), which involves the notions of of flower trees and flower gardens.
\begin{df}[Flower trees and flower gardens \cite{DH21,DH23}]\label{defflower} We define a \emph{flower tree} to be a ternary tree $\Tc$ with one particular leaf $\ff$ fixed. This leaf $\ff$ is called the \emph{flower} of $\Tc$. There is a unique path from the root $\rf$ to the flower $\ff$, which we call the \emph{stem} of $\Tc$. Similarly we can define a \emph{flower garden} $\Gc$ of width $2R$ to be a collection of $2R$ flower trees whose roots have signs $(+,-,\cdots,+,-)$, such that all the \emph{non-flower} leaves are completely paired\footnote{We may also pair the flowers (for example pair the flower of the first tree with the second, the third with the fourth, etc.) to make $\Gc$ a genuine garden, as was done in \cite{DH23}, but this is not necessary in the proof below.} as in Definition \ref{defgarden}. For any $k,k'\in\Zb_L^d$ we define a $(k,k')$-decoration of $\Tc$ or $\Gc$ to be a decoration in the sense of Definition \ref{defdec} such that $k_\rf=k$ and $k_\ff=k'$ for the root $\rf$ and flower $\ff$ of each tree.

We define the layering of a flower garden as in Definition \ref{deflayer}, but with the additional requirement that all the nodes on the stem of each tree (from the root to the flower) are in layer $p$. We also define the collection $\Hs_{p+1}$ of \emph{canonical layered flower gardens} $\Gc$ in the same way as in Definition \ref{defcanon} (but without any assumption about irreducibility of $\Gc$), where we start from $2R$ flower trees with all branching nodes, flowers and paired leaves in layer $p$, then divide the set of remaining (non-flower) leaves into subsets of size $\geq 4$ and replace each subset by an irreducible proper layered garden in $\Gs_p$. Note that Proposition \ref{canonequiv} also holds for $\Hs_{p+1}$ with essentially the same proof.
\end{df}
For any flower tree $\Tc$ and layered flower garden $\Gc$ of order $n$, and $k,k'\in\Zb_L^d$ and $p+1>t>s>p$, define the quantities
\begin{equation}\label{jtflower}\widetilde{\Jc}_\Tc(t,s,k,k')=\bigg(\frac{\delta}{2L^{d-1}}\bigg)^n\zeta(\Tc)\sum_\Ds\epsilon_\Ds\int_{\Dc}\prod_{\nf\in\Nc}e^{\zeta_\nf\pi i \cdot\delta L^{2\gamma}\Omega_\nf t_\nf}\mathrm{d}t_\nf\cdot\dirac(t_{\ff^{\mathrm{pr}}}-s)\prod_{\ff\neq\lf\in\Lc}a_{k_\lf}(p),\end{equation}
\begin{equation}\label{kqflower}\widetilde{\Kc}_\Gc(t,s,k,k')=\bigg(\frac{\delta}{2L^{d-1}}\bigg)^{n}\zeta(\Gc)\sum_\Is\epsilon_\Is\int_{\Ic}\prod_{\nf\in\Nc}e^{\zeta_\nf\pi i \cdot\delta L^{2\gamma}\Omega_\nf t_\nf}\mathrm{d}t_\nf\prod_{\ff}\dirac(t_{\ff^{\mathrm{pr}}}-s){\prod_{\ff\neq\lf\in\Lc}^{(+)}F_{\Lf_\lf}(k_\lf)},\end{equation} which are slight modifications of (\ref{defjt}) and (\ref{defkg0}). Here in (\ref{jtflower}), $\Ds$ is a $(k,k')$-decoration of $\Tc$, $\Dc$ is defined as in (\ref{timetree}), and the other objects are associated with the tree $\Tc$. In (\ref{kqflower}), $\Is$ is a $(k,k')$-decoration of $\Gc$, the other objects are associated with the garden $\Gc$, and the set $\Ic$ is defined as in (\ref{timegarden}); the second product is taken over all the flowers $\ff$ in $\Gc$, where $\ff^{\mathrm{pr}}$ is the parent node of $\ff$, and in the last product we assume $\lf$ has sign $+$ and is not any flower $\ff$ in $\Gc$.

Now we define the parametrix of the operator $1-\Ls_1$. Define the operator
\begin{equation}\label{operatorx}(\Xs\textit{\textbf{b}})_k(t)=\sum_{\zeta\in\{\pm\}}\sum_{k'}\int_p^{t}\Xs_{kk'}^\zeta(t,s)b_{k'}(s)\,\mathrm{d}s;\quad \textit{\textbf{b}}=b_{k'}(s),
\end{equation} with the kernel given by
\begin{equation}\label{operatorxker}\Xs_{kk'}^\zeta(t,s)=\sum_\Tc\widetilde{\Jc}_\Tc(t,s,k,k'),
\end{equation} where the summation is taken over all flower trees $\Tc$ of order $\leq N_{p+1}$ such that $\zeta_\rf=+$ and $\zeta_\ff=\zeta$ for the root $\rf$ and flower $\ff$.

Let $\Ys=(1-\Ls_1)\Xs-1$ and $\Ws=\Xs(1-\Ls_1)-1$. Note that the operator $\Ls_1$ defined in (\ref{eqnbk1.5}) formally corresponds to attaching two trees to a single node as siblings, as in the proof of Proposition 11.2 of \cite{DH21}, we then see that $\Ys$ and $\Ws$ both have the same expression as in (\ref{operatorx}) and (\ref{operatorxker}), but with the sum in (\ref{operatorxker}) taken over different sets of $\Tc$ (cf. Section 11.1.1 of \cite{DH23}). Namely, for $\Ys$ we require that (Y-1) the order $n$ of each tree $\Tc$ in the sum satisfy $n>N_{p+1}$, and (Y-2) the subtree rooted at each child node of $\rf$ has order $\leq N_{p+1}$; for $\Ws$ we require that (W-1) the order $n>N_{p+1}$, and (W-2) the subtree rooted at each sibling node of $\ff$ has order $\leq N_{p+1}$, and (W-3) the flower tree obtained by replacing the parent $\ff^{\mathrm{pr}}$ of $\ff$ with a new flower has also order $\leq N_{p+1}$. See Figure \ref{fig:flowertree} for illustration. Note that the above requirement imposes that the order of $\Tc$ satisfies $n\in[N_{p+1}+1,3N_{p+1}+1]$ for both $\Ys$ and $\Ws$.
  \begin{figure}[h!]
  \includegraphics[scale=.46]{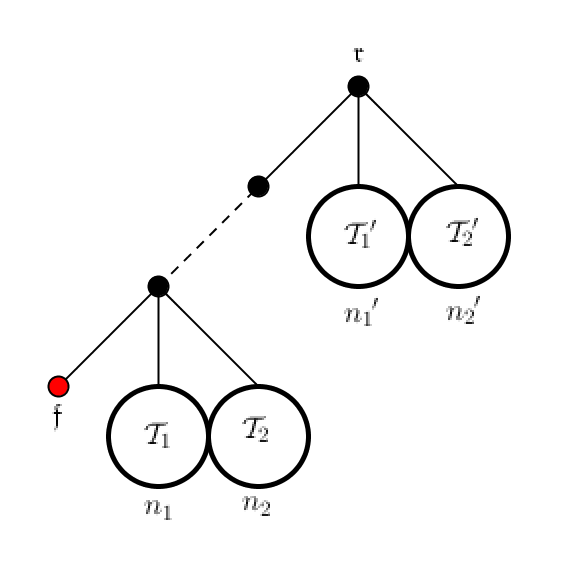}
  \caption{A flower tree $\Tc$ with root $\rf$ and flower $\ff$. Let the order of subtrees $\Tc_j$ be $n_j$ etc., and the order of $\Tc$ be $m$. Then the requirements for $\Ys_m$ is that $N_{p+1}<m\leq N_{p+1}+n_1'+n_2'+1$ and $n_j'\leq N_{p+1}$, the requirements for $\Ws_m$ is that $N_{p+1}<m\leq N_{p+1}+n_1+n_2+1$ and $n_j\leq N_{p+1}$.}
  \label{fig:flowertree}
\end{figure} 
\begin{prop}\label{flowerest1} Let $P$ be an integer that is large enough depending on $C_1$, assume also that $\delta$ is small enough depending on $P$. Let $\Eb_p$ be as in Proposition \ref{propansatz}. Then for any $(t,s,k,k')$ we have
\begin{equation}\label{flowerest2}\Eb_p|\Xs_{kk'}^\zeta(t,s)|^{2P}\leq\langle k-\zeta k'\rangle^{-4\Lambda_{p+1}P}L^{(100d)^3P}.
\end{equation} The same estimate holds for $\Ys$ and $\Ws$, but with an extra factor $e^{-100N_{p+1}}$ on the right hand side.

Next, let $M_1,M_2\geq L$ be two fixed dyadic numbers. Consider the definition (\ref{jtflower}) of $\widetilde{\Jc}_\Tc$ that occurs in (\ref{operatorxker}). Suppose we restrict to \[\max\{|k_\lf|:\ff\neq\lf\in\Lc\}\sim_0 M_1\] (or with $\sim_0$ replaced by $\lesssim_0$ if $M_1=L$) by a smooth cutoff function for the tree $\Tc$, where $\ff$ is the flower of $\Tc$ and $\Lc$ the set of leaves, then (\ref{flowerest2}) holds with an extra factor  $M_1^{-80dP}$ on the right hand side; the same is true for $\Ys$ and $\Ws$.

Moreover, suppose $|k|\geq M_1^{10}$, and in (\ref{jtflower}) we further restrict to \[\max_{\nf\in\Sc}|\Omega_\nf|\sim_0 M_2\] (or with $\sim_0$ replaced by $\lesssim_0$ if $M_2=L$) by a smooth cutoff function for the tree $\Tc$, where $\Sc$ is the set of all branching nodes $\nf$ in the stem of $\Tc$ \emph{other than} the parent node $\ff^{\mathrm{pr}}$ of $\ff$, then (\ref{flowerest2}) holds with an extra factor $M_2^{-P/9}$ on the right hand side. The same is true for $\Ys$ and $\Ws$.
\end{prop}
\begin{proof} We follow the arguments in Section \ref{secduhamel} used to expand (\ref{cmattr}) based on (\ref{defjt}), and now apply similar arguments to expand the left hand side of (\ref{flowerest2}) based on (\ref{jtflower}); note that now we are dealing with moments instead of cumulants of $(a_{k_\lf}(p))$, so we will make no irreducibility assumption about flower gardens. This gives, upon application of \eqref{eqncm2}, that
\begin{equation}\label{flowerexp2}\Eb_p|\Xs_{kk'}^\zeta(t,s)|^{2P}=\sum_{\Gc}\widetilde{\Kc}_\Gc(t,s,k,k')+\Es^{(1)},\end{equation} where the sum is taken over all flower gardens $\Gc=(\Tc_1,\cdots,\Tc_{2R})\in\Hs_{p+1}^{\mathrm{tr}}$ of width $2P$, and $\Es^{(1)}$ is a remainder term containing one of $\Es_p$ or $\widetilde{\Es}_p$ in (\ref{ansatz1}) and (\ref{ansatz3}). Here $\Hs_{p+1}^{\mathrm{tr}}$ is a subset of $\Hs_{p+1}$ defined similar to $\Gs_{p+1}^{\mathrm{tr}}$ in Definition \ref{defgtr}, where we require in Definition \ref{defflower} that each of the initial flower trees has order $\leq N_{p+1}$, and that each of the layered gardens replacing leaf subsets has order $\leq N_p$. For $\Ys$ and $\Ws$ we have the same formulas, but we also require the initial trees to satisfy (Y-1)--(Y-2) or (W-1)--(W-3) above.

We may treat $\Es^{(1)}$ just as in the proof of Proposition \ref{propansatz} in Section \ref{proofmain}; moreover, if the order of $\Gc$ is $n(\Gc)>N_{p+1}^4$, then the resulting term $\Es^{(2)}$ can also be treated just as in the proof of Proposition \ref{propansatz}. Thus we may restrict to $n(\Gc)\leq N_{p+1}^4$ in (\ref{flowerexp2}), so the definition of $\Hs_{p+1}^{\mathrm{tr}}$ is equivalent to each tree of $\Gc$ having $\leq N_{p+1}$ layer $p$ branching nodes. Similarly, the extra requirements for $\Ys$ and $\Ws$ are also just (Y-1)--(Y-2) or (W-1)--(W-3), but with the notion of order in these requirements being replaced by the number of layer $p$ branching nodes.

Now, to estimate (\ref{flowerexp2}), we simply repeat the arguments in Sections \ref{layerobject1}--\ref{redcount} devoted to the proof of Proposition \ref{layergarden}. Below we only discuss the necessary (minor) adjustments needed in the flower garden case here, as most of the proof is completely analogous.
\begin{enumerate}[{(1)}]
\item Regarding the decay factor $\langle k-\zeta k'\rangle^{-4\Lambda_{p+1}P}$: this is obtained by the same arguments as in Section \ref{redcount0}, using the decay of $F_{\Lf_\lf}(k_\lf)$ in (\ref{kqflower}); note that an algebraic sum of $k_\lf$ for all non-flower leaves in each tree equals $k-\zeta k'$. The extra factor $M_1^{-80dP}$ when assuming $\max\{|k_\lf|:\ff_j\neq\lf\in\Lc_j\}\sim_0 M_1$ comes in the same way (note that if we insert some cutoff to the tree $\Tc$ associated to $\Xs$, then the same cutoff will appear in each tree $\Tc_j$ occurring in the garden $\Gc$ associated with the expansion (\ref{flowerexp2})).
\item Regarding the extra power $L^{(100d)^3P}$: compared to (\ref{defkg0}), the modification in (\ref{kqflower}) in the flower garden case only involves two nodes (the flower $\ff$ and its parent $\ff^{\mathrm{pr}}$) for each tree, and each such modification may cause a loss of $\leq L^{80d}$. For example, this loss may come from taking the $\Sf^{40d,40d}$ norm of the input function $\mathbf{1}_{k_\ff=k'}$ which occurs due to the definition of $(k,k')$-decorations, or from the atoms corresponding to these one or two nodes in the molecule. They may also come from the $\dirac(t_{\ff^{\mathrm{pr}}}-s)$ factor which removes the integration in the $t_{\ff^{\mathrm{pr}}}$ variable, but $\Omega_{\ff^{\mathrm{pr}}}$ has at most $L^{10d}$ possible values once we fix each $k_\lf$ to a unit ball, so we may fix the value of $\Omega_{\ff^{\mathrm{pr}}}$ and safely get rid of the factor $e^{\zeta_\nf\pi i \cdot\delta L^{2\gamma}\Omega_\nf t_\nf}$ for $\nf=\ff^{\mathrm{pr}}$ at a loss of at most $L^{10d}$. In any case the total loss is $\leq L^{200dP}$ for the total of $2P$ trees.
\item Regarding the arbitrariness of $k$ and $k'$: all the proofs in Section \ref{layerobject1} through Section \ref{redcount} are translation invariant\footnote{Except those leading to the $\langle k_j\rangle^{-4\Lambda_{p+1}}$ factors in Proposition \ref{layergarden}; but these factors are replaced by the $\langle k-\zeta k'\rangle^{-4\Lambda_{p+1}P}$ factor, which is treated in (1).}, so this is not a problem once $k$ and $k'$ are fixed.
\item Regarding LF twisting and cancellation: we only need to consider those LF twists for $\Gc$ that may affect the property of belonging to $\Hs_{p+1}^{\mathrm{tr}}$. Note that the definition of $\Hs_{p+1}^{\mathrm{tr}}$, including the modified versions for $\Ys$ and $\Ws$, only involves the number of layer $p$ branching nodes in  (i) each tree $\Tc_j$ of $\Gc$, (ii) each of the trees rooted at the children nodes of each root $\rf_j$, and (iii) each of the trees rooted at the sibling nodes of each flower $\ff_j$, see Figure \ref{fig:flowertree}. Therefore, if the realization of a vine $\Vb$ in $\Gc_{\mathrm{sk}}$ does not contain any of the ``special nodes", i.e. the root $\rf_j$, any of the children nodes of $\rf_j$, the flower $\ff_j$, the parent and any sibling node of $\ff_j$, then any LF twist at this vine will not affect the property of belonging to $\Hs_{p+1}^{\mathrm{tr}}$. Clearly there are less than $20P$ vines containing one of these special nodes; for each of them we can use Proposition \ref{ladderl1new} to treat its ladder parts, and sum and integrate trivially over the $k_\ell$ and $t_v$ variables for $\ell$ and $v$ outside of the ladders (there are at most $20$ such variables per vine). This leads to a loss of at most $L^{20d}$ per vine, and at most $L^{400dP}$ in total.
\item Similarly, the property that all the nodes on the stem of each tree must belong to layer $p$, as in the definition of flower gardens in Definition \ref{defflower}, is also preserved under LF twisting. For example, consider $\Gc$ whose skeleton contains a Vine (II-a) in Figure \ref{fig:block_mole}, while the skeleton of its LF twist $\Gc'$ contains a Vine (II-b). We may assume the flower is in the tree $\Tc_1$ or $\Tc_2$ in the notation of Figure \ref{fig:block_mole} (otherwise it is not affected by LF twisting); now if all stem nodes in $\Gc$ are in layer $p$, then in particular $\Lf_{\uf_2}=p$, which implies that $\Lf_{\uf_3}=\Lf_{\uf_4}=p$ as $\Lf_{\uf_2}\leq\min(\Lf_{\uf_3},\Lf_{\uf_4})$. Thus all stem nodes in $\Gc'$ must also be in layer $p$, including those in the regular tree $\Tc_{(\uf_2)}$, because the layer of $\uf_2$ in the skeleton is given by the layer of the lone leaf of the regular tree at $\uf_2$ in $\Gc$ or $\Gc'$.
\item Regarding the extra factor $M_2^{-P/9}$: if we insert the $M_2$ cutoff, then for each $j$ we fix a branching node $\nf_j\in\Sc_j$ and insert the cutoff such that $|\Omega_{\nf_j}|\sim_0 M_2$ (otherwise these cutoffs play no role in the proof).  Note for each $j$, the expression (\ref{kqflower}) involves an integration of the exponential factor $e^{\zeta_{\nf_j}\pi i\cdot\delta L^{2\gamma}\Omega_{\nf_j}t_{\nf_j}}$ in the time variable $t_j$. By integrating in $t_j$ before manipulating everything else, and arguing similarly as in the proof of Corollary 11.3 in \cite{DH21}, we gain a factor $M_2^{-P/9}$ from this integration. This also leads to a loss of at most $L^{20dP}$ by fixing the values of $\Omega_{\nf_j}$ as in (2) , which is acceptable. We then proceed with the same arguments in Section \ref{layerobject1} through Section \ref{redcount} to treat the rest of the expression (\ref{kqflower}) involving the rest of the garden. All the possible loss (such as those described in (2)) that may come along the way again add up to at most $L^{100dP}$ which is acceptable.
\item Regarding LF twisting in the context of (6): here we are assuming $|k|\geq M_1^{10}$. For any (CL) vine $\Vb$ in $\Mb(\Gc_{\mathrm{sk}})$ that is SG, let the nodes $(\uf_1,\uf_{11},\uf_{21},\uf_{22})$ be as in Proposition \ref{block_clcn}, then we have the following: if the flower is a descendant of $\uf_1$, then it must be a descendant of $\uf_{11}$. In fact, if this is not true, then the flower must be a descendant of $\uf_{21}$ or $\uf_{22}$, which means that exactly one of the nodes $\uf_{21}$ and $\uf_{22}$ is on the stem. But this is not possible, as $|k_{\uf_{21}}-k_{\uf_{22}}|\leq 1$ by the SG assumption, while the difference $|k_\nf-k_{\nf'}|$ for any node $\nf$ on stem and $\nf'$ off stem is at least $|k|/2$ since $|k|\geq M_1^{10}$. Therefore, see Figure \ref{fig:block_mole}, we know that any LF twisting for $\Gc$ will not affect the structure of the stem. In particular it will not affect the value of any $\Omega_{\nf}$ for any $\nf$ on the stem of any tree, nor the smooth cutoff functions of these $\Omega_\nf$ involving $M_2$. The rest of the discussion for LF twisting is the same as in (4) and (5).
\end{enumerate}
With the above discussions, we can then obtain the desired bound for (\ref{flowerexp2}): if we restrict to $\Gc$ of order $n$, then this contribution is bounded by
\[\langle k-\zeta k'\rangle^{-4\Lambda_{p+1}P}L^{(100d)^3P}\cdot \delta^{2\nu n},\]
with the corresponding extra factors if the cutoff functions involving $M_1$ or $M_2$ are present. By adding up all $n\geq 0$ for $\Xs$, and adding up all $N_{p+1}+1\leq n\leq N_{p+1}^4$ for $\Ys$ and $\Ws$, this completes the proof of (\ref{flowerest2}), as well as the variants with the cutoff functions.
\end{proof}
\subsection{Proof of Proposition \ref{fixedpoint}} In this section we prove Proposition \ref{fixedpoint}.

\uwave{Part 1: Proof of (\ref{fixedest1}) and (\ref{fixedest2}).} Recall that
\begin{equation}\label{trterm}a_k^{\mathrm{tr}}(t)=\sum_{\Tc}(\Jc_\Tc)_k(t)\end{equation} from (\ref{defjn}) and (\ref{defbk}), where the sum is taken over all trees of sign $+$ and order $\leq N_{p+1}$. By repeating the same arguments in the proof of Proposition \ref{propansatz} in Section \ref{proofmain}, for each fixed $k\in\Zb_L^d$ and $t\in[p,p+1]$ we can bound
\begin{equation}\label{trterm1}\Eb_p|(a_k^{\mathrm{tr}})(t)|^2\leq L^2\langle k\rangle^{-4\Lambda_{p+1}}.\end{equation}

In the same way, by (\ref{eqnbk1.5}), we can write $(\Ls_0)_k(t)$ in a form similar to (\ref{trterm}), but with the restriction that the order of $\Tc$ is $>N_{p+1}$, while the order of the subtree rooted at each child node of the root $\rf$ is $\leq N_{p+1}$. The same proof can be carried out, and with the lower bound on the order of $\Tc$, using (\ref{taylorbound}) and (\ref{estcouple2}), we get the same bound (\ref{trterm1}) for $\Ls_0$ but with extra factor $e^{-100N_{p+1}}$ on the right hand side. In addition, we can treat the time derivative $\partial_ta_k^{\mathrm{tr}}$ (which simply contains an extra $\dirac(t_{\rf}-t)$ factor in $\Jc_\Tc$ similar to the one in (\ref{jtflower})) using the same arguments in the proof of Proposition \ref{flowerest1}, to get the same bound (\ref{trterm1}) with an additional $L^{100d}$ loss. The same holds also for $\partial_t(\Ls_0)_k(t)$.

Now note that, (\ref{fixedest1}) and (\ref{fixedest2}), compared to (\ref{trterm1}), requires large deviation estimates for the \emph{supremum} of $\langle k\rangle^{(3/2)\Lambda_{p+1}}|a_k^{\mathrm{tr}}(t)|$ and $\langle k\rangle^{(3/2)\Lambda_{p+1}}|(\Ls_0)_k(t)|$. This can be solved using the following trick in \cite{DH21}: by Sobolev in $t$ we have
\begin{equation}\label{sobolev}\sup_{t,k}\big(\langle k\rangle^{(3/2)\Lambda_{p+1}}|a_k^{\mathrm{tr}}(t)|\big)\lesssim_0\big\|\langle k\rangle^{(3/2)\Lambda_{p+1}}a_k^{\mathrm{tr}}(t)\big\|_{L_{k,t}^2}+\big\|\langle k\rangle^{(3/2)\Lambda_{p+1}}\partial_ta_k^{\mathrm{tr}}(t)\big\|_{L_{k,t}^2},\end{equation} where the norm is taken for $k\in\Zb_L^d$ and $t\in[p,p+1]$ (here $L_k^2$ is with respect to the counting measure). Then, using summability in $k$, we can bound the second moment of the above supremum by the right hand side of (\ref{trterm1}) with an additional $L^{200d}$ loss. The bounds (\ref{fixedest1}) and (\ref{fixedest2}) then follow from the standard Chebyshev inequality.

\uwave{Part 2: Proof of (\ref{fixedest3}).} Let $Z:=Z^{(5/4)\Lambda_{p+1}}([p,p+1])$, and recall the definition of $\Xs,\Ys,\Ws$ in Section \ref{flower}. For any operator $\Ks$ with kernel $\Ks_{kk'}^\zeta(t,s)$ as in (\ref{operatorx}), it is clear that
\[\|\Ks\|_{Z\to Z}\leq L^d\sup_{\zeta,k,k',t,s}\langle k-\zeta k'\rangle^{(3/2)\Lambda_{p+1}}|\Ks_{kk'}^\zeta(t,s)|:=L^d\|\Ks\|_{\Nf}.\] Our goal is to prove large deviation estimates for $\|\Xs\|_\Nf$; the cases of $\Ys$ and $\Ws$ are analogous.

Let $M_1,M_2\geq L$. We insert smooth cutoff functions, which are just the ones described in Proposition \ref{flowerest1}, to the $\widetilde{\Jc}_\Tc$ terms defined in (\ref{jtflower}) that occur in (\ref{operatorxker}). Denote the resulting contribution by $\Xs_{M_1M_2}$ (for $|k|\leq M_1^{10}$ we do not insert the $M_2$ cutoff, in which case we write $\Xs_{M_1}$ instead of $\Xs_{M_1M_2}$). The large deviation estimate we will prove is of form\begin{equation}\label{newnormn}\Eb_p\|\Xs_{M_1M_2}\|_\Nf^{2P}\leq L^{C_2}M_1^{-P}M_2^{-P/18}
\end{equation} (for $\Xs_{M_1}$ there is no $M_2^{-P/18}$ factor), and similar bounds for $\Ys$ and $\Ws$ but with the extra factor $e^{-100N_{p+1}}$ on the right hand side. The proof of (\ref{newnormn}) is very similar to the proof of Proposition 12.2 in \cite{DH21}, where we apply Lemma \ref{finitelem} to represent $k$ by a finite system and thus reduce to finitely many possibilities of $k$ (under suitable cutoffs), and estimate the moment for each fixed $k$ by Proposition \ref{flowerest1}. Below we will only sketch the main points.
\begin{enumerate}[{(1)}]
\item Like the cases of (\ref{fixedest1}) and (\ref{fixedest2}) above, if we fix $(k,k',t,s)$, then the desired bound follows from (\ref{flowerest2}); the main difficulty comes from the fact that we are taking $L^\infty$ norms in $(k,k')$ and $(t,s)$.
\item For the $L^\infty$ norm in $(k,k')$, we only need to consider the case when $|k|\geq M_1^{10}$. In fact if $|k|\leq M_1^{10}$ then $(k,k')$ has at most $M_1^{20d}$ choices, so we may replace the $L_{k,k'}^\infty$ norm by the $L_{k,k'}^{2P}$ norm which allows to fix $(k,k')$ and apply Proposition \ref{flowerest1}. The subsequent summation in $(k,k')$ loses at most $M_1^{20d}$ which is acceptable in (\ref{newnormn}).
\item In the same way, we may also assume $|k|\geq M_2^{10}$ (otherwise we lost at most $M_2^{20d}$ from summation in $(k,k')$, which can be covered by the $M_2^{-P/9}$ gain in Proposition \ref{flowerest1} now that we know $|k|\geq M_1^{10}$). Using the definition of $M_2$ and the formula for $\Omega_\nf$, this in particular implies that all branching nodes in the stem of $\Tc$ must have the same sign (as, if two consecutive branching nodes in the stem have opposite signs, then the corresponding resonance factor will satisfy $|\Omega_\nf|\geq |k|^2$ which is not possible).
\item Note that the expression (\ref{jtflower}), hence $\Xs_{M_1M_2}$, depends on $k$ and $k'$ \emph{only via} the $\Omega_\nf$ for branching nodes $\nf$ in the stems of the flower trees $\Tc$; moreover each such $\Omega_\nf$ has the form $\langle k,z\rangle+\widetilde{\Omega}$, where $z$ is a vector depending only on the $k_\lf$ variables for non-flower leaves $\lf$ (in particular $|z|\leq M_1^2$), and $\widetilde{\Omega}$ also depends only on these $k_\lf$ variables (in particular $|\widetilde{\Omega}|\leq M_1^4$).
\item In addition, note from (\ref{jtflower}) and the $\dirac(t_{\ff^{\mathrm{pr}}}-s)$ factor that $\widetilde{\Jc}_\Tc$ can be written as a sum of terms of form $e^{\zeta_\nf\pi i \cdot\delta L^{2\gamma}\Omega_\nf s}\cdot {\Jc}_\Tc^\dagger$, where $\nf=\ff^{\mathrm{pr}}$ and ${\Jc}_\Tc^\dagger$ does not depend on $\Omega_{\ff^{\mathrm{pr}}}$. This sum is taken over $(k_\mf)$ for siblings $\mf$ of $\ff$ (which has at most $M_1^{10d}$ choices), and once these $(k_\mf)$ are fixed the value of $\Omega_{\ff^{\mathrm{pr}}}$ is also a fixed function of $k$ and $k'$. Thus, at a loss of at most $M_1^{10d}$ (which is amplified to $M_1^{20dP}$ with the $2P$ power, but still can be absorbed by the $M_1^{-80dP}$ factor in Proposition \ref{flowerest1}), we may get rid of the exponential factor and replace $\widetilde{\Jc}_\Tc$ by $\Jc_\Tc^\dagger$ in the definition of $\Xs_{M_1M_2}$, so now $\Xs_{M_1M_2}$ depends on $k$ and $k'$ only via the $\Omega_\nf$ for branching nodes $\nf\neq \ff^{\mathrm{pr}}$ in the stems of the flower trees $\Tc$.
\item Since $|\Omega_\nf|\leq M_2$ for all branching nodes $\nf\neq\ff^{\mathrm{pr}}$ on the stem, we know that also $|\langle k,z\rangle|\leq 2M_2$ (and $|z|\leq M_1^2$) in the context of (4) above. By Lemma \ref{finitelem}, we can find $k_*\in\mathfrak{Rep}$, where $\mathfrak{Rep}$ is a fixed subset of $\Zb_L^d$ of at most $(M_1M_2)^{C_0}$ elements\footnote{This part of argument is easier than the irrational torus case \cite{DH21}, as in \cite{DH21} we also have $y\in\Rb^q$ and thus need an extra Sobolev argument in $y$.}, such that (\ref{represent}) holds. In this case, using the above bounds for $\langle k,z\rangle$, we conclude that $(\Xs_{M_1M_2})_{kk'}^\zeta(t,s)=(\Xs_{M_1M_2})_{k_*k_*'}^\zeta(t,s)$ with the $M_2$ cutoffs, where $k_*'$ is such that $k_*-k_*'=k-k'$ (we may also assume $k_*'\in\mathfrak{Rep}$ by enlarging $\mathfrak{Rep}$, as $|k-k'|\leq M_1^2$). We can then replace the supremum in $(k,k')$ by the supremum in $(k_*,k_*')$ and control it by the $L_{(k_*,k_*')\in \mathfrak{Rep}^2}^{2P}$ norm, with a loss of at most $(M_1M_2)^{C_0}$ due to summation in $(k_*,k_*')$, which is acceptable in view of the gain $M_1^{-80dP}M_2^{-P/9}$ in Proposition \ref{flowerest1}.
\item Finally, for the $L^\infty$ norm in $(t,s)$, we shall apply Gagliardo-Nirenberg inequality in slightly more precise way than (\ref{sobolev}) (this applies to any of the cases considered above). Namely, concerning only the $(t,s)$ part, we have
\begin{equation}\label{sobolev2}\sup_{t,s}|\Ks(t,s)|^{2P}\leq M_2^{P/100}\cdot\|\Ks\|_{L_{t,s}^{2P}}^{2P}+M_2^{-100P}\cdot\|\partial_{t,s}\Ks\|_{L_{t,s}^{2P}}^{2P}\end{equation} for any function $\Ks$. Note that the bound in Proposition \ref{flowerest1} without $\partial_{t,s}$ derivative gains a factor $M_2^{-P/9}$ which cancels the $M_2^{P/100}$ factor in (\ref{sobolev2}).
\item On the other hand, the term in (\ref{sobolev2}) with $\partial_{t,s}$ derivative can be estimated in the same way as in Proposition \ref{flowerest1} \emph{after replacing} $\widetilde{\Jc}_\Tc$ by $\Jc_\Tc^\dagger$ as in (5) above (here the $\partial_t$ derivative introduces an extra $\dirac(t_\rf-t)$ factor, while the $\partial_s$ derivative introduces an extra $\dirac(t_\nf-s)$ factor with $\nf$ being the parent or a non-flower child of $\ff^{\mathrm{pr}}$, and all these factors can be treated just as the one in (\ref{jtflower}); the case when $\ff^{\mathrm{pr}}=\rf$ only needs some trivial adjustments). This leads to a bound similar to (\ref{flowerest2}) but with extra loss of at most $L^{(100d)^3P}$ and possibly without the $M_2^{-P/9}$ decay, but we still have the $M_2^{-100P}$ weight in (\ref{sobolev2}) which provides the needed decay.
\end{enumerate}

By putting (1) through (8) together, we arrive at the proof of (\ref{newnormn}). Now, with (\ref{newnormn}) and a simple Chebyshev argument, it is easy to see that by excluding a subset $\Gf\subset \Ff_p$, which has gauge and space translation symmetry and has measure $\leq e^{-4N_{p+1}}$ (so $\Ff_{p+1}=\Ff_p\backslash\Gf$), we have that
\[\|\Ys\|_{Z\to Z}+\|\Ws\|_{Z\to Z}\leq 1/2,\quad \|\Xs\|_{Z\to Z}\leq e^{2N_{p+1}}.\] By Neumann series, this means that the operators $(1-\Ls_1)\Xs$ and $\Xs(1-\Ls_1)$ are both invertible. Therefore $1-\Ls_1$ has both left and right inverses, hence $1-\Ls_1$ is invertible and\[\|(1-\Ls_1)^{-1}\|_{Z\to Z}\leq 2\|\Xs\|_{Z\to Z}\leq e^{3N_{p+1}}.\] This proves (\ref{fixedest3}).

\uwave{Part 3: Proof of (\ref{fixedest4}).} This is essentially trivial once we have (\ref{fixedest1})--(\ref{fixedest3}), if we just use the easily established facts that
\[\|\Ls_3(\textit{\textbf{b}},\textit{\textbf{b}},\textit{\textbf{b}})\|_{Z}\leq L^{30d}\|\textit{\textbf{b}}\|_Z^3,\quad \|\Ls_2(\textit{\textbf{b}},\textit{\textbf{b}})\|_{Z}\leq L^{30d}\|\textit{\textbf{b}}\|_Z^2\cdot\|a_k^{\mathrm{tr}}\|_Z\leq e^{4N_{p+1}}\|\textit{\textbf{b}}\|_Z^2.\] This completes the proof of Proposition \ref{fixedpoint}.
\appendix
\section{Miscellaneous results}\label{appendmisc}
In this appendix we collect some useful lemmas needed in the main proof.
\subsection{Cumulants}\label{appendcm} We start with the properties of cumulants.
\begin{lem}[Property of cumulants]\label{propertycm} Let $S_i\,(1\leq i\leq r)$ be pairwise disjoint index sets, and $X_j\,(j\in S_i)$ be complex random variables. Let $Y_i:=\prod_{j\in S_i}X_j$, then we have
\begin{equation}\label{eqncm3}\Kb(Y_1,\cdots,Y_r)=\sum_{\Pc}\prod_{B\in\Pc}\Kb\big((X_j)_{j\in B}\big),
\end{equation} where the sum is taken over all partitions $\Pc$ of $S_1\cup\cdots\cup S_r$, under the assumption that there is no proper subset $A\subset\{1,\cdots,r\}$ such that $\cup_{i\in A}S_i$ equals the union of \emph{some of the sets} $B\in\Pc$. The same results hold if the cumulants are taken assuming a particular event $\Ff$.
\end{lem}
\begin{proof} Note that by (\ref{eqncm2}) we have
\begin{equation}\label{cmproof1}\Eb(Y_1\cdots Y_r)=\Eb\bigg(\prod_{j\in X_1\cup\cdots\cup X_r}X_j\bigg)=\sum_{\Pc}^{(*)}\prod_{B\in\Pc}\Kb\big((X_j)_{j\in B}\big),\end{equation} where the sum is taken over all partitions of $S_1\cup\cdots\cup S_r$, without the restriction made in (\ref{eqncm3}). If we define those partitions satisfying the restriction in (\ref{eqncm3}) as irreducible partitions, then any partition $\Pc$ can be uniquely decomposed into irreducible partitions; i.e. there is a unique partition $\Pc_0$ of $\{1,\cdots,r\}$ such that $\Pc=\cup_{A\in\Pc_0}\Pc_A$, where $\Pc_A$ is an irreducible partition of $\cup_{i\in A}S_i$. This implies that
\begin{equation}\label{cmproof2}\Eb(Y_1\cdots Y_r)=\sum_{\Pc_0}\prod_{A\in\Pc_0}\sum_{\Pc_A}\prod_{B\in\Pc_A}\Kb\big((X_j)_{j\in B}\big),\end{equation}where $\Pc_0$ and $\Pc_A$ is as above. Also by (\ref{eqncm2}) we have that
\begin{equation}\label{cmproof3}\Eb(Y_1\cdots Y_r)=\sum_{\Pc_0}\prod_{A\in\Pc_0}\Kb\big((Y_i)_{i\in A}\big).\end{equation} Note that both (\ref{cmproof2}) and (\ref{cmproof3}) remain true (with obvious changes) if $\{1,\cdots,r\}$ is replaced by any subset $A\subset\{1,\cdots,r\}$. It then follows, by induction on $|A|$, that (\ref{eqncm3}) also holds (with obvious changes) if $\{1,\cdots,r\}$ is replaced by any subset $A\subset\{1,\cdots,r\}$, and lastly setting $A=\{1,\cdots,r\}$ yields (\ref{eqncm3}). The results assuming $\Ff$ follows in the same way.
\end{proof}
\subsection{Analytical lemmas} Next we prove some analytical lemmas concerning time integrations.
\begin{lem}\label{timeintlem} Consider the operators
\begin{equation}\label{defichi}
\begin{aligned}
\Ic_\chi F(t,s)&=\int_0^t\int_0^s\chi(t)\chi(t')\chi(s)\chi(s')F(t',s')\,\mathrm{d}t'\mathrm{d}s',\\
\Jc_\chi F(t,s,k)&=\int_s^t\int_s^{t'}\chi(t)\chi(t')\chi(s)\chi(s')F(t',s',k)\,\mathrm{d}s'\mathrm{d}t',
\end{aligned}
\end{equation} where $\chi$ is a smooth cutoff function. Then, for \emph{complex valued} $F$, we have
\begin{equation}\label{estichi}\|\Ic_\chi F\|_{\Xf^\eta}\lesssim_1\|F\|_{L^{1+3\eta}(\Rb^2)},\qquad \|\Jc_\chi F\|_{\Xf^{-\eta,0,0}}\lesssim_1\|F\|_{L_k^\infty L^1(\Rb^2)}.
\end{equation}
\end{lem}
\begin{proof} For the first inequality, from Lemma 3.1 of \cite{DNY21} we have
\begin{equation}\label{proofichi}\widehat{\Ic_\chi F}(\xi,\sigma)=\int_{\Rb^2}K(\xi,\sigma,\xi',\sigma')\widehat{F}(\xi',\sigma')\,\mathrm{d}\xi'\mathrm{d}\sigma',
\end{equation} where the kernel $K$ satisfies
\begin{equation}\label{proofichi-1}|K|\lesssim_0\frac{1}{\langle \xi\rangle\langle \xi-\xi'\rangle\langle \sigma\rangle\langle \sigma-\sigma'\rangle}.\end{equation} By the definition of $\Xf^\eta$ norm (and assuming $F$ has compact support), we then have
\begin{align}
\|\Ic_\chi F\|_{\Xf^\eta}&\lesssim_1\int_{\Rb^4}\langle \xi\rangle^{\eta-1}\langle \sigma\rangle^{\eta-1}\langle \xi-\xi'\rangle^{-1}\langle \sigma-\sigma'\rangle^{-1}|\widehat{F}(\xi',\sigma')|\,\mathrm{d}\xi\mathrm{d}\sigma\mathrm{d}\xi'\mathrm{d}\sigma'\nonumber\\
&\lesssim_1\int_{\Rb^2}\langle \xi'\rangle^{\eta-1}\langle \sigma'\rangle^{\eta-1}|\widehat{F}(\xi',\sigma')|\,\mathrm{d}\xi'\mathrm{d}\sigma'\nonumber\\
\label{proofichi-2}&\lesssim_1\|\widehat{F}\|_{L^{1/(2\eta)}(\Rb^2)}\lesssim_1\|F\|_{L^{1/(1-2\eta)}(\Rb^2)}\lesssim_1\|F\|_{L^{1+3\eta}(\Rb^2)}.
\end{align}
As for the second inequality, we can calculate
\begin{align}
\widehat{\Jc_\chi F}(\xi,\sigma,k)&=\int_{\Rb^2} e^{-i(t\xi+s\sigma)}\int_s^t\int_s^{t'}\chi(t)\chi(t')\chi(s)\chi(s')F(t',s',k)\,\mathrm{d}s'\mathrm{d}t'\nonumber\\
&=\int_{t'>s'}\chi(t')\chi(s')F(t',s',k)\,\mathrm{d}t'\mathrm{d}s'\int_{t'}^\infty\int_{-\infty}^{s'}\chi(t)\chi(s)e^{-i(t\xi+s\sigma)}\,\mathrm{d}s\mathrm{d}t\nonumber\\
\label{proofichi-3}&+\int_{t'<s'}\chi(t')\chi(s')F(t',s',k)\,\mathrm{d}t'\mathrm{d}s'\int_{-\infty}^{t'}\int_{s'}^\infty\chi(t)\chi(s)e^{-i(t\xi+s\sigma)}\,\mathrm{d}s\mathrm{d}t.
\end{align}
From this we easily see that $|\widehat{\Jc_\chi F}(\xi,\sigma,k)|\lesssim_1\langle\xi\rangle^{-1}\langle \sigma\rangle^{-1}\|F(t',s',k)\|_{L^1(\Rb^2)}$ for each $k$, which finishes the proof in view of the definition (\ref{normX}).
\end{proof}
\begin{lem}\label{timeintlem2} Consider the following operators
\begin{equation}\label{timeintlem2-1}
\begin{aligned}
\Uc_\chi F(t,s,k)&=\int_s^t\int_s^{t_1}\chi(t)\chi(t_1)\chi(s)\chi(s_1)e^{i\Gamma_1(k)(t_1-s_1)}F(t_1,s_1,k)\,\mathrm{d}s_1\mathrm{d}t_1,\\
\Vc_\chi F(t,s,k)&=\int_{s}^t\int_{s}^{t_1}\int_s^{t_2}\int_s^{s_2}\chi(t)\chi(t_1)\chi(t_2)\chi(s)\chi(s_1)\chi(s_2)\\&\qquad\qquad\qquad\qquad\times e^{i\Gamma_1(k)(t_1-t_2)}e^{i\Gamma_2(k)(s_2-s_1)}F(t_2,s_2,k)\,\mathrm{d}s_1\mathrm{d}s_2\mathrm{d}t_2\mathrm{d}t_1,
\end{aligned}
\end{equation} then for complex valued $F$ we have
\begin{equation}\label{timeintlem2-2}
\|\Uc_\chi F\|_{\Xf^{1/4,0,0}}\lesssim_0\|\partial_{t_1}\partial_{s_1}F\|_{L_k^\infty L^1(\Rb^2)},\qquad\|\Vc_\chi F\|_{\Xf^{1/4,0,0}}\lesssim_0\|\partial_{t_2}\partial_{s_2}F\|_{L_k^\infty L^1(\Rb^2)}.
\end{equation}
\end{lem}
\begin{proof} Assume $\|\partial_{t_j}\partial_{s_j}F\|_{L_k^\infty L^1(\Rb^2)}=1$, where $j$ is $1$ or $2$ in the case of $\Uc_\chi$ or $\Vc_\chi$. We may assume $F$ is compactly supported, then a simple calculation yields that
\begin{equation}\label{weakdecay}|\widehat{F}(\xi,\sigma,k)|\lesssim_0\langle \xi\rangle^{-1}\langle\sigma\rangle^{-1},\qquad |\Fs(F\cdot\mathbf{1}_{t_j>s_j})(\xi,\sigma,k)|\lesssim_0\langle \xi\rangle^{-1/2}\langle\sigma\rangle^{-1/2}
\end{equation} uniformly in $k$, where $\Fs$ is the Fourier transform in $(t,s)$.

Consider first $\Uc_\chi$. By calculations similar to those for $\Jc_\chi$ in the proof of Lemma \ref{timeintlem}, we get
\begin{multline}
\label{timeintproof2-1}
\widehat{\Uc_\chi F}(\xi,\sigma,k)=\int_{t_1>s_1}\chi(t_1)\chi(s_1)e^{i\Gamma_1(k)(t_1-s_1)}F(t_1,s_1,k)\,\mathrm{d}t_1\mathrm{d}s_1\cdot\int_{t_1}^\infty \chi(t)e^{-it\xi}\,\mathrm{d}t\cdot\int_{-\infty}^{s_1}\chi(s)e^{-is\sigma}\,\mathrm{d}s\\
+(\mathrm{similar\ terms}).
\end{multline} We may restrict to the case $|\xi|,|\sigma|\geq 1$ (otherwise the proof is similar and much easier), and integrate by parts in the two integrals in $t$ and $s$. The resulting bulk term involving $\chi'$ has the same form but has better decay in $\xi$ and $\sigma$, and can be treated in the same way; for convenience we will only consider the hardest case, which is the boundary term. With this simplification we thus write
\begin{multline}\label{timeintproof2-2}
\widehat{\Uc_\chi F}(\xi,\sigma,k)=\frac{1}{\xi\sigma}\int_{t_1>s_1}\chi^2(t_1)\chi^2(s_1)F(t_1,s_1,k)e^{i(\Gamma_1(k)-\xi)t_1}\cdot e^{i(-\Gamma_1(k)-\sigma)s_1}\,\mathrm{d}t_1\mathrm{d}s_1\\+(\mathrm{similar\ or\ better\ terms}).
\end{multline}Using (\ref{weakdecay}) we get that
\begin{equation}|\widehat{\Uc_\chi F}(\xi,\sigma,k)|\lesssim_0\frac{1}{\langle \xi\rangle\langle \sigma\rangle}\cdot\frac{1}{\langle \Gamma_1(k)-\xi\rangle^{1/2}\langle \Gamma_1(k)+\sigma\rangle^{1/2}}\lesssim_0\frac{1}{\langle \xi\rangle\langle \sigma\rangle\langle \xi+\sigma\rangle^{1/2}}
\end{equation} uniformly in $k$, which easily implies (\ref{timeintlem2-2}) for $\Uc_\chi$.

The estimate for $\Vc_\chi$ is similar. We have
\begin{equation}\label{timeintproof2-3}
\widehat{\Vc_\chi F}(\xi,\sigma,k)=\int_{t_2>s_2}\chi(t_2)\chi(s_2)F(t_2,s_2,k)\cdot G(t_2,\xi,k)\cdot H(s_2,\sigma,k)\,\mathrm{d}t_2\mathrm{d}s_2+(\mathrm{similar\ terms}),
\end{equation} where
\begin{equation}\label{timeintproof2-4}
\begin{aligned}
G(t_2,\xi,k)&=\int_{t>t_1>t_2}\chi(t)\chi(t_1)e^{i\Gamma_1(k)(t_1-t_2)}e^{-it\xi}\,\mathrm{d}t\mathrm{d}t_1,\\
H_j(s_2,\sigma,k)&=\int_{s<s_1<s_2}\chi(s)\chi(s_1)e^{i\Gamma_2(k)(s_2-s_1)}e^{-is\sigma}\,\mathrm{d}s\mathrm{d}s_1.
\end{aligned}
\end{equation} We will assume $|\xi|,|\sigma|\geq 1$. In the expression of $G$, we first integrate in $t$ and may replace this integral by the boundary term $(i\xi)^{-1}\chi(t_1)e^{-it_1\xi}$ as above; this yields that
\[G(t_2,\xi,k)=(i\xi)^{-1}e^{-it_2\xi}\int_{t_2}^\infty\chi^2(t_1)e^{i(\Gamma_1(k)-\xi)(t_1-t_2)}\,\mathrm{d}t_1+(\mathrm{better\ terms}),\] and notice that this remaining integral is bounded in Schwartz class with respect to $t_2$ (for example by a change of variable $t_1-t_2:=u$), uniformly in $\xi$ and $k$. The same happens for $H$, and putting together into (\ref{timeintproof2-3}) we get that
\begin{equation}|\widehat{\Vc_\chi F}(\xi,\sigma,k)|\lesssim_0\frac{1}{\langle \xi\rangle^{3/2}\langle \sigma\rangle^{3/2}}
\end{equation} uniformly in $k$, which implies (\ref{timeintlem2-2}) for $\Vc_\chi$.
\end{proof}
\subsection{Combinatorial lemmas} Next we present some combinatorial results needed in the proof. Most of them are the same as in \cite{DH21,DH23}, but we include them here for reader's convenience.
\begin{lem}\label{treelemma}
Given a tree $\Tc$ of order $m$.
\begin{enumerate}[{(1)}]
\item We can order all the branching nodes of $\Tc$ as $\nf_1,\cdots,\nf_m$, such that the following holds for any decoration $(k_\nf)$: for each $j$ there exists $j'<j$ such that $k_{\nf_{j'}}-k_{\nf_j}$ is an algebraic sum of at most $\Lambda_j$ distinct terms in $(k_\lf)$ where $\lf$ ranges over the leaves, and $\prod_{j=1}^m\Lambda_j\leq C_0^m$.
\item Moreover, consider the collections of all decorations $(k_\nf)$ suc that $|k_\lf-k_\lf^0|\leq 1$ for each leaf $\lf$, where $k_\lf^0\in\Zb_L^d$ are fixed vectors for each leaf $\lf$. Then this collection can be divided into at most $C_0^m$ sub-collections, such that for any decoration $(k_\nf)$ in a given sub-collection, each $k_\nf$ belongs to a fixed unit ball that depends only on $(k_\nf^0)$, $\nf$ and the sub-collection but not on $(k_\nf)$ itself.
\end{enumerate}
\end{lem}
\begin{proof} See Lemma 6.6 of \cite{DH21} and Lemma A.4 of \cite{DH23}.
\end{proof}
\begin{lem}\label{subsetvc} Consider a given molecule with a given decoration. Suppose $\Ub$ is either a DV, or vine-like object that is LG or ZG, and assume $\Ub$ contains some disjoint SG vine-like objects $\Ub_j$ as subsets. Then each $\Ub_j$ is either one double bond or two double bonds or Vine (V), or the adjoint of a VC denoted by $\Ub_j'$, which is either Vine (V) or two double bonds. In any case, $\Ub$ still remains connected after removing all the $\Ub_j$ (in case $\Ub_j$ is a vine) or $\Ub_j'$ (in case $\Ub_j'$ is a vine).
\end{lem}
\begin{proof} First assume $\Ub$ is a VC or HVC, which is concatenated by one or multiple vine ingredients. Then each $\Ub_j$ must be a subset of a single vine ingredient, because the gap of $\Ub_j$ is not the same as the gap of $\Ub$, hence $\Ub_j$ cannot be concatenated by any of the vine ingredients of $\Ub$. Therefore we may focus on a single vine ingredient, which is Vine (I)--(VIII) as in Figure \ref{fig:vines}. The result is then self-evident by examining the structures of these vines. The case of DV is similar, since each $\Ub_j$ must be a subset of one of the two vines (V) forming $\Ub$.
\end{proof}
\begin{lem}\label{finitelem} Fix $M\geq L$. Then for any $k=(k^1,\cdots,k^d)\in \Zb_L^d$, there must exist $0\leq q \leq d$ and nonzero orthogonal vectors $v_j\in \Zb_L^d\,(1\leq j\leq q)$, and vector $y=(y^1,\cdots,y^q)\in(L^{-2}\Zb)^q$, such that $|v_j|\leq M^{C_0}$ and $|y|\leq M^{C_0}$, where $C_0$ is a constant depending only on $d$. Moreover the system $(q,v_1,\cdots,v_q,y)$ \emph{represents} $k$, in the following sense:\begin{enumerate}
\item If $z=(z^1,\cdots,z^d)\in \Zb_L^d$, $|z|\leq M$ and $|\langle k,z\rangle|\leq M$, then $z$ is a linear combination of $\{v_1,\cdots,v_q\}$;
\item If $z=\gamma^1v_1+\cdots +\gamma^qv_q$, then we have $\langle k,z\rangle=y^1\gamma^1+\cdots +y^q\gamma^q$.
\end{enumerate} In particular, there exists a subset $\mathfrak{Rep}\subset\Zb_L^d$ with at most $M^{C_0}$ elements, such that for any $k\in\Zb_L^d$ there exists $k_*\in\mathfrak{Rep}$ and
\begin{multline}\label{represent}
\mathrm{For\ any\ }z\in\Zb_L^d,\,|z|\leq M,\mathrm{\ we\ have\ that\ }|\langle k,z\rangle|\leq M\Leftrightarrow|\langle k_*,z\rangle|\leq M;\\
\mathrm{in\ this\ case\ we\ also\ have\ }\langle k,z\rangle=\langle k_*,z\rangle.
\end{multline}
\end{lem}
\begin{proof} Upon multiplying everything by $L$, we may replace $\Zb_L$ by $\Zb$. Fix $k\in\Zb^d$, consider the set $H$ of $z\in\Zb^d$ such that $|z|\leq M$ and $|\langle k,z\rangle|\leq M$ (clearly $0\in H$). Let $q$ be the maximal number of linearly independent vectors in $H$, we may fix a maximum independent set $\{w_1,\cdots, w_q\}\subset H$, and apply Gram-Schmidt process to get orthogonal vectors $(v_1,\cdots,v_q)$. Since each $w_j\in\Zb^d$ and $|w_j|\leq M$, we can easily make $v_j\in\Zb^d$ and $|v_j|\leq M^{C_0}$. Moreover, since $|\langle k,w_j\rangle|\leq M$ for $1\leq j\leq q$, we know that also $|y|\leq M^{C_0}$ where $y=(y^1,\cdots,y^q)$ with $y^j=\langle k,v_j\rangle$. This means that $k$ is represented by the system $(q,v_1,\cdots,v_q,y)$, as desired. Clearly, if both $k$ and $k_*$ are represented by the same system $(q,v_1,\cdots,v_q,y)$, then (\ref{represent}) holds by definition; this completes the proof as the number of choices of $(q,v_1,\cdots,v_q,y)$ is at most $M^{C_0}$.
\end{proof}
\begin{rem} Lemma \ref{finitelem} is adapted to the current setting of square torus. A similar (and slightly more complicated) version was proved in \cite{DH21}, which is adapted to generic irrational tori. If we consider more general (partly rational and partly irrational) tori, then the corresponding result can be stated and proved by combining the arguments here and in \cite{DH21}. We omit the details.
\end{rem}

\begin{lem}\label{lincomblem}
Let $S$ be any finite set of vectors in $\Rb^d$. Then there exist $q\leq d$ vectors $X_1, \cdots X_q \in S$ such that any $Y\in S$ is a linear combination of $(X_k)_{1\leq k \leq q}$ with coefficients bounded by $2^{d-1}$.
\end{lem}

\begin{proof}
Let $X_1\in S$ be such that $|X_1|$ is maximal. For $k\geq 2$, pick $X_k\in S$ successively such that $|X_1\wedge X_2\cdots \wedge X_k|$ is positive and maximal, if such an $X_k$ exists. This process must terminate after at most $q\leq d$ steps, which results in a subset $\{X_1,\cdots,X_q\}$. Clearly they are linearly independent, and any $Y\in S$ can be written as a unique linear combination $Y=\lambda_1 X_1 +\cdots+\lambda_q X_q$.

Now we have $|X_1 \wedge \cdots \wedge X_{q-1}\wedge Y|=|\lambda_q|\cdot|X_1 \wedge \cdots \wedge X_{q-1}\wedge X_q|$; by the choice of $X_q$ we have $|\lambda_q|\leq 1$. Next consider $|X_1 \wedge \cdots \wedge X_{q-2}\wedge (Y-\lambda_q X_q)|=|\lambda_{q-1}||X_1 \wedge \cdots \wedge X_{q-2}\wedge X_{q-1}|$; since both $Y$ and $X_q$ belong to $S$, by the choice of $X_{q-1}$ we have $|\lambda_{q-1}|\leq |\lambda_q|+1$. In the same way we get that $|\lambda_k|\leq |\lambda_{k+1}|+\cdots+|\lambda_q|+1$ for all $1\leq k \leq q-1$, which gives the result.
\end{proof}
\subsection{Counting lemmas} Finally we recall some counting estimates from \cite{DH23}.
\begin{lem}\label{countlem} Fix $\alpha,\beta\in\Rb$ and $r,v\in\Zb_L^d$, such that $|r|\sim_0 P$ with $P\in[L^{-1},1]\cup\{0\}$ (with $|r|\gtrsim_0 1$ if $P=1$). Let each of $(x,y,z)$ belong to $\Zb_L^d$ intersecting a fixed unit ball, assume $\xi\in\{x,y,x+y\}$ and $\zeta\in\{x-v,x+y-v,x+y-v-z\}$, and define $\Xf:=\min((\log L)^2,1+\delta L^{2\gamma}P)$ as in Proposition \ref{ladderl1old}. Then we have the following counting bounds
\begin{align}
\label{basiccount01}\big\{(x,y):|x\cdot y-\alpha|\leq \delta^{-1}L^{-2\gamma},\,\, |r\cdot\xi-\beta|\leq \delta^{-1}L^{-2\gamma}\big\}&\lesssim_1\delta^{-1}L^{2(d-\gamma)}\Xf^{-1},\\\label{basiccount02}
\big\{(x,y,z):|x\cdot y-\alpha|\leq \delta^{-1}L^{-2\gamma},\,\, |\zeta\cdot z-\beta|\leq \delta^{-1}L^{-2\gamma}\big\}&\lesssim_1\delta^{-2}L^{3(d-\gamma)-\gamma_0}.
\end{align} 
\end{lem}
\begin{proof} This follows from Lemma A.2 of \cite{DH23} (the bounds (\ref{basiccount01})--(\ref{basiccount02}) needed here are much weaker than in \cite{DH23}).
\end{proof}
\begin{lem}\label{atomcountlem}
Let $r\in\{2,3\}$ and $\epsilon_j\in\{\pm1\}$, and $k_j\,(1\leq j\leq r)$ be in $\Zb_L^d$ intersecting a fixed unit ball, that solve the following system
\begin{equation}\label{atomsystem}
\sum_{j=1}^r\epsilon_jk_j=k,\quad \bigg|\sum_{j=1}^r\epsilon_j|k_j|^2-\beta\bigg|\leq\delta^{-1}L^{-2\gamma},
\end{equation} where $k\in\Zb_L^d$ and $\beta\in\Rb$ are fixed parameters. Let the number of solutions to this system be $\Ef_r$, then we have
\begin{equation}\label{atomcountbd}\Ef_2\lesssim_1 \delta^{-1}\min(L^d,L^{d+1-2\gamma}),\quad\,\,\Ef_2\lesssim_1\delta^{-1}L^{d-\gamma-\eta}\,\,\mathrm{if}\,\, |r|\geq L^{-\gamma+\eta};\quad\,\,\Ef_3\lesssim_1\delta^{-1}L^{2(d-\gamma)}.
\end{equation} Moreover, suppose we have four vector variables $k_j\,(1\leq j\leq 4)$ as above, that (\ref{atomsystem}) with $r=3$ holds for $(k_1,k_2,k_3)$, and that the same equation holds also for the three vectors $(k_1,k_2,k_4)$ with the same $\epsilon_j$ and possibly different $(k,\beta)$, then the number of solutions is bounded by
\begin{equation}\label{atomcountbd3}\Ef_4\lesssim_1\delta^{-1}L^{2(d-\gamma)}\Xf^{-1},
\end{equation} where $\Xf$ is defined as in Proposition \ref{ladderl1old} from $P$, and $P\sim_0|k_3-k_4|$.
\end{lem}
\begin{proof} This follows from (parts of) Lemma A.3 of \cite{DH23}.
\end{proof}
\section{Failure of a physical space $L^1$ bound}\label{physicalL1} Consider the reduced Fourier coefficients $a_k(t)$ in (\ref{reducedcoef}). In the Gibbs measure case of \cite{LS11}, these coefficients are not independent, but their cumulants satisfy the (discrete analog of) physical space $L^1$ bound, namely
\begin{equation}\label{physicall1-2}\int_{(\Tb_L^d)^3}\big|\Kb\big(\check{a}(0,t),\overline{\check{a}(x_1,t)},\check{a}(x_2,t),\overline{\check{a}(x_3,t)}\big)\big|\,\mathrm{d}x_1\mathrm{d}x_2\mathrm{d}x_3\leq C_0L^{-\gamma}\end{equation} for $t=0$ (and any $t>0$ due to invariance), where $\check{a}(x,t)$ is the inverse Fourier transform of $a$ as defined in (\ref{fourier}). The bound (\ref{physicall1-2}) is crucial in the derivation in \cite{LS11}, see the discussions in Section \ref{strategyintro}; however it is not true for $t>0$, in the off-equilibrium setting in this paper. To see this, just choose $t=1$, expand $a_k(1)$ into Taylor series consisting of multilinear expressions of $a_k(0)$, and take lowest order terms. By using (\ref{akeqn})--(\ref{akeqn2}) and Lemma \ref{propertycm}, we get that the cumulant $\Kb\big(a_{k_1}(1),\overline{a_{k_2}(1)},a_{k_3}(1),\overline{a_{k_4}(1)}\big)$ contains a leading contribution
\begin{multline}\label{physicall1-3}\Kb_0(k_1,k_2,k_3,k_4):=\frac{-\delta i}{2L^{d-\gamma}}\mathbf{1}_{k_1-k_2+k_3-k_4=0}\int_0^1 e^{-\pi i\cdot \delta L^{2\gamma}\Omega s}\,\mathrm{d}s\\\times \varphi_{\mathrm{in}}(k_1)\varphi_{\mathrm{in}}(k_2)\varphi_{\mathrm{in}}(k_3)\varphi_{\mathrm{in}}(k_4)\bigg(\frac{1}{\varphi_{\mathrm{in}}(k_1)}-\frac{1}{\varphi_{\mathrm{in}}(k_2)}+\frac{1}{\varphi_{\mathrm{in}}(k_3)}-\frac{1}{\varphi_{\mathrm{in}}(k_4)}\bigg),\end{multline} where
\[\Omega:=|k_1|^2-|k_2|^2+|k_3|^2-|k_4|^2\] and we have omitted the $\epsilon$ factor in (\ref{defcoef0}) (as (\ref{physicall1-3}) vanishes if $\epsilon\neq 1$). Note the other contributions contain either higher powers of $\delta$ or more negative powers of $L$, so we may ignore these other terms and replace $\Kb$ by $\Kb_0$.

Now, choose $k_j^*$ such that $k_1^*-k_2^*+k_2^*-k_4^*=|k_1^*|^2-|k_2^*|^2+|k_3^*|^2-|k_4^*|^2=0$, but
\[\frac{1}{\varphi_{\mathrm{in}}(k_1^*)}-\frac{1}{\varphi_{\mathrm{in}}(k_2^*)}+\frac{1}{\varphi_{\mathrm{in}}(k_3^*)}-\frac{1}{\varphi_{\mathrm{in}}(k_4^*)}\neq 0;\] this is possible due to the off-equilibrium setting. Using the definition of $\check{a}$, we can calculate
\begin{multline}\int_{(\Tb_L^d)^3}\Kb\big(\check{a}(0,t),\overline{\check{a}(x_1,t)},\check{a}(x_2,t),\overline{\check{a}(x_3,t)}\big)\cdot e^{-2\pi i(k_2^*\cdot x_1-k_3^*\cdot x_2+k_4^*\cdot x_3)}\,\mathrm{d}x_1\mathrm{d}x_2\mathrm{d}x_3\\=L^d\sum_{k_1}\Kb_0(k_1,k_2^*,k_3^*,k_4^*)=L^d\Kb_0(k_1^*,k_2^*,k_3^*,k_4^*)\sim L^\gamma,\end{multline} which means that (\ref{physicall1-2}) cannot be true.
\section{Schwartz solutions to (\ref{wke})}\label{schwartzwke} In this appendix we prove Proposition \ref{wkelwp}. Recall the equation (\ref{wke}) and its nonlinearity, which is given by the trilinear form (\ref{wke2}). We can decompose (\ref{wke2}) as 
\[\Kc(\varphi_1,\varphi_2,\varphi_3)=\sum_{j=0}^3(-1)^j\Kc_j(\varphi_1,\varphi_2,\varphi_3),\] where (note that $\Kc_1$ and $\Kc_3$ are essentially equivalent up to permutation of input functions)
\begin{multline}\label{schappend1}
\Kc_0(\varphi_1,\varphi_2,\varphi_3)(k)=\int_{(\Rb^d)^3}\varphi_1(k_1)\varphi_2(k_2)\varphi_3(k_3)\dirac(k_1-k_2+k_3-k)\\\times\dirac(|k_1|^2-|k_2|^2+|k_3|^2-|k|^2)\,\mathrm{d}k_1\mathrm{d}k_2\mathrm{d}k_3,
\end{multline}
\begin{multline}\label{schappend2}
\Kc_1(\varphi_1,\varphi_2,\varphi_3)(k)=\int_{(\Rb^d)^3}\varphi_1(k)\varphi_2(k_2)\varphi_3(k_3)\dirac(k_1-k_2+k_3-k)\\\times\dirac(|k_1|^2-|k_2|^2+|k_3|^2-|k|^2)\,\mathrm{d}k_1\mathrm{d}k_2\mathrm{d}k_3,
\end{multline}
\begin{multline}\label{schappend3}
\Kc_2(\varphi_1,\varphi_2,\varphi_3)(k)=\int_{(\Rb^d)^3}\varphi_1(k_1)\varphi_2(k)\varphi_3(k_3)\dirac(k_1-k_2+k_3-k)\\\times\dirac(|k_1|^2-|k_2|^2+|k_3|^2-|k|^2)\,\mathrm{d}k_1\mathrm{d}k_2\mathrm{d}k_3,
\end{multline}
\begin{multline}\label{schappend4}
\Kc_3(\varphi_1,\varphi_2,\varphi_3)(k)=\int_{(\Rb^d)^3}\varphi_1(k_1)\varphi_2(k_2)\varphi_3(k)\dirac(k_1-k_2+k_3-k)\\\times\dirac(|k_1|^2-|k_2|^2+|k_3|^2-|k|^2)\,\mathrm{d}k_1\mathrm{d}k_2\mathrm{d}k_3.
\end{multline}

Clearly each $\Kc_j$ is a positive trilinear operator. Next, in (\ref{schappend1}), by making the change of variables $k_1=k+\ell_1$ and $k_3=k+\ell_3$ (so $k_2=k+\ell_1+\ell_3$) we can rewrite
\begin{equation}\label{schappend5}
\Kc_0(\varphi_1,\varphi_2,\varphi_3)(k)=\int_{(\Rb^d)^2}\varphi_1(k+\ell_1)\varphi_2(k+\ell_1+\ell_3)\varphi_3(k+\ell_3)\dirac(-2\langle \ell_1,\ell_3\rangle)\,\mathrm{d}\ell_1\mathrm{d}\ell_3,
\end{equation} noticing also that $|k_1|^2-|k_2|^2+|k_3|^2-|k|^2=-2\langle \ell_1,\ell_3\rangle$. By Leibniz rule it then follows that
\begin{equation}\label{schappend6}\nabla\Kc_0(\varphi_1,\varphi_2,\varphi_3)=\Kc_0(\nabla\varphi_1,\varphi_2,\varphi_3)+\Kc_0(\varphi_1,\nabla\varphi_2,\varphi_3)+\Kc_0(\varphi_1,\varphi_2,\nabla\varphi_3),
\end{equation}
where $\nabla$ refers to $\nabla_k$, and clearly the same is true for each $\Kc_j\,(1\leq j\leq 3)$. Moreover, the integral in (\ref{schappend1}) is supported in $k_1-k_2+k_3-k=0$, so similarly we get
\begin{equation}\label{schappend7}
k^i\cdot\Kc_0(\varphi_1,\varphi_2,\varphi_3)=\Kc_0((k_1)^i\varphi_1,\varphi_2,\varphi_3)-\Kc_0(\varphi_1,(k_2)^i\varphi_2,\varphi_3)+\Kc_0(\varphi_1,\varphi_2,(k_3)^i\varphi_3),
\end{equation} where $k^i$ represents any coordinate for $1\leq i\leq d$. For other $\Kc_j\,(1\leq j\leq 3)$ a simpler version of (\ref{schappend7}) holds, with only one term on the right hand side that involves $k^i\cdot\varphi_j$, and the other two terms has coefficient $0$. Finally, both (\ref{schappend6}) and (\ref{schappend7}) can be iterated to get corresponding equalities for multi-indices $\alpha$.

\medskip
Now we can start the proof of Proposition \ref{wkelwp}. Fix any $s>d/2-1$ and define
\[\|\varphi\|_{Z_s}:=\|\langle k\rangle^s\varphi\|_{L^2};\quad \|\varphi\|_{Z_{s,m}}:=\sup_{|\alpha|+|\beta|\leq m}\|k^{\alpha}\nabla^{\beta} \varphi\|_{Z_s}\] for nonnegative integers $m$. It is proved in \cite{GIT20}, see Lemmas 4.1--4.3 of \cite{GIT20} (also Remark 2.6 of \cite{GIT20}), that each $\Kc_j$ satisfies the trilinear estimate
\begin{equation}\label{schappend8}\|\Kc_j(\varphi_1,\varphi_2,\varphi_3)\|_{Z_s}\leq C_0(s)\|\varphi_1\|_{Z_s}\cdot\|\varphi_2\|_{Z_s}\cdot\|\varphi_3\|_{Z_s}\end{equation} for some constant $C_0(s)$ depending on $s$. It then follows that (\ref{wke}) is locally well-posed in $Z_s$.

The main point in our proof is to establish the following: for any $T>0$ and any $s$ and $m$, if the solution $\varphi$ to (\ref{wke}) exists in $Z_s$ space until and including time $T$ and \begin{equation}\label{schappend9}\sup_{\tau\in[0,T]}\|\varphi(\tau)\|_{Z_s}\leq A,\end{equation} and if the initial data satisfies $\|\varphi_{\mathrm{in}}\|_{Z_{s,m}}\leq B$, and the solution $\varphi(\tau)\in Z_{s,m}$ for all $\tau\in[0,T]$, then we have the a priori estimate \begin{equation}\label{schappend10}\sup_{\tau\in[0,T]}\|\varphi(\tau)\|_{Z_{s,m}}\leq C(s,m,T,A,B)\end{equation} for some constant $C=C(s,m,T,A,B)$.

To prove (\ref{schappend10}) we induct on $m$; the case $m=0$ is known. Suppose (\ref{schappend10}) is true for smaller $m$, then to control the $Z_{s,m}$ norm of $\varphi(\tau)$, we only need to consider multi-indices $\alpha,\beta$ such that $|\alpha|+|\beta|=m>0$. By iterating (\ref{schappend6}) and (\ref{schappend7}), and applying Duhamel formula, we get
\begin{equation}\label{schappend11}
k^\alpha\nabla^\beta (\varphi(\tau)-\varphi_{\mathrm{in}})=\sum_{j=0}^3\sum_{\substack{\alpha_1+\alpha_2+\alpha_3=\alpha\\\beta_1+\beta_2+\beta_3=\beta}}\iota\cdot\int_0^{\tau}\Kc_j\big(k^{\alpha_1}\nabla^{\beta_1}\varphi(\tau'),k^{\alpha_2}\nabla^{\beta_2}\varphi(\tau'),k^{\alpha_3}\nabla^{\beta_3}\varphi(\tau')\big)\,\mathrm{d}\tau',
\end{equation} where $\iota\in\{-1,0,1\}$ is a coefficient  depending on $j$ and $\alpha_i,\beta_i$ etc. In each of these terms, either $\alpha_i+\beta_i<m$ for all $i$, or $\alpha_i=\alpha$ and $\beta_i=\beta$ for exactly one $i$, and $\alpha_{i'}=\beta_{i'}=0$ for all other $i'$. Now we apply (\ref{schappend8}) to (\ref{schappend11}), and use the induction hypothesis, to get
\[\|k^\alpha\nabla^\beta\varphi(\tau)\|_{Z_s}\leq B+4(m+1)^{6d}C_0(s)T\cdot C(s,m-1,T,A,B)^3+12A^2C_0(s)\cdot\int_0^{\tau}\|k^\alpha\nabla^\beta\varphi(\tau')\|_{Z_s}\,\mathrm{d}\tau'.\] By Gronwall this then implies (\ref{schappend10}) with \begin{multline*}C(s,m,T,A,B)=C(s,m-1,T,A,B)+Be^{12A^2TC_0(s)}\\+4T(m+1)^{6d}C_0(s)e^{12A^2TC_0(s)}\cdot C(s,m-1,T,A,B)^3.\end{multline*}

Now we can finish the proof of Proposition \ref{wkelwp}. For $s>d/2-1$ and $m$, consider the maximal time $\tau_{\mathrm{max}}^{(s,m)}$ such that (\ref{wke}) has a (unique) solution in $C_t ([0,\tau_{\mathrm{max}}^{(s,m)})\to Z_{s,m})$. Since $\varphi_{\mathrm{in}}\in\Sc$ and (\ref{wke}) is locally well-posed in $Z_{s,m}$ (because the trilinear estimate (\ref{schappend8}) holds for $Z_{s,m}$ with constant $C_0=C_0(s,m)$, which follows from the same proof as above), we know $\tau_{\mathrm{max}}^{(s,m)}>0$. Moreover, if $\tau_{\mathrm{max}}^{(s,m)}<\tau_{\mathrm{max}}^{(s,0)}$, then (\ref{schappend10}) implies that $\varphi(\tau)$ is \emph{uniformly bounded} in $Z_{s,m}$ for all $\tau\in[0,\tau_{\mathrm{max}}^{(s,m)})$, so by applying local well-posedness in $Z_{s,m}$ to $\varphi(\tau_{\mathrm{max}}^{(s,m)}-\epsilon)$ for sufficiently small $\epsilon$, we can construct the solution to (\ref{wke}) that belongs to $Z_{s,m}$ up to time $\tau_{\mathrm{max}}^{(s,m)}+\epsilon$, which is impossible.

Now we know $\tau_{\mathrm{max}}^{(s,m)}=\tau_{\mathrm{max}}^{(s,0)}$. But for any $s_1<s_2$, we can find $m$ such that the $Z_{s_1,m}$ norm is stronger than $Z_{s_2,0}$ norm, thus $\tau_{\mathrm{max}}^{(s_2,0)}\leq \tau_{\mathrm{max}}^{(s_1,0)}=\tau_{\mathrm{max}}^{(s_1,m)}\leq \tau_{\mathrm{max}}^{(s_2,0)}$. Therefore the values of $\tau_{\mathrm{max}}^{(s,m)}$ are equal for all $(s,m)$, and define this to be $\tau_{\mathrm{max}}$. Then for $\tau\in[0,\tau_{\mathrm{max}})$ the solution $\varphi(\tau)$ belongs to all $Z_{s,m}$ spaces, and hence is Schwartz; moreover if $\tau_{\mathrm{max}}<\infty$ then the blowup criterion (\ref{blowuptest}) has to hold, since otherwise we can extend the solution in $Z_{s,0}$. This proves Proposition \ref{wkelwp}.
\section{Table of notations}\label{appendtable} In this appendix we collect some important notations used in this paper. Table \ref{table1} contains notations about trees, couples and gardens. Table \ref{table2} contains notations about molecules. Table \ref{table3} contains notations about layerings and others, which are specific to this paper.

\medskip
\begin{center}
\NiceMatrixOptions
{
    custom-line =
    {letter = I ,
    tikz = {line width=0.08em},
    width=0.08em
    },  
}
\begin{NiceTabular}{Il c lI}
\toprule
Concept           &  Symbol                & Where defined \\
\midrule
Tree              &$\Tc$           & Definition \ref{deftree}\\
Root, node, leaf  &$\rf,\nf,\lf$ &Definition \ref{deftree}\\
Couple, garden & $\Qc,\Gc$&Definition \ref{defgarden}\\
Paired tree & --- & Definition \ref{defgarden}\\
Lone leaf & $\lf_{\mathrm{lo}}$ & Definition \ref{defgarden}\\
Leaf set & $\Lc$ & Definitions \ref{deftree}, \ref{defgarden}\\
Branching node set&$\Nc$ &Definitions \ref{deftree}, \ref{defgarden}\\
Order & $n(\Tc)$ etc. & Definitions \ref{deftree}, \ref{defgarden}\\
Decoration (trees etc.) & $\Ds,\Is$ & Definition \ref{defdec}\\
Regular couple, regular tree & --- & Definition \ref{defreg}\\
Dominant couple &--- &Proposition \ref{propstructure}\\
Prime couple& --- & Proposition \ref{propstructure2}\\
Skeleton& $\Gc_{\mathrm{sk}}$ & Proposition \ref{propstructure2}\\
\bottomrule
\end{NiceTabular}
\captionof{table}{Notations about trees, couples and gardens.}\label{table1}

\medskip
\begin{NiceTabular}{Il c lI}
\toprule
Concept           &  Symbol                & Where defined \\
\midrule
Molecule, atom, bond & $\Mb,v,\ell$ & Definition \ref{defmol}\\
Circular rank & $\chi$ & Definition \ref{defmol}\\
Molecule from garden &$\Mb(\Gc)$ & Definition \ref{defcplmol}\\
Decoration (molecules) & --- &Definition \ref{defdecmol}\\
Block, realization & $\Bb,\Gc[\Bb]$ & Definition \ref{defblock}, Remark \ref{rem_realiz}\\
(CL) and (CN) blocks, splicing & --- & Proposition \ref{block_clcn}\\
Vines (I)--(VIII), ladder & $\Vb,\Lb$ &Definition \ref{defvine}\\
HV, VC, HVC, DV, vine-like object &$\Ub$ & Definition \ref{defvine}, Lemma \ref{vinechainlem}\\
Vines (I-a), (II-b) etc. & --- & Proposition \ref{molecpl}\\
Twist & --- & Definition \ref{twist}\\
Gap, SG, LG, ZG &$r$ &Definition \ref{defdiff}\\
(CL) vines set, result after splicing & $\Vs,\Vs_0,\Gc_{\mathrm{sb}}$ & Definition \ref{defcong}, Section \ref{redsplice}\\
Degenerate and tame atoms &--- & Proposition \ref{subpro} (3)\\
Cut, $\alpha$- and $\beta$-atoms &---& Definition \ref{defcut}\\
\bottomrule
\end{NiceTabular}
\captionof{table}{Notations about molecules.}\label{table2}

\medskip
\begin{NiceTabular}{Il c lI}
\toprule
Concept           &  Symbol                & Where defined \\
\midrule
Cumulant & $\Kb$ & Definition \ref{defcm}\\
Layering, pre-layering & $\Lf[\Gc]$ & Definition \ref{deflayer}\\
Canonical layered garden &$\Gs_{p}$ & Definition \ref{defcanon}\\
Canonical layered couple &$\Cs_{p}$ & Definition \ref{defcanon}\\
Expression for layered objects & $\Kc_\Gc$ etc. & Definition \ref{defkg}\\
Coherent regular objects &--- & Definition \ref{defcoh}\\
Incoherency index (regular objects) &$g$ & Definition \ref{defcoh}\\
Layered full (LF) twist &--- &Definition \ref{lftwist}\\
Coherent vines/ladders &--- & Definition \ref{cohmol}\\
Incoherency index (vines/ladders) &$g$ & Definition \ref{cohmol}\\
Expression for (CL) vine chains & $\Kc^\Ub$, $\widetilde{\Kc}^\Ub$ etc. & Proposition \ref{vineest} \\
\bottomrule
\end{NiceTabular}
\end{center}
\captionof{table}{Notations about layerings and others.}\label{table3}

\end{document}